\documentclass[reqno, 10pt]{amsart}
\newcommand{\VE}{ }
\newcommand{\tysup}[1]{^{\scriptscriptstyle #1}}

  \usepackage[margin=1.3in]{geometry}

\usepackage[usenames,dvipsnames]{color}

\usepackage{amsthm,amsfonts,amssymb,amsmath,amsxtra}%,fdsymbol}
\usepackage[all]{xy}
\SelectTips{cm}{}
\usepackage{xr-hyper}
\usepackage[colorlinks=
   citecolor=Black,
   linkcolor=Red,
   urlcolor=Blue%, backref=page
]{hyperref}
\usepackage{verbatim}
 \usepackage{cancel}

\usepackage{mathrsfs}

\RequirePackage{xspace}
\RequirePackage{etoolbox}
\RequirePackage{varwidth}
\RequirePackage{enumitem}
\RequirePackage{tensor}
\RequirePackage{mathtools}
\RequirePackage{longtable}
\RequirePackage{multirow}

\usepackage{scalerel}

\newcommand\reallywidehat[1]{\arraycolsep=0pt\relax%
\begin{array}{c}
\stretchto{
  \scaleto{
    \scalerel*[\widthof{\ensuremath{#1}}]{\kern-.5pt\bigwedge\kern-.5pt}
    {\rule[-\textheight/2]{1ex}{\textheight}} %WIDTH-LIMITED BIG WEDGE
  }{\textheight} % 
}{0.8ex}\\           % THIS SQUEEZES THE WEDGE TO 0.8ex HEIGHT
#1\\                 % THIS STACKS THE WEDGE ATOP THE ARGUMENT
\rule{-1ex}{0ex}
\end{array}
}

\newcommand{\sL}{\ensuremath{\mathscr{L}}\xspace}

\newcommand{\fkp}{\ensuremath{\mathfrak{p}}\xspace}

\newcommand{\fks}{\ensuremath{\mathfrak{s}}\xspace}
\newcommand{\fkt}{\ensuremath{\mathfrak{t}}\xspace}

\newcommand{\fkM}{\ensuremath{\mathfrak{M}}\xspace}

\newcommand{\diam}{{\diamond}}

\newcommand{\BA}{\ensuremath{\mathbb {A}}\xspace}

\newcommand{\BC}{\ensuremath{\mathbb {C}}\xspace}

\newcommand{\BF}{\ensuremath{\mathbb {F}}\xspace}
\newcommand{\BG}{\ensuremath{\mathbb {G}}\xspace}

\newcommand{\BP}{\ensuremath{\mathbb {P}}\xspace}
\newcommand{\BQ}{\ensuremath{\mathbb {Q}}\xspace}
\newcommand{\BR}{\ensuremath{\mathbb {R}}\xspace}
\newcommand{\BS}{\ensuremath{\mathbb {S}}\xspace}

\newcommand{\BX}{\ensuremath{\mathbb {X}}\xspace}

\newcommand{\BZ}{\ensuremath{\mathbb {Z}}\xspace}

\newcommand{\CA}{\ensuremath{\mathcal {A}}\xspace}
\newcommand{\CB}{\ensuremath{\mathcal {B}}\xspace}

\newcommand{\CE}{\ensuremath{\mathcal {E}}\xspace}
\newcommand{\CF}{\ensuremath{\mathcal {F}}\xspace}
\newcommand{\CG}{\ensuremath{\mathcal {G}}\xspace}
\newcommand{\CH}{\ensuremath{\mathcal {H}}\xspace}
\newcommand{\CI}{\ensuremath{\mathcal {I}}\xspace}

\newcommand{\CL}{\ensuremath{\mathcal {L}}\xspace}
\newcommand{\CM}{\ensuremath{\mathcal {M}}\xspace}
\newcommand{\CN}{\ensuremath{\mathcal {N}}\xspace}
\newcommand{\CO}{\ensuremath{\mathcal {O}}\xspace}
\newcommand{\CP}{\ensuremath{\mathcal {P}}\xspace}
\newcommand{\CQ}{\ensuremath{\mathcal {Q}}\xspace}

\newcommand{\CS}{\ensuremath{\mathcal {S}}\xspace}
\newcommand{\CT}{\ensuremath{\mathcal {T}}\xspace}
\newcommand{\CU}{\ensuremath{\mathcal {U}}\xspace}
\newcommand{\CV}{\ensuremath{\mathcal {V}}\xspace}

\newcommand{\RB}{\ensuremath{\mathrm {B}}\xspace}

\newcommand{\RD}{\ensuremath{\mathrm {D}}\xspace}
\newcommand{\RE}{\ensuremath{\mathrm {E}}\xspace}
\newcommand{\RF}{\ensuremath{\mathrm {F}}\xspace}
\newcommand{\RG}{\ensuremath{\mathrm {G}}\xspace}
\newcommand{\RH}{\ensuremath{\mathrm {H}}\xspace}

\newcommand{\RK}{\ensuremath{\mathrm {K}}\xspace}

\newcommand{\RM}{\ensuremath{\mathrm {M}}\xspace}

\newcommand{\RO}{\ensuremath{\mathrm {O}}\xspace}
\newcommand{\RP}{\ensuremath{\mathrm {P}}\xspace}

\newcommand{\RU}{\ensuremath{\mathrm {U}}\xspace}
\newcommand{\RV}{\ensuremath{\mathrm {V}}\xspace}

\newcommand{\RX}{\ensuremath{\mathrm {X}}\xspace}

\newcommand{\ab}{{\mathrm{ab}}}

\newcommand{\ad}{{\mathrm{ad}}}

\DeclareMathOperator{\Aut}{Aut}

\DeclareMathOperator{\diag}{diag}

\renewcommand{\div}{{\mathrm{div}}}

\DeclareMathOperator{\End}{End}

\DeclareMathOperator{\Gal}{Gal}
\newcommand{\GL}{\mathrm{GL}}

\newcommand{\GO}{\mathrm{GO}}
\newcommand{\GSpin}{\mathrm{GSpin}}
\newcommand{\GSp}{\mathrm{GSp}}
\newcommand{\GU}{\mathrm{GU}}

\newcommand{\id}{\ensuremath{\mathrm{id}}\xspace}

\newcommand{\inv}{{\mathrm{inv}}}

\DeclareMathOperator{\Lie}{Lie}

\DeclareMathOperator{\Nm}{Nm}

\DeclareMathOperator{\rank}{rank}

\newcommand{\PGL}{{\mathrm{PGL}}}
\DeclareMathOperator{\Pic}{Pic}

\newcommand{\red}{\ensuremath{\mathrm{red}}\xspace}

\DeclareMathOperator{\Res}{Res}

\DeclareMathOperator{\Spec}{Spec\,}
\DeclareMathOperator{\Spd}{Spd\,}
\DeclareMathOperator{\Spf}{Spf\,}

\DeclareMathOperator{\tr}{tr}

\newcommand{\U}{\mathrm{U}}

\newcommand{\wt}{\widetilde}
\newcommand{\wh}{\widehat}

\newcommand{\ov}{\overline}

\newcommand{\bs}{\backslash}

\newcommand{\lps}{[\![}
\newcommand{\rps}{]\!]}
\newcommand{\llps}{(\!(}
\newcommand{\lrps}{)\!)}

\newtheorem{theorem}[subsubsection]{Theorem}
\newtheorem{proposition}[subsubsection]{Proposition}
\newtheorem{lemma}[subsubsection]{Lemma}
\newtheorem {conjecture}[subsubsection]{Conjecture}
\newtheorem{corollary}[subsubsection]{Corollary}

\theoremstyle{definition}
\newtheorem{definition}[subsubsection]{Definition}
 \newtheorem{example}[subsubsection]{Example}
\newtheorem{remark}[subsubsection]{Remark}
\newenvironment{altenumerate}
   {\begin{list}
      {\textup{(\theenumi)} }
      {\usecounter{enumi}
       \setlength{\labelwidth}{0pt}
       \setlength{\labelsep}{0pt}
       \setlength{\leftmargin}{0pt}
       \setlength{\itemsep}{\the\smallskipamount}
       \renewcommand{\theenumi}{\roman{enumi}}
      }}
   {\end{list}}
\newenvironment{altitemize}
   {\begin{list}
      {$\bullet$}
      {\setlength{\labelwidth}{0pt}
	   \setlength{\itemindent}{5pt}
       \setlength{\labelsep}{5pt}
       \setlength{\leftmargin}{0pt}
       \setlength{\itemsep}{\the\smallskipamount}
      }}
   {\end{list}}

\numberwithin{equation}{subsection}
\numberwithin{subsubsection}{subsection}

\setitemize[0]{leftmargin=0.3in,itemsep=\the\smallskipamount}
\setenumerate[0]{leftmargin=0.3in,itemsep=\the\smallskipamount}

\renewcommand{\to}{%
   \ifbool{@display}{\longrightarrow}{\rightarrow}%
   }
\let\shortmapsto\mapsto
\renewcommand{\mapsto}{%
   \ifbool{@display}{\shortmapsto}{\shortmapsto}%
   }
\newlength{\olen}
\newlength{\ulen}
\newlength{\xlen}
\newcommand{\xra}[2][]{%
   \ifbool{@display}%
      {\settowidth{\olen}{$\overset{#2}{\longrightarrow}$}%
       \settowidth{\ulen}{$\underset{#1}{\longrightarrow}$}%
       \settowidth{\xlen}{$\xrightarrow[#1]{#2}$}%
       \ifdimgreater{\olen}{\xlen}%
          {\underset{#1}{\overset{#2}{\longrightarrow}}}%
          {\ifdimgreater{\ulen}{\xlen}%
             {\underset{#1}{\overset{#2}{\longrightarrow}}}
             {\xrightarrow[#1]{#2}}}}%
      {\xrightarrow[#1]{#2}}
   }
\makeatother
\newcommand{\xyra}[2][]{%
   \settowidth{\xlen}{$\xrightarrow[#1]{#2}$}%
   \ifbool{@display}%
      {\settowidth{\olen}{$\overset{#2}{\longrightarrow}$}%
       \settowidth{\ulen}{$\underset{#1}{\longrightarrow}$}%
       \ifdimgreater{\olen}{\xlen}%
          {\mathrel{\xymatrix@M=.12ex@C=3.2ex{\ar[r]^-{#2}_-{#1} &}}}%
          {\ifdimgreater{\ulen}{\xlen}%
             {\mathrel{\xymatrix@M=.12ex@C=3.2ex{\ar[r]^-{#2}_-{#1} &}}}
             {\mathrel{\xymatrix@M=.12ex@C=\the\xlen{\ar[r]^-{#2}_-{#1} &}}}}}%
      {\mathrel{\xymatrix@M=.12ex@C=\the\xlen{\ar[r]^-{#2}_-{#1} &}}}%
   }
\makeatletter
\newcommand{\xla}[2][]{%
   \ifbool{@display}%
      {\settowidth{\olen}{$\overset{#2}{\longleftarrow}$}%
       \settowidth{\ulen}{$\underset{#1}{\longleftarrow}$}%
       \settowidth{\xlen}{$\xleftarrow[#1]{#2}$}%
       \ifdimgreater{\olen}{\xlen}%
          {\underset{#1}{\overset{#2}{\longleftarrow}}}%
          {\ifdimgreater{\ulen}{\xlen}%
             {\underset{#1}{\overset{#2}{\longleftarrow}}}
             {\xleftarrow[#1]{#2}}}}%
      {\xleftarrow[#1]{#2}}
   }
\newcommand{\isoarrow}{%
   \ifbool{@display}{\overset{\sim}{\longrightarrow}}{\xrightarrow\sim}%
   }

\newcommand{\quash}[1]{}

 \usepackage{relsize} 
\usepackage[bbgreekl]{mathbbol} 
\usepackage{amsfonts} 
\DeclareSymbolFontAlphabet{\mathbb}{AMSb} %to ensure that the meaning of \mathbb does not change
\DeclareSymbolFontAlphabet{\mathbbl}{bbold}

\newcommand{\und}{\underline}

 \newcommand{\br}{\breve}

\newcommand{\Adm}{{\rm {Adm}}}

 \newcommand{\Mloc}{{\rm M}}

\newcommand{\ti}{\tilde}
\newcommand{\der}{{\rm der}}

\newcommand{\bslash}{\symbol{92}}

\newcommand{\Mlocroot}{{\mathrm M}^{\sqrt{S}}}

\newcommand{\Mlocrooto}{{\mathrm M}^{\sqrt{\phantom{a}}}}

\newcommand{\ga}{\gamma}
\newcommand{\de}{\delta}

\newcommand{\LBS}{\, {\bs}_{\vphantom{A}{L}}\, }

\newcommand{\Rcom}{\wh{R[u]}_{(u(u-p))}}

\begin{document}
%------------------------------------------------------

\setcounter{tocdepth}{1}
%------------------------------------------------------

\title[Toric schemes and integral models for Shimura varieties]{Toric schemes and integral models for Shimura varieties with $\Gamma_1(p)$-type level}
\author[G. Pappas]{Georgios Pappas}
\address{Department of Mathematics, Michigan State University, E. Lansing, MI 48824, USA}
\email{pappasg@msu.edu}

\author[M. Rapoport]{Michael Rapoport}
\address{Mathematisches Institut der Universit\"at Bonn, Endenicher Allee 60, 53115 Bonn, Germany}
\email{rapoport@math.uni-bonn.de}

\date{\today}

\begin{abstract} We propose a conjectural theory of $p$-integral models of Shimura varieties with level structure at $p$ given by a class of normal subgroups of parahoric subgroups with abelian quotient group. The role of the theory of local models  is played in this context by a certain \emph{root stack} over the local model for parahoric level. The construction of this root stack is based on the \emph{divisor theorem} (a foundational fact about local models) and on the theory of toric varieties in this context, both of which are of independent interest. We prove our conjecture in the case of Shimura varieties of PEL type when the parahoric is an Iwahori (under some additional conditions). 
\end{abstract}

\maketitle

\tableofcontents

\date{\today}
\maketitle

\section{Introduction} 

\subsection{}
In the theory of Shimura varieties, the construction and investigation of $p$-adic integral models plays a prominent role. Here $p$ is a prime number that will be fixed throughout the paper. There is one class of level structure where we have, at least conjecturally, a systematic picture: when the $p$-component of the level subgroup is parahoric. In this case, we expect to have a \emph{canonical} integral model \cite{PRg} which is flat over $\BZ_p$ and which should come with  a \emph{local model diagram} that relates the local structure of the integral model of the Shimura variety with the local structure of the corresponding \emph{local model}. This expectation has now been realized for 
essentially all Shimura varieties of abelian type, see \cite{KPZ} and the references there for some of the most recent results. The local model is a projective scheme over the localization at the distinguished place of the ring of integers of the reflex field. The local model has been investigated thoroughly in the last 30 years by several authors: See \cite{PRS} and the references there for the earlier history of the subject, and \cite{PZ},   \cite{SWberkeley}, \cite{HRi}, \cite{AGLR}, for more recent developments. It is known that it is normal and, provided $p$ is odd, Cohen-Macaulay. Furthermore, for Shimura varieties of PEL type (i.e., related to moduli schemes of abelian varieties with endomorphisms and polarization), there is a moduli-theoretic description of the corresponding integral models, at least in ``unramified'' cases.

Beyond the parahoric case, the situation is less clear. Inspired by the dichotomy in the representation theory of $p$-adic Lie groups where representations of depth zero have behavior completely different from representations of positive depth, one may hope that a class of level structure for which one may expect a systematic theory of integral models is given by open compact subgroups that contain the pro-unipotent radical of a parahoric subgroup. The present paper takes a small step in this direction.

The goal of this paper is to develop a theory  of $p$-integral models of Shimura varieties when the $p$-component of the level subgroup is a normal subgroup of a parahoric with ``abelian  reductive quotient''. In general,  if $K=\CG(\BZ_p)$ is parahoric over $\BZ_p$, we set  $K_1\subset K$ for the kernel of
\[
 \CG(\BZ_p)\to \bar\CG(\BF_p)\to (\bar\CG_{\red})_{\rm ab}(\BF_p) ,
\]
where $(\bar\CG_{\red})_\ab$ is the maximal torus quotient of the maximal reductive quotient of the  reduction modulo $p$ of $\CG$. More generally, we can take  as level any subgroup $K'\subset K$ which contains $K_1$ and which is the inverse image of $Q(\BF_p)$ for a subtorus $Q$ of $(\bar\CG_{\red})_\ab$. We set $T=(\bar\CG_{\red})_{\rm ab}/Q$. Let $\RK=K^p K\subset G(\BA_f)$ be an open compact subgroup, where $K^p\subset G(\BA_f^p)$ is sufficiently small, and set $\RK'=K^pK'$.  There is a corresponding \'etale $T(\BF_p)$-cover of Shimura varieties 
\[
{\rm Sh}_{\RK'}(\RG,\RX)\to {\rm Sh}_{\RK}(\RG,\RX)
\]
 defined over the reflex field $\RE$. Let $E=\RE_\nu$ be a $p$-adic completion of $\RE$. We denote by $\CS_\RK$ the $O_E$-integral model for the parahoric $K$ mentioned above. In other words, let us assume that $\CS_\RK$ exists, with its corresponding   local model $\Mloc_{\CG,\mu}$ and its local model diagram.  In the present paper, we are interested in extending the cover ${\rm Sh}_{\RK'}(\RG,\RX)\to {\rm Sh}_{\RK}(\RG,\RX)$ to a $T(\BF_p)$-cover 
 \[
 \CS_{\RK'}\to \CS_\RK
 \]
  where the (\'etale) local structure is modeled on a kind of local model. In addition, in the case of Shimura varieties of PEL type, when often $\CS_\RK$ has a moduli description, we would like our models $\CS_{\RK'}$ also to have some type of a moduli description. For simplicity, we will mostly discuss the case $K'=K_1$, $\RK'=\RK_1=K^pK_1$, here and in the rest of the paper. Before we describe our method, let us discuss the case of the modular curve. 

\subsection{}
In the case of $\GL_2$ and the standard parahoric $K=\Gamma_0(p)$, the integral model  $\CS_\RK$ is the moduli space over $\BZ_p$ of triples $(E, \alpha, G)$, where $(E, \alpha)$ is an elliptic curve with $K^p$-level structure and where $G\subset E$ is a subgroup scheme of order $p$. In this case, the local model $\Mloc_{\CG,\mu}$ is the blow-up of $\BP^1_{\BZ_p}$ in the origin  of its special fiber. It represents the functor which associates to a $\BZ_p$-algebra $R$ the set of commutative diagrams
\begin{equation}\label{loG2}
\begin{aligned}
\xymatrix{
R \oplus R \ar[r]^{\begin{tiny} \begin{pmatrix} p & 0 \\ 0 & 1 \end{pmatrix} \end{tiny}} & R \oplus R \ar[r]^{\begin{tiny} \begin{pmatrix} 1 & 0 \\ 0 & p \end{pmatrix} \end{tiny}} & R \oplus R \\
\CF_0 \ar[u] \ar[r] & \CF_1 \ar[u] \ar[r] & \CF_0, \ar[u] 
}
\end{aligned}
\end{equation}
where the vertical arrows are the inclusions of  locally direct summands of rank one.
The subgroup $K_1$ is given as
\[
K_1=\left\{\begin{pmatrix}1 &*\\ 0& 1\end{pmatrix}\,{\rm mod}\ p  \right\}\subset \GL_2(\BZ_p). 
\]
The integral model $\CS_{\RK_1}$ is the moduli space over $\BZ_p$ of tuples $(E, \alpha, G, P_1, P_2)$, where $(E, \alpha, G)$ is as before and where $P_1$ is a \emph{generator} of $G$ and $P_2$ is a generator of $E[p]/G$. Here the notion of a generator of a group scheme of order $p$ is not the naive one (for instance $0$ can be a generator). It can be defined in various ways (using Oort--Tate theory, or via the Katz-Mazur notion of \emph{full set of sections} or via the Kottwitz-Wake notion of \emph{primitive elements}).  Using Oort--Tate theory, one can prove that \'etale-locally around a supersingular point, the map $\CS_{\RK_1}\to \CS_\RK$ is given by
\[
\Spec(\BZ_p[u, v]/(u^{p-1}v^{p-1}-p))\to \Spec(\BZ_p[x, y]/(xy-p)), \quad x\mapsto u^{p-1},\, y\mapsto v^{p-1} .
\]
In particular, the scheme $\CS_{\RK_1}$ is normal and the covering $\CS_{\RK_1}\to \CS_\RK$ is finite and flat. Incidentally, instead of $K_1$ it is common to consider instead the subgroup $K'$ between $K_1$ and $K$ given by 
\[
K'=\left\{\begin{pmatrix}1 &*\\ 0& *\end{pmatrix}\,{\rm mod}\ p  \right\}\subset \GL_2(\BZ_p),
\]
which also falls in our framework. Then $\CS_{\RK'}$ parametrizes quadruples $(E, \alpha, G, P)$, where $(E, \alpha, G)$ is as before and where $P$ is a generator of $G$. Again, we obtain a normal scheme which is finite flat over $\CS_\RK$ and which is locally of the form
\[
\Spec(\BZ_p[u, v]/(u^{p-1}v-p))\to \Spec(\BZ_p[x, y]/(xy-p)), \quad x\mapsto u^{p-1},\, y\mapsto v .
\]
What is remarkable in both cases is that the covering map is given by monomials in the parameters in the local ring upstairs. This leads us to the idea, underlying our construction, of modeling the searched-for cover $\CS_{\RK_1}\to \CS_\RK$ by using the theory of toric schemes to construct appropriate covers of the local model $\Mloc_{\CG,\mu}$. 

\subsection{} Our construction is based on the following theorem of independent interest, cf. Theorem \ref{tameconj}.

\begin{theorem}\label{divconj-intro}
Let $(G, \{\mu\}, \CG)$ be a local model triple over $\BQ_p$. 
Suppose that the group $G$ splits over a tamely ramified extension 
of $\BQ_p$ and that $p$ does not divide the order of the algebraic fundamental group $\pi_1(G_\der)$.  There exists a pair $({\rm P}_{\CG,\mu}, s_{\CG,\mu})$, unique up to unique isomorphism, consisting of a $\CG$-equivariant $T_{\CG}$-torsor ${\rm P}_{\CG,\mu}\to \Mloc_{\CG,\mu}$ over the local model, and a $\CG$-equivariant trivialization (section) $s_{\CG,\mu}$ of this torsor over the generic fiber $\Mloc_{\CG,\mu}\otimes_{O_E}E$, such that
the following condition holds: 

For a  character $\chi: T_{\CG, \br \BZ_p} \to \BG_{m, \br\BZ_p}$, consider the 
$\BG_{m,O_{\br E}}$-torsor 
\[
{\rm P}_\chi:=\BG_{m,O_{\br E}}\times_{\chi, T_\CG}{\rm P}_{\CG,\mu} 
\]
over $\Mloc_{\CG,\mu}\otimes_{O_E}O_{\br E}$, obtained by pushing out ${\rm P}_{\CG,\mu}\otimes_{O_E}O_{\br E}$ by $\chi$. Denote by $\CL_\chi$ the line bundle over 
$\Mloc_{\CG,\mu}\otimes_{O_E}O_{\br E}$ which corresponds to the 
$\BG_{m,O_{\br E}}$-torsor ${\rm P}_\chi$. The section $s_{\CG,\mu}$ induces a section $s_\chi$ of the $\BG_{m,{\br E}}$-torsor ${\rm P}_\chi\otimes_{O_{\br E}}\br E$ over $\Mloc_{\CG,\mu}\otimes_{O_E}\breve E$ and, hence, a (meromorphic) section of the  line bundle
$\CL_\chi$ over $\Mloc_{\CG,\mu}\otimes_{O_E}O_{\br E}$. The condition is that, for all such $\chi$, the  divisor of this meromorphic section of $\CL_\chi$ is equal to 
\begin{equation*}
D_{\chi}:=\sum_{W^Kt_{\bar\mu'} W^K, \bar\mu'\in \Lambda_{\mu}}e\cdot \langle \bar\mu',\chi\rangle\cdot  Z_{\bar\mu'},
\end{equation*}
where  $e$ is the ramification index of $E$ over $\BQ_p$.

{\rm Here the sum is over the set indexing the  irreducible components $Z_{\bar\mu'}$} of the geometric special fiber of the local model $\Mloc_{\CG,\mu}$, and $\langle\ ,\ \rangle$ is the pairing defined in (\ref{pair2}); the coefficients $e\cdot \langle \bar\mu',\chi\rangle$ are integers, cf. Remark \ref{rem:coef}.
\end{theorem}

We note that the $Z_{\bar\mu'}$ are only Weil divisors in general and that the fact that the RHS is a Cartier divisor is far from trivial.  Our proof uses the  construction of local models in \cite{PZ} and Zhu's work on the coherence conjecture \cite{ZhuCoh}.  We conjecture that the statement of Theorem \ref{divconj-intro} holds in general, cf. Conjecture \ref{divconj}.
 Note that, under quite general assumptions, the Picard group   of the local model agrees with the Picard group of a corresponding affine partial flag variety, see \cite[Cor. 5.19]{FHLR}. Under this identification, the  map
\[
X^*(T_\CG)\to \Pic(\Mloc_{\CG,\mu}\otimes_{O_E}O_{\br E}),\quad \chi\mapsto \CL_\chi,
\]
obtained by Theorem \ref{divconj-intro} gives the standard construction of 
the subgroup of classes of line bundles with \emph{zero central charge}. Indeed, this subgroup is the image of the map. See, for example, \cite[p. 28]{ZhuCoh} or the proof of Cor. 5.19 (3) in \cite{FHLR}.

Let us illustrate Theorem \ref{divconj-intro} in the case $G=\GL_2$ for the Iwahori level. In this case $T_\CG=\BG_m^2$ and the torsor ${\rm P}_{\CG,\mu}$ is given in terms of the moduli functor \eqref{loG2} by $(\CF_1^{-1}\otimes\CF_0, \CF_0^{-1}\otimes\CF_1)$, and the section $s_{\CG,\mu}$ is given by the arrows in the bottom of \eqref{loG2} which are isomorphisms in the generic fiber. 

 From Theorem \ref{divconj-intro} (under the hypotheses made, or by Conjecture \ref{divconj} in general), the $\CG $-equivariant $T_{\CG }$-torsor ${\rm P}_{\CG,\mu}\to \Mloc_{\CG,\mu}$ and 
 the $G_E=\CG\otimes_{\BZ_p}E$-equivariant  trivialization $s_E$ of the general fiber define the  $T_{\CG,E}$-equivariant morphism over $E$
\begin{equation}\label{defdelintro}
\delta_E={\rm pr}\cdot s_{\CG, \mu}^{-1}: {\rm P}^{(-1)}_{\CG,\mu}
\otimes_{O_E}E
\to T_{\CG,E} 
\end{equation}
which is $G_E$-equivariant for the trivial action on the target. 
Its fibers are isomorphic to the partial flag variety $\CF(G, \{\mu\})=G_E/P_{\{\mu\}}$. 
Here  ${\rm P}^{(-1)}_{\CG,\mu}={\rm P}_{\CG,\mu}$ as a scheme, but with the inverse $T_{\CG}$-action, i.e. composed with $t\mapsto t^{-1}$.
  
\subsection{} We next introduce our tori and toric schemes for them. To simplify, we assume in the sequel  that $\CG$ is an Iwahori group scheme. Only at the end of the introduction do we discuss the case of an arbitrary parahoric group scheme. We denote by $T_\CG$ the unique lift over $\BZ_p$ of the torus $T=(\bar\CG_{\red})_{\rm ab}$. We now define an affine toric embedding of $T_\CG$.  Let
\[
S_{ \mu}=\{ \chi\in X^*(T_\CG)\mid \langle \bar\mu', \chi\rangle\geq 0, \forall\, \bar\mu'\in \Lambda_\mu\}.
\]
 In the Iwahori case, $\Lambda_\mu\subset X_*(T_\CG)_\BQ$;   when $G$ splits over $\br\BQ_p$, it is simply the Weyl orbit of $\mu$, see \S \ref{ss:LM}.
Then $S_{ \mu}$ is a saturated semi-group. If $G_\der$ is simply connected, $G_\ab$ is unramified and $\mu$ is non-trivial, or if $(G, \mu)$ is of local Hodge type,  $S_\mu$ generates $X^*(T_\CG)$ as a group, cf. Corollary \ref{cor:Hodgetype}. From now on we assume this. Hence $S_\mu$ defines an affine toric embedding 
\[
T_{\CG, O_{E_0}}\hookrightarrow Y_{\CG, \mu}
\]
over $O_{E_0}$.  Here $E_0$ denotes the maximal unramified subextension of $E$. The dual cone of $S_{ \mu}$ is given as 
\begin{equation*}
\sigma_{\mu}:=\{\sum\nolimits_{\mu'}r_{\mu'}\cdot \mu'\mid\ r_{\mu'}\geq 0\}\subset X_*(T_\CG)\otimes_{\BZ}\BR .
\end{equation*}
The affine toric scheme  $Y_{\CG,\mu}$ plays a central role in our constructions.

\begin{remark}
Let $T_{\CG, \ad}$ be the image of $T_\CG$ in the adjoint group $G_\ad$. Then  the image $\sigma_{\mu_\ad}$ of $\sigma_{\mu}$ in $(T_{\CG, \ad})_\BR$ defines a  toric variety for the torus $T_{\CG, \ad}$, by taking the cones over the images of the extreme rays of $\sigma_{\mu}$. In contrast to our toric variety $Y_{\CG, \mu}$ which is affine, this toric variety is projective, provided that $\mu$ projects non-trivially to all simple factors of $G_\ad\otimes_{\BQ_p}\breve\BQ_p$. Here we are using the amusing fact that for any irreducible root system $(V, R)$ and any non-trivial element $\mu\in V^*$, the convex hull of the orbit of $\mu$ under the Weyl group contains the origin in its interior, cf. Lemma \ref{lim}.
There are various ways  in the literature by which a projective toric variety is associated to an adjoint linear algebraic group but we were unable to find 
this particular one. It is for instance different from the toric varieties which arise from generic torus orbits in partial flag varieties, cf. \S \ref{ss:othcon} and the literature cited there.  Considering the fact that already at the birth of the theory of toric varieties, prominent examples arose from the theory of semisimple algebraic groups, it may be useful to add these toric varieties to the stock of standard examples of toric varieties.
\end{remark}

 It is not difficult to see that the map $\delta_E$ in \eqref{defdelintro} extends 
to a $T_{\CG}$-equivariant morphism  $\delta: {\rm P}^{(-1)}_{\CG,\mu}\to Y_{\CG,\mu}$ which is $\CG_{}$-equivariant for the trivial action on the target  (in the case of $\GL_2$, this amounts to the fact that the arrows in the bottom of \eqref{loG2} extend over the whole local model). 
In stacks language, we obtain,  after dividing out by $T_\CG$,   a morphism which we call the \emph{divisor map}, 
 \begin{equation}\label{divmap-intro}
 \Delta=\Delta_\CG: [\CG \bslash \Mloc_{\CG,\mu}]\to  [T_\CG \bslash Y_{\CG,\mu}].
 \end{equation}
 More generally, for a 
choice of saturated semi-group $S \subset S_{\CG, \mu} $, also generating $X^*(T_\CG)$ and satisfying suitable additional conditions (cf.  \S \ref{moregen}), we obtain  the toric scheme $Y_S$,
and   the corresponding divisor maps  $\delta_S$ and
 $
\Delta_S=\Delta_{\CG, S}:   [\CG \bslash \Mloc_{\CG,\mu}]\to  [T_\CG \bslash Y_{S} ] .
 $
 The point set of source and target of \eqref{divmap-intro} are finite posets: indeed, the point set of $[\CG \bslash \Mloc_{\CG,\mu}]$ can be identified with the admissible set $\Adm(\mu)$; the point set of $[T_\CG \bslash Y_{\CG,\mu} ]$ can be identified with the face poset of the cone $\sigma_{\mu}$. We give a  purely combinatorial description (the \emph{face map}) of the induced map between these finite posets, cf. Theorem \ref{conj:fac}. Results of Q.~Yu show that the face map is useful in the analysis of the fine structure of the admissible set, cf. Remark \ref{rem:yu}.

\subsection{}  The next step in our construction involves the Lang map $L: T\to T$ and its unique lifting over $\BZ_p$,
 \begin{equation}
 L\colon T_\CG\to  T_\CG,\quad x\mapsto {\rm Fr}(x)x^{-1} .
 \end{equation}
 Combining this with our toric embedding, we obtain  a Cartesian diagram
  \begin{equation*}\label{LangNormalization-intro}
\begin{aligned}
 \xymatrix{
        T _{\CG, O_{E} }\ar@{^{(}->}[r] \ar[d]_{L}  &  \wt Y_{S, O_E}  \ar[d]^L \\
        T_{\CG,  O_{E}} \ar@{^{(}->}[r]  & Y_{S,O_E} ,
        }
        \end{aligned}
\end{equation*}
 where  $\wt Y_{S}$ is the normalization of $ Y_{S}$ in the Lang covering map $L\colon T_{\CG}\to T_{\CG}$. Note that $\wt Y_{S}$ is a toric scheme for $T_\CG$ in its own right. 
 \begin{remark}
 Note that if $T_\CG$ is a split torus (as for $G=\GL_2$), then the Lang map is simply the map $x\mapsto x^{p-1}$. Note that this exponent is exactly the one appearing in the description of the cover $\CS_{\RK_1}\to\CS_\RK$ in the case of $\GL_2$. 
 \end{remark}
 We now define  the stack $\Mlocroot_{\CG,\mu}$ as the fiber product   of the Lang cover and $\delta_S$,
\begin{equation*}
\begin{aligned}
 \xymatrix{
      \Mlocroot_{\CG,\mu} \ar[r] \ar[d]_L  &   [T_\CG \bs \wt Y_S ] \ar[d]^L \\
       \Mloc_{\CG,\mu}\ar[r]^{\delta_S} &  [T_\CG \bs Y_S ]  .
        }
        \end{aligned}
\end{equation*}
 If $S=S_{\mu}$, we will  denote $\Mlocroot_{\CG,\mu}$ by $\Mlocrooto_{\CG,\mu}$.
 
 \begin{remark}\label{intro:rootstack}
The notation $\Mlocroot_{\CG,\mu}$ is supposed to be a reminder of the fact that this stack
 arises as a \emph{root stack}. To illustrate this point, assume that  the toric stack $[T_\CG \bs Y_S] $ is $[\BG_m\bs \BA^1]$. Recall that the $S$-valued points of $[\BG_m\bs \BA^1]$ correspond to a $\BG_m$-torsor $P$ over $S$ and an equivariant morphism  $P\to \BA^1$. By the  correspondence between $\BG_m$-torsors and invertible sheaves, an $S$-valued point of $[\BG_m\bs \BA^1]$ corresponds therefore to a pair $(\CL, s)$ consisting of an invertible sheaf $\CL$ over $S$ and  a global section $s$ of its dual $\CL^{-1}$, cf. \cite[Ex. 5.13]{Ol}. The $n$-root stack over $S$ corresponding to  $(\CL, s)$ is defined to be the fiber product in the following diagram,
\begin{equation*}\label{GMrootstack}
\begin{aligned}
 \xymatrix{
      X_{\CL, s, n} \ar[r] \ar[d]  &   [\BG_m \bs \BA^1 ] \ar[d]^{n} \\
       S\ar[r]^{\delta_{\CL, s}} &  [\BG_m \bs \BA^1 ]  ,
        }
        \end{aligned}
\end{equation*}
cf. \cite[Def. 2.2.1]{Cadman}. In \cite{Cadman}, this construction is mostly applied to a pair $(\CL, s)$ of the form $(\CO(-D), \CO\to\CO(D))$, defined by an effective Cartier divisor on $S$. Our construction is a generalization of the Cadman construction (in fact, when $T_\CG$ is split and $Y_S$ is smooth so that $[T_\CG \bs Y_S ]\simeq [\BG_m\bs \BA^1]^r$, our construction coincides with the $r$-fold power of Cadman's construction for $n=p-1$). 
\end{remark}

Let $\fkM_{\CG,\mu}=[\CG\bs  \Mloc_{\CG,\mu}]$ and $\fkM\tysup{\sqrt{S}}_{\CG,\mu}=[\CG\bs  \Mlocroot_{\CG,\mu}]$. Then the divisor morphism  $\Delta_S: \fkM_{\CG,\mu}\to [T_\CG\bs Y_S]$ gives  a fiber product
 diagram
 \begin{equation}\label{CDrootstackIntr}
\begin{aligned}
 \xymatrix{
      \fkM^{\sqrt{S}}_{\CG,\mu} \ar[r] \ar[d]  &   [T_\CG \bs \wt Y_S ] \ar[d]^L \\
       \fkM_{\CG,\mu}\ar[r]^{\Delta_S} &  [T_\CG \bs Y_S ]  .
        }
        \end{aligned}
\end{equation}
  
 \subsection{}  The following conjecture would give a systematic construction of $O_E$-integral models of $K_1$-level. In its statement, we let the prime to $p$ component $\RK^p$ vary, i.e., we consider  ${\rm Sh}_{K_1}(\RG, \RX)_E$, resp. ${\rm Sh}_{K}(\RG, \RX)_E$ (the inverse limit  over the prime to $p$ subgroups $K^p$). Let 
  \[
  \varphi\colon \CS_{K}\to  \fkM_{\CG,\mu}
  \]
 be the stacks version of the local model diagram. 
 
\begin{conjecture}\label{globconj-intro}
 There exists an $O_E$-integral model  of the Shimura variety 
 ${\rm Sh}_{K_1}(\RG, \RX)_E$, i.e. an $O_E$-scheme $\CS_{K_1, S}$ together with an isomorphism $\CS_{K_1, S}\otimes_{O_E}E\simeq {\rm Sh}_{K_1}(\RG, \RX)_E$,   which has the following properties:
 
 \begin{itemize}
\item[i)] There is a morphism \[\CS_{K_1, S}\to \mathcal \CS_{K}
\] which extends $\pi: {\rm Sh}_{\RK_1}(\RG, \RX)_E\to {\rm Sh}_{\RK}(\RG, \RX)_E$ on the generic fibers.
 
 \item[ii)] The action of $T_\CG(\BF_p)$ on ${\rm Sh}_{K_1 }(\RG, \RX)_E$ extends to $\CS_{K_1, S}$ and the morphism 
 $\CS_{K_1, S}\to \mathcal \CS_{K}$ of (i)
  identifies $\CS_{K}$ with the scheme quotient $T_\CG(\BF_p)\bs \CS_{K_1, S}$.
 
  \item[iii)] The scheme $\CS_{K_1, S}$ and the $T_\CG(\BF_p)$-cover $\CS_{K_1, S}\to \mathcal \CS_{K}$ are $\RG({\BA}_f^p)$-equivariant for a $\RG({\BA}_f^p)$-action that extends the natural action on the generic fiber.
  
  \item[iv)] There is a morphism $\varphi_{1, S}: \CS_{K_1, S}\to \fkM^{\sqrt{S}}_{\CG,\mu}$ which fits  in a $2$-commutative diagram
  \begin{equation*} 
\begin{aligned}
 \xymatrix{
     \CS_{K_1, S}\ \ar[r]^{\varphi_{1,S}} \ar[d]  &  \fkM^{\sqrt{S}}_{\CG,\mu} \ar[d] \\
       \CS_{K}\ar[r]^{\varphi}  & \fkM_{\CG,\mu},
        }
        \end{aligned}
\end{equation*}
 and which induces
   an isomorphism of stacks 
 \[
 [T_\CG(\BF_p)\bs \CS_{K_1, S}]\xrightarrow{\sim}  \CS_{K}\times_{\fkM_{\CG,\mu}}\fkM^{\sqrt{S}}_{\CG,\mu }.
 \]
 \end{itemize}
 \end{conjecture} 
The conjecture implies a local description of $\CS_{\RK_1}$ in the following sense. Let $x\in \CS_{\RK}$, with corresponding point $y\in \Mloc_{\CG,\mu}$ (only the $\CG$-orbit of $y$ is well-determined). Then there is an isomorphism
\begin{equation}\label{eq:sh2-intro}
 \CS_{K_1, S}\times_{\CS_{K}}\Spec(\CO^{\rm sh}_{\CS_{K}, x})\simeq  \wt Y_S\times_{Y_S,\delta}\Spec(\CO^{\rm sh}_{\Mloc_{\CG,\mu}, y})
 \end{equation}
 which respects the $T_\CG(\BF_p)$-actions on both sides.  This can be made completely explicit. For instance if 
 $T_\CG\simeq \BG_m^r$ is split (so that the Lang map is simply the dilation by $p-1$), then  \eqref{eq:sh2-intro} is the spectrum of the ring
\[
\CO^{\rm sh}_{\Mloc_{\CG,\mu}, y}[u_1, \ldots, u_n]/(I_S, u_1^{p-1}-\delta^*(s_1), \ldots , u_n^{p-1}-\delta^*(s_n)).
\]
Here $s_1, \ldots, s_n$ is a minimal generator set for $S$, and $I_S$ is the ideal of relations presenting the semigroup $S$, and  $t\in T_\CG(\BF_p)$ acts by $t\cdot u_i=\chi_i(t) u_i$, where $\chi_i(t)$ is the image of $t$ under
\[
T_\CG(\BF_p) \subset T_\CG(\BZ_p)\xrightarrow{s_i} \BZ_p^*.
\]
(The inclusion is the Teichmuller lift.)
 
  In the direction of Conjecture \ref{globconj-intro}, our main result is the following.
 
\begin{theorem}\label{thm:PEL-intro} (cf. Theorem \ref{thm:PEL})
 Let $(\RB,\RV,( \,,\,), *, h, O_\RB, \CL )$ be integral PEL data as in \cite[Ch. 6]{RZbook}. Let
 $\CG={\rm Aut}_{O_B,(\, )}(\mathscr L)$ be
 the group scheme of isomorphisms of the polarized self-dual lattice multichain $\sL$ up to common similitude in $\BZ_p^*$.  Assume that $p\neq 2$, 
 $\CG\otimes_{\BZ_p}\BQ_p$ splits over a tamely ramified extension of $\BQ_p$, and that $\CG$ is an Iwahori group scheme.
Assume furthermore  that condition (S) is satisfied, see \S \ref{sss:MainPEL}. 
 Then Conjecture  \ref{globconj-intro} for the 
 $T_\CG(\BF_p)$-cover of the corresponding PEL type Shimura variety  with $\CG(\BZ_p)$-level at $p$ is true.
 \end{theorem}
 
 Note that the group $G=\CG\otimes\BQ_p$ might not be connected.
 By definition, an Iwahori subgroup of a non-connected $p$-adic reductive group   is an Iwahori subgroup of its neutral component. Our assumption ``$\CG$ is Iwahori" in the statement of the theorem more precisely means
 that we require that $\CG\otimes\BF_p$ is connected and that the neutral component $\CG^\circ$ is a Iwahori group scheme for the connected reductive group $G^\circ$; then also $\CG^\circ(\BZ_p)=\CG(\BZ_p)$ is an Iwahori subgroup of $G^\circ(\BQ_p)$.
 If $G$ is connected (i.e. for types A or C) this amounts to simply requiring that $\CG$ is Iwahori in the standard sense.
  We will see that our assumption implies that the stabilizer of $\CL$ in ${\rm Aut}_{O_B}(\mathscr L)$ is an Iwahori in $\GL_\RB(\RV)$. 
  Condition (S) is an additional requirement on the periodic polarized lattice multichain $\CL$. It requires that $T_\CG$ is an induced torus
  which is, roughly, presented in a specific manner and it is often satisfied when $G_\ad$ is absolutely simple.

 The Shimura variety ${\rm Sh}_\RK$ is open and closed in a moduli scheme of abelian varieties with additional structure and this is how the integral model $\CS_\RK$ is defined.
  The integral model $\CS_{\RK_1, S}$ is then constructed by adding moduli data to the abelian scheme representing a point of $\CS_\RK$.

\subsection{}  Assuming the truth of Conjecture \ref{globconj-intro}, let us enumerate properties of the integral models $\CS_{\RK_1, S}$. The local  properties are deduced from properties of the stack   $\Mlocroot_{\CG,\mu}$ above; these, in turn depend on properties of the Lang covers $L\colon \wt Y_S\to Y_S$.
 We continue to assume that $\CG$ is Iwahori. For (2) and (3) below, assume  $p\neq 2$ and also that
 $\sigma_\mu$ is full-dimensional, or equivalently, that $Y_{S_\mu}$ has no torus factor, as in \S \ref{blanket}. It turns out that the desire to find hypotheses that simultaneously  make both this cover ``nice" and the  
 stack   $\Mlocroot_{\CG,\mu}$ ``nice" can only be realized in the Drinfeld case (to be made more precise below).
 
\begin{enumerate}
\item The cover $\CS_{\RK_1, S}\to\CS_{\RK}$ is finite and surjective. In fact, $\CS_{\RK}$ is the scheme theoretic quotient of $\CS_{K_1, S}$ by the action of the finite group $T(\BF_p)$. Hence, assuming the flatness of $\CS_{\RK}$ over $O_E$, the scheme $\CS_{\RK_1, S}$ is topologically flat over $O_E$. In particular, $\CS_{\RK_1, S}$ is a  good model for topological questions, like, e.g., the analysis of the sheaf of nearby cycles.

\item If $S=S_{\mu}$, then the scheme $\CS_{\RK_1, S}$ is regular in codimension one, provided that $E$ is unramified over $\BQ_p$, cf. Corollary \ref{cor:R1}. By contrast, if $S\neq S_{\mu}$ and $T_\CG$ is split, then $\CS_{\RK_1, S}$ is very often \emph{not} regular in codimension one, see Proposition \ref{prop:R1c} for a precise criterion. In particular, this happens when $S\neq S_\mu$ is free, see Remark \ref{rem:5.2.12}. 

\item Assume the validity of Conjecture \ref{conjLco} that the Lang cover $L\colon \wt Y_{S_\mu}\to Y_{S_\mu}$ is flat if and only if $Y_{S_\mu}$ is smooth. This holds when $T_\CG$ is split, but also in many more PEL cases, see Remark \ref{rem:Altmann}. Then the cover $\CS_{\RK_1, S_\mu}\to\CS_{\RK}$ is flat if and only if $Y_{S_\mu}$ is smooth, cf. Lemma \ref{rootback}. Assuming   furthermore that $(G,\{\mu\})$ is ab-nondegenerate and that $G_\ad$ is absolutely simple, this last property holds  only if the pair $(G_\ad\otimes_{\BQ_p} \breve\BQ_p, \mu)$ is isomorphic to $(\PGL_n, \varpi_1^\vee)$ (the Drinfeld case), cf. Corollary \ref{classfree}. 

\item We expect that, outside the Drinfeld case, the scheme $\CS_{\RK_1, S}$ is rarely normal--in fact, sometimes it is not even flat over $\BZ_p$, comp. Remark \ref{rem:HilbertSiegel}.  Also, computer calculations suggest that even the flat closure of the generic fiber can be bad (not normal, special fiber not reduced), comp. Remark \ref{rem:badSiegel}.  
\end{enumerate}

Here the notion of the ``Drinfeld case'' has to interpreted with caution. Indeed, our theory is sensitive to changes of $\mu$ by a central cocharacter. For instance the cases $(\GL_n, \varpi_1^\vee)$ and $(\GL_n, \varpi_{n-1}^\vee)$ are isomorphic when passing to the adjoint group (and hence share a common local model $ \Mloc_{\CG,\mu}$) but yield radically different $ \Mlocroot_{\CG,\mu}$. It is this sensitivity to central cocharacters that prevents us so far from stating a conjecture on Lang covers of affine toric varieties that would have Conjecture \ref{conjLco} as consequence. 

We mention here an amusing group-theoretic ingredient in the proof of 3).  Let $(R, V)$ be an irreducible root system with Weyl group $W$. Then for any proper parabolic subgroup $W_J$ of $W$  there is the inequality
$$\# (W/W_J)\geq \dim V +1, $$
with equality if and only if $(V, R)$ is of type $A_n$ and $J=\{s_1,  \ldots, s_{n-1}\}$ or $J=\{s_2,  \ldots, s_{n}\}$. This is the Weyl group analogue of the inequality in \cite{OR} comparing $\dim G/P$ and ${\rm rk}\, G$. 

 At this point, one may ask why we do not define the finite cover $\CS_{\RK_1, S}$ of $\CS_{\RK}$ by simply taking the normal closure, which by (4) above would be different from our construction in most cases. The reason is that there seems no systematic way of analyzing the local structure of the normalization.

\subsection{}  Let us now explain the method of proof of Theorem \ref{thm:PEL-intro}. First, the $T_\CG(\BF_p)$-cover ${\rm Sh}_{\RK_1}\to {\rm Sh}_{\RK}$ gives rise to a pair $(Q, a)$ consisting of a $T_\CG$-torsor $Q$ over $ {\rm Sh}_{\RK}$ and a section $a$ of the push-out $L_*(Q)$ of $Q$ under the Lang map $L:T_\CG\to T_\CG$, cf. Proposition \ref{prop:torsors}. 

In  light of Theorem \ref{divconj-intro} (or its conjectural generalization, freeing it from the hypotheses made in  Theorem \ref{divconj-intro}), Conjecture \ref{globconj-intro} is reduced to the following statement.   Namely, that there exists a $T_\CG$-torsor $P$ over $\CS_K$ extending the $T_\CG$-torsor $Q$ and an isomorphism of pairs 
 \[
 \beta: \big(L_*(P), a: T_{\CG}[1/p]\xrightarrow{\sim} L_*(P)[1/p]\big)\xrightarrow{\sim} \varphi^*\big(\RP^{(-1)}_{\CG,\mu}, s_{\CG,\mu}\big),
 \]
 where the pair $(\RP_{\CG,\mu}^{(-1)}, s_{\CG,\mu})$ comes from  Theorem  \ref{divconj-intro} and the pull-back is by the local model diagram morphism $\varphi: \CS_{K}\to \fkM_{\CG,\mu}
 =[\CG\bs \Mloc_{\CG,\mu}]$.  Note that this statement is independent of the semi-group $S$. 
 
 We are therefore reduced to constructing the $T_\CG$-torsor $P$ over $\CS_K$ with its section after push-out under the Lang isogeny. Here the main ingredient is the Oort--Tate--Raynaud theory of finite group schemes and their generators in the sense of \cite{P95} (and the condition  (S) is taylored to enable us to apply  this theory). Indeed, a $\BF_{p^d}$-Raynaud group  scheme $G$ over a flat $\BZ_p$-scheme $S$, corresponding in terms of that theory to the tuple $(\CL_i, \gamma_i, \delta_i)_{i\in\BZ/d}$, defines the $\Res_{W(\BF_{p^d})/\BZ_p}(\BG_m)$-torsor $P_G=(\CL_i)_{i\in\BZ/d}$ with $L_*(P_G)=(\CL_{i-1}^{p}\otimes\CL_i^{-1})_{i\in\BZ/d}$ and sections $(\delta_{i-1}: \CO_S\to\CL_{i-1}^{-p}\otimes\CL_i)_{i\in\BZ/d}$.
 
 We note that one may use the  $T_\CG$-torsor $P$ and its section of its push-out under the Lang isogeny to formulate a relative moduli problem for $\CS_{K_1}$ over $\CS_K$, cf. Proposition \ref{prop:Conj}. However, this moduli problem lacks in general a catchy formulation. In a few cases, one can give a formulation directly in terms of Raynaud group schemes and their generators. Here are a few instances of such formulations.
 
 \begin{itemize}
 \item In the Siegel case ($G=\GSp_g$), for $S=S_\mu$, the integral model $\CS_{K_1, S}$ parametrizes the set of Oort--Tate generators $z_i$ of the group schemes $G_i, i=1, \ldots 2g$, such that the products $z_i z_{2g+1-i}$ induced by the duality pairings are independent of $i$. Here the group schemes $G_i$ arise from the universal chain of $p$-isogenies $A_1\to A_2\to\ldots\to A_{2g}\to A_1$ over $\CS_K$, and the pairing is induced by the universal principal polarization over $\CS_K$, cf. Corollary  \ref{sieOT}. 
 \item In the ``split unitary case'' ($p$ splits), for $S$ the evident free semi-group in $S_\mu$, there is a similar moduli description in terms of Oort--Tate generators of universal group schemes $G_i$ of order $p$, cf. \S \ref{sss:split}.
 \item In the case of the fake unitary group leading to $p$-adic uniformization, for $S=S_\mu$, the integral model $\CS_{K_1, S}$ parametrizes the pairs $(P, u)$, where $u$ is a $p-1$-st root of $w_p$ and where $P$ is a generator of the Raynaud group scheme $X[\Pi]$, where $X$ is the universal $p$-divisible group with $O_D$-action over $\CS_K\otimes_{O_K}W(\kappa_D)$, cf. \S \ref{ss:fak}
 \end{itemize}
 
 Let us now drop the assumption, made so far in  this introduction, that $\CG$ is an Iwahori group scheme. The construction of the root stack $ \Mlocroot_{\CG,\mu}$ extends to general parahoric group schemes, and so does the formulation of Conjecture \ref{globconj-intro}. However, our method of using Raynaud theory to prove it does not extend (not even to other PEL cases). In fact, even in the context of Theorem \ref{thm:PEL-intro}, when we drop some of our conditions   but still insist that the neutral component $\CG^\circ$ be  Iwahori, Raynaud theory is insufficient. Considering the fact that Raynaud theory can be viewed as an early attempt at addressing modulo $p$ phenomena  in $p$-adic Hodge theory (and  has not been really integrated in the subsequent advances, like prismatic theory), it seems to us promising to try to apply these more recent developments to Conjecture \ref{globconj-intro}. 
 
\subsection{}  Our theory of integral models of $\Gamma_1(p)$-type seems the first attempt at a systematic theory of such integral models. However, there are quite a number of papers in the literature that develop special cases. We mention the following instances. 
 
 \begin{itemize}
 \item The case of modular curves, as also described in this introduction, is treated in \cite{DeligneRa} and by Katz-Mazur in \cite{KatzMazur}.
 \item The Hilbert-Blumenthal case is studied by the first-named author in \cite{P95} who introduces the notion of a ``generator" for a Raynaud group scheme (see loc. cit. \S 5), the same as in this paper, cf. \S \ref{sss:Generators}. The level subgroup considered is
 \[
 K'=\left\{\begin{pmatrix}1 &*\\ 0&  *\end{pmatrix}\,{\rm mod}\ \mathfrak p  \right\}\cap \{g\in \GL_2(O_{F_{\mathfrak p}})\ |\ \det(g)\in \BZ_p^*\},
 \]
 where $\mathfrak p$ is a prime of the corresponding totally real field $F$, comp. Rem. \ref{remarkHB}.
 \item  M. Harris and R. Taylor \cite{HT} consider  unitary similitude Shimura varieties  at a prime which splits in the quadratic extension. In particular they consider the case of a quadratic extension of $\BQ$ of signature $(1, n-1)$ and the subgroup of a parahoric (which is not an Iwahori)
 \[
K'=\left\{\begin{pmatrix}1_1 &*\\ 0& *_{n-1}\end{pmatrix}\,{\rm mod}\ p  \right\}\subset K=\left\{\begin{pmatrix}*_1 &*\\ 0& *_{n-1}\end{pmatrix}\,{\rm mod}\ p  \right\} .
\]
They use Oort--Tate theory to show that the normalization of $\CS_K$ in ${\rm Sh}_{K'}$ is a regular scheme mapping by a finite flat morphism to $\CS_K$. 
A similar construction for the Klingen parahoric in the Siegel case $\GSp_4$ is considered by Genestier-Tilouine \cite{GenTi}.
\item T. Haines and the second-named author \cite{HR} consider the unitary similitude Shimura variety of signature $(1, n-1)$ at a prime which splits in the quadratic extension, and consider the pro-unipotent radical $K_1$  of an Iwahori $K$.
They give a moduli problem for $\CS_{K_1}$ in terms of generators of Oort--Tate group schemes and show that $\CS_{K_1}$ is a regular scheme finite and flat over $\CS_K$. 
\item The Siegel case and the unitary similitude Shimura variety of general signature $(r, s)$ at a prime which splits in the quadratic extension
was considered by R.~Shadrach \cite{Shadrach}, by T.~Haines--B.~Stroh \cite{HS} in unpublished work, and T. Haines--Q. Li-B. Stroh \cite{HLS}, in work in preparation. In all these cases, the authors give 
moduli schemes in terms of Oort--Tate generators on the universal isogenies, comp. \S \ref{ss:Siegel}, \S \ref{sss:split}. 
These authors take $S=S_\mu$ in the Siegel case, but make a different choice $S\neq S_\mu$ in the split unitary case. Essentially the same moduli problems were studied by G. Marazza \cite{Marazza}.  He proves that these schemes are rarely normal and in the Siegel case not even flat over $\BZ_p$, see Remark \ref{rem:badSiegel}.

\item The Hilbert-Siegel case for primes unramified in the corresponding totally real field is considered by S. Liu  \cite{SLiu}. He uses
a different notion of a Raynaud generator to give a moduli scheme, see Remark \ref{rem:Generators}. See Remark \ref{rem:HilbertSiegel} (3)
for some comments on \cite{SLiu}.

 \end{itemize}
 \subsection{} We also develop  a local analogue of our theory, concerning coverings of \emph{integral local Shimura varieties}.  Prominent examples of integral local Shimura varieties are the Lubin-Tate formal moduli scheme of one-dimensional formal $p$-divisible groups of height $n$ and the Drinfeld moduli scheme of special formal $O_D$-modules, where $D$ is the central division algebra with invariant $1/n$ over $\BQ_p$. The latter case leads to the integral local Shimura variety given by the formal model $\wh{\Omega}^n_{\BQ_p}$ of the Drinfeld $p$-adic halfspace. In this case, such coverings were considered earlier by J.~Teitelbaum \cite{T}, H.~Wang \cite{W} and L.~Pan \cite{P}. Here the first two papers take an indirect approach through the consideration of the rigid-analytic generic fiber; the last paper is close in spirit to ours, and constructs the integral cover by adding a generator of the natural Raynaud group scheme over the Drinfeld moduli scheme.  
  
\subsection{} Let us indicate some future directions that seem promising to us. 
We already mentioned that we conjecture that Theorem \ref{divconj-intro} holds without the hypotheses of tameness.  A proof should involve  the recent advances in the theory of local models  beyond the tamely ramified case.

 We also mentioned already the challenge of generalizing Theorem \ref{thm:PEL-intro}. Here the example at the end of  \S \ref{ss:localex} could be a first test case. As we argue in \S \ref{sss:ConjImplyConj}, the construction of $\CS_{K_1, S}$ rests on constructing a $2$-commutative 
 diagram of stacks 
\begin{equation}\label{Intro-StackDia1}
\begin{aligned}
 \xymatrix{
    {\rm Sh}_{\RK_0}(\RG, X)_E \ar[r]^{\rm\ \ incl} \ar[d]_{[\pi]}  &     \CS_{K_0}\ar[r]^{\varphi} \ar[d]_{\wt{[\pi]}} & [\CG\bs \Mloc_{\CG,\mu}]\ar[d]_{\Delta_S}\\
      [T_\CG  \LBS T_\CG ] \ar[r]  & [T_\CG \LBS Y_S ]  \ar[r]  & [T_\CG\, \bs\, Y_S ].
        }
        \end{aligned}
\end{equation}
 As a weakening of the conjecture, one might aim at constructing a similar diagram of $v$-stacks over perfectoid spaces. This would give directly a $v$-sheaf which should be represented by the integral model $\CS_{K_1,S}$ when this is shown to exist.  Such a construction could be simpler than proving the full Conjecture \ref{globconj-intro}. This $v$-sheaf version should be enough to capture the information needed for some applications, e.g. for calculating nearby cycles. Work in progress by the first-named author  gives hope that this can be accomplished by building on the methods of \cite{PRg}.  We also mention in this context that Y. Takaya \cite{Tak} gave a  construction of covers of tubular neighborhoods of local Shimura varieties with ``principal subgroup level".

Here are some other questions. 

 The local model diagram relates $\CS_K$ to the local model $\Mloc_{\CG,\mu}$, which is a closed subscheme of a Beilinson-Drinfeld affine partial flag variety ${\rm Gr}_{\CG,O}=\CL\und\CG/\CL^+\und\CG$, cf. \S \ref{ss:Bundles}.  Vaguely speaking, our root stack $\Mlocroot_{\CG,\mu}$ is related to a partial compactification of a quotient $\CL\und\CG/\CL^1\und\CG$ (which is not ind-projective), depending on the toric embedding $T_\CG\subset Y_S$. One might wonder whether this partial compactification arises from taking a quotient by $\CL^1\und\CG$ of a ``toroidal partial compactification'' of  $\CL\und\CG$, in the spirit of Mumford's construction of a toroidal partial compactification of a loop group, cf. \cite[ch. IV, \S 3, p. 202 et seq.]{TEI}.

Another circle of questions arises when we drop the assumption that the cone $\sigma_\mu$ is totally convex. This assumption is satisfied when the Shimura datum $(\RG, \RX)$ is of Hodge type, but not always when $(\RG, \RX)$ is of abelian type or more general. For instance, does there exist a parallel theory of coverings $\CS_{K_1}\to \CS_K$ when $\RG$ is an adjoint group? 

\VE
 
 \subsection{} We now explain the lay-out of the paper. 

In \S 2, we define the pairing which appears in the statement of Theorem \ref{divconj-intro}.
 and then prove this theorem. The proof involves the construction of the local models  in \cite{PZ} via  Beilinson-Drinfeld affine Grassmannians. 
 
 In \S 3, we define the toric schemes $Y_{\CG,\mu}$ and $Y_S$ which are obtained from cones spanned by  Weyl orbits of coweights. Since these could have independent interest, 
we give many examples and a discussion of various properties and relations to other similar constructions in the literature. 

In \S 4, we define the divisor map $\Delta$ and the toric compactification of the Lang cover of the torus $T_\CG$. We  discuss   the flatness of this cover and its ramification structure along the boundary of the toric embeddings. 

In \S 5, we define the local model root stacks $\Mloc^{\sqrt{S}}_{\CG,\mu}$ and discuss their algebraic-geometric properties. In particular, we give criteria for when $\Mloc^{\sqrt{S}}_{\CG,\mu}$ are regular in codimension $1$ and for when the morphism $\Mloc^{\sqrt{S}}_{\CG,\mu}\to \Mloc_{\CG,\mu}$ is flat.

The main Conjecture \ref{globconj-intro} is stated in \S 6 where we also give its implications for the \'etale local structure of the integral models of the covers. In \S \ref{ss:generalconstruction} we give a more concrete interpretation of the conjecture in terms of line bundles equipped with suitable sections. This involves the   diagram \eqref{Intro-StackDia1} (labeled \eqref{StackDia1} in the main text) which is key for the   proof of Theorem \ref{thm:PEL-intro} and the analysis of the examples in \S \ref{s:PELex}.   

In \S 7, we discuss the case of Shimura varieties of PEL type, recall the notion of a generator of an Oort--Tate/ Raynaud group scheme and give the proof of Theorem \ref{thm:PEL-intro}. 

In \S 8, we give several key examples of PEL type Shimura varieties. We discuss cases where the integral models can be obtained as moduli schemes for moduli problems involving generators of Raynaud group schemes and compare with prior results in the literature as in the list above. 

Finally, in \S 9 we discuss the parallel set-up for integral models of  local Shimura varieties.

 \subsection{AI usage} The paper was revised after using ChatGPT-6: We sharpened the hypotheses for Conjecture \ref{conjLco} and corrected a mistake in the definition of the face map in the general parahoric case.

\subsection{Acknowledgements} We thank T.~Haines and B.~Stroh for informing us of their results in the Siegel case and of their forthcoming joint work with Q. Li.   We thank Q.~Yu for his interest in our conjectures concerning the face map. We thank Q.~He and R.~Zhou for an e-mail concerning complete self-dual lattice chains for orthogonal groups. We especially thank Q. Li for several comments and corrections to  first versions of the paper and K.~Altmann for the many helpful explanations concerning the theory of toric varieties and for his interest in our questions about the flatness of the Lang covers.

This work was supported by the Deutsche Forschungsgemeinschaft (DFG, German Research Foundation) under Germany's Excellence Strategy – EXC-2047/2 – 390685813. G. P. was also supported by Simons Foundation grant SFI-MPS-TSM-00013296.

\subsection{Notations} Suppose $B$ is an $A$-algebra. If $X$ is a scheme over $\Spec(A)$ we will often write $X\otimes_{ A}B$ or simply $X_B$ for the base change $X\times_{\Spec(A)}\Spec(B)$. Suppose that $H$ is a group scheme over $\Spec(A)$ and $Y$ a scheme over $\Spec(B)$. Often, to simplify notation, we will refer to an $H_B$-action on $Y$  as simply ``an $H$-action on $Y$", omitting the base change.
Similarly, we will then   just write $[H\bs Y]$ for the quotient stack $[H_B\bs Y]$. We write $\CL^\vee$ or $\CL^{-1}$ for the dual of an invertible sheaf $\CL$. If $s: \CO\to \CL$ is a section which gives an isomorphism, we will sometimes write $s^{-1}:\CO\to \CL^{-1}$ for the dual of the inverse $ \CL\to \CO$ of $s$.

\VE

\section{Divisors on local models}

\subsection{The pairing}

\subsubsection{Tori attached to parahorics}\label{ss:211}
Let $\CG$ be a parahoric group scheme over $\BZ_p$ for $G=\CG\otimes_{\BZ_p}\BQ_p$. 
Consider the maximal reductive quotient $\ov \CG_\red:=(\overline \CG)_{\red}$ of the special fiber $\ov\CG=\CG\otimes_{\BZ_p}\BF_p$ and its maximal abelian quotient 
$\overline \CG_{\red,\ab}=(\ov \CG_{\red})_{\ab}$  which is a torus over $\BF_p$. Recall that, by rigidity of tori, reduction modulo $p$
gives an equivalence  between the categories of tori over $\BZ_p$ and tori over $\BF_p$ (with maps given by group scheme homomorphisms). We let $T_\CG$ be the unique (up to unique isomorphism) torus over $\BZ_p$ which lifts $(\overline \CG)_{\red,\ab}$. We have 
\[
X_*(T_\CG)=X_*(\overline \CG_{\red,\ab})
\]
as $\hat\BZ=\Gal(\br\BZ_p/\BZ_p)=\Gal(\ov\BF_p/\BF_p)$-modules; here we somewhat abuse notation and we write $X_*(T_\CG)$
instead of $X_*(T_\CG\otimes_{\BZ_p}\br\BZ_p)$.

Suppose that $\CG$, $\CG'$ are two parahoric group schemes for $G$ for which there is a (unique) group scheme homomorphism  $\CG\to\CG'$ extending the identity on the generic fiber. Then, we obtain $\overline \CG_{\red,\ab}\to \overline \CG'_{\red,\ab}$. Hence, by the above construction, we also have a group scheme homomorphism
\[
T_{\CG}\to T_{\CG'}.
\]
Associating $T_\CG\to T_{\CG'}$ to $\CG\to \CG'$ is functorial, i.e. this association  respects compositions.

\subsubsection{Relation to the standard torus}\label{ss:standardtorus} Suppose now that $\CG$ is the parahoric group scheme associated to a point $x$ in the extended Bruhat-Tits building $\CB(G, \BQ_p)$.
Let $A$ be a maximal $\BQ_p$-split torus of $G$ such that $x$ belongs to the apartment of $A$. We choose a maximal $\br\BQ_p$-split torus $S$ of $G\otimes_{\BQ_p}\br\BQ_p$ which contains $A$ and is defined over $\BQ_p$; such a torus exists by \cite[5.1.12]{BTII}. Let $T=Z_G(S)$ be the centralizer 
of $S$ in $G$. Then, using Steinberg's theorem, we see that $T$ is a maximal torus of $G$ which is defined over $\BQ_p$. 
By the construction of $\CG$ in \cite{BTII}, the connected ft Neron model $\CT=\CN(T)$  of $T$ is contained in $\CG$. More precisely, $T\hookrightarrow G$ extends to a closed immersion of group schemes 
\[
\CT\hookrightarrow \CG.
\]
Since $\ov\CT_{\red}$ is a maximal torus in the connected reductive group $\overline \CG_{\red}$, the homomorphism 
$\CT\to \CG$ induces a surjective homomorphism of tori $\ov\CT_{\red}\to \overline \CG_{\red,\ab}$ and hence a $\hat\BZ$-equivariant
\begin{equation}\label{map1}
X_*(\ov\CT_{\red})\to X_*(\overline \CG_{\red,\ab})=X_*(T_\CG)
\end{equation}
which is surjective up to torsion.
Note that when $\CG=\CI$ is an Iwahori group scheme, then by construction,
\[
\overline \CG_{\red,\ab}=\ov \CI_{\red}=\ov\CT_{\red}.
\]
 Hence, in the Iwahori case,
$
X_*(\ov\CT_{\red})=X_*(T_\CG).
$

By \cite[Prop. B.7.9, p. 678]{KalethaPrasad}, there is a natural  
  identification
  \[
  X_*(\ov\CT)=X_*(\ov\CT_{\red})=X_*(T)^I.
  \]
Hence, (\ref{map1}) becomes
\begin{equation}\label{map2}
X_*(T)^I\to X_*(T_\CG).
\end{equation}
In the Iwahori case,   (\ref{map2}) is an isomorphism and, in general, it is surjective up to torsion.

\subsubsection{Relation to the cocenter of $G$}\label{ss:cocenter} Let $\CG$ and $T$ be as above.
Note that the homomorphism $G\to G_\ab$ induces $T\to G_\ab$ whose kernel is a torus
\[
1\to T_1=T\cap G_\der\to T\to G_\ab\to 1.
\]
The homomorphism $G\to G_\ab$ also induces a homomorphism of group schemes over $\BZ_p$,
$$
\CG \to \CN(G_\ab).
$$
Here again, $\CN(G_\ab)$ denotes the connected ft Neron model of $G_\ab$ over $\BZ_p$. Since, ${\rm H}^1(\br\BQ_p, T_1)=(0)$, the map
$T(\br\BQ_p)\to G_\ab(\br\BQ_p)$ is surjective. Hence, by \cite[9.6, Lemma 2]{BLR}, so is $\CT(\br\BZ_p)\to \CN(G_\ab)(\br\BZ_p)$, cf.  \cite[proof of Cor. 11.7.4]{KalethaPrasad}. 
It follows that $T\to G_\ab$ extends to a fppf surjective homomorphism 
\[
\CT\to \CN(G_\ab)
\]
between connected ft Neron models. This induces $\ov\CT_{\red}\to  \ov{\CN(G_\ab)}_\red $, and also
\begin{equation}\label{map2b}
 \overline \CG_{\red,\ab} \to  \ov{\CN(G_\ab)}_\red,
\end{equation}
which are surjective homomorphisms of tori.
The latter gives  
\begin{equation}\label{map2c}
X^*(\ov{\CN(G_\ab)})\hookrightarrow X^*( \overline \CG_{\red,\ab})=X^*(T_\CG)
\end{equation}
and
\begin{equation}\label{map3}
X_*(\overline \CG_{\red,\ab})=X_*(T_\CG)\to X_*(\ov{\CN(G_\ab)})=X_*(G_\ab)^I,
\end{equation}
where the last identification is by \cite[Prop. B.7.9, p. 678]{KalethaPrasad}, as above.
We obtain a commutative diagram
\begin{equation}\label{commTG}
  \begin{aligned}
   \xymatrix{
         X_*(T)^I \ar[r]^{(\ref{map2})} \ar[rd]& X_*(T_\CG) \ar[d]^{(\ref{map3})}  \\
    & X_*(G_\ab)^I,
        }
        \end{aligned}
    \end{equation}
with the diagonal map induced by $T\to G_\ab$.

\subsubsection{The pairing}\label{ss:pairing} Again, we let $\CG$ and $T$ be as above.
Here, we will define a natural bilinear and $\Gal(\bar\BF_p/\BF_p)$-equivariant pairing
  \begin{equation}\label{pair2}
  \langle\ ,\ \rangle: X_*(T)_I\times   X^*(T_\CG)\to \BQ
   \end{equation}
which plays a crucial role in our constructions.  

Consider the ``averaging" homomorphism 
\[
  X_*(T)\to X_*(T)^I\otimes_{\BZ}\BQ,
 \]
given by
  \[
  \lambda\mapsto \lambda^\diam:=\frac{1}{[I:I_\lambda]}\sum_{\bar \gamma\in I/I_\lambda}\bar\gamma\cdot \lambda,
  \]
  where $ I=I_{\BQ_p}\subset \Gal(\bar\BQ_p/\BQ_p)$ is the inertia and $I_\lambda$ is the subgroup of $I$ fixing $\lambda$  (cf. the construction in \cite[\S 2.8]{KoIsoI} applied to $I=I_{\BQ_p}$).
This homomorphism factors through the coinvariants $X_*(T)_I$ as a composition
\[
  X_*(T)\to  X_*(T)_I\xrightarrow{\diam}  X_*(T)^I\otimes_{\BZ}\BQ,
  \]
 where  
  \[
  X_*(T)_I\xrightarrow{\diam} X_*(T)^I\otimes_{\BZ}\BQ
  \]
  is a $\Gal(\bar\BF_p/\BF_p)$-equivariant 
homomorphism.

 Recall the identification $X_*(\ov\CT_{\red})=X_*(T)^I$ which we can compose with 
 $X_*(T)_I\to X_*(T)^I\otimes_{\BZ}\BQ$ to obtain $X_*(T)_I\to X_*(\ov\CT_{\red})\otimes_{\BZ}\BQ$.
Combining this with the natural pairing 
  \[
X_*(\ov\CT_{\red})\times   X^*(\ov\CT_{\red})\to \BZ,
  \]
  gives a bilinear pairing
  \begin{equation}\label{pair3}
  \langle\ ,\ \rangle: X_*(T)_I\times   X^*(\ov\CT_{\red})\to \BQ.
   \end{equation}
  This is the desired pairing (\ref{pair2}) in the Iwahori case $\CG=\CI$, when $X^*(\ov\CT_{\red})=X^*(T_\CI)$.
  
  In general, to obtain (\ref{pair2}), we compose (\ref{pair3}) with $X^*(T_\CG)\to X^*(\ov\CT_{\red})$ given by 
  $\ov\CT_{\red}\to \ov \CG_{\red,\ab}$ (or given by $T_\CI\to T_\CG$).
  
\begin{remark}\label{remark:psi}  Consider the composition
  \[
 \phi_\CG:  X_*(T)_I\xrightarrow{\diam} X_*(T)^I_\BQ\xrightarrow{(\ref{map2})_\BQ} X_*(T_\CG)_\BQ.
  \]
 Then, by the construction  of $\langle\ ,\ \rangle$, we have
 \begin{equation}
   \langle \lambda , \chi \rangle=   \langle \phi_\CG(\lambda) , \chi \rangle_{T_\CG},
 \end{equation}
 where 
  \[
    \langle\ ,\ \rangle_{T_\CG}: X_*(T_\CG)_\BQ \times   X^*(T_\CG)_\BQ\to \BQ,
\]
is the natural duality pairing for the torus $T_\CG$.
\end{remark}

 \bigskip

\subsection{Divisors on local models and the divisor  conjecture}

\subsubsection{Local models}\label{ss:LM}

Let $(\CG,\{\mu\})$ be local model data, so that $\CG$ is a parahoric group scheme for $G$ and $\{\mu\}$ is the $G(\bar\BQ_p)$-conjugacy class of a minuscule cocharacter $\mu: (\BG_m)_{\bar\BQ_p}\to G_{\bar\BQ_p}$.
As usual, we denote by $E\subset \bar\BQ_p$ the reflex field of $\{\mu\}$ over $\BQ_p$. We denote by $O_E$ the ring of integers of $E$ and by $k_E$ its residue field. We set $k=\ov k_E$ for an algebraic closure of the finite field $k_E$. We also denote by $E_0$ the maximal unramified extension of $\BQ_p$ which is contained in $E$. We let $\Mloc_{\CG,\mu}$ be the corresponding local model, which is a flat projective scheme over $O_E$ with
$\CG_{O_E}$-action.  

Suppose a maximal torus $T$ of $G$ is chosen as in \S \ref{ss:standardtorus} above, with its connected ft N\'eron model $\CT$. We denote by $W$ its (absolute) Weyl group and by $\wt W=N_T(\br\BQ_p)/\CT(\br\BZ_p)$ its Iwahori Weyl group. Recall  from \cite[\S 4.3]{PRS} the definition of the subset $\Lambda_{\{\mu\}}\subset X_*(T)_I\subset \widetilde W$ of the Iwahori-Weyl group $\widetilde W$: 
We can suppose that a member $\mu$ of $\{\mu\}$ factors through $T_{\bar\BQ_p}$ and consider the corresponding $W$-conjugacy class $\{\mu\}_T$ of geometric cocharacters of $T\subset G$. Let $\tilde \Lambda_{\{\mu\}} $ be the subset of elements of $\{\mu\}_T$ whose images in $X_*(T)\otimes_\BZ \BR$ are contained in some (absolute)
Weyl chamber corresponding to a Borel subgroup of $G$ containing $T$ and defined
over $\br\BQ_p$. Then $\tilde \Lambda_{\{\mu\}}$ forms a single $W_0$-conjugacy class, i.e. a $W_0$-orbit $\tilde \Lambda_{\{\mu\}}=W_0\cdot \mu$, since all such Borels are
$W_0$-conjugate. Here, $W_0=N_T(\br\BQ_p)/T(\br\BQ_p)$ is 
the (relative) Weyl group of $G$ over $\br\BQ_p$. We write $\Lambda_{\{\mu\}} $ or occasionally just $\Lambda_\mu$ for the image of $\tilde \Lambda_{\{\mu\}}=W_0\cdot \mu$ in $X_*(T)_I\subset \widetilde W$.

The $\{\mu\}$-admissible set $\Adm(\{\mu\})$ of the Iwahori Weyl group $\widetilde W$ 
of $G$ is defined to be the set of $w\in\wt W$ such that $w\leq t_{\bar\mu'}$, for some  $\bar\mu'\in \Lambda_{\{\mu\}}\subset X_*(T)_I$.  Here we use the Bruhat order wrt a fixed alcove $\frak a$ in the apartment $\CA$ defined by $S\subset T$, cf. \S \ref{ss:standardtorus}. In all of this, we identify
$\bar\mu'$ with the corresponding translation elements $t_{\bar\mu'}\in \widetilde W$. 
 Then $\Adm(\{\mu\})$ contains the image  $\ov{\{\mu\}}_T$ in $X_*(T)_I$ of $\{\mu\}_T$, cf. \cite[equivalence of (4.1) and (4.2)]{Hai}, hence $\Adm(\{\mu\})$ can also be defined to be the set of $w\in\wt W$ such that $w\leq t_{\bar\mu'}$, for some  $\bar\mu'\in \ov{\{\mu\}}_T$. 

\subsubsection{Statement of the divisor conjecture}
Set $K=\CG(\br\BZ_p)\subset G(\br\BQ_p)$ and consider the subgroup $W^K=N_T(\br\BQ_p)\cap \CG(\br\BZ_p)/ \CT(\br\BZ_p)$ of the Iwahori-Weyl group $\widetilde W$. 
Recall that the double cosets $W^Kt_{\bar\mu'} W^K\in W^K\bs \widetilde W/ W^K$ with $\bar\mu'\in \Lambda_{\{\mu\}}$ parametrize the irreducible components of  $\Mloc_{\CG,\mu}\otimes_{O_E}k$, the geometric special fiber of the local model $\Mloc_{\CG,\mu}$, see
\cite{AGLR}, \cite{GL}, for the general case, and \cite[Thm. 9.3]{PZ} for the local models we will use below. Hence, for each element $\bar\mu'\in \Lambda_{\{\mu\}}$ there is an irreducible component 
$Z_{\bar\mu'}$ of $\Mloc_{\CG,\mu}\otimes_{O_E}k$; we may also think of $Z_{\bar\mu'}$ as a (fibral, irreducible) Weil divisor on the scheme $\Mloc_{\CG,\mu}\otimes_{O_E}O_{\br E}$.

\begin{conjecture}\label{divconj}
There exists a unique, up to unique isomorphism, pair $({\rm P}_{\CG,\mu}, s_{\CG,\mu})$ of a $\CG$-equivariant $T_{\CG}$-torsor ${\rm P}_{\CG,\mu}\to \Mloc_{\CG,\mu}$ over the local model $\Mloc_{\CG,\mu}$, and a $\CG$-equivariant trivialization (section) $s_{\CG,\mu}$ of this torsor over the generic fiber $\Mloc_{\CG,\mu}\otimes_{O_E}E$, such that
the following condition holds: 

For a  character $\chi: T_{\CG, \br \BZ_p} \to (\BG_m)_{\br\BZ_p}$, consider the 
$(\BG_m)_{O_{\br E}}$-torsor 
\[
{\rm P}_\chi:=(\BG_m)_{O_{\br E}}\times_{\chi, T_\CG} \RP_{\CG,\mu} 
\]
over $\Mloc_{\CG,\mu}\otimes_{O_E}O_{\br E}$, obtained by pushing out ${\rm P}_{\CG,\mu}\otimes_{O_E}O_{\br E}$ by $\chi$. Denote by $\CL_\chi$ the invertible sheaf over 
$\Mloc_{\CG,\mu}\otimes_{O_E}O_{\br E}$ which corresponds to the 
$(\BG_m)_{O_{\br E}}$-torsor ${\rm P}_\chi$. The section $s_{\CG,\mu}$ induces a section $s_\chi$ of the $(\BG_m)_{\br E}$-torsor ${\rm P}_\chi\otimes_{O_{\br E}}\br E$ over $\Mloc_{\CG,\mu}\otimes_{O_E}\breve E$ and, hence, a (meromorphic) section of the  invertible sheaf
$\CL_\chi$ over $\Mloc_{\CG,\mu}\otimes_{O_E}O_{\br E}$. The condition is that, for all such $\chi$, the  divisor of this meromorphic section of $\CL_\chi$ is equal to 
\begin{equation}\label{divsum}
D_{\chi}:=\sum_{W^Kt_{\bar\mu'} W^K, \mu'\in \Lambda_{\{\mu\}}}e\cdot \langle \bar\mu',\chi\rangle\cdot  Z_{\bar\mu'},
\end{equation}
where $\langle\ ,\ \rangle$ is the pairing (\ref{pair2}) and $e$ is the ramification index of $E$ over $\BQ_p$.
\end{conjecture}

\begin{remark}\label{rem:coef}
{\rm a) In the above, by saying a ``$\CG$-equivariant $T_\CG$-torsor ${\rm P}_{\CG,\mu}\to \Mloc_{\CG,\mu}$", we mean as usual that 
${\rm P}_{\CG,\mu}$ supports a $\CG $-action which lifts the $\CG $-action on $\Mloc_{\CG,\mu}$ and commutes with the 
$T_\CG $-action.

b) Since $e=[I:I_{\mu'}]$, the values $\langle \bar\mu',\chi\rangle$ lie in $e^{-1}\BZ$, and so $e\cdot \langle \bar\mu',\chi\rangle$ are all integers.

c)  It is not obvious that the linear combination $D_\chi$ of the $Z_{\bar\mu'}$ in (\ref{divsum}) above is a Cartier divisor on $\Mloc_{\CG,\mu}\otimes_{O_E}O_{\br E}$ ; this is one consequence of the conjecture.}
\end{remark}

\begin{remark}\label{rem:Slodowy}
  The formula \eqref{divsum} bears a resemblance to a classical formula which expresses the first Chern classes of the line bundles on an (affine or classical) Schubert variety as a linear combination of the Weil divisors given by its Schubert subvarieties which are ``at the boundary'', see e.g. \cite[Cor. 3]{Slodowy} and the references therein. 
  
A question arising in this context is to give an explicit description for the divisor class group ${\rm Cl}(\br\Mloc_{\CG,\mu})$ of  $\br\Mloc_{\CG,\mu}=\Mloc_{\CG,\mu}\otimes_{O_E} O_{\br E}$. More precisely, let 
  \[{\rm Cl}^\red(\br\Mloc_{\CG,\mu})=\ker({\rm Cl}(\br\Mloc_{\CG,\mu})\to {\rm Cl}(\br\Mloc_{\CG,\mu}\otimes_{O_{\br E}}\br E).
  \]
   Assume that  $G_\ad$  is simple.  Then one may ask whether there is a $\Gal(\br E/ E)$-equivariant isomorphism 
  \begin{equation}
  {\rm Cl}^\red(\br\Mloc_{\CG,\mu})\simeq \BZ^{\Lambda_{\CG,\{\mu\}}}/\BZ . 
  \end{equation}
  Here $\BZ$ is diagonally embedded (this embedding reflects the fact that the special fiber is a principal divisor), and $\Lambda_{\CG,\{\mu\}}$ denotes the index set of \eqref{divsum} and can be identified with the set of irreducible components of $\Mloc_{\CG,\mu}\otimes_{O_E} k$, with its Galois action. 
\end{remark}

\subsection{The divisor conjecture in the tamely ramified case}

In this subsection, we will prove Conjecture \ref{divconj} in the tame case:

\begin{theorem}\label{tameconj}
Suppose that the group $G$ splits over a tamely ramified extension 
of $\BQ_p$ and that $p$ does not divide the order of the algebraic fundamental group $\pi_1(G_{\der, \bar\BQ_p})$. 
Then Conjecture   \ref{divconj} is true.
\end{theorem}

We will use the fact that, under these two conditions, the local model $\Mloc_{\CG,\mu}$ is constructed in \cite{PZ} by using mixed characteristic Beilinson-Drinfeld affine Grassmannians.

\subsubsection{Spread-out group schemes} 
Let $(G,\{\mu\})$ be a local Shimura pair over $F=\BQ_p$ or $\br\BQ_p$ with reflex field $E$.  Set $ O= O_F$  and denote by $k=k_F$ the residue field. Let $\CG$ be a parahoric group scheme for $G$. Suppose that $G$ splits over a tamely ramified extension of $F$
so the construction of \cite[\S 4]{PZ} applies. In particular, this construction produces  a smooth and affine group scheme $\und \CG$ over $ O[u]$ such that
\begin{itemize}
\item $\und\CG$ is reductive over $O[u, u^{-1}]$,
\item   $\und\CG\otimes_{ O[u], u\mapsto p} O\simeq \CG$.
\end{itemize}
Set 
\[
\CG_0=\und\CG\otimes_{ O[u], u\mapsto 0} O.
\]
This is a smooth and affine group scheme over $ O$, but no longer reductive. Its reduction modulo $p$ is isomorphic to the special fiber of the parahoric $\CG$. There is an exact sequence
\begin{equation}
1\to \CU\to \CG_0\to \CG_{0,\red}=(\CG_{ 0})_{\red}\to 1 ,
\end{equation}
where $\CG_{0,\red}$ is smooth, affine and reductive over $ O$ and $\CU$ is smooth unipotent over $ O$. This exact sequence lifts a similar exact sequence
\[
1\to U\to \bar\CG\to \bar\CG_{\red}\to 1
\]
for the special fiber $\bar\CG=\CG\otimes_{ O}k$ of the parahoric $\CG$. The maximal abelian quotient $(\CG_{0,  \red})_\ab$
of $\CG_{0,\red}$ is a torus over $ O$ which lifts the torus $(\bar\CG_{\red})_{\ab}$. Hence, we can identify $(\CG_{0,  \red})_\ab$ with the torus
$T_\CG$ of the previous section, i.e.
\[
(\CG_{0,  \red})_\ab=T_\CG.
\]
In particular, the construction gives a group scheme homomorphism
\begin{equation}\label{push}
\tau_0: \CG_0\to T_\CG
\end{equation}
over $ O$.

\subsubsection{The $T_\CG$-torsor over the local model}\label{ss:Bundles} Now recall  from \cite{PZ} some aspects of the construction of the local model $\Mloc_{\CG,\mu}$  over $ O_E$. Consider the Beilinson-Drinfeld affine Grassmannian 
\[
{\rm Gr}_{\und\CG,  O[u]}\to \Spec( O[u]).
\]
Let $R$ be an $ O[u]$-algebra with $u\mapsto r\in R$. By definition,  
an $R$-valued point of ${\rm Gr}_{\und\CG, O[u]}$ lifting the corresponding point of 
$\Spec( O[u])$ is given by a $\und\CG$-torsor over $\Spec(R[u])$ together with a 
trivialization of its restriction to $\Spec(R[u][(u-r)^{-1}])$. We also need
\[
{\rm Gr}_{\und\CG,  O}={\rm Gr}_{\und\CG,  O[u]}\otimes_{ O[u],u\mapsto p} O.
\]
By definition, ${\rm Gr}_{\und\CG,  O}$ parametrizes 
 isomorphism classes of pairs of a $\und\CG$-torsor over $\Spec(R[u])$ together with a trivialization of its restriction to the open subscheme $\Spec(R[u][ (u-p)^{-1}])$ and is representable by an ind-projective ind-scheme over $O$, cf. \cite[Prop. 6.5]{PZ}.
 Consider the loop/jet groups $\CL \und\CG$ and $\CL^+\und\CG$ over $O$ given by
 \begin{equation}\label{loop1BD}
\CL \und\CG(R)=\und\CG(R\lps u-p\rps[(u-p)^{-1}]),\quad \CL^+ \und\CG(R)=
\und\CG(R\lps u-p\rps).
 \end{equation}
Here $R\lps u-p\rps$ denotes the completion of $R[u]$ at the ideal $(u-p)$.
 By Beauville-Laszlo glueing 
 \[
 {\rm Gr}_{\und\CG,  O}=\CL \und\CG/\CL^+ \und\CG,
 \]
 see \cite[\S 6.2.4]{PZ}, \cite[\S 3.1]{ZhuIntro}. By definition \cite[Def. 7.1]{PZ}, the local model is a closed subscheme 
\[
\Mloc_{\CG,\mu}\hookrightarrow {\rm Gr}_{\und\CG,  O }\otimes_{ O}O_E.
\]
Recall that $\Mloc_{\CG,\mu}$ admits an action of $\CG $.
The main construction here is:

\begin{proposition}\label{mainTorsor}
a) There exists a $\CG $-equivariant $T_\CG$-torsor  
\[
\pi_{\Mloc}: \RP_{\CG,\mu}\to \Mloc_{\CG,\mu},
\]
i.e, a $T_\CG$-torsor such that the natural action of $\CG $ on $\Mloc_{\CG,\mu}$ lifts to 
$\RP_{\CG,\mu}$ and the lifted action commutes with the action of $T_\CG$ on $\RP_{\CG,\mu}$. 

b) The restriction of the $T_\CG$-torsor $\pi_{\Mloc}$ over the generic fiber $ \Mloc_{\CG,\mu}\otimes_{ O_E}E$ admits a $G_E=\CG\otimes_{ O}E$-equivariant section
\[
s_{\CG,\mu}: \Mloc_{\CG,\mu}\otimes_{ O_E}E\to \RP_{\CG,\mu}\otimes_{ O_E}E.
\]
\end{proposition}

 In Subsection \ref{ss:mult}, we will show that the pair $( {\rm P}_{\CG,\mu}, s_{\CG, \mu})$ satisfies the condition in Conjecture \ref{divconj}. This will provide the desired proof of the conjecture in this tamely ramified case, where the local model is given by the construction of \cite{PZ}.

 \begin{proof} Note that, by \cite[Thm. 19.5.1]{SWberkeley}, it is enough to construct the $T_\CG$-torsor as a cohomological fppf sheaf torsor in the sense of loc. cit.. We will need the Beilinson-Drinfeld affine Grassmannian  ${\rm Gr}_{\und\CG, O, 0,p}$ over $O$
 for $\und\CG$ and the two points $u=0$ and $u=p$, see \cite[\S 3.1]{ZhuIntro}. By definition, this parametrizes 
 isomorphism classes of pairs of a $\und\CG$-torsor over $\Spec(R[u])$ together with a trivialization of its restriction to the open subscheme $\Spec(R[u][u(u-p)^{-1}])$. 
 Restricting specializations along $\Spec(R[u][(u(u-p))^{-1}])\subset \Spec(R[u][(u-p)^{-1}])$ gives  
 \begin{equation}\label{restrict2BD}
 {\rm Gr}_{\und\CG,  O}\to {\rm Gr}_{\und\CG, O, 0,p}.
 \end{equation}
 Consider the loop/jet groups $\CL_{0,p}\und\CG$ and $\CL^+_{0,p}\und\CG$ given by
 \begin{equation}\label{loop2BD}
\CL_{0,p}\und\CG(R)=\und\CG(\Rcom[(u(u-p))^{-1}]),\quad \CL^+_{0,p}\und\CG(R)=
\und\CG(\Rcom),
 \end{equation}
 over $O$. Here, $\Rcom$ denotes the completion of $R[u]$ along the ideal generated by $u(u-p)$. 
 Note that we have natural $R$-algebra homomorphisms
 \[
 \Rcom\to R \lps u\rps,\quad  \Rcom\to R\lps u-p\rps,
 \]
where $R\lps u-p\rps=\wh{R[u]}_{(u-p)}$. By Beauville-Laszlo glueing,   $\CL_{0,p}\und\CG$ acts on  
${\rm Gr}_{\und\CG, O, 0,p}$ which we can write as a fpqc quotient over $O$,
\[
{\rm Gr}_{\und\CG,  O, 0,p}=\CL_{0,p}\und\CG/\CL^+_{0,p}\und\CG,
\]
cf. \cite[Prop. 3.2.9]{ZhuIntro}. For our purposes, we do not need to know that ${\rm Gr}_{\und\CG,  O, 0,p}$ is ind-representable over $O$ and we can just treat it as an fpqc sheaf. (However, this ind-representability should be true and can probably be shown following \cite[\S 6.2]{PZ}.
See \cite[Rem. 3.1.4]{ZhuIntro} for the function field case.)
In any case, we can see that both the generic and special fiber of ${\rm Gr}_{\und\CG,  O, 0,p}$ over $O$ are ind-representable:

If $p$ is a unit in $R$, then   
\[
\Rcom=R\lps u\rps\times R\lps  u-p \rps.
\]
This gives the  factorization of the generic fiber, cf. \cite[\S 3.2]{ZhuIntro},
\begin{equation}\label{factor1BD}
{\rm Gr}_{\und\CG, O, 0,p} [1/p]={\rm Gr}_{\und\CG, O, 0 }[1/p]\times_{\Spec( O[1/p])} 
{\rm Gr}_{\und\CG, O, p}[1/p].
\end{equation}
On the other hand, there is a canonical equivariant isomorphism between special fibers
\begin{equation}\label{factorSpBD}
{\rm Gr}_{\und\CG, O, 0,p}\otimes_O k={\rm Gr}_{\und\CG, O}\otimes k.
\end{equation}
(This isomorphism can also be viewed as part of the ``factorization structure" of the affine Beilinson-Drinfeld Grassmannians, cf. \cite[\S 3.2]{ZhuIntro}.)

We now continue with the construction. Since $\CG_0=\und\CG\otimes_{O[u], u\mapsto 0}O$, evaluating at $u=0$  gives 
\[
\und\CG( \Rcom )\to \CG_0(R)  
\]
which defines a  surjective  homomorphism
\[
\psi_0: \CL^+_{0,p}\und\CG \to \CG_0.
\]
Similarly, evaluating at $u=p$ gives  
\[
\psi_p: \CL^+_{0,p}\und\CG \to \CG.
\]
Composing $\psi_0$ with $\CG_0\to  T_\CG$ gives a   homomorphism
\[
\tau: \CL^+_{0,p}\und\CG \to  T_\CG.
\]
We can now consider the contracted product
\[
\CP_{\und\CG}:=\CL_{0,p}\und\CG \times_{\CL^+_{0,p}\und\CG, \tau} T_\CG
\]
which is a $T_\CG$-torsor 
\[
\pi_{0,p}: \CP_{\und\CG}\to {\rm Gr}_{\und\CG,  O, 0,p}
\]
 over ${\rm Gr}_{\und\CG,  O, 0,p}=\CL_{0,p}\und\CG/\CL^+_{0,p}\und\CG$, equivariant for the action of $\CL_{0,p}\und\CG$. For clarity, we note 
that we form the contracted product as the quotient fpqc sheaf of  $\CL_{0,p}\und\CG \times  T_\CG $ by the equivalence relation which identifies  $(\und g\cdot h, t)$ with $(\und g, \tau(h)\cdot t)$, when $h$ is a point of $\CL^+_{0,p}\und\CG$. We now equip $\CP_{\und \CG}$ with a (left) action of $\CL^+_{0,p}\und\CG$ given on points by
\begin{equation}\label{Newaction}
h\cdot (\und g, t)=(h\cdot \und g, \tau(h)^{-1}t).
\end{equation}
(We can see that  this respects the equivalence relation  defining the contracted product.) This action commutes with the action of $T_\CG$ and lifts the natural action of $\CL^+_{0,p}\und\CG$ on ${\rm Gr}_{\und\CG,  O, 0,p}$. 

\begin{remark}\label{rem:differentaction} This action \emph{differs} from an other 
natural $\CL^+_{0,p}\und\CG$-action on the contracted product which is obtained 
by acting only on the left factor. Under \eqref{Newaction}, $\CL^+_{0,p}\und\CG$ acts trivially on the fiber of   $\pi_{0,p}$ over the neutral point given by the trivial coset 
of $\CL_{0,p}\und\CG/\CL^+_{0,p}\und\CG$. 
\end{remark}

Pulling back $\pi_{0,p}$ along \eqref{restrict2BD} and then base changing to $O_E$ gives a $T_\CG$-torsor  
\[
\pi:  {\rm Gr}^1_{\und\CG,  O_E}\to  {\rm Gr}_{\und\CG,  O_E}
\]
  over $\Spec( O_E)$. We define the $T_\CG$-torsor 
\[
\pi_{\Mloc}:  \RP_{\CG,\mu}\to \Mloc_{\CG,\mu}
\]
as the restriction of $\pi$ to $\Mloc_{\CG,\mu}\hookrightarrow  {\rm Gr}_{\und\CG,  O_E}$.

\begin{lemma}\label{lem:embLMop}
The composition 
\[
f: \Mloc_{\CG, \mu}\to {\rm Gr}_{\und\CG,  O_E}\to {\rm Gr}_{\und\CG, O_E, 0,p}={\rm Gr}_{\und\CG, O, 0, p}\otimes_{O}{O_E}
\]
is relatively representable by closed immersions.
\end{lemma}

\begin{proof} It is easy to see that $f$ is a closed immersion on the generic fibers using \eqref{factor1BD}. Also, 
$f$ is a closed immersion on the special fibers since by  \eqref{factorSpBD}, the morphism ${\rm Gr}_{\und\CG,  O_E}\to {\rm Gr}_{\und\CG, O_E, 0,p}$ is an isomorphism on special fibers. Since $\Mloc_{\CG,\mu}$ is proper  the result now follows by an argument as in the end of the proof of \cite[Prop. 8.1]{PZ}.
\end{proof}

By Lemma \ref{lem:embLMop} we can view $\Mloc_{\CG,\mu}$ as a fpqc subsheaf of  ${\rm Gr}_{\und\CG, O_E, 0,p}$. This is preserved by the $\CL^+_{0,p}\und\CG$-action 
because $\Mloc_{\CG,\mu}$ is flat and this is true on the generic fiber (see also below). For the same reason, we see that the resulting $\CL^+_{0,p}\und\CG$ action on 
$\Mloc_{\CG,\mu}$ factors through the homomorphism $\psi_p: \CL^+_{0,p}\und\CG\to \CG$ given by   $u\mapsto p$ and it comes from the usual $\CG$-action on $\Mloc_{\CG,\mu}$.
The morphism $f$  is $\CL^+_{0,p}\und\CG$-equivariant with the action on the source as described above. It follows that  the action \eqref{Newaction} of $\CL^+_{0,p}\und\CG$ on $ \CP_{\und\CG}$   restricts to an action on $\RP_{\CG, \mu}$ which lifts the natural action on $\Mloc_{\CG,\mu}$ via $\psi_p$ as above. To show (a) it now remains to prove that the $\CL^+_{0,p}\und\CG$-action on $\RP_{\CG,\mu}$ factors through $\psi_p$. Then it gives an $\CG$-action on $\RP_{\CG,\mu}$ which lifts the $\CG$-action on $\Mloc_{\CG,\mu}$; the fact that this commutes with the $T_\CG$-action follows from the above.

We now discuss the generic fibers. This is useful for both completing the proof of (a) and for proving (b). 
In what follows, for simplicity, we omit some subscripts and write 
$\Mloc$ instead of $\Mloc_{\CG,\mu}$ and also write $\Mloc[1/p]$ for the generic fiber $\Mloc\otimes_{O_E}E$. In general, we write $X[1/p]$ instead of $X\otimes_{\BZ_p}\BQ_p$.
Recall the  factorization \eqref{factor1BD}
\begin{equation*}
{\rm Gr}_{\und\CG, O, 0,p} [1/p]={\rm Gr}_{\und\CG, O, 0 }[1/p]\times_{\Spec( O[1/p])} 
{\rm Gr}_{\und\CG, O, p}[1/p].
\end{equation*}
The group homomorphism $\tau$ factors
\[
\tau : \CL^+_{0,p}\und\CG  \to \CL^+_0\und\CG \xrightarrow{\   \ }    \CG_0\to T_\CG,
\]
where $\CL^+_0\und\CG(R)=\und\CG(R\lps u\rps)$, the first arrow is given by the projection, and $  \CL^+_0\und\CG\to \CG_0$ is obtained by evaluation at $u=0$. This implies that the $T_\CG$-torsor $\pi_{0,p}[1/p]$ is obtained by pulling back the generic fiber of the $T_\CG$-torsor
\[
\pi_0: \CL_{0}\und\CG \times_{\CL^+_{0}\und\CG, \tau} T_\CG\to \CL_{0}\und\CG/\CL^+_{0}\und\CG={\rm Gr}_{\und\CG,  O, 0 } 
\]
by    the projection
\[
{\rm Gr}_{\und\CG, O, 0,p}[1/p] \to {\rm Gr}_{\und\CG, O, 0 }[1/p]
\] 
to the first factor in \eqref{factor1BD}. On the other hand, under
\begin{equation}\label{LMtoBD}
\Mloc\hookrightarrow  {\rm Gr}_{\und\CG,  O_E}\to {\rm Gr}_{\und\CG, O, 0, p}\otimes_{O}{O_E}
\end{equation}
the generic fiber $\Mloc[1/p]$   embeds in the base change to $E$ of the second factor
in \eqref{factor1BD}. However, the $T_\CG$-torsor $\pi_M$ is obtained by restricting $\pi_{0,p}$ along \eqref{LMtoBD}. Hence we obtain an isomomorphism
\[
\RP_{\CG,\mu}[1/p]\simeq T_{\CG, E}\times_{\Spec(E)} \Mloc[1/p] 
\]
between $\RP_{\CG,\mu}[1/p]$ and the trivial $T_\CG$-torsor over $\Mloc[1/p]$. Under this, the action  of $\CL^+_{0,p}\und\CG$ on $\RP_{\CG,\mu}[1/p]$ is trivial 
on the factor $T_{\CG, E}$ and is the standard action on $\Mloc[1/p]$. In particular, it factors through the quotient
\[
  \CL^+_{0,p}\und\CG[1/p]\to \CG[1/p]=G
\]
obtained by reducing modulo $(u-p)$. Hence, it follows that the $T_\CG [1/p]$-torsor $\pi_\Mloc[1/p]$ is $\CG_E=G_E$-equvariant and  equivariantly trivial. This gives (b) provided
we complete the proof of (a).

For this it remains to show that
the $\CL^+_{0,p}\und\CG$-action on $\RP_{\CG,\mu}$ factors through $\psi_p$, i.e. through reducing modulo $u-p$. 

  This is clear on the generic fibers by the above discussion.
To show it in general, we start by  observing that, since $\RP_{\CG,\mu}$ is of finite type, the action factors 
  through the group scheme of $m$-jets
\[
\CL^{+,m}_{0,p}\und\CG:=\Res_{O[u]/((u(u-p))^m)/O}(\und\CG\otimes_{O[u]}O[u]/((u(u-p))^m)
\]
for some $m\gg 0$. Since $\und\CG$ is smooth over $O[u]$, $\CL^{+,m}_{0,p}\und\CG$ is smooth   over $O$. Write
\[
1\to \ker_m\to \CL^{+,m}_{0,p}\und\CG\to \CG\to 1.
\]
 Let $H\subset \ker_m$ be the flat closure of $\ker_m[1/p]$ in $\CL^{+,m}_{0,p}\und\CG$. The fppf quotient $\CL^{+,m}_{0,p}\und\CG/H$
 is representable (\cite{Ana}) by a smooth group scheme $\CG'$ with $G$ as generic fiber. It acts on $\RP_{\CG,\mu}$ since $\CL^{+,m}_{0,p}\und\CG$ acts on $\RP_{\CG,\mu}$ and $H$ acts trivially, since $H[1/p]$ acts trivially on $\RP_{\CG,\mu}[1/p]$ and $\RP_{\CG,\mu}$ is flat over $O_E$. We have a group scheme homomorphism $\CG'\to \CG$ extending the identity on the generic fiber $G$.
Now, by smoothness, $\CL^{+,m}_{0,p}\und\CG(\br O)\to \CG(\br O)$ is surjective and so $\CG'(\br O)$ also surjects onto $\CG(\br O)$. But $\CG'(\br O)$, $\CG(\br O)$ are both subgroups of $G(\br O[1/p])$,
so $\CG'(\br O)=\CG(\br O)$ and, by \cite[\S 1.7]{BTII}, $\CG'=\CG$.
The proof of (a) follows. Part (b) now also follows from the dicusssion about the generic fibers above.

The above completes the proof of the statement but in fact, we can also construct an ``explicit" canonical  section $s_{\CG,\mu}$ in (b), as follows. This section is functorial for change of groups and its construction is useful in the sequel.

By the construction, an $R$-valued point of $\RP_{\CG,\mu}$ is the
  isomorphism class 
$[ (\CE,\und\alpha, \beta_0))]$ of
\[
 (\CE,\und\alpha, \beta_0):  \quad \CE,\quad \und\alpha: \und\CG\xrightarrow{\sim} \CE_{|u\neq p}, \quad
\beta_0:   T_\CG\xrightarrow{\sim} T_\CG\times_{\tau_0, \CG_0} \CE_{|u=0},
\]
where $\CE$ is a $\CG$-torsor over $R[u]$ and the trivialization $\und\alpha$ is ``bounded by $\mu$'' along $u=p$. An $R$-valued point of $\Mloc$ is the
  isomorphism class 
$[ (\CE,\und\alpha)]$ of a pair as above; the $T_\CG$-torsor $\pi_\Mloc$ is obtained by forgetting $\beta_0$. 

Assume now $p$ is invertible in $R$. Then
\[
\Spec(R)=\Spec(R/uR)\hookrightarrow   \Spec(R[u])
\]
 factors 
\[
\Spec(R) \xrightarrow{\iota} \Spec(R[u][(u-p)^{-1}])\subset \Spec(R[u]).
\]
We can consider the restriction $\iota^* \und\alpha: \CG_0\xrightarrow{\sim} \CE_{|u=0}$ which is an isomorphism of  $\CG_0$-torsors over $\Spec(R)$. This gives a $T_\CG$-trivialization $\tau_0(\iota^* \und\alpha)$ of the push out $T_\CG\times_{\tau_0, \CG_0} \CE_{|u=0}$. We can now give a section $s=s_{\CG,\mu} $ of $\pi_M[1/p]$  
by setting
\begin{equation}\label{cansection}
s[(\CE, \und\alpha)]:=[(\CE, \und\alpha, \tau_0(\iota^*\und\alpha))]
\end{equation}
on $R$-valued points.  

We now check that $s$ is $\CG$-equivariant. By Beauville-Laszlo glueing we can interpret the $R$-valued points 
of $\RP_{\CG,\mu}[1/p]$, $\Mloc[1/p]$, as isomorphism classes
\[
 [(\CE^\wedge , \und\alpha^\wedge, \beta_0)]  \ ,\quad   [(\CE^\wedge , \und\alpha^\wedge)].
\] 
Here, $\CE^\wedge$ is a $\und\CG$-torsor over $\Rcom=R\lps u\rps\times R\lps u-p\rps$, and 
$\und\alpha^\wedge$ is a trivialization over $R\llps u\lrps\times R\llps u-p\lrps$ which extends to a trivialization over $R\lps u\rps$, and $\beta_0:   T_\CG\xrightarrow{\sim} T_\CG\times_{\tau_0, \CG_0} \CE_{|u=0}$
 is a trivialization over $\Spec(R)$.
Given $g\in \CG(R)=G(R)$, we lift it to 
\[
   \und g =(\wh g_0, \wh g_p)\in \CG(\Rcom)=\und \CG(R\lps u\rps)\times \und \CG(R\lps u-p\rps).
\]
Then  
\[
\und  g \cdot (\CE^\wedge, \und\alpha^\wedge\cdot \und g^{ -1}, \beta_0\cdot \tau((\wh g_0)_{u=0})^{-1})
,\qquad 
 \und g\cdot (\CE^\wedge, \und\alpha^\wedge)= (\CE,  \und\alpha^\wedge\cdot \und g^{-1}).
\]
(The corresponding isomorphism classes are independent of the choice of lift $\und g$ of $g$.)
We have $\iota^*(\und g)=(\wh g_0)_{|u=0}$.  Hence,
\begin{align*}
s(\und g\cdot (\CE^\wedge, \und\alpha^\wedge))&= s(\CE^\wedge,  \und\alpha^\wedge\cdot \und g^{-1})=(\CE^\wedge,  \und\alpha^\wedge\cdot \und g^{-1}, \tau(\iota^*(\und\alpha^\wedge\cdot \und g^{-1})))=\\
&=(\CE^\wedge,  \und\alpha^\wedge\cdot \und g^{-1}, \tau(\iota^*(\und\alpha^\wedge)\cdot  (\hat g_0)^{-1}))=\\
&=\und g\cdot  (\CE^\wedge, \und\alpha, \tau_0(\iota^*\und\alpha)\cdot \tau_0((\wh g_0)_{u=0})^{-1})=\und g\cdot s (\CE^\wedge, \und\alpha^\wedge).
\end{align*}
This shows the $\CG$-equivariance of $s_{\CG,\mu}$.
\end{proof}

\begin{remark}\label{rem:padicComplete}
If $R$ is $p$-adically complete and separated, sending $u$ to $0$ gives an $R$-algebra map $R\lps u-p\rps\to R$.
 In that case, the ``evaluation at $u=0$" we use in the proof above factors
\[
\Rcom\to R\lps u-p\rps\to R.
\] 
In this case, $R$-valued points of $\RP_{\CG,\mu}$ are given by isomorphism classes of triples $(\CE^\wedge, \und\alpha^\wedge, \beta_0)$, with
$\CE^\wedge$ a $\und\CG$-torsor over $R\lps u-p\rps$, and $\und \alpha:\und\CG\xrightarrow{\sim}(\CE^\wedge)_{|u\neq p}$, $\beta_0:T_\CG\xrightarrow{\sim} T_\CG\times_{\tau_0, \CG_0} \CE_{|u=0}$, trivializations. An element  $g\in \CG(R)$  acts on $\RP_{\CG,\mu}(R)$ as follows: We lift $g$ to $\und g\in \und\CG(R\lps u-p\rps)$ and set
\[
g\cdot (\CE^\wedge, \und\alpha^\wedge, \beta_0)=(\CE^\wedge, \und\alpha^\wedge\cdot \und g^{-1}, \beta_0\cdot \tau_0(\und g_{|u=0})^{-1})
\]
with isomorphism class independent of the lift. On the other hand,   $t\in T_\CG(R)$ acts by changing $\beta_0$ to $t\cdot \beta_0$. Since we do not need this we omit the details.
\end{remark}

 \subsubsection{Line bundles over the local model}\label{ss:BundlesLines} We continue with the above set-up. Suppose  that 
\[
\chi:   T_\CG\to \BG_m
\]
 is a character; by   pushing out $\pi_{\Mloc}$ along $\chi$, we obtain a $\BG_m$-torsor, i.e. a line bundle over $\Mloc_{\CG,\mu}$ which corresponds to an invertible sheaf $\CL_\chi$ so that 
 \[
 \BG_m\times_{\chi, T_\CG} \RP_{\CG,\mu}={\rm Isom}(\CL^\vee_\chi,\CO_{\Mloc_{\CG,\mu}}).
 \]
By the above, the invertible sheaf $\CL_\chi$  is $\CG $-equivariant, i.e. admits a $\CG $-action which lifts the $\CG $-action on $\Mloc_{\CG,\mu}$. The section $s=s_{\CG,\mu}$ above gives a non-zero global section $s_\chi$ of $\CL_\chi[1/p]$ over the generic fiber $\Mloc_{\CG,\mu}[1/p]$. We can think of $s_\chi$ as a meromorphic section of $\CL_\chi$ over $\Mloc_{\CG,\mu}$ whose divisor $\div(s_\chi)$ is supported on the special fiber.  Showing the condition in Conjecture \ref{divconj} amounts to proving that the multiplicities of $\div(s_\chi)$ are given by the formula (\ref{divsum}). As we will see later, for $\chi$ in a certain ``positive cone'', $s_\chi$ is an actual section of $\CL_\chi$ and so the divisor $\div(s_\chi)$ is effective for such $\chi$.

\begin{remark}\label{rem:factorBD}
Recall the  isomorphism \eqref{factorSpBD} of special fibers
\[
{\rm Gr}_{\und\CG, O, 0,p}\otimes_O k={\rm Gr}_{\und\CG, O}\otimes k.
\]
By \eqref{LMtoBD} we have
\[
\Mloc_{\CG,\mu}\otimes_{O_E}k_E\to {\rm Gr}_{\und\CG, O, 0,p}\otimes_O k_E={\rm Gr}_{\und\CG, O}\otimes k_E
\]
which agrees with the closed immersion
\[
\Mloc_{\CG,\mu}\otimes_{O_E}k_E\to  {\rm Gr}_{\und\CG, O}\otimes k_E
\]
given by the definition of $\Mloc_{\CG,\mu}$. 
 We also have an identification
 \[
 {\rm Gr}_{\und\CG, O}\otimes k_E=(L H/L^+ \CH)_{k_E},
 \]
where $L H/L^+ \CH$ is the   affine flag variety ${\rm Gr}_{\CH}$ for the twisted loop group $H=\und G\otimes_{O[u]}k\llps u\lrps$ over $k\llps u\lrps$ and
 parahoric $\CH=\und \CG\otimes_{O[u]}k\lps u\rps$. Here, 
  $ \CH\otimes_{k\lps u\rps}k = \CG\otimes_O k $, see \cite[Cor. 4.2]{PZ}.
 
 Suppose now 
$
 \chi:   T_\CG\to \BG_m
$
 is a character over $O$ as above. Using the identifications
 \[
 (\CH\otimes_{k\lps u\rps}k)_{\red}=( \CG\otimes_O k)_{\red}=(\CG_0\otimes_O k)_{\red},
 \]
 we see that $\chi$ gives a character $\ov\chi : L^+\CH\to \BG_m$ over $k$. The usual construction
 \[
 L H \times_{L^+\CH,\ov\chi}k
 \]
 now defines a $LH$-equivariant line bundle over ${\rm Gr}_{\CH}$. The restriction of this line bundle to $\Mloc_{\CG,\mu}\otimes_{O_E}k_E$  is isomorphic to the line bundle $\CL_\chi\otimes_{O_E}k_E$ obtained from the special fiber of the $T_{\CG}$-torsor $\RP_{\CG,\mu}\to \Mloc_{\CG,\mu}$. This follows from the  proof of Proposition \ref{mainTorsor}. Indeed,
  the $T_\CG$-torsor $\pi_\Mloc$ is given by restricting a similar contracted product. Note however, that the $L^+\CH$-structure which is obtained from the natural $LH$-structure on $\CL_\chi$ differs from the equivariant structure given by Proposition \ref{mainTorsor}, see also Remark \ref{rem:differentaction}.
\end{remark}

 \subsubsection{An example}\label{sss:ExampleLoop}

Suppose $G=\GL_n={\rm Aut}_{\BQ_p}(\BQ^n_p)$ and $\CG$ is the Iwahori group scheme  given as the stabilizer of the standard periodic $\BZ_p$-lattice chain $\{\Lambda_m\}_{m\in \BZ}$
\[
\cdots \subset   \Lambda_{i-1}\subset \Lambda_i\subset  \cdots ,
\] with 
\[
\Lambda_{-i}=p \BZ_p e_1\oplus \cdots \oplus p\BZ_p e_i\oplus \BZ_p e_{i+1}\oplus\cdots \oplus  \BZ_p e_n,
\]
for $0\leq i\leq n-1$, and $\Lambda_{j}=p^k\Lambda_i$, for $j=i-nk$, where $e_i$ is the standard basis of $\BQ_p^n$.

In this case, the group scheme $\und\CG$ over $\BZ_p[u]$ is obtained as the automorphisms of ``the corresponding" periodic chain of free $\BZ_p[u]$-modules
$\{\und\Lambda_i\}_{i\in \BZ}$ 
\[
  \cdots \subset \und\Lambda_{i-1}\subset \und\Lambda_i\subset \cdots   ,
\] 
with
\[
\und \Lambda_{-i}=u \BZ_p[u] e_1\oplus \cdots \oplus u\BZ_p[u] e_i\oplus \BZ_p[u] e_{i+1}\oplus\cdots \oplus  \BZ_p[u] e_n,
\]
for $0\leq i\leq n-1$, and $\und\Lambda_{j}=u^k\und\Lambda_i$, for $j=i-nk$. The corresponding global affine Grassmannian
${\rm G}_{\und\CG,\BZ_p}$ over $\BZ_p$ represents the functor which to a $\BZ_p$-algebra $R$ associates the set of chains $\{\und\CE_i\}_{i\in \BZ}$ of
projective finitely generated $R[u]$-modules  of rank $n$  contained in $\und\Lambda_0\otimes_{\BZ_p[u]}R[u, u^{-1}, (u-p)^{-1}]=R[u, u^{-1}, (u-p)^{-1}]^n$:
\[
  \cdots \subset \und\CE_{i-1}\subset \und\CE_i\subset \cdots  ,
\]
which are periodic $\und \CE_{i-kn}=u^k\und\CE_i$, for all $i$ and $k$, and with the following properties (see \cite[\S 6, \S 7.2]{PZ}):
 \begin{itemize}
\item For each $i$, the quotient $\und\CE_i/\und\CE_{i-1}$ is killed by $u$ and is  locally free of rank $1$ over $R$.
\item  There is some $N\geq 1$ such that
 \[
 (u-p)^N  \und \Lambda_i\otimes_{\BZ_p[u]}R[u]\subset \und\CE_i\subset  (u-p)^{-N}  \und \Lambda_i\otimes_{\BZ_p[u]}R[u],
 \]
for  all $i$.
\end{itemize}
  For $-n+1\leq i\leq 0$, set
\[
\CL_i:={\rm Hom}_R(\und\CE_i/\und\CE_{i-1}, (\und\Lambda_{i}/\und\Lambda_{i-1})\otimes_{\BZ_p}R)\simeq (\und\CE_{i}/\und\CE_{i-1})^\vee,
\]
which are locally free $R$-modules of rank $1$. The universal modules $\und\CL^{\rm univ}_i$, $-n+1\leq i\leq 0$, define $n$ invertible sheaves over ${\rm Gr}_{\und\CG,\BZ_p}$
and hence a $T_\CG=\BG_m^n$-torsor given as 
\[
\prod_{i=-n+1}^{i=0}{\rm Isom}(\und\CE^{\rm univ}_{i}/\und\CE^{\rm univ}_{i-1}, \und\Lambda_{i}/\und\Lambda_{i-1})=\prod_{i=-n+1}^{i=0}{\rm Isom}((\und\CL^{\rm univ}_i)^\vee, \CO).
\]
 By definition, this agrees with the $T_\CG$-torsor 
\[
\pi:  {\rm Gr}^1_{\und\CG,\BZ_p}\to {\rm Gr}_{\und\CG,\BZ_p}
\]
which appears in the proof of Prop. \ref{mainTorsor}.

Now fix the coweight $\mu=(1^{(r)}, 0^{(n-r)})$ and recall the comparison isomorphism between the Iwahori $\GL_n$ local model from \cite{RZbook} or \cite{Goertz} given by the ``classical definition"
and the local model $\Mloc_{\CG,\mu}\subset {\rm Gr}_{\CG,\BZ_p}$ given by the group-theoretic definition in \cite{PZ}.
 This comparison is explained, for example, in \cite[\S 7.2.1]{PZ}. Under this isomorphism, $R$-points of the classical local model given by $R$-submodules $\CF_i\subset \Lambda_i\otimes_{\BZ_p}R$ which are locally direct summands of rank $r$, correspond to $(u-p)  \und \Lambda_i\subset \und\CE_i\subset    \und \Lambda_i\otimes_{\BZ_p}R[u]$, by taking $\und\CE_i$ to be the inverse image of $\CF_i$ under the map $\und\Lambda_i\otimes_{\BZ_p[u]}{R[u]}\to \Lambda_i\otimes_{\BZ_p}R$ given by $u\mapsto p$. For simplicity of notation, set $\und\Lambda_{i, R}:=\und\Lambda_i\otimes_{\BZ_p[u]}{R[u]}$. Set also
 \[
 \CQ_i:=(\Lambda_i\otimes_{\BZ_p}R)/\CF_i=\und\Lambda_{i, R}/\und\CE_i.
 \]
 The standard property of the determinant functor on the derived category of
 perfect complexes of $R$-modules (\cite[Ch. I, Def. 4, Thm. 2]{KnM}) applied to the exact sequence of complexes
 \[
0\to [\und\CE_{i-1}\rightarrow  \und\CE_{i}]\to [ \und\Lambda_{i-1, R}\rightarrow \und\Lambda_{i, R}] \to [\CQ_{i-1}\to \CQ_i]\to 0
\]
 gives a canonical isomorphism between the locally free rank $1$ $R$-modules
 \[
 \det(\CQ_i)\otimes_R \det(\CQ_{i-1})^{-1}
 \]
 and
 \[
 (\und\Lambda_{i, R}/\und\Lambda_{i-1, R})\otimes_R (\und\CE_i/\und\CE_{i-1})^{\otimes-1}={\rm Hom}_R(\und\CE_i/\und\CE_{i-1}, (\und\Lambda_{i}/\und\Lambda_{i-1})\otimes_{\BZ_p}R).
 \]
 So we have
\begin{equation}\label{isoDet}
 \det(\CQ_i)\otimes_R \det(\CQ_{i-1})^{-1}\xrightarrow{\sim} {\rm Hom}_R(\und\CE_i/\und\CE_{i-1}, (\und\Lambda_{i}/\und\Lambda_{i-1})\otimes_{\BZ_p}R)=\CL_i.
\end{equation}

The isomorphisms (\ref{isoDet}) give a straightforward expression for the $T_\CG$-torsor $\RP_{\CG,\mu}$ over the local model $\Mloc_{\CG,\mu}$, in this example. In particular, sections of the $T_\CG$-torsor $\RP_{\CG,\mu}$ correspond via (\ref{isoDet}) to nowhere zero sections of the line bundles  $\det( \CQ^{\rm univ}_i)\otimes \det(\CQ^{\rm univ}_{i-1})^{-1}$, $i=-n+1,\ldots ,0$.  We can see that the section $s_{\CG,\mu}$ 
of the $T_\CG$-torsor $\RP_{\CG,\mu}[1/p]$ over the generic fiber $\Mloc_{\CG,\mu}[1/p]$, corresponds to the collection of  sections of $ \det( \CQ^{\rm univ}_i)\otimes \det(\CQ^{\rm univ}_{i-1})^{-1} $ induced by the isomorphisms $\CQ^{\rm univ}_{i-1}\xrightarrow{\sim} \CQ^{\rm univ}_i$ over $\Mloc_{\CG,\mu}[1/p]$.
 
\begin{remark}\label{rem:PEL}
{\rm Versions of the isomorphism (\ref{isoDet}) can be given in many other EL and PEL cases, and in particular in the case of  
${\rm GSp}_{2g}$, by comparing with corresponding
 explicit ``classical" descriptions of the local models $\Mloc_{\CG,\mu}$, see \cite[\S 7, \S 8]{PZ}. We will revisit this in Section \ref{s:PEL}.  }
\end{remark}

\subsubsection{Proof of Conjecture \ref{divconj} in the tame case;  multiplicities}\label{ss:mult}
We continue with the notations of the previous paragraph. Again, we first extend the base to $\br \BQ_p$ and  set $O=\br\BZ_p$. For simplicity, we will often omit the base change to $O$ and write $\CG$ instead of $\CG_O$.

We would like to 
 calculate the multiplicity $m_{\bar\mu', \chi}$ of the divisor $\div(s_\chi)$ along the irreducible  component $Z_{\bar\mu'}$ of the geometric special fiber of $\Mloc_{\CG,\mu}$ which corresponds to an extreme  element $\bar\mu'$ of the $\{\mu\}$-admissible set ${\rm Adm}^{\CG}(\mu)\subset W^K\bs \widetilde W/W^K$:
\[
  \div(s_\chi)=\sum_{\bar \mu'} m_{\bar\mu', \chi}\cdot Z_{\bar\mu'}.
  \]

Our goal is to show:

 \begin{proposition}\label{multiProp}
 Let $\chi:  T_\CG \to \BG_m$. Let 
   $Z_{\bar\mu'}$ be the component of the geometric special fiber of $\Mloc_{\CG,\mu}$  corresponding to the coset of the translation element $\bar \mu'\in \Lambda_{\{\mu\}}\subset X_*(T)_I\subset \widetilde W$.  The  multiplicity of the divisor $\div(s_\chi)$ along $Z_{\bar\mu'}$ is given by
  \[
  m_{ \mu', \chi}=[\br E:\br\BQ_p]\cdot \langle\bar\mu',  \chi\rangle,
  \]
  where $\langle\bar\mu',  \chi\rangle$ is the value of the pairing (\ref{pair2}).
  \end{proposition}

By the above, this implies Theorem \ref{tameconj}, i.e. the divisor conjecture \ref{divconj} in the tame case.
   
\begin{remark} a) Note that  for $\Gamma=I$, $[\Gamma:\Gamma_{\mu'}]=[\br E:\br\BQ_p]$, so  the values $\langle\bar\mu',  \chi\rangle$ of the pairing (\ref{pair2}) lie in $ [\br E:\br\BQ_p]^{-1}\cdot \BZ$, hence the RHS is indeed an integer. 

b)  Proposition \ref{multiProp} says that the value $\langle\bar\mu', \bar \chi\rangle$ gives the ``absolute" multiplicity of the divisor $\div(s_\chi)$, i.e. the $p$-adic valuation of a local equation for $\div(s_\chi)$   at the generic point $\eta_{\bar \mu'}$ of  $Z_{\bar\mu'}$, i.e., 
  \[
  \div(p^{\langle\bar\mu', \bar \chi\rangle})=\div(s_\chi)
  \]
 in $\Spec(\CO_{\Mloc, \eta_{\bar\mu'}})$.
  \end{remark}
  
  \begin{proof} Recall that we have base changed to $O=\br\BZ_p$, so  $\CG$, $\BG_m$, etc., is over $O$. Recall the choices of tori $S\subset T=Z_G(S)$  above. By the tameness assumption, there exists a finite tamely ramified extension $\ti F/\br\BQ_p$ with $\Gamma=\Gal(\ti F/\br\BQ_p)\simeq \BZ/e\BZ$, such that  
 \[
G\simeq ({\rm Res}_{\ti F/\br\BQ_p}  (H\otimes_{\br\BQ_p}\ti F))^{\Gamma},\quad T\simeq ({\rm Res}_{\ti F/\br\BQ_p} (T_H\otimes_{\br\BQ_p}\ti F))^{\Gamma},
\]
where  $H$ is a split Chevalley form of $G$, and $T_H=\BG_{m}^r$ is a maximal torus 
of $H$, for a ``twisted'' action of $\Gamma$. For simplicity, we set $F=\br\BQ_p$.
  
  Under this tameness assumption, the constructions of \cite{PZ} apply to the torus $T$ and we have smooth affine group schemes $\und\CT$, resp. $\CT_0$, over $O[u]$, resp. $O$, which extend the parahoric  group scheme $\CT$ (i.e. connected ft Neron model), and its special fiber $\bar\CT=\CT\otimes_{O}k$. Note that by the construction in \cite{PZ}, there are natural  identifications of   Galois groups
\[
\Gamma=\Gal(\ti F/F)=\Gal(O[v]/O[u])=\Gal(F\llps v-p\lrps/F\llps u-p\lrps)=\Gal(k\llps v\lrps/k\llps u\lrps) ,
\]
 where $v^e=u$, and a
  $\Gamma$-equivariant identification
\[
X_*(T)=X_*(\und\CT_{F\llps t\lrps})=X_*(\und\CT_{k\llps t\lrps})
\]
  where $t=u-p$. In particular, we can identify the  coinvariants 
  \[
    X_*(T)_I= X_*(\und\CT_{k\llps t\lrps})_I.
  \]
  Here, $X_*(T)_I$ and $X_*(\und\CT_{k\llps t\lrps})_I$ are the targets of the Kottwitz homomorphism for  the tori obtained from $\und\CT$ by specializing to the two discretely valued fields $F$ and $k\llps u\lrps$, by $O[u]\to F$, $u\mapsto p$, and $O[u]\to k\llps u\lrps$.  (On the other hand, the extension $F\llps v-p\lrps/F\llps u-p\lrps$ of discretely valued fields,  is unramified.)

 For simplicity, write $\ti O$ for $O_{\ti F}$ and, as above, $t=u-p$. The following  is a mixed characteristic version of \cite[Prop. 3.4]{ZhuCoh}, see also \cite[\S 7.1.1]{PZ}. 
  For $\lambda\in X_*(T)$, we denote by 
 $E_\lambda$ its field of definition, this is a subfield of $\ti F$ containing $F=\breve\BQ_p$.

  \begin{proposition}\label{ZhuProp}
  For each $\lambda\in X_*(T)$, there is a morphism  over $\Spec(O)$,
  \[
  \fks_\lambda: \Spec(O_{E_\lambda})\to \CL\und\CT,
  \]
  such that, denoting by $\eta$, resp. $s$, the generic, resp. special point of $\Spec(O_{E_\lambda})$, 
  \begin{itemize}
  \item[i)] The element 
 $
  \fks_\lambda(\eta)\in \CL\und\CT(E_\lambda)=\und\CT(E_\lambda\llps u-p\lrps) 
 $
 maps under the Kottwitz  homomorphism 
 \[
 \CT(E_\lambda\llps u-p\lrps)\to X_*(\und\CT_{E_\lambda\llps u-p\lrps})
 \]
  to the image of $\lambda$ under the identification $X_*(T)=X_*(\und\CT_{E_\lambda\llps u-p\lrps})$.

 \item[ii)] The element
  $
  \fks_\lambda(s)\in \CL\und\CT(k )=\und\CT(k\llps t\lrps),
  $
   maps under the Kottwitz homomorphism
    \[
   \CT(k\llps t\lrps)\to X_*(\und\CT_{k\llps t\lrps})_I,
   \]
    to the image of $\lambda$ under the natural map $X_*(T)\to X_*(T)_I=X_*(\und\CT_{k\llps t\lrps})_I$.
 \end{itemize}
\end{proposition}

\begin{proof}
Recall we consider $ O[u]\to O[v]$, $u\mapsto v^e$, with Galois group $\Gamma=\BZ/e\BZ$, so that $\ti O/O$ 
is the base change of this cover over $u=p$, i.e. $\ti O=O[v]/(v^e-p)$. Then $v\,{\rm mod}\, (u-p)$ defines $\varpi\in \ti O$ with $\varpi^e=p$. We have  models $\und{\ti\CT}$ and $\und\CT$ of the torus over $O[u]$ which specialize 
to $\ti\CT$ and $\CT$ under $u\mapsto p$. For example, choose a basis $\omega_1,\ldots , \omega_r$ of the $\Gamma$-module $X^*(T_H)=\BZ^r$ which corresponds to $T_H\simeq \BG_m^r$. Then
\[
\und{\ti\CT}=({\rm Res}_{O[v]/O[u]}( \BG^r_m\otimes_{\BZ_p} O[v]))^{\Gamma},
\]
where the $\Gamma$-action extends in the obvious way, and $\und{\CT}$ is the neutral component of $\und{\ti\CT}$.
For an $O$-algebra $R$,  
 \[
 \CL\und{\ti\CT}(R)=\und{\ti\CT}(R\lps u-p\rps[(u-p)^{-1}]).
 \]
This is equal to the $\Gamma$-fixed points of 
\[
({\rm Res}_{O[v]/O[u]}\BG_m^r)(R\lps u-p\rps[(u-p)^{-1}])=\prod_{j=1}^r (R\lps u-p\rps[(u-p)^{-1}]\otimes_{O[u]} O[v])^\times.
\]
Recall $T_H=\BG_{m, F}^r$ and  $T(\ti F)=T_H(\ti F)=(\ti F^\times)^r$ but with ``twisted" $\Gamma$ action. 
Write
\[
\ti O=O[v]/(v^e-p),\quad \varpi=v\,{\rm mod}\,(v^e-p)
\]
and let $\zeta$ be a primitive $e$-th root of unity in $O$. We can choose a generator $\gamma$ of $\Gamma$ such that 
so $\gamma(\varpi)=\zeta\cdot\varpi$ and  $\gamma(v)=\zeta\cdot v$.
For $1\leq i\leq e$, we set
\[
x_i:=-\varpi\otimes 1+\zeta^i\otimes v\in \ti O\lps u-p\rps \otimes_{O[u]}O[v].
\]
We will consider reduction modulo $(u)$:
\[
(\ti O\lps u-p\rps\otimes_{O[u]}O[v])/(u)=\ti O\otimes_{O}O[v]/(v^e)=\ti O[v]/(v^e),
\]
and
\[
(\ti O\lps u-p\rps[\frac{1}{u-p}]\otimes_{O[u]}O[v])/(u)=\ti F\otimes_{O}O[v]/(v^e)=\ti F[v]/(v^e),
\]
followed by reduction modulo $(v)$. 

Note that the image of $x_i$ under reduction modulo $(v)$ is the constant
\[
x_i\equiv -\varpi \in \ti O\subset \ti O[v]/(v^e)\subset \ti F[v]/(v^e)=\ti F\oplus \ti F\cdot v\oplus \cdots \oplus \ti F\cdot v^{e-1}
\]
which is a unit in $\ti F[v]/(v^e)$. Hence,  $x_i$ is a unit,
\[
x_i\in (\ti O\lps u-p\rps[\frac{1}{u-p}]\otimes_{O[u]}O[v])^\times.
\]
Recall the basis $\omega_1,\ldots , \omega_r$ of $X^*(T_H)=\BZ^r$. To distinguish with our other pairing, we use the notation 
\[
\langle\ ,\ \rangle_H: X_*(T_H)\times X^*(T_H)=X_*(T)\times X^*(T)\to \BZ
\]
for the standard  pairing.

As in \cite[Prop. 3.4]{ZhuCoh}, we define
 \[
 \fks_\lambda\in \prod_{j=1}^r(\ti O\lps u-p\rps[(u-p)^{-1}]\otimes_{O[u]}O[v])^\times.
\]
  by setting 
 \[
 \omega_j(\fks_\lambda):=\prod_{i=1}^e x_i^{\langle \lambda, \gamma^i\cdot \omega_j\rangle_H}.
 \]
The action of $\Gamma$ on $(\ti O\lps u-p\rps[(u-p)^{-1}]\otimes_{O[u]}O[v])^\times$ obtained 
by acting by Galois automorphisms on the second factor $O[v]$ is such that $\gamma(x_i)=x_{i+1}$.
This implies that $ \fks_\lambda$ is  fixed for the twisted $\Gamma$-action and so
  \[
 \fks_\lambda\in \CL\und{\ti\CT} (\ti O).
 \]
 In fact, we can see that $\fks_\lambda\in \CL\und{\ti\CT} (O_{E_\lambda})$, where we recall  the reflex field $E_\lambda\subset \ti F$ 
 of $\lambda\in X_*(T)$.
Note that we can view $\fks_\lambda\in \CL\und{\ti\CT} (\ti O)$ as a morphism over $\Spec(O[u])$,
\[
\fks_\lambda: \Spec(\ti O\lps u-p\rps[\frac{1}{u-p}])\to \und{\ti\CT} .
\]
 To show that this factors through the neutral component $\und{\CT}\subset \und{\ti\CT}$, we 
consider the reduction modulo $(u)$. Now $\fks_\lambda$ induces 
\[
\fks_\lambda\,{\rm mod}\,(u): \Spec(\ti O\lps u-p\rps[\frac{1}{u-p}]/(u))=\Spec(\ti F)\to \und{\ti\CT}_{|u=0} .
\]
We compose this map with 
\[
\und{\ti\CT}_{|u=0}=(\prod_{j=1}^r(\Res_{O[v]/(v^e)/O}(\BG_m\otimes_{O}O[v]/(v^e))))^\Gamma\to  (\und{\ti\CT}_{|u=0})^\red=(\prod_{j=1}^r \BG_{m, O})^\Gamma.
\] 
The composition is
\[
\Spec(\ti F)\to (\prod_{j=1}^r \BG_{m, O})^\Gamma\subset \prod_{j=1}^r \BG_{m, O},
\]
with $j$-component
\[
(-\varpi)^{\langle \lambda, {\rm Tr}_\Gamma \omega_j\rangle_H}.
\]
So this composition 
\[
\Spec(\ti F)\to (\und{\ti\CT}_{|u=0})^\red=(\prod_{j=1}^r \BG_{m, O})^\Gamma\subset \BG_{m,O}^r,
\]
is also
\begin{equation}\label{initialterm}
\Spec(\ti F)\xrightarrow{-\varpi} \BG_{m,O}\xrightarrow{{\rm Tr}_\Gamma\cdot \lambda} \BG_{m,O}^r.
\end{equation}
But this lands in the neutral component (the torus part) of $(\und{\ti\CT}_{|u=0})^\red=(\prod_{j=1}^r \BG_{m, O})^\Gamma$, as desired. 
Hence  $\fks_\lambda$ factors through $\CL\und{\CT} $ and we obtain
$
 \fks_\lambda\in \CL\und{\CT} (\ti O).
$
It then also follows that
\[
 \fks_\lambda\in \CL\und{\CT} (O_{E_\lambda})
\]
as desired.

Let us now consider the special fiber $ \fks_\lambda(s)\in \CL\und{\CT} (k)$; this is given by
\[
\fks_\lambda(s)= \fks_\lambda\,{\rm mod}\, (\varpi): \Spec(k\llps t\lrps)\to \und{\CT}\otimes_{O[t]}k[t]
\]
with $t=u$. Hence, by the above definition,
\[
\omega_j( \fks_\lambda\,{\rm mod}\, (\varpi))=
\prod_{i=1}^e (\zeta^i \otimes v)^{\langle \lambda, \gamma^i\cdot \omega_j\rangle_H}=\pm v^{\langle {\rm Tr}_\Gamma\cdot \lambda, \omega_j\rangle_H }.
\]
This and the compatibility of the Kottwitz invariant with norms, gives part (ii). 

It remains to show part (i). Since $ \ti F\llps u-p\lrps/E_\lambda \llps u-p\lrps$ is an unramified extension of  discretely valued fields, it is enough to calculate the Kottwitz invariant of $\fks_\lambda(\eta)\in \und\CT(E_\lambda\llps u-p\lrps)$ after this extension, i.e. consider $\fks_\lambda(\eta)$ in $\und\CT(\ti F\llps u-p\lrps)$. Now 
\[
  v^e-\varpi^e=(v-\varpi)\cdot\prod_{i=1}^{e-1}(\zeta^i v-\varpi)
\]
\[
\ti F\lps u-p\rps\otimes_{O[u]}O[v]=\prod_{i=1}^e\ti F\lps \zeta^i v-\varpi\rps.
\]
Using this we can see that $\und\CT(\ti F\llps u-p\lrps)=T_H(\ti F\llps v-\varpi\lrps)$. We now have
\[
\omega_j( \fks_\lambda(\eta))=(v-\varpi)^{\langle\lambda, \omega_j\rangle_H}\cdot A(\lambda, j) ,
\]
where
\[
A(\lambda, j)=\prod_{i=1}^{e-1}(\zeta^i v-\varpi)^{\langle\lambda, \gamma^i\omega_j\rangle_H} ,
\]
a unit in $\ti F\lps v-\varpi\rps$. This shows that the Kottwitz invariant of 
\[
\fks_\lambda(\eta)\in \und\CT(\ti F\llps u-p\lrps)=T_H(\ti F\llps v-\varpi\lrps)
\]
 is given by $\lambda$ in $X_*(T)$.\end{proof}

The above now gives a point  
\[
[\fks_\lambda]: \Spec(O_{E_\lambda})\to {\rm Gr}_{\und\CT, O}=\CL\und\CT /\CL^+\und\CT ,
\]
  of the Beilinson-Drinfeld affine Grassmannian for the torus $\und\CT$.

\begin{remark}\label{rem:spoints}
(i) We can also interpret the points $[\fks_\lambda]$
following \cite[Prop. 21.3.1]{SWberkeley}, as follows.
Consider the $\Gal(\bar F/F)$-module $X_*(T)$
which defines an \'etale $\Spec(F)$-scheme $\und X$ and consider the normalization 
$\und X^{\rm int}$ of $\Spec(O)$  in $\und X$. Each $\lambda\in X_*(T)$ defines a point \[
[\lambda]:\Spec(O_{E_\lambda})\to \und X^{\rm int}.
\]
The points $[\fks_\lambda]$ combine to give a morphism of $\Spec(O)$-schemes
\[
\iota: \und X^{\rm int}\to {\rm Gr}_{\und\CT,O}
\]
such that $[\fks_\lambda]=\iota\cdot [\lambda]$, for each $\lambda$.

(ii) If $\lambda$ is defined over $\breve \BQ_p$, the point  $\fks_\lambda\in \CL\und\CT_{O}(O_{E_\lambda})$ is simply given by the image of $t=u-p\in \CL\und\BG_{m, O}(O)=(O\lps u-p\rps(u-p)^{-1})^* $ under $\CL\und\BG_{m, O}\to\CL\und\CT $. 
\end{remark}

We now continue with the proof of Proposition \ref{multiProp}: 
\smallskip

\emph{Step 1. The torus case.} Consider $(\CT,\{\lambda\})$ and a character $\chi: \CT_0\to \BG_m$. Recall that $\CT_0=\und\CT_{u=0}:=\und\CT\otimes_{O[u], u\mapsto 0}O$ and that $\chi$ factors through $\CT^{\red}_0=(\und\CT_{|u=0})^\red$. The construction in \S \ref{ss:BundlesLines}, for $\CG=\CT$, $\chi$, and $R=O_{E_{\lambda}}$, gives a $\BG_m$-torsor over $\Spec(O_{E_\lambda})$ which is
\[
\CL_\chi=\BG_m\cdot \chi((\fks_\lambda^{-1})_{|u=0}).
\]
The calculation in the proof of Proposition \ref{ZhuProp}, see in particular (\ref{initialterm}), gives
 \[
 \chi((\fks_\lambda^{-1})_{|u=0})=(-\varpi)^{\langle\lambda, {\rm Tr}_\Gamma\cdot \chi\rangle_H}.
 \]
The section of $\CL_\chi$ over $E_\lambda$, constructed in \S \ref{ss:BundlesLines}, is given by $\chi((\fks_\lambda^{-1})_{|u=0})\in E^\times_\lambda$. Hence, the corresponding divisor over $\Spec(O_{E_\lambda})$ is
\[
{\rm val}_{E_\lambda}((-\varpi)^{\langle\lambda, {\rm Tr}_\Gamma\cdot \chi\rangle_H})\cdot [s],
\]
with $[s]$ the closed point of $\Spec(O_{E_\lambda})$. Since $\varpi$ is a uniformizer of $\ti F$ and $\ti F/E_\lambda$
is totally ramified, this is equal to 
\[
\frac{1}{[\ti F:E_\lambda]}\cdot \langle\lambda, {\rm Tr}_\Gamma\cdot \chi\rangle_H\cdot [s].
\]
The multiplicity is
\[
\frac{1}{[\ti F:E_\lambda]}\cdot \langle\lambda, {\rm Tr}_\Gamma\cdot \chi\rangle_H=\frac{1}{[\ti F:E_\lambda]}\cdot 
\langle{\rm Tr}_\Gamma\cdot \lambda, \chi\rangle_H.
\]
Since $[\ti F:E_\lambda]=|\Gamma_\lambda|$, we have
\[
\frac{1}{[\ti F:E_\lambda]}\cdot {\rm Tr}_\Gamma\cdot \lambda=\frac{1}{|\Gamma_\lambda|}\cdot \sum_{\gamma\in\Gamma}\gamma\cdot\lambda=\sum_{\bar\gamma\in \Gamma/\Gamma_\lambda}\bar\gamma\cdot \lambda.
\]
It follows that this multiplicity is 
\[
[\Gamma:\Gamma_\lambda]\cdot \langle \lambda^\diamond, \chi\rangle_H=[E_\lambda: F]\cdot \langle \lambda^\diamond, \chi\rangle_H=[E_\lambda: \br\BQ_p]\cdot \langle \bar\lambda, \chi\rangle,
\]
where again $\bar\lambda \in X_*(T)_I$ denotes the image of $\lambda$, as claimed. This shows Proposition \ref{multiProp} in the case of a torus.
 \smallskip

  \emph{Step 2: The general case.} We show the general case by  reducing it to the case of a torus, which was handled above.
We start with $(\CG,\{\mu\})$ and  the local model $\Mloc_{\CG,\mu}$. As usual, it is enough to base change to $O=\br\BZ_p$ and consider the situation over this base.   In particular, $E$ is a finite extension of $\br\BQ_p$.  We choose $S\subset T=Z_G(S)$ in $G_{F}$ as above, such that the parahoric $\CG$ corresponds to a point in the apartment of $S$, as usual.  Let $\lambda=\mu'\in W_0\cdot \mu\subset X_*(T)$; then $\lambda\in \{\mu\}$ and $E_\lambda=E_\mu=E$. By the construction of \cite{PZ}, we have $\und\CT\hookrightarrow \und\CG$, and  
this gives
\[
\CT_0\hookrightarrow \CG_0, \quad
{\rm Gr}_{\und\CT,O}\to {\rm Gr}_{\und\CG, O}.
\]
 By composing with $[\fks_\lambda]$, we obtain a point
 $
 [\fkt_\lambda]: \Spec(O_{E })\to {\rm Gr}_{\und\CG, O}\otimes_{O}O_E
 $
which factors through the local model, 
\[
 [\fkt_\lambda]: \Spec(O_{E })\to \Mloc_{\CG,\mu}\subset {\rm Gr}_{\und\CG, O}\otimes_{O}O_E
\]
because its generic fiber does so. Note that $\Spec(O_{E })$ here is naturally identified with the local model $\Mloc_{\CT, \lambda}$
for the torus. Let now $\chi: \CG_0\to \BG_m$ be a character, inducing the pair $(\CL_\chi, s_\chi)$
over $\Mloc_{\CG,\mu}$ by the construction in  \S \ref{ss:BundlesLines}. By the construction of this pair, we see that the restriction  $[\fkt_{\lambda}]^*(\CL_\chi, s_\chi)$   to $\Mloc_{\CT, \lambda}=\Spec(O_{E})\hookrightarrow {\rm Gr}_{\und\CT,O}$, is isomorphic to the pair obtained, by the same construction applied to the local model pair $(\CT,\{\lambda\})$ and the composition 
$\chi_{|\CT_0}: \CT_0\to \CG_0\xrightarrow{\chi} \BG_m$.

We can now see  that the special fiber $[\fkt_\lambda](s)$
lands in the $\CG$-orbit  $Z_{\bar\lambda}$ of $\Mloc_{\CG,\mu}\otimes_{O_E}k$ given by the coset $W^K\cdot t_{\bar\lambda}\cdot W^K\subset \widetilde W$ of the 
translation element $t_{\bar\lambda}$ in the Iwahori-Weyl group $\widetilde W$ which corresponds to the image $\bar\lambda$ of $\lambda$ in $X_*(T)_I$. This follows from Proposition \ref{ZhuProp} (ii).

Since $\Mloc_{\CG,\mu}$ is smooth over $O_{E}$ along the stratum $Z_{\bar\lambda}$ and $[\fks_{\lambda}]$ is an $O_{E}$-valued point, the multiplicity of the divisor $\div(s_\chi)$ along $Z_{\bar\lambda}$, is determined by the restriction $[\fkt_{\lambda}]^*(\CL_\chi, s_\chi)$ of the pair $(\CL_\chi, s_\chi)$ to $\Spec(O_{E})$; the restricted pair is an invertible sheaf over $\Spec(O_{E})$ with a trivialization over the 
generic point $\Spec(E)$. However, the restriction  $[\fkt_{\lambda}]^*(\CL_\chi, s_\chi)$   to $\Mloc_{\CT, \lambda}=\Spec(O_{E})\hookrightarrow {\rm Gr}_{\und\CT, O}$ is isomorphic to the pair obtained, for the local model pair $(\CT,\{\lambda\})$ and the composition 
$\chi_{|\CT_0}: \CT_0\to \CG_0\xrightarrow{\chi} \BG_m$. Hence, the multiplicity is given by the
formula we showed above in the torus case. This gives the desired result and completes the proof of Proposition \ref{multiProp} and hence also of Theorem \ref{tameconj}.
\end{proof}

\begin{remark}\label{rem:simpler}
 a) The proof of Theorem \ref{tameconj} is significantly simpler if $G$ splits over $\br\BQ_p$.
Indeed, then $T$ also splits over $\br\BQ_p$ and the construction of the points ${\mathfrak s}_\lambda$ in Proposition \ref{ZhuProp} is  straightforward, see Remark \ref{rem:spoints} (ii).

b) Proposition \ref{mainTorsor} but also the proof of Conjecture \ref{divconj} in the tame case (Theorem \ref{tameconj}), should  generalize  to local models for groups which are the Weil restriction of scalars of a tame group by a wildly ramified extension, by using Levin's construction   \cite{Levin} in that case. For the general case, one can hope to use the constructions of \cite{Lourenco}, \cite{FHLR}, 
\cite{AGLR} etc., as informed by Scholze-Weinstein's theory of $v$-sheaf local models. For example, the construction of $[\fks_\lambda]$ is straightforward for $v$-sheaf local models, see \cite[Prop. 21.3.1]{SWberkeley}. 

\end{remark}

\section{Toric schemes }

\subsection{Toric embeddings of $T_\CG$}\label{ss:toric31}

\subsubsection{Semigroups}\label{ss:semigroups} Recall $\Lambda_{\{\mu\}}\subset X_*(T)_I$. This generates
  the (rational polyhedral) convex cone
\begin{equation}
\sigma_{\{\mu\}}:=\{\sum\nolimits_{\lambda \in \Lambda_{\{\mu\}}}r_\lambda\cdot \lambda\ |\ r_\lambda\geq 0\}\subset (X_*(T)_I)\otimes_{\BZ}\BR .
\end{equation}
(We will sometimes say that $\sigma_{\{\mu\}}$ is the \emph{positive hull} of $\Lambda_{\{\mu\}}$.)
Here we recall from \S \ref{ss:pairing} that the ambient vector space can be identified with
$$
X_*(T)_I\otimes_{\BZ}\BR=X_*(T)^I\otimes_{\BZ}\BR=X_*(\ov \CT_\red)\otimes_{\BZ}\BR.
$$
We also write $\sigma_\mu$ for $\sigma_{\{\mu\}}$.  On the character group side, this defines a saturated finitely generated semigroup $S_\mu\subset X^*(\ov \CT_\red)$,  
\begin{equation}
S_{\mu}=\{\chi\in X^*(\ov \CT_\red)\mid \langle \mu', \chi\rangle\geq 0,\forall\, \mu'\in \Lambda_{\{\mu\}}\}.
\end{equation}
 In other words,
 \[
 S_\mu=\sigma_\mu^\vee\cap X^*(\ov \CT_\red),
 \]
where $\sigma_\mu^\vee$ is the dual cone of $\sigma_\mu$.

More generally, let $\CG$ be a parahoric group scheme with corresponding torus $T_\CG$ over $\BZ_p$, as in \S \ref{ss:211}. Then we define
\begin{equation}
\begin{aligned}
\sigma_{\CG, \mu}&= \varphi_\CG(\sigma_{\mu})\subset X_*(T_\CG)\otimes_{\BZ}\BR,\\
S_{\CG,\mu}&=\{\chi\in X^*(T_\CG)\mid \langle \mu', \chi\rangle\geq 0,\forall\, \mu'\in \Lambda_{\{\mu\}}\}.
\end{aligned}
\end{equation}
Recall here the map $\varphi_\CG:X_*(T)^I\to X_*(T_\CG)$ from Remark \ref{remark:psi}.
Then $S_{\CG, \mu}$ is a  finitely generated, saturated sub semigroup of $X^*(T_\CG)$. 

\subsubsection{The toric scheme}\label{ss:toricscheme} 
Continuing with the set-up above we set
\[
\br Y_{\CG,\mu}:=\br Y_{S_{\CG, \mu}}=\Spec(\br\BZ_p[S_{\CG, \mu}])
\]
which comes with an action of $T_\CG$ and an equivariant morphism
\begin{equation}\label{tomaY}
T_{\CG}\otimes_{\BZ_p}\br\BZ_p\to \br Y_{ \CG,\mu}
\end{equation}
over $\Spec(\br\BZ_p)$. Here the affine ring of $\br Y_{\CG,\mu}$ is given by the  semigroup ring of $S_{\CG,\mu}$, and the morphism \eqref{tomaY} by the inclusion
\[
\Gamma(\br Y_{\CG,\mu}, \CO)=\br\BZ_p[S_{\CG,\mu}]\subset \Gamma(T_{\CG}\otimes_{\BZ_p}\br\BZ_p, \CO)=\br\BZ_p[X^*(T_{\CG})].
\]
By the stability of $\Lambda_{\{\mu\}}$ under $\Gal(\br\BQ_p/E_0)$, the semigroup $S_{\CG,\mu}$
is also stable and  so the affine scheme $\br Y_{ \CG,\mu}$  and the morphism \eqref{tomaY} descend to define an affine scheme $Y_{ \CG,\mu}$ with $T_\CG$-action and equivariant $T_{\CG}\otimes_{\BZ_p}{O_{E_0}}\to  Y_{ \CG, \mu}$ over $O_{E_0}$.  
(Recall from \S \ref{ss:LM} that  $E_0$ is the maximal unramified extension of $\BQ_p$ contained  in $E$.) We have 
\[
Y_{\CG, \mu}=({\rm Res}_{\BZ_{p^m}/O_{E_0}}\br Y_{S_{\CG, \mu}})^{\Gal(\BQ_{p^m}/E_0)} ,
\]
with $m$ large enough so that the action of $\Gal(\br\BQ_{p}/E_0)$ on $S_{\CG, \mu}$ factors through $\Gal(\BQ_{p^m}/E_0)$.

\begin{remark}
Note that $S_{\CG, \mu}$  contains the subgroup  of $X^*(T_\CG)$,
\begin{equation}
X^*(T_\CG)^{\mu, 0}=\sigma_{\CG, \mu}^\perp\cap X^*(T_\CG)=\{\chi\in X^*(T_\CG)\mid \langle \mu', \chi\rangle=0,\forall\, \mu'\in \Lambda_{\{\mu\}}\}.
\end{equation}
Here $X^*(T_\CG)^{\mu, 0}\subset X^*(T_\CG)$ is co-torsion free and stable under $\Gal(\breve\BQ_p/E_0)$. Hence  the quotient is the character group of a subtorus $T_{\CG, \mu}\subset T_\CG\otimes_{\BZ_p} O_{E_0}$ defined over $O_{E_0}$,  i.e.,
\begin{equation}
X^*(T_{\CG, \mu})=X^*(T_\CG)/X^*(T_\CG)^{\mu, 0}.
\end{equation} 
We denote by $\ov S_{\CG, \mu} $ the image of $S_{\CG,\mu}$ in $X^*(T_{\CG, \mu})$. We obtain an affine variety $\br Z_{S_{\CG,\mu}}=\br Z_{\ov S_{\CG,\mu}}$ over $\br\BZ_p$ with an action of $T_{\CG, \mu}\otimes_{O_{E_0}}\br\BZ_p$ and an equivariant map 
\begin{equation}\label{torma}
T_{\CG, \mu}\otimes_{O_{E_0}}\br\BZ_p\to \br Z_{S_{\CG,\mu}} .
\end{equation}
The affine ring of $\br Z_{S_{\CG,\mu}}$ is given by the  semigroup ring of $\ov S_{\CG,\mu}$, and the morphism \eqref{torma} by the inclusion
\[\Gamma(\br Z_{S_{\CG,\mu}}, \CO)=\br\BZ_p[\ov S_{\CG, \mu}]\subset \Gamma(T_{\CG, \mu}\otimes_{O_{E_0}}\br\BZ_p, \CO)=\br\BZ_p[X^*(T_{\CG, \mu})].\]
By the stability of $\Lambda_{\{\mu\}}$ under $\Gal(\bar\BQ_p/E_0)$, the semigroup $\ov S_{\CG,\mu}$ is also  invariant and  so the affine scheme $\br Z_{S_{\CG,\mu}}$  and the morphism \eqref{torma} descend to $T_{\CG, \mu} \to  Z_{S_{\CG,\mu}}$ over $O_{E_0}$.

We can also give the affine $O_{E_0}$-scheme  $Y_{\CG,\mu}=Y_{S_{\CG, \mu}}$ and the morphism $T_{\CG,O_{E_0}}\to Y_{S_{\CG, \mu}}$ by pushout 
of $Z_{S_{\CG,\mu}}$  and   $T_{\CG, \mu} \to  Z_{S_{\CG,\mu}}$ along $T_{\CG,\mu}\hookrightarrow 
T_{\CG,O_{E_0}}$, in particular 
\begin{equation}\label{torma2}
T_{\CG,O_{E_0}}\to Y_{S_{\CG, \mu}}=T_{\CG, O_{E_0}}\times_{T_{\CG, \mu}}Z_{S_{\CG, \mu}}.
\end{equation}
  When the inclusion $T_{\CG, \mu}\subset T_{\CG, O_{E_0}}$ is proper, the toric variety $Y_{S_{\CG, \mu}}$  has a \emph{toric factor} in the sense of \cite[\S 3.3]{CLS}. Note that there is a natural identification of stacks 
  $$
  [T_{\CG,\mu} \bs Z_{S_{\CG, \mu}}]=[T_\CG \bs  Y_{S_{\CG, \mu}}].
  $$
 
In the next subsection, we will introduce ``non-degeneracy" conditions under which $\ov S_\mu$ generates the group $X^*(T_{\CG, \mu})$ and hence the morphisms \eqref{torma} and \eqref{torma2}  are open  embeddings and define   toric schemes $T_{\CG, \mu}\hookrightarrow  Z_{S_{\CG, \mu}}$, resp. $T_{\CG,O_{E_0}}\hookrightarrow  Y_{S_{\CG, \mu}}$.
\end{remark}
 
 \begin{remark}\label{rem:toruscase}
Let us consider the extreme case when $G=T$ is a torus. In this case, there is only one parahoric group scheme, which is an Iwahori group scheme $\CI$. Then $\{\mu\}$ consists of a single element $\mu$. Let $\bar\mu$ be the image of $\mu$ in $ X_*(T)_I\otimes_{\BZ}\BR$. If $\bar\mu=0$, then $T_{\CI, \mu}$ is trivial and there are no nontrivial torus embeddings. If $\bar\mu$ is non-trivial (this case falls under the name of \emph{ab-nondegenerate}, cf. Definition \ref{defabn} below), then $T_{\CI, \mu}$ is a one-dimensional torus over $O_{E_0}$. Since  the Frobenius element of $\Gal(\br\BQ_p/E_0)$ fixes $\bar\mu$, we see that then $T_{\CI, \mu}$ is isomorphic to $\BG_m$ over $O_{E_0}$. In this case, the morphism  $T_{\CI, \mu}\otimes O_{E_0}\to Z_{S_\mu}$ is the embedding $\BG_{m, O_{E_0}}\hookrightarrow \BA^1_{O_{E_0}}$. 
\end{remark}

 \subsubsection{More general semigroups}\label{ss:generalsemigroups}\label{moregen}
More generally, let $S$ be a semigroup of $X^*(T_\CG)$ satisfying the following conditions: 
\begin{enumerate}
\item $S$  is saturated. 
\item $S$ is contained in $S_{\CG,\mu}$ and generates the same group.
\item $S$ contains $X^*(T_\CG)^{\mu, 0}$.
\item $S$ is stable under the action of $\Gal(\bar\BQ_p/E_0)$. 
\end{enumerate}
On the dual side, we can also prescribe $S$ as $S=\tau^\vee\cap X^*(T_\CG)$, where  $\tau\subset X_*(T_\CG)_\BR$
is a choice of a rational polyhedral cone  such that $\tau\supset\sigma_{\CG, \mu}$, and satisfying the following conditions, in which we set $\sigma=\sigma_{\CG, \mu}$: 
\begin{enumerate}
\item Let $W_\tau$, resp. $W_\sigma$ the maximal sub vector space contained in $\tau$, resp. $\sigma$. Then the inclusion $W_\sigma\subset W_\tau$ is an equality. 
\item The inclusion $\tau^\perp\subset\sigma^\perp$ of subspaces of $(X^*(T_\CG))\otimes_{\BZ}\BR$ is an equality.
\item $\tau$ is stable under the action of $\Gal(\br\BQ_p/E_0)$. 
\end{enumerate}
Note that the cone $\tau$ is uniquely determined by $S$ and vice versa, via $\tau^\vee=\BR_+S$.

We let $\ov S$ be the image of $S$ in $X^*(T_{\CG, \mu})$, and define as before the affine variety $Z_S$
  with an action of $T_{\CG, \mu}$. Under these conditions we obtain a $T_{\CG, \mu}$-equivariant map over $O_{E_0}$ 
\begin{equation}\label{relt}
Z_{S_{\CG,\mu}}\to Z_{S} .
\end{equation}
Via push-out, we obtain as before a $T_{\CG}$-equivariant map
\begin{equation}\label{relt2}
Y_{S_{\CG, \mu}}\to Y_{S} .
\end{equation}
 When $S=\tau^\vee\cap X^*(T_\CG)$, we will sometimes write $Y_\tau$ instead of $Y_S$.

\subsubsection{Nondegeneracy}\label{ss:nondegeneracy}

In this subsection, we consider a pair $(G, \{\mu\})$, where $\mu$ is not necessarily minuscule. The pair $(G, \{\mu\})$ is called \emph{strictly convex} if  the convex cone $\sigma_{\{\mu\}}$ of $ X_*(T)_I \otimes_{\BZ}\BR$ is strictly convex, i.e. it does not contain an $\BR$-line through the origin. Equivalently, the dual cone of $\sigma_{\mu}$ in $ X^*(T)^I \otimes_{\BZ}\BR$ contains a non-empty open subset.  

We also introduce the notion of a pair $(\CG, \{\mu\})$ to be strictly convex: This means that 
\[
\sigma^\vee_{\CG, \mu}:=\{ \chi\in X^*(T_\CG)_\BR\ |\ \langle\lambda ,\chi \rangle\geq 0, \ \forall\, \lambda\in \Lambda_{\{\mu\}}\},
\]
 contains a non-empty open subset of $X^*(T_\CG)_\BR$. Equivalently, $S_{\CG, \mu}$ generates $X^*(T_\CG)$ as a group, or $\ov S_{\CG, \mu}$ generates $X^*(T_{\CG, \mu})$ as a group. Note that $(\CI, \{\mu\})$ is strictly convex iff $(G, \{\mu\})$ is strictly convex. 

Recall the homomorphism
\begin{equation}
\phi: X_*(T)_I\to \Omega_G=\pi_1(G)_I\to \pi_1(G_{\rm ab})_I =X_*(G_{\rm ab})_I.
\end{equation}
Under the composed map, all $\lambda\in \Lambda_{\{\mu\}}\subset X_*(T)_I$ have the same image $\bar\mu_{\rm ab}$ in $X_*(G_{\rm ab})_I$, since they all have the same image in $\Omega_G$. 
\begin{definition}\label{defabn}
The pair $(G, \{\mu\})$ is called \emph{ab-nondegenerate} if $\bar\mu_{\rm ab}$ is not a torsion element, i.e., if the image of $\bar\mu_{\rm ab}$ in $ X_*(G_{\rm ab})_I\otimes\BR$ is non-zero.
\end{definition}
Since the second map above is an isogeny, ab-nondegeneracy is equivalent to the condition that the image of $\Lambda_{\{\mu\}}$ in $\pi_1(G)_I$ (which consists of a single element)  is not a torsion element. 
\begin{proposition}\label{propAo}
 If  $(G, \{\mu\})$ is ab-nondegenerate, then $(G, \{\mu\})$ is  strictly convex. Similarly, then $(\CG, \{\mu\})$ is  strictly convex, for any parahoric $\CG$ of $G$.
\end{proposition}

\begin{proof} Assume that $\bar\mu_{\rm ab}\in X_*(G_{\rm ab})_I$ is not torsion, so $\bar\mu_{\rm ab}$ is not zero in the vector space $ X_*(G_{\rm ab})_I\otimes_{\BZ}\BR$. Suppose by contradiction that there is $x\in \sigma_{\{\mu\}}$ such that $-x\in \sigma_{\{\mu\}}$ also. 
We can write $x=\sum_{\lambda\in \Lambda_\mu}r_\lambda\cdot \lambda$, and also
$-x=\sum_{\lambda\in \Lambda_\mu}r'_\lambda\cdot \lambda$, with  $r_\lambda\geq 0$ and $r'_\lambda\geq 0$. Apply the homomorphism 
\[
\phi: X_*(T)_I\to  X_*(G_{\rm ab})_I
\]
as above.  We obtain 
\[
\phi(-x)=(\sum_{\lambda\in \Lambda_\mu}r'_\lambda)\cdot \bar\mu_{\rm ab}=-\phi(x)=-(\sum_{\lambda\in \Lambda_\mu}r_\lambda)\cdot \bar\mu_{\rm ab}.
\]
Hence, $\sum_{\lambda\in \Lambda_\mu}(r'_\lambda+r_\lambda)=0$, which gives a contradiction unless $x=0$, hence $\sigma_{\{\mu\}}$ is strictly convex. 

Now let us consider $\CG$ and $\sigma_{\CG, \mu}=\varphi_\CG(\sigma_\mu)\subset X_*(T_\CG)\otimes_{\BZ}\BR$. Using \eqref{commTG}, the same argument proves that $\sigma_{\CG, \mu}$ is strictly convex. But this is equivalent to the fact that $\ov S_{\CG, \mu}$ generates $X^*(T_{\CG, \mu})$ as a group. 
\end{proof}
 
\begin{corollary}\label{cor:Hodgetype} Under any one of the following hypotheses $(G, \{\mu\})$ is ab-nondegenerate:
\begin{enumerate}
\item if there exists a character $\chi: G\to \BG_m$ such that $\chi\circ \mu$ is not trivial.

\item If $(G, \{\mu\})$ is of local Hodge type, i.e. $\mu$ is minuscule and   there exists an embedding
\[
  (G, \{\mu\})\hookrightarrow (\GL_n,\{\mu_d\})
\]
with $\mu_d(a)=\diag(a^{(d)}, 1^{(n-d)})$,  $1\leq d<n$. 

\item If $\mu$ is non-trivial minuscule, $G_{\rm ab}$ splits over an unramified extension,
and $G_\der$ is simply-connected.
\end{enumerate}
\end{corollary}
\begin{proof}
(1)  Indeed,  then $\bar\mu_{\rm ab}\in X_*(G_{\rm ab})_I$ maps to a non-zero integer under $X_*(G_{\rm ab})_I\to X_*(\BG_m)_I=X_*(\BG_m)=\BZ$, hence it is not torsion.

(2) This follows from (1) because the composition with the determinant gives   a character $\chi$ as in (1).

(3) The hypotheses imply that 
$X_*(G_{\rm ab})=X_*(G_{\rm ab})_I$ is torsion-free. It is enough to check that $ \mu_{\rm ab}\in X_*(G_{\rm ab})$ is not trivial. But if it were, 
$\mu$ would factor through $G_\der=G_{\rm sc}$ which would contradict the assumption that $\mu$ is minuscule. 
\end{proof}
On the other hand, if $G$ is adjoint, then $(G, \{\mu\})$ is not ab-nondegenerate. The following proposition shows that Proposition  \ref{propAo}  is sharp under mild hypotheses.

 \begin{proposition}\label{ab-stric} Consider $(G, \{\mu\})$ such that  $\mu$ is non-trivial and such that the torus $G_{\rm ab}$ splits over an unramified extension. If  the  cone $\sigma_{\{\mu\}}$ of $ X_*(T)_I\otimes_{\BZ}\BR$ is  strictly convex, then  $(G, \{\mu\})$ is ab-nondegenerate. 
 \end{proposition}
 \begin{proof}
   Indeed, assume that $\bar\mu_{\rm ab}$ is torsion. After extension of scalars to $\br F$, we may assume that $G$ is quasi-split and that a positive multiple of $\mu$ factors through $G_{\rm der}$. Passing to the adjoint group, we may assume that $G$ is simple and $\mu$ is non-trivial. Denote by $S$ a maximal split torus and by $T$ the centralizer of $S$, a maximal torus of $G$. The set $\Lambda_{\{\mu\}}$ consists of non-zero elements, all conjugate under the relative Weyl group $W_0$. The assertion follows from the next lemma applied to the reduced root system associated to the system of relative roots  in $X_*(S)\otimes\BR=X_*(T)_I\otimes\BR$.  
\end{proof}
\begin{lemma}\label{lim}
Let $(V, R)$ be an irreducible root system with Weyl group $W$, and let $\mu\in V^*$ be non-zero. Then the convex hull of the finite set $\{\mu'=w\mu\mid w\in W\}$ contains an open neighborhood of the origin. 
\end{lemma}
\begin{proof} Let $\mu'=w\mu$ and choose the positive closed Weyl chamber $\ov C$ such that $\mu'\in \ov C$. Consider the translate by $\mu'$ of the negative of the obtuse Weyl chamber, 
\[
\ov C^\vee=\{ x\in V^*\otimes \BR\mid x=\sum x_i\alpha_i^\vee, \ x_i\geq 0, \ \forall\, i\},
\]
where $\alpha_i^\vee$ ranges over the simple coroots.  We claim that this translate contains the origin. Indeed, this is equivalent to the existence of $x_i\geq 0$ such that $\mu'=\sum x_i\alpha_i^\vee$. This follows from the inclusion $\ov C\subset\ov C^\vee$, cf. \cite[ch. VI, \S 1.6, Prop. 18]{Bou}.  In fact, $0$ lies in the interior of this translate because $\ov C\setminus \{0\}$ is contained in the interior of $\ov C^\vee$, cf. \cite[Lem. 2.18]{Lim}. The assertion follows because the convex hull in question is the intersection of the above translates for varying $\mu'$ in the $W$-orbit, cf. \cite[Lem. 12.14]{AB}. 
\end{proof}

\begin{remark}
Let $(G, \{\mu\})$ be ab-nondegenerate.  Then  the map \eqref{torma2} is an open embedding and  $Y_{\CG, \mu}$ is a toric embedding of $T_\CG\otimes_{\BZ_p}O_{E_0}$. Furthermore, for $S\subset S_{\CG,\mu}$ as in \S \ref{ss:generalsemigroups}, $Y_S$ is also a toric embedding of $T_\CG\otimes_{\BZ_p}O_{E_0}$ and the map \eqref{relt2} induces the identity on $T_\CG\otimes O_{E_0}$.
\end{remark}

\begin{proposition}\label{dimT1}
Assume that $(G, \{\mu\})$ is ab-nondegenerate, and let $\CG$ be a parahoric group scheme. Then
\[
\dim T_{\CG, \mu}\leq 1+\sum_i {\rm rank}_{\breve \BQ_p}(G_\ad\otimes\breve\BQ_p)_{ i}, 
\]
where the sum is over the simple factors  of $G_\ad\otimes\breve\BQ_p$ with non-trivial component $\mu_{\ad, i}$ of $\mu_\ad$. Here ${\rm rank}_{\breve \BQ_p}$ means the split rank over $\breve\BQ_p$. 

Furthermore, if $\CG$ is an Iwahori group scheme, there is equality between the two sides. 
\end{proposition}
\begin{proof} By the functoriality in the parahoric \S \ref{ss:211}, \S \ref{ss:standardtorus}, we obtain for a suitable Iwahori $\CI\to \CG$ a surjective homomorphism of tori $T_{\CI, \mu}\to T_{\CG, \mu}$. Hence we are reduced to the case where $\CG$ is an Iwahori group scheme. Then
\[
\dim T_{\CG, \mu}={\rm rank}\, X^*(T)^I-{\rm rank}\, (X^*(T)^{I})^{  \mu, 0} ,
\]
where $ (X^*(T)^{I})^{  \mu, 0} =\{\chi\in X^*(T)^I\mid \langle \mu', \chi\rangle=0,\forall\, \mu'\in \Lambda_{\{\mu\}}\}$. Consider the following diagram with exact rows, 
\begin{equation}\label{CharGroupsDiag}
\begin{aligned}
 \xymatrix{
     0 \ar[r]    &  C_\BQ\ar[r]& X^*(T_\der)^I_\BQ \ar[r]& D_\BQ\ar[r] &0  \\
      0 \ar[r]   &  (X^*(T)^I)^{\mu, 0}_\BQ \ar[r]\ar[u] & X^*(T)^I_\BQ \ar[r]\ar[u] & X^*(T_{\CI, \mu})_\BQ\ar[r]\ar[u] &0\\
      0 \ar[r]   &  A_\BQ\ar[r]\ar[u] & X^*(G_\ab)^I_\BQ \ar[r]\ar[u]  & B_\BQ\ar[r]\ar[u]  &0, 
        }
        \end{aligned}
\end{equation}
where the subscript $\BQ$ means, as usual, tensoring with $\BQ$. Here $A=(X^*(T)^I)^{\mu, 0}\cap X^*(G_\ab)$, and $B$, $C$, $D$ are defined as cokernels of the obvious maps between lattices. The hypothesis implies that ${\rm rank}\, B=1$. On the other hand, ${\rm rank}\, D={\rm rank}\, X^*(T_\der)^I-{\rm rank}(X^*(T_\ad)^I)^{\mu, 0}$. By Lemma \ref{lim}, we  have $(X^*(T_{\ad, i})^I)^{\mu, 0}=(0)$ when $\mu_{\ad, i}$ is non-trivial, and 
$(X^*(T_{\ad, i})^I)^{\mu, 0}=X^*(T_{\ad, i})^I$ when $\mu_{\ad, i}$ is trivial. We deduce that
\[
\begin{aligned}
\dim T_{\CI, \mu}=\, &{\rm rank} \, X^*(T_\der)^I-{\rm rank}\, (X^*(T_\ad)^I)^{\mu, 0} +1\\
 =\, &1+\sum_i {\rm rank}_{\breve \BQ_p}(G_{\ad, i}\otimes_{\BQ}\br\BQ_p)
\end{aligned}
\]
with the sum over the simple factors as in the statement.
\end{proof}

 \subsubsection{Blanket assumptions}\label{blanket}
 For the rest of the paper, we  make the following simplifying assumption on $(\CG, \{\mu\})$:  
 \begin{altenumerate}
 \item The pair $(\CG, \{\mu\})$ is strictly convex.  This is equivalent to the dual cone $\sigma_{\CG, \mu}^\vee$ having the full dimension. As a consequence, our toric schemes contain the torus  $T_\CG\otimes_{\BZ_p}O_{E_0}$ as an open subscheme. 
 \vskip0.01in
 
In addition, we often assume:
 \vskip0.01in
 
 \item  The subgroup $X^*(T_\CG)^{\mu, 0}$ of  $X^*(T_\CG)$ is trivial, i.e., $T_{\CG, \mu}=T_\CG$. 
 This is equivalent to the cone $\sigma_{\CG, \mu}$ having the full dimension. Then our toric schemes have no torus factor. 

 \end{altenumerate}
 
 We will point out when we  assume (ii) or when (ii) fails. 
 
% Equivalently, both cones $\sigma_{\CG, \mu}$ and $\sigma_{\CG, \mu}^\vee$ have the full dimension.

\subsection{Examples}\label{ss:Exam}

 Here we give some examples of the toric schemes $Y_{\CG,\mu}$ for various $(\CG,\{\mu\})$. In this list  all parahoric subgroups $\CG$ are Iwahori. In all cases, $(\CG, \{\mu\})$ is strictly convex and even ab-nondegenerate. Also, in all cases $T_{\CG, \mu}=T_\CG$. 
 \smallskip
 
 1) $(G,\{\mu\})=(\GL_n, \{\varpi_1^\vee\})$, where $\varpi^\vee_1=(1,0,\ldots, 0)$.
 (Here and below, we will denote by $(a_1,\ldots, a_n)$ the coweight $\BG_m\to \GL_n$ given by $t\mapsto \diag(t^{a_1},\ldots , t^{a_n})$.)
  Take $T$ to be the standard maximal diagonal torus. Then the Weyl group orbit $\Lambda_{\{\mu\}}=W_0\cdot \mu\subset X_*(T)$ is  the  standard basis  given by $e_i=(0,\ldots,0,1,0,\ldots, 0)$, with $1$ in the $i$-th place. Then
 \[
 \sigma_{\CG,\mu}=\{(a_1,\ldots, a_n)\in\BR^n\ |\ a_i\geq 0, \forall\, i\}, \quad S_{\CG,\mu}=\{(x_1,\ldots, x_n)\in\BZ^n\ |\ x_i\geq 0, \forall\, i\} .
 \] 
 The toric scheme is $Y_{\CG,\mu}=\BA^n_{\BZ_p}$.
 \smallskip

 2) $(G,\{\mu\})=(\GL_n, \{\varpi_2^\vee\})$, where $\varpi^\vee_2=(1,1, 0,\ldots, 0)$. Take $T$ to be the standard maximal diagonal torus. Then $\Lambda_{\{\mu\}}=W_0\cdot \mu$ is the set of all vectors with exactly two coordinates $1$ and the rest $0$; there are $n\choose 2$ such vectors. Then
 \[
S_{\CG,\mu}= \sigma^\vee_{\CG,\mu}\cap X^*(T)=\{(x_1,\ldots, x_n)\in \BZ^n\ |\ x_i+x_j\geq 0, \forall\, i\neq j\}.
 \]
 We can see that the Hilbert basis (i.e. minimal set of semi-group generators)  of $S_{\CG,\mu}$ 
 is the set of $e_i=(0,\ldots,0,1,0,\ldots, 0)$, with $1$ in the $i$-th place, and of $f_i=(1,\ldots, 1, -1,1,\ldots, 1)$ with $-1$ in the $i$-th place.
 Then $Y_{\CG,\mu}$ is the spectrum of the quotient of
 \[
 \BZ_p[e_1,\ldots, e_n, f_1,\ldots, f_n]
 \]
 by the ideal generated by 
 \[
 e_if_i-e_1\cdots \wh e_i \cdots e_n,\quad 1\leq i\leq n,
 \]
 \[
  f_if_j-\prod_{k\not\in \{i,j\} }e_k^2, \quad 1\leq i, j\leq n, \quad i\neq j.
 \]
Here the hat means that we omit the term.
\smallskip

3) $(G,\{\mu\})=(\GL_n, \{\varpi_{n-1}^\vee\})$,  where $\varpi^\vee_{n-1}=(1,1,\ldots, 1, 0)$.
Note that when $n>2$ this is not the same as example (1) above, even though the corresponding local models are isomorphic.  For $n>2$ the corresponding toric variety $Y_{\CG,\mu}$ is simplicial but not smooth. 
Indeed, then
 \[
S_{\CG,\mu}= \sigma^\vee_{\CG,\mu}\cap X^*(T)=\{(x_1,\ldots, x_n)\in \BZ^n\, |\, x_1+\cdots +\widehat{x_i}+\cdots +x_n\geq 0, \forall\, i\}.
 \]
 The dual cone has extremal rays given by the $n$ vectors obtained from $(1,\ldots,1, 2-n)$ by permuting its coordinates; these vectors give a basis of $X^*(T)_\BR$, generate $X^*(T)$ as a group, and form a subset of the minimal generating set (Hilbert basis) of $S_{\CG,\mu}$. However, this is a proper subset since these $n$ vectors do not generate $S_{\CG,\mu}$: Indeed, we can easily see for example that $(1,0,\ldots , 0)\in S_{\CG, \mu}$ but it is not a {\sl positive integral} linear combination of these $n$ vectors. In particular, $S_{\CG,\mu}$ is not free and $Y_{\CG,\mu}$ is not smooth. 
In fact, as we will explain later, when $(p, n-1)=1$, $Y_{\CG,\mu}$ can be identified with the quotient of $\BA^n_{\BZ_p}$ by the diagonal action of the group scheme $\mu_{n-1}$ and it is the affine cone of the $(n-1)$-th Veronese embedding of $\BP^{n-1}_{\BZ_p}$.
 (See Example \ref{DrversAnti}.)
\smallskip

4) $(G, \{\mu\})=(\GL_n, \{\varpi_j^\vee\})$, $ \varpi_j^\vee=(1^{(j)}, 0^{(n-j)})$, for $1\leq j\leq n-1$. 
 In all these cases, the  semigroup $S=\{(x_1,\ldots, x_n)\in \BZ^n\ |\ x_i\geq 0\}\subset S_{\CG,\mu}$ satisfies our assumptions. Then $Y_S=\BA^n_{\BZ_p}$ which supports a 
 birational $T$-equivariant morphism $Y_{\CG,\mu}\to Y_S$. 
\smallskip
 
 5) $(G,\{\mu\})=(\GSp_{2g}, \{\mu_{\rm std}\})$.  Here, $\GSp_{2g}$ is the symplectic similitude group defined by the symplectic form $(\, ,\,)$ with matrix
 \[
\begin{pmatrix}0 & J\\ -J & 0\end{pmatrix}
\]
where $J$ is the antidiagonal $g\times g$ matrix with coefficients equal to $1$ and $\mu_{\rm std}$ is the standard non-trivial minuscule coweight of type $\varpi_g^\vee$.  Take $T$ to be the standard maximal diagonal torus, so
\[
T=\{{\rm diag}(r_1,\ldots, r_g, cr^{-1}_g,\ldots , cr^{-1}_1)\}=\{(r_1,\ldots, r_g, c)\}=\BG_m^{g+1}.
\]
In the second more compact presentation of this torus   $\mu_{\rm std}$ is given by $x\mapsto (x,x,\ldots, x)$ and the Weyl orbit $\Lambda_{\{\mu\}}=W_0\cdot \mu_{\rm std}$ consists of the homomorphisms  $x\mapsto (a_1,\ldots, a_g, x)$, where $a_i$ are either $1$ or $x$. So there are $2^g$ such $\mu'$. Denote by $(x_1,\ldots, x_g, y)$ coordinates of $X^*(T)_\BR$. Then the inequalities defining the dual cone $\sigma^\vee_{\CG,\mu}$ are:
\begin{equation}\label{ineq}
 (\sum_{i\in U}x_i)+y\geq 0,\ \forall\ U\subset \{1,\ldots ,g\}.
\end{equation}
Here $U$ runs over all the subsets (including $\emptyset$) of $\{1,\ldots ,g\}$. 
The Hilbert basis   of $S_{\CG,\mu}$ are the following $2g$ elements of $X^*(T)$:
\[
e_1=(1,0,0,\ldots, 0,0),  e_2=(0,1,0,\ldots,0,0),\ldots, e_g=(0,0,\ldots, 1,0),
\]
\[
f_1=(-1,0,0,\ldots, 0,1) , f_2=(0,-1,0,\ldots,0,1),\ldots, f_g=(0,0,\ldots, -1,1).
\]
To see these generate $S_{\CG,\mu}$ take $a:=(x_1,\ldots, x_g, y)$ which satisfies the inequalities above; in particular $y\geq 0$. By subtracting a sum  
of the  $f_i$'s, we can assume that $y=0$.
Then $a=(x_1,\ldots, x_g, 0)$, with all $x_i\geq 0$. But these points are sums of points of the form 
 $e_i$.
(Note that, in the above, $(A_1,\ldots, A_g, B)$ denotes the   character which takes the value $\prod_{i=1}^g r_i^{A_g} c^B$ on $(r_1,\ldots, r_g,c)$.
Hence, $e_1,\ldots, e_g,f_1,\ldots, f_g$ 
are all the weights for the action of $T$ on the standard representation $V$ of $G=\GSp(V)$.)
We can now also easily see that this set of generators is minimal.
The  toric scheme $Y_{\CG,\mu}$ is the spectrum of the quotient of 
 \[
 \BZ_p[e_1,\ldots, e_g, f_1,\ldots, f_g]
 \]
 by the ideal generated by 
\[
e_if_i-e_jf_j,\quad\ \hbox{\rm for all}\ \ 1\leq i, j\leq g.
\]
This toric scheme  appears again in  \S \ref{ss:Siegel}, see also \cite{Marazza}, \cite{HS}, \cite{HLS}.
\smallskip

$5'$) $(G,\{\mu\})=(\GSp_{2g}, \{\mu_{\rm std}\})$, as in (5) above. We can consider the sub-semigroup $S$ of $S_{\CG, \mu}$ which is generated by $e_1,\ldots , e_g, f_1$. Then $S$ is free and still satisfies our assumptions.
We have $Y_S=\BA^{g+1}_{\BZ_p}$. This choice of $S$ and the corresponding toric scheme $Y_S$ is implicitly used in \cite[\S 4.1, \S 4.2]{Shadrach} and \cite{Marazza}. It still fits in our framework   (see Sect. \ref{sec:SV}). However, from the point of view of the current paper, this $S$ and the corresponding toric scheme $Y_S$ are not the canonical choices; the canonical choices are  $S_{\CG,\mu}$ and $Y_{\CG,\mu}$ of Ex. (5) above.
\smallskip

6) $(G, \{\mu\})=({\rm GSpin}(V),\{\mu_1\})$, where $V$ is a split quadratic space over $\BQ_p$ of dimension $2g+1$ and $\mu_1$ the standard non-trivial minuscule coweight of type $\varpi^\vee_1$,
landing in the standard maximal split torus $T$, both as described for example in \cite[\S 4.2.1]{HP}  (it is important to pin down precisely the data, as our theory is sensitive to changes of the center). We again take $\CG$ to be an Iwahori, given by an alcove in the apartment of $T$. 

Using the duality ${\rm GSpin}^\vee_{2g+1}\simeq \GSp_{2g}$, we see that the dual  (character) cone $\sigma^\vee_{\CG,\mu_1}$ identifies with the (cocharacter) cone for
$(\GSp_{2g}, \{\mu_{\rm std}\})$ as in example (5) above. As in  (5), we can see that the latter cone is the positive hull of
the set of $(a_1,\ldots, a_g, 1)$, where $a_i=0$ or $1$ (so there are $2^g$ elements) in $X_*(T_{\GSp_{2g}})\simeq \BZ^{g+1}$. It follows that $\sigma^\vee_{\CG,\mu}\subset X^*(T)_\BR$
is the convex hull of the same subset of $X^*(T)\simeq X_*(T_{\GSp_{2g}})\simeq \BZ^{g+1}$. It is now not hard to check that the elements $(a_1,\ldots, a_g, 1)$, with $a_i=0, 1$, form a Hilbert basis
 of the semigroup  $S_{\CG,\mu}=\sigma^\vee_{\CG,\mu}\cap X^*(T)\subset \BZ^{g+1}$. Hence, the  toric scheme $Y_{\CG,\mu}$ is isomorphic the spectrum of the quotient 
of the algebra $\BZ_p[\{x_U\}_{U\subset \{1,\ldots ,g\}}]$, with generators $x_U$ parametrized by the subsets $U$ of $\{1,\ldots ,g\}$, by the ideal of all  relations between the $x_U$'s  
given by corresponding identities in $\BZ^{g+1}$ by mapping $x_U$ to the element  $(a_1,\ldots ,a_g, 1)$ with $a_i=1$ iff $i\in U$. This ideal is generated 
by the   elements 
\[
x_{U}\cdot x_{U'}-x_{U\cap U'}\cdot x_{U\cup U'},
\]
for any pair $U$, $U'$ of subsets of $\{1,\ldots ,g\}$.

7) \emph{The genuine Drinfeld case.} $(G,\{\mu\})=(D^*, \{\varpi_1^\vee\})$, where $D$ is a central division $\BQ_p$-algebra of index $d$. Again $\CG$ is the unique Iwahori
given by the units $O_D^*$ of the maximal oder $O_D$. Set $q=p^d$. Then $T_\CG\simeq \Res_{\BZ_{q}/\BZ_p}\BG_m$. We have $E_0=\BQ_{p}$ and
\[
Y_{O_D^*, \varpi_1^\vee} \simeq \Res_{\BZ_q/\BZ_p}\BA^1.
\] 
This is the torus embedding $\Res_{\BZ_{q}/\BZ_p}\BG_m\hookrightarrow \Res_{\BZ_q/\BZ_p}\BA^1$.
\smallskip

8)  \emph{Ramified Weil restriction of scalars.} Let $F/\BQ_p$ be a finite totally ramified extension  and consider the pair $(G, \{\mu\})$, where $G={\rm Res}_{F/\BQ_p}H$, where  $H$ is a split reductive group over $F$ and $\mu=(\mu_\phi)_{\phi: F\hookrightarrow \ov\BQ_p}$, with  $\mu_\phi: (\BG_m)_{/\ov\BQ_p}\to H\times_F\ov\BQ_p$ a minuscule coweight or trivial, for each 
field embedding $\phi: F\hookrightarrow \ov\BQ_p$. Take $T=\Res_{F/\BQ_p}T_H $ with $T_H\simeq (\BG^n_m)_{/F}$ a split maximal torus of $H$ and fix a Borel subgroup $T_H\subset B_H\subset H$.
Then $X_*(T)_I$ can be identified with $\BZ^n= X_*(T_H)$ using the averaging map and, under this identification, the positive hull of $\Lambda_{\{\mu\}}\subset X_*(T)_I$ is the positive hull of the $W_0$-orbit of the average $\sum_\phi \mu^{\rm dom}_\phi$, where $\mu^{\rm dom}_\phi\in X_*(T_H)=\BZ^n$ is a dominant representative of the conjugacy class $\{\mu_\phi\}$. In this, $W_0$ identifies with the Weyl group $N(T_H)/T_H$ of $H$.

We can apply this to the case $H=\GL_n$, which is already very interesting: Then we take $T=\Res_{F/\BQ_p}(\BG^n_m)_{/F}$ to be the standard diagonal torus of $G$ and $\mu_\phi=(1^{(r_\phi)}, 0^{(r_\phi)})$, for some $r_\phi$, with $0\leq r_\phi\leq n$, for each 
field embedding $\phi: F\hookrightarrow \ov\BQ_p$. As above, $X_*(T)_I$ can be identified with $ \BZ^n$ using the averaging map. Under this identification, $\Lambda_{\{\mu\}}\subset X_*(T)_I$ 
is given as
\begin{equation*}
\text{$\Lambda_{\{\mu\}}$\,=\,$S_n$-orbit of $(s_1,s_2,\ldots ,s_n)$, where $s_i=\#\{\phi\ |\ r_\phi\geq i\}$.}
\end{equation*} We assume that $\Lambda_{\{\mu\}}$ generates the vector space $X_*(T)_I\otimes\BR=\BR^n$ and that at least one of the $\mu_\phi$ is not trivial. (The second condition guarantees that $(G,\{\mu\})$ is ab-nondegenerate.) 

Now consider the  Iwahori $\CG=\Res_{O_F/\BZ_p}\CI$, with $\CI$ the standard Iwahori for $\GL_n$ over $O_F$.  The cone $\sigma_{\CG,\mu}$ is the positive hull of $\Lambda_{\{\mu\}}$,  a general ``permutohedral cone".  We have
\[
S_{\CG,\mu}=\{(a_1,a_2,\ldots, a_n)\in \BZ^n\ |\ \sum_{i=1}^n s_{w(i)}\cdot a_i\geq 0,\ \forall\, w\in S_n\},
\]
and $Y_{\CG,\mu}$ is the toric affine scheme given by the spectrum of the corresponding semigroup algebra over $\BZ_p$.
These schemes can quickly become very complicated: For example, the $3$-dimensional permutohedral cone for $(s_1,s_2,s_3)=(5,2,1)$ 
gives a semigroup whose Hilbert basis has $30$ elements. According to Macaulay 2, the corresponding toric ideal of relations for the semigroup algebra has $1181$ generators.

A similar construction can be performed for more general $H$, $F$ and $(\mu_\phi)_\phi$ as  above.

9)  \emph{Ramified unitary similitudes.} $(G, \{\mu\})=({\rm GU}_n, \{\mu_{r,s}\})$. Let us explain the data involved in the definition of this pair. We take $p>2$, $F/\BQ_p$ a ramified quadratic extension, and $V$ an $n$-dimensional $F$-vector space equipped with a $F/\BQ_p$-hermitian form $h: V\times V\to F$. We assume $n\geq 3$ and that $h$ is split, i.e. there is an $F$-basis of $V$ such that $h(e_i, e_{n+1-j})=\delta_{ij}$. For a $\BQ_p$-algebra $R$
set $V_R=V\otimes_{\BQ_p}R$ and denote by $h_R: V_R\times V_R\to F\otimes_{\BQ_p}R$   the base change of $h$. Then ${\rm GU}_n(R)={\rm GU}(V,h)(R)$ is the group of   $g\in {\rm Aut}_{F\otimes_{\BQ_p}R}(V_R)$ for which there exists  $c(g)\in R^*$ such that  
\[
h_R(gv, gw)=c(g)\cdot h_R(v,w),  \quad \forall\, v, w\in V_R.
\]
Denote by $a\mapsto \ov a$ the conjugation of $F$ over $\BQ_p$ which gives an involution on the $F$-rational points ${\rm GU}_n(F)$.  There is an isomorphism   
\[
\psi: {\rm GU}_n(F)\xrightarrow{\sim}\GL_n(F)\times F^*,
\]
which carries this involution
to 
$
(A, c)\mapsto (\ov c\cdot (A^*)^{-1}, \ov c)
$
where $A^*$ denotes the hermitian adjoint of $A$. For a pair $(r, s)$ of non-negative integers with $n=r+s$,  we take the coweight $\mu_{r,s}$ defined by 
\[
 \mu_{r,s}(a)=\psi^{-1}((\diag(a^{(r)}, 1^{(s)}), a)).
\]

Take $T$ to be the maximal torus of ${\rm GU}_n$ which corresponds via $\psi$ to the 
standard maximal torus $\diag\times \BG_m\subset \GL_n\times \BG_m$.  We can see that the Galois involution on $X_*(T)=\BZ^n\times \BZ$ is given by
\[
\tau(x_1,\ldots, x_n, y)=(y-x_n, y-x_{n-1},\ldots, y-x_2, y-x_1,y).
\]

 a) Suppose $n=2m+1$ is odd. 
Then 
\[
X_*(T)^I=\{(x_1,\ldots, x_m, x_{m+1},2x_{m+1}-x_m,\ldots, 2x_{m+1}-x_1, 2x_{m+1})\}\subset \BZ^n\times\BZ
\]
and the projection to the first $m+1$ coordinates $(x_1,\ldots, x_m, x_{m+1})$ gives $X_*(T)^I\xrightarrow{\sim} \BZ^{m+1}$.

b) Suppose $n=2m$ is even. Then 
\[
X_*(T)^I=\{(x_1,\ldots, x_m,y-x_m,\ldots, y-x_1, y)\}\subset \BZ^n\times\BZ
\]
and we can project to $(x_1,\ldots, x_m, y)$ to obtain $X_*(T)^I\xrightarrow{\sim} \BZ^{m+1}$.

In either case, we consider the composition
\[
a: X_*(T)\to X_*(T)^I_\BQ\xrightarrow{\sim}\BZ^{m+1}\otimes_\BZ\BQ,
\]
where the first arrow is given by the averaging map $\lambda\mapsto \lambda^\diam$.

Now take  $\mu=\mu_{1, n-1}=(1,0,\ldots ,0,1)\in X_*(T)$. Then $\tau\mu=(1,1,\ldots,1,0,1)$ and so 
\[
\mu^\diam=\frac{\mu+\tau\mu}{2}=(1,1/2,\ldots, 1/2,0,1)\in X_*(T)^I_\BQ
\]
Under the isomorphism $X_*(T)^I_\BQ\xrightarrow{\sim} \BQ^{m+1}$ this maps to $(1,1/2,\ldots, 1/2)$ if $n=2m+1$ is odd, and to
$(1,1/2,\ldots, 1/2, 1)$ if $n=2m$ is even. 

We now consider the action of the Weyl group $W_0$.  Recall that the reduced root system $\Sigma$ on $X_*(T)_I\otimes\BQ $ is of type $B_m$ (if $n$ is even) or $C_m$ (if $n$ is odd), see \cite{PRLM3} for more details. In both cases, the Weyl group $W_0$ is isomorphic to $S_m\rtimes (\BZ/2\BZ)^m$.
We let $S_m\rtimes (\BZ/2\BZ)^m$ act
on $X_*(T)=\BZ^n\times\BZ=\{(x_1,\ldots , x_n, y)\}$ by fixing $y$ and permuting the indices $1,\ldots , n$, so that  the partition of $\{1,\ldots , n\}$ into subsets $\{i, n+1-i\}$ is preserved (so also fixing the ``middle" $ m+1$, if $n$ is odd).  We let $W_0$ act on $X_*(T)^I$ as follows. If $n$ is odd, the first factor in the semi-direct product $S_m\rtimes (\BZ/2\BZ)^m$ acts on $(x_1, \ldots  ,x_m, x_{m+1})$ by permuting the first $m$ entries. The element $(-1)_i$ in the second factor maps $(x_1, \ldots  ,x_m, x_{m+1})$ to the vector with $i$th entry equal to $2x_{m+1}-x_i$ and all other entries unchanged. If $n$ is even, then $S_m$ acts on $(x_1, \ldots  ,x_m, x_{m+1})$ by permuting the first $m$ entries, whereas the factor $(-1)_i$ maps $(x_1, \ldots  ,x_m, x_{m+1})$ to the vector with $i$th entry equal to $x_{m+1}-x_i$ and all other entries unchanged.  Then the natural map $X_*(T) \to X_*(T)_I\otimes\BQ $ respects these actions. 

The orbit of $\mu=(1,0,\ldots ,0,1)$ by the $W_0$-action on $\BZ^n\times \BZ$ as above, consists of the $2m$ elements $(0,\ldots, 0,1,0,\ldots 0,1)$, where the first $1$ is in any of the $2m$ possible spots (after excluding the middle spot in the odd case $n=2m+1$). These elements map under $a$ to 
\[
(1/2,\ldots, 1/2, 1,1/2,\ldots, 1/2,1/2), \quad\hbox{\rm and}\quad (1/2,\ldots, 1/2, 0,1/2, \ldots, 1/2,1/2),
\]
 in $X_*(T)^I_\BQ=\BQ^{m+1}$ if $n=2m+1$, and to 
\[
(1/2,\ldots, 1/2, 1,1/2,\ldots, 1/2,1), \quad\hbox{\rm and}\quad (1/2,\ldots, 1/2, 0,1/2, \ldots, 1/2,1),
\]
 in $X_*(T)^I_\BQ=\BQ^{m+1}$ if $n=2m$,
with the   $1$ or the $0$ taking any of the first $m$ possible spots.  The resulting $2m$ elements in $X_*(T)^I_\BR$ span the extremal rays of the cone $\sigma_{\CG,\mu}$. (Here  $\CG$ is a standard Iwahori of $G$).  

Some more explicit examples: For $n=4$, $(r,s)=(1,3)$, by the above,  the extremal rays of $\sigma_{\CG,\mu}$ are spanned by
$(1,1/2,1)$, $(1/2,1,1)$,  $(1/2,0,1) $, $(0,1/2,1) $. The corresponding
primitive elements in $X_*(T_\CG)=X_*(T)^I\simeq \BZ^3$ on these rays are $(2,1 ,2 )$, $(1 ,2,2 )$,  $(1 ,0,2) $, $(0,1 ,2 ) $. For $n=3$, $(r,s)=(1,2)$, the extremal rays are spanned by $(1,1/2) $, $(0,1/2)$; the corresponding primitive elements of $X_*(T)^I\simeq \BZ^2$ on these rays are $(2,1)$, $(0,1)$.  
In this last case, the toric scheme $Y_{\CG,\mu}$ is isomorphic to the spectrum of 
$\BZ_p[x_1,x_2,x_3]/(x_1^2-x_2x_3)$.

 \subsection{Further properties; relation to other constructions}\label{ss:othcon}
 
In this section we give some further properties and constructions relating to the toric schemes $Y_{\CG,\mu}$.

 \subsubsection{Central isogenies}\label{cisogenies}
Let $(G, \{\mu\})\to (G', \{\mu'\})$ be a central isogeny, i.e., $f: G\to G'$  is surjective, the kernel is  a finite central subgroup $C$ of $G$ and $\{\mu'\}=\{f\cdot \mu\}$. Assume  that  $C$ splits over $\breve\BQ_p$, i.e., the action of $\Gal(\bar\BQ_p/\BQ_p)$ on $C(\bar\BQ_p)$ is unramified and that the order of $C(\bar\BQ_p)$ is prime to $p$. Let $\CG$, resp. $\CG'$, be parahorics that correspond to each other under the isogeny.  
Then $C(\bar\BQ_p)\subset G(\breve\BQ_p)$ and even $C(\bar\BQ_p)=C(\breve\BQ_p)\subset \CG(\breve\BZ_p)$ since $C(\breve\BQ_p)$ fixes the point in the Bruhat-Tits building corresponding to $\CG$ and  the Kottwitz map  $G(\breve \BQ_p)\to \pi_1(G)_I$ evaluated on $C(\breve\BQ_p)$ is trivial. Furthermore the closure $C_\CG$ of $C$ in $\CG$ has special fiber $\bar C_\CG$ \'etale of multiplicative type. We obtain an exact sequence of groups of multiplicative type over $\BZ_p$,
\[
1\to  C_\CG\to T_\CG\to T_{\CG'}\to 1 .
\] 
Assume $(\CG, \{\mu\})$, and hence also $(\CG', \{\mu'\})$, satisfies both our blanket assumptions \ref{blanket} (i) and (ii).   Under the identification $X_*(T_\CG)_\BR=X_*(T_{\CG'})_\BR$, the two cones $\sigma_{\CG, \mu}$ and $\sigma_{\CG',\mu'}$ coincide, hence $S_{\CG', \mu'}=S_{\CG, \mu}\cap X^*(T'_{\CG'})$, where  $X^*(T'_{\CG'})\subset X^*(T_{\CG})$ is an inclusion with finite index. We obtain a finite morphism $Y_{\CG, \mu}\to Y_{\CG', \mu'}$ equivariant wrt the isogeny $T_\CG\to T_{\CG'}$. In fact, $Y_{\CG', \mu'}$  is the quotient of $Y_{\CG, \mu}$ by the action of $C_\CG$. 

\begin{example}\label{DrversAnti}
Consider the central isogeny  
\[
f: \GL_n\to \GL_n\quad A\mapsto \wedge^{n-1}A.
\]
Assume $p$ is coprime to $n-1$.
This gives an isogeny $f: T\to T$ of the standard maximal torus with the corresponding $X_*(T)\to X_*(T)$ given by 
\[
(x_1,\ldots, x_n)\mapsto (\sum_i x_i-x_1,\ldots, \sum_i x_i-x_n)
\]
so that $\varpi^\vee_1=(1,0,\ldots ,0)$ maps to $\varpi^\vee_{n-1}=(0,1,\ldots ,1,1)$. The kernel of $f$ is 
the group scheme of roots of unity $\mu_{n-1}$.
The isogeny gives a toric morphism $Y_{\varpi^\vee_1}\to Y_{\varpi^\vee_{n-1}}$ and induces $Y_{\varpi^\vee_{n-1}}=Y_{\varpi^\vee_1}/\mu_{n-1}$.
Since $Y_{\varpi^\vee_1}=\BA^n_{\BZ_p}$ (cf. Example 1 in \S \ref{ss:Exam}), we get
\[
Y_{\varpi^\vee_{n-1}}=\BA^n_{\BZ_p}/\mu_{n-1} ,
\]
the quotient for the diagonal scaling action. This is consistent with Example 3  in \S \ref{ss:Exam}, reconfirming that $Y_{\varpi^\vee_{n-1}}$ is not smooth. The affine coordinate ring of this quotient is the invariants 
\[
(\BZ_p[x_1,\ldots, x_n])^{\mu_{n-1}}=\bigoplus_{i\geq 0}\BZ_p[x_1,\ldots, x_n]_{{\rm deg}=i(n-1)}.
\]
This is the ``$(n-1)$-th Veronese algebra'' and   $Y_{\varpi^\vee_{n-1}}=\BA^n_{\BZ_p}/\mu_{n-1}$ is the $(n-1)$-st Veronese cone, i.e. the affine cone corresponding to the line bundle $\CO_{\BP^{n-1}}(n-1)$ over $\BP^{n-1}_{\BZ_p}$.
\end{example} 

\subsubsection{The projective toric scheme $X_{\CG, \mu}$}\label{ss:projtor}

 The triple $(G, \{\mu\}, \CG)$ defines $(G_\ad, \{\mu_\ad\}, \CG_\ad)$.  Let us  assume that $\mu_\ad$ is non-trivial in all simple factors of $G_\ad\otimes_{\BQ_p}\breve\BQ_p$, cf. Proposition \ref{dimT1}.  We consider the polytope $P_{\mu_\ad}=P_{\mu_\ad}^\CG$ which is the convex hull 
of the image of the Weyl orbit $\Lambda_{\{\mu\}}$  in $X_*(T_{\CG_\ad})_\BR$. 
 By Lemma \ref{lim}, the polytope $P_{\mu_\ad}$ contains the origin in its interior.  Let $\Sigma=\Sigma^\CG_{\mu_\ad}$ be the associated fan. The extreme rays of $\Sigma$ are in bijection with the extreme rays of $\sigma_{\CG,\mu}$, when $\sigma_{\CG,\mu}$ is strictly convex. Let $X_{\CG, \mu}=X_\Sigma$ be the corresponding toric scheme (over $\BZ_p$) for the torus $T_{\CG_\ad, \mu_\ad}$. Since $|\Sigma|=X_*(T_{\CG_\ad})_\BR$, we see that $X_{\CG, \mu}$ is a projective toric scheme. It depends only on  $(G_\ad, \{\mu_\ad\}, \CG_\ad)$. In the rest of this paragraph, we try to relate the schemes  $Y_{\CG,\mu}$ and  $X_{\CG,\mu}$.

 To simplify our discussion, we will suppose, in addition to both the blanket assumptions (i) and (ii), that $\CG$ is Iwahori and that $(G,\{\mu\})$ is ab-nondegenerate. Consider the exact sequence
\[
1\to Z\to G\to G_\ad\to 1 .
\]
By Proposition \ref{dimT1}, it follows that $\dim Z=1$. Let us also assume that the neutral component $Z^o$ is isomorphic to $\BG_m$ and that the finite abelian group scheme $\pi_0(Z)$ is unramified and has rank prime to $p$. This exact sequence induces an exact sequence 
\begin{equation}
1\to Z_{\BZ_p}\to T_\CG\to T_{\CG_\ad}\to 1 
\end{equation}
where we denote by $Z_{\BZ_p}$ the smooth model of $Z$ over $\BZ_p$ obtained 
as the Zariski closure of $Z$ in $T_\CG$; then $Z^o_{\BZ_p}= \BG_m$ over $\BZ_p$ and   $\pi_0(Z_{\BZ_p})=\pi_0(Z)_{\BZ_p}$ is the unique finite \'etale abelian group scheme over $\BZ_p$ with  $\pi_0(Z)$
as generic fiber.

The first observation is that $X_{\CG, \mu}$ is the quotient of $Y_{\CG, \mu}$ by the action of $Z_{\BZ_p}$ in the following sense: We normalize the isomorphism $Z^o_{\BZ_p}=\BG_m$ such that the corresponding element $z\in X_*(T_\CG)$ lies in $\sigma_{\mu}$. Then the action of $Z^o_{\BZ_p}=\BG_m$ on the affine ring $A=A_{\CG, \mu}$ of $Y_{\CG, \mu}$ defines a grading $A=\bigoplus_{m\geq 0} A_m$.   Let $\wt X_{\CG, \mu}={\rm Proj} (A)$. Then $X_{\CG, \mu}$  is the quotient of $\wt X_{\CG, \mu}$ under the action of the finite abelian group scheme
$\pi_0(Z)_{\BZ_p}$.

Next, we give a condition which implies that $Y_{\CG,\mu}$ is the affine cone of a line bundle over $\wt X_{\CG,\mu}$.

  \begin{proposition}\label{prop:free} Under the above assumptions, the $\BG_m$-action on $Y_{\CG, \mu}\setminus \{0\}$ is free if and only if $\langle z, \chi_\rho\rangle=1$ for all generators $\chi_\rho$ of the extremal rays $\rho$ of $\sigma^\vee_\mu$.  Then $Y_{\CG, \mu}\setminus \{0\}\to \wt X_{\CG, \mu}$ is a $\BG_m$-torsor and
  $Y_{\CG, \mu}$ is the affine cone for the  line bundle over $\wt X_{\CG, \mu}$
  which corresponds to this $\BG_m$-torsor.  
  \end{proposition}

\begin{proof}
 Recall that each extremal ray $\rho$ of the dual cone $\sigma^\vee_\mu$ corresponds to a ($1$-codimensional) facet $F=F(\rho)$ of the cone $\sigma_\mu$. In turn, by
  the face-orbit correspondence, this corresponds to a $1$-dimensional torus orbit $O(F)$ in $Y_{\CG,\mu}$. If $y\in Y_{\CG, \mu}\setminus \{0\}$ is a fixed point of $t_0\in z(\BG_m)$, then the whole $T_\CG$-orbit of $y$ is fixed by $t_0$ which, by considering closures, implies that $t_0$ fixes all points of some  orbit $O(F)$. But the subgroup of $z(\BG_m)\subset T_\CG$ which fixes all the points of $O(F)$ has order $\langle z, \chi_\rho\rangle$. The claim in the first part of the statement follows. 
  
  We now  show the second part of the statement. Note that the Hilbert homogeneous basis of $A$ is given by the indecomposable elements in $S_{\CG, \mu}$. Among them are  the generators $\chi_\rho$ of the extremal rays $\rho$ of the dual cone $\sigma^\vee_\mu$, but there may be more. 
 Let $A'$ be the $\BZ_p$-subalgebra of $A$ which is generated by the degree $1$ part $A_1$. By our assumption $\langle z, \chi_\rho\rangle=1$, $A_1$ contains the ray generator $\chi_\rho$, for each extremal ray $\rho$ of $\sigma_\mu^\vee$. It is easy to see that $S_{\mu}$ is the saturation of its semigroup $S'_\mu$
which is  generated in $X^*(T_\CG)$ by the ray generators $\chi_\rho$. Since $A$ is normal, this implies that $A$ is the normalization of $A'$ and we have $\wt X_{\CG,\mu}={\rm Proj}(A)={\rm Proj}(A')$, since $A$ and $A'$ coincide in high degree. Since $A'$ is generated by $A_1$,  the tautological twisting sheaf $O_{\wt X_{\CG,\mu}}(1)$ is a line bundle and the canonical homomorphism
  \[
  A'\to \wt A':= \bigoplus_{m\geq 0}\Gamma(\wt X_{\CG,\mu}, O_{\wt X_{\CG,\mu}}(1)^{\otimes m})
  \]
  identifies $\wt A'$ with the normalization $A$ of $A'$. Hence
  \[
   A=\bigoplus_{m\geq 0}\Gamma(\wt X_{\CG,\mu}, O_{\wt X_{\CG,\mu}}(1)^{\otimes m})
  \]
  (Stacks 27.8-27.10). The rest of the statement now follows.
  \end{proof}

\begin{remark}\label{ex:notfree} The condition in Proposition \ref{prop:free} is often not satisfied. Then the action of $Z$ on $Y_{\CG, \mu}\setminus \{0\}$ is not free. Consider the case $(G=\GL_4,\mu=\varpi^\vee_2, \CG=\CI)$ of Ex. 2, \S \ref{ss:Exam}.  Then $Z=\BG_m$ and the extremal ray generators are the elements $e_i$ and $f_i$ there;
 the $e_i$' s have degree $1$ but the $f_i$'s have degree $2$. In this situation, there are
 stabilizer groups of order $2$: The closed subscheme $W$ of $Y_{\CG, \mu}$ defined by the ideal
  $(e_1,e_2,e_3,e_4, f_1f_2,f_1f_3,f_1f_4,f_2f_3,f_2f_4, f_3f_4)$ has dimension $1$, so it is not the closed orbit. We can see that  $\mu_2\subset Z=\BG_m$ acts trivially on $W$.   \end{remark} 
 
 \begin{example}\label{DrversAnti2}
 In the situation of Example \ref{DrversAnti}, the projective toric scheme $X_{\CG,\mu}$ associated to $(G,\{\mu\},\CG)=(\GL_n,\{\varpi^\vee_{n-1}\},\CI)$ is $\BP^{n-1}_{\BZ_p}$ and $Y_{\CG, \mu}$ is the $(n-1)$-st Veronese cone.
 \end{example}

\subsubsection{Constructions starting with a weight; duality}\label{duality}
The following ``dual" construction of toric varieties/schemes is perhaps more common and variations have appeared in many contexts,   comp., e.g., the references in \cite{MoR}. 
The resulting toric embeddings are sometimes called ``Weyl orbit toric varieties".

Let $H$ be a reductive group over $\BQ_p$ which is quasi-split and splits over $F=\br\BQ_p$. Choose a maximal torus $S$ and a Borel subgroup $ B$ in $ H_{F}$ over $F$; we can assume that both $S$ and $B$ are defined over $\BQ_p$. Suppose that $\lambda\in X^*(S)$ is a dominant weight $\neq 0$ and consider the Weyl orbit $\{\lambda\}=W_0\cdot \lambda\subset X^*(S) $. The rational polyhedral convex cone
\[
\tau^\vee_\lambda=\{ a\in X_*(S)_\BR\ |\ \langle  w\cdot\lambda, a\rangle\geq 0, \forall\, w\in W_0\}
\]
determines an affine $S_{O_F}$-toric scheme $V_{H,\lambda}$ over $O_F$, as follows: the cone $\tau^\vee_\lambda$ is the dual of the positive hull $\tau_\lambda={\rm Cone}(W_0\cdot \lambda) \subset X^*(S)_\BR$, and we set
\[
V_{H,\lambda}:=\Spec(O_F[\tau_\lambda\cap X^*(S)]).
\]

There is also a projective variant. Assume for simplicity that $H$ is absolutely  almost simple. Consider the ``Weyl polytope" ${\rm WP}_{\lambda}$ defined as the convex hull of the Weyl orbit 
$\{\lambda \}=W_0\cdot \lambda $ in $X^*(S )_\BR$ and let  $Q=X^*(S_\ad)\subset X^*(S_\der)$ be the root lattice.   We can define a projective toric scheme embedding $U_{H,\lambda}$ for the torus $S_{\ad,O_F}$, as follows:  the corresponding fan in $X_*(S_\ad)_\BR$ is the projection from $X_*(S)_\BR $ to $X_*(S_\ad)_\BR$ of the normal (dual) fan of the polytope ${\rm WP}_\lambda$. Here we refer to \cite[Ch. 1, \S 2.3]{CLS} for the definition of the normal fan of a lattice polytope in $X^*(S)_\BR$ (the normal fan is called the \emph{dual fan} in \cite{MoR}). 

By Dabrowski  \cite[Thm. 3.2] {Dabrowski}\footnote{\cite{Dabrowski} assumes the base is an algebraically closed field but the proof applies in our situation too. Here the qualifier  ``generic'' is meant in the sense of \cite[Def. 1.1]{Dabrowski}, which implies that the Zariski closure is normal.} the toric scheme $U_{H,\lambda}$ is isomorphic to the Zariski closure of  a generic $S$-orbit in  $H/P_\lambda$ (see also \cite[Prop. 3]{MoR} for this description of the toric variety defined by a generic orbit closure).

 Let us now consider the dual reductive group $G=H^\vee$ with corresponding dual maximal torus $T=S^\vee\subset H^\vee=G$. We have identifications $X^*(S)=X_*(T)$, $\chi\mapsto \chi^\vee$. Using the dual $\mu:=\lambda^\vee$, which is now a  coweight of $G$, we construct  the pair $(G, \{\mu\})=(H^\vee, \{\lambda^\vee\})$. Take $\CG$ to be an Iwahori of $G$ defined by an alcove in the apartment of $T$. Then, under the above identification $X_*(T)=X^*(S)$, the cone $\sigma_{\CG,\mu}$ is equal to $\tau^\vee_\lambda$, and its dual 
$\sigma_{\CG,\mu}^\vee$ is equal to $\tau_\lambda={\rm Cone}(W_0\cdot \lambda)$, as above. It follows that the $T_{O_F}$-toric scheme 
$Y_{\CG,\mu}\otimes_{O_E}O_F$ is isomorphic to the dual (in the sense of toric varieties, as in \cite[Def. 3.9.1]{BL} for example) of the $S_{O_F}=T^\vee_{O_F}$-toric scheme $V_{H,\lambda}=V_{G^\vee,\mu^\vee}$ given as above. Note that these considerations generalize some of the arguments in Ex. 6, \S \ref{ss:Exam}.

\begin{remark}
One might ask for a more direct connection between the toric schemes $Y_{\CG,\mu}$ which are defined from the Weyl orbit of a \emph{coweight} and the toric schemes  $V_{H,\lambda}$ which are defined from the Weyl orbit of a \emph{weight}.

Let us start with $(G,\{\mu\})$, with $\mu$ a non-trivial coweight, such that $G$ is quasi-split and splits over $F=\br\BQ_p$. Let $T$ be a maximal torus defined over $\BQ_p$ and take $\CG$ to be an Iwahori of $G$ defined by an alcove in the apartment of $T$. Now take $H=G$ and $S=T$ in the above. We ask if there is a weight $\lambda\in X^*(S)=X^*(T)$ such that $Y_{\CG,\mu}\otimes_{O_E}O_F\simeq V_{H,\lambda}$ as $T_{O_F}$-toric schemes. This happens if $\tau_\lambda=\sigma^\vee_\mu$ in $X^*(T)_\BR$. 
Now recall that the extremal rays of $\sigma^\vee_\mu$ correspond to facets of the cone $\sigma_\mu$. Hence, a necessary 
and sufficient condition for the existence of $\lambda\in X^*(S)$ with $\tau_\lambda=\sigma^\vee_\mu$ is that the facets of $\sigma_\mu$ consist of one Weyl orbit.
We can list the cases when this happens by applying work of Maxwell \cite{Maxwell}: For simplicity, assume that, in addition, $G$ is absolutely almost simple. Let $\Delta$ be the set of simple roots regarded as vertices of the Dynkin diagram of $G$. Consider
\[
J(\mu):=\{s\in \Delta\ |\ s(\mu)=\mu\}.
\]
Then by \cite{Maxwell}, see also \cite[Cor. 1.3]{Renner}:
\begin{itemize}
\item[] \emph{The set of Weyl orbits of the facets of $\sigma_\mu$ is in bijection with 
the set of subsets $J\subset \Delta$ of
size $|J|=|\Delta|-1$ which satisfy the following condition:
No connected component of $J$ is contained in $J(\mu)$.}
\end{itemize}
It is not hard to see that if the Dynkin diagram   has branches ($D$, $E$ types)
then, for each $\mu$, there exists more than one such subset $J$, hence also more than one Weyl orbit of facets. For the unbranched diagrams ($A$, $B$,
$C$, $F$, $G$ types) there is a unique $J$ if and only if $J(\mu)$ is obtained by just omitting one of
the extreme points of the Dynkin diagram. Then the unique $J $ is given by
omitting the opposite extreme point. Let us now   restrict consideration to minuscule $\mu$. Then, by the above,  there is a single orbit of facets of $\sigma_\mu$ in exactly the following cases:
\begin{itemize}
\item $A$ type: $\{\mu_\ad\}$ is the Drinfeld $\{\varpi^\vee_1\}$ (or anti-Drinfeld $\{\varpi^\vee_{n-1}\}$) class (Ex. 1, resp. Ex. 3, \S \ref{ss:Exam}),
\item $B$ type: $\{\mu_\ad\}$ is the unique minuscule class (Ex. 6, \S \ref{ss:Exam}),
\item $C$ type: $\{\mu_\ad\}$ is the unique minuscule class (Ex. 5, \S \ref{ss:Exam}).
\end{itemize}

\end{remark}

  \section{The divisor map and Lang covers}

\subsection{The divisor map}\label{ss:divm}

\subsubsection{Definitions} 
Assume now that the divisor conjecture, i.e. Conjecture \ref{divconj}, holds. Then there is a $\CG $-equivariant $T_{\CG }$-torsor ${\rm P}_{\CG,\mu}\to \Mloc_{\CG,\mu}$ and 
 a $G_E=\CG\otimes_{\BZ_p}E$-equivariant  trivialization 
\[
s: T_{\CG,E}\times_{\Spec(E)} (\Mloc_{\CG,\mu}\otimes_{O_E}E)\xrightarrow{\sim} {\rm P}_{\CG,\mu}
\otimes_{O_E}E
\]
 over the generic fiber,  satisfying the conditions of Conjecture \ref{divconj}. The trivialization $s$ gives, by using the projection, a $T_{\CG,E}$-equivariant morphism over $E$
\begin{equation}\label{defq}
\delta_E={\rm pr}\cdot s^{-1}: {\rm P}_{\CG,\mu}^{(-1)}
\otimes_{O_E}E
\to T_{\CG,E} 
\end{equation}
 with fibers isomorphic to the partial flag variety $\CF(G, \{\mu\})=G_E/P_{\{\mu\}}$. Here  ${\rm P}^{(-1)}_{\CG,\mu}={\rm P}_{\CG,\mu}$ as a scheme, but with the inverse $T_{\CG}$-action, i.e. composed with $t\mapsto t^{-1}$. Since $s$ is $G_E $-equivariant, $\delta_E$ is $G_E$-equivariant for the trivial action on the target. 
 
 \begin{proposition}\label{prop:delta}
Under the above assumptions (in particular assuming Conjecture \ref{divconj}), there is a unique extension of the morphism $\delta_E$  to a $T_{\CG,O_E}$-equivariant morphism  
\[
\delta: {\rm P}^{(-1)}_{\CG,\mu}\to Y_{\CG,\mu}\otimes_{O_{E_0}} O_E.
\]
Furthermore, $ \delta$ is $\CG_{O_E}$-equivariant for the trivial action on the target. 
\end{proposition} 

 Note that   $ \delta$ amounts to a morphism $\Mloc_{\CG,\mu}\to [T_\CG \bslash Y_{\CG,\mu}]$ which is $\CG $-equivariant for the trivial action on the target. We postpone the proof to give:
 
\begin{definition}\label{def:DivisorMap}
 Assume the divisor conjecture for $(\CG,\{\mu\})$. Then the \emph{divisor map} for the  pair $(\CG,\{\mu\})$ is  the  morphism of stacks
  \[
 \Delta_\CG: [\CG \bslash \Mloc_{\CG,\mu}]\to  [T_\CG \bslash Y_{\CG,\mu}]
 \]
obtained from the $\CG $- and $T_{\CG }$-equivariant morphism $\delta$ above. More generally, for a 
choice of semi-group $S $  as in \S \ref{moregen}, defining the toric scheme $Y_S$,
 we define the divisor map  
 \[
\Delta_{\CG, S}:   [\CG \bslash \Mloc_{\CG,\mu}]\to  [T_\CG \bslash Y_S ]
 \]
 to be the composition of  $\Delta_\CG$ above with  $[T_\CG \bslash Y_{\CG,\mu}]\to [T_\CG \bslash Y_S ]$.
 \end{definition} 
  
 We now give the proof of Proposition \ref{prop:delta}.

\begin{proof}
 The uniqueness statement follows using that $\Mloc_{\CG,\mu}$ and hence ${\rm P}_{\CG,\mu}$ is flat over $O_E$. A similar argument proves that the $\CG$-equivariance
 of any extension $\delta$ of $\delta_E$ follows from the $G_E$-equivariance of $\delta_E$.
To extend $\delta_E$ to $ \delta$, it  is enough to base change to $O_{\br E}$  and check that, if $\chi\in S_\mu$, then the element $\delta_E^*(\chi)\in \Gamma(\br{\rm P}^{(-1)}_{\CG,\mu}[1/p], \CO_{\br{\rm P}_{\CG,\mu}}[1/p])$ actually lies in $\Gamma(\br{\rm P}^{(-1)}_{\CG,\mu}, \CO_{\br{\rm P}_{\CG,\mu}})$.
Since ${\rm P}_{\CG,\mu}\to \Mloc_{\CG,\mu}$ is smooth and $\Mloc_{\CG,\mu}$ normal and flat over $O_E$, ${\rm P}_{\CG,\mu}^{(-1)}={\rm P}_{\CG,\mu}$ is also normal and flat over $O_E$, and it is enough to verify that $\delta_E^*(\chi)$ is regular at the generic points
 of the special fiber of $\br{\rm P}_{\CG,\mu}$. The result then follows from the multiplicity formula (\ref{divsum}) in Conjecture \ref{divconj}  and the positivity appearing in the definition of $S_{\CG,\mu}$: Indeed, on sufficiently small open charts $\Spec(R)$ of $\Mloc_{\CG,\mu}\otimes_{O_E}O_{\br E}$ covering the generic points of its special fiber, we have ${\rm P}_{\CG,\mu}^{(-1)}=T_\CG\times \Spec(R)=\Spec(R[X^*(T_\CG)])$ and $\delta_E$ gives 
 \[
 f: \Spec(R[1/p])\to \Spec(\br\BQ_p[\omega^\pm_1,\ldots ,\omega^\pm_r]).
 \]
 For $\chi=\omega_1^{a_1}\cdots\omega^{a_r}_r\in S_{\CG,\mu}$, we have 
$ f^*(\chi)\in R$ 
 since $D_\chi$ is an effective divisor on $\Mloc_{\CG, \mu}\otimes_{O_E}O_{\br E}$ by (\ref{divsum}) and by the definition of $S_{\CG,\mu}$. So this gives
 \[
 \bar f: \Spec(R)\to \Spec(\br\BZ_p[S_{\CG,\mu}])=\br Y_{\CG,\mu}
 \]
 which, together with $T_\CG\to   Y_{\CG,\mu}$,  defines the extension $ \delta: T_\CG\times \Spec(R)\to \br Y_{\CG,\mu}$ over $\Spec(R)$.
 In particular, $\delta_E^*(\chi)$ is regular over $\Spec(R)$.
\end{proof}

%SKATA

\subsubsection{Combinatorics of the divisor map}\label{ss:comfa} 

 Let us assume that $\sigma_{\CG, \mu}$ is a strictly convex cone. 
 Consider the partially ordered set of faces ${\mathscr F}(\sigma_{\CG,\mu})$ of the cone $\sigma_{\CG,\mu}\subset X_*(T_\CG)_\BR$ with $\tau'\leq \tau$
if $\tau'$ is a face of $\tau$. This is in bijection 
$\tau\mapsto O(\tau)$ with the set of
$T_\CG$-orbits in $\br Y_{\CG,\mu}$ and $\tau'\leq \tau$ if and only if $O(\tau)\subset \ov{O(\tau')}$, \cite[Thm. 3.2.6]{CLS}. Note that  ${\mathscr F}(\sigma_{\CG,\mu})$  has a unique maximal element, namely $\sigma_{\CG,\mu}$  and a unique minimal element, namely $\{0\}$.

Also, let ${\rm Adm}^\CG(\{\mu\})\subset W^\CG\bs \wt W/W^\CG$ be the $\{\mu\}$-admissible set which parametrizes the $\CG_k$-orbits $S_w$
in $\Mloc_{\CG,\mu}\otimes_{O_E}k$, cf. \cite[\S 4.4]{PRS}. This   also has a partial order induced by the Bruhat order on $\wt W$ with $w'\leq w$ iff $S_{w'}\subset \ov{S_w}$, cf. \cite[Prop. 4.18]{PRS}.
Note that ${\rm Adm}^\CG(\{\mu\})$ contains a smallest element $\tau$ which corresponds to the unique closed $\CG_k$-orbit on 
$\Mloc_{\CG,\mu}\otimes_{O_E}k$.  The maximal elements are exactly given by the elements of  the image of $\Lambda_{\{\mu\}}$ in ${\rm Adm}^\CG(\{\mu\})$.

The map of stacks $\Delta_\CG$  gives a map,
\begin{equation}\label{divcomb}
|\Delta_\CG|: {\rm Adm}^\CG(\{\mu\})\to {\mathscr F}(\sigma_{\CG,\mu}).
\end{equation}
Recall that the set 
 $$
 \{\phi_\CG(\bar\mu')\in X_*(T_\CG)_\BR\mid \bar\mu'\in\Lambda_{\{\mu\}}\}
 $$
 generates the cone $\sigma_{\CG,\mu}$. Let 
 \[
 \rho^\CG_{\bar\mu'}:=\BR_{>0}\cdot \phi_\CG(\bar\mu').
 \]
When $\phi_\CG(\bar\mu')\neq 0$, this is the ray generated by $\phi_\CG(\bar\mu')$.

  \begin{proposition}\label{prop:deltastrata}
  The  map $|\Delta_\CG|: {\rm Adm}^\CG(\{\mu\})\to {\mathscr F}(\sigma_{\CG,\mu})$ reverses the  orders on source and target.
     \end{proposition}
   
\begin{proof}
Let $S^1_w\subset {\rm P}_{\CG,\mu}$ be the inverse image of the stratum $S_w\subset \Mloc_{\CG,\mu}\otimes_{O_E}k$
under the $T_\CG$-torsor $\pi: {\rm P}_{\CG,\mu}\to \Mloc_{\CG,\mu}$. Since $\CG_k$ acts transitively on $S_w$ and $\delta$ is $\CG$- and torus-equivariant, $\delta(S^1_w)$ is a single torus orbit which, by definition, is $O(|\Delta_\CG|(w))$.
Suppose $w'\leq w$. This implies $S_{w'}\subset \ov{S_w}$, so $S^1_{w'}=\pi^{-1}(S_{w'})\subset \pi^{-1}(\ov{S_w})=\ov{S^1_w}$, since the projection $\pi$ is smooth, hence open. Now
\[
O(|\Delta_\CG |(w'))=\delta(S^1_{w'})\subset \delta(\ov {S^1_{w}})\subset \ov{\delta(S^1_{w})}=\ov{O(\Delta_\CG(w))}.
\]
Hence, $|\Delta_\CG |(w)\leq |\Delta_\CG|(w')$.
 \end{proof}

 \subsubsection{The face map}

We have the following description of the map $|\Delta_\CG |$, which is completely in combinatorial terms.  We define as follows an order reversing map of posets called the \emph{face map},
\begin{equation}
|\Delta_\CG|^{\rm f}: {\rm Adm}^\CG(\{\mu\})\to {\mathscr F}(\sigma_{\CG,\mu}).
\end{equation}
For $w\in\Adm^\CG(\{\mu\})$, we set  
\[
|\Delta_\CG|^{\rm f}(w):=F_w={\rm Face}_{\sigma_{\CG,\mu}}\{\phi_\CG(\bar\mu')\ |\ \bar\mu'\in \Lambda_{\{\mu\}}, w\leq t^{\bar\mu'}\} ,
\]
which is, by definition, the smallest face of the cone $\sigma_{\CG,\mu}$ containing all the images $\phi_\CG(\bar\mu')$ for   $\bar\mu'\in \Lambda_{\{\mu\}}$ and $w\leq t^{\bar\mu'}$ in the parahoric Bruhat order.

\begin{theorem}\label{conj:fac}
 
  The map $|\Delta_\CG |$ coincides with the map $|\Delta_\CG |^{\rm f}$.

\end{theorem}

\begin{proof} To simplify the notation, we set $O=\br \BZ_p$, $\RM=\br\RM_{\CG,\mu}\otimes_{\BZ_p}O$, $\sigma=\sigma_{\CG,\mu}$, $T=T_\CG$, $S=S_{\CG,\mu}$, $Y=\br Y_{\CG,\mu}$,
just for this proof. We also write $v_\lambda=\phi_\CG(\lambda)\in X_*(T)_\BR$, for $\lambda=\bar \mu'\in \Lambda_{\{\mu\}}$ an extreme element of the admissible set. Then the cone $\sigma$ is generated by the $v_\lambda$. 

Choose $x \in S_w$ and an open $U$ in $\RM$ which contains $x$ and is such that our $T$-torsor $P\to \RM $ is trivial over $U$. Then, after choosing a section, the map $\delta: \RM\to [T\backslash   Y]$ gives
$$
\delta: U\to Y= \Spec(O[S]) .
$$
We want to show that $\delta(x)$ lies in the $T$-orbit $O(F_w)$ in $Y$ (actually in the $T$-orbit $O(F_w)$ in the special fiber of $Y$). 

We start by recalling that the $T$-orbits of $Y$ are characterized by the elements of the semigroup $S$ which give invertible functions on them, cf. \cite[(3.2.7)]{CLS}. More precisely:

\begin{itemize}
\item[] \emph{A  point $y$ of $Y$ lies on the orbit $O(\tau)$, if and only if, for any $u\in S$, we have $u(y)\neq 0$ exactly when $u$ lies in the perpendicular $\tau^\perp$.}
\end{itemize}
 By applying the above to $\tau=F_w$ we obtain that

\begin{itemize}
\item[] \emph{
$\delta(x)$ is in $O(F_w)$ if and only if,  for any $u$ in $S$, we have $u(\delta(x))\neq 0$  precisely when  $u$ lies in the perpendicular $F_w^\perp$. }
\end{itemize}

\smallskip

But since $U$ is normal, $u(\delta(x))\neq 0$ is equivalent to the divisor   of $\delta^*(u)$ missing $x$.
Recall that we assume the divisor conjecture, i.e. Conjecture \ref{divconj}. This gives
$$
{\rm supp}(D_u)=\bigcup_{\{\lambda\mid \langle u, v_\lambda\rangle>0\}} Z_\lambda .
$$
Moreover, $x$ is in $Z_\lambda$ if and only if  $S_w\subset Z_\lambda$ and this happens if and only if $ w\leq t^\lambda$.
Hence, 
\[
u(\delta(x))\neq 0\ \hbox{\it if and only if\ }  \langle u, v_\lambda\rangle=0\ \hbox{\it for all $\lambda\in \Lambda_{\{\mu\}}$ with $w\leq t^\lambda$}.
\]
Since $\sigma\cap u^\perp$ is a face of $\sigma$, by the defining minimality of $F_w$, the last condition is equivalent to
$
F_w\subseteq u^\perp$, or equivalently $u\in F_w^\perp$.
Thus, for every $u\in S$, $u(\delta(x))\neq 0$
 if and only if 
$u\in F_w^\perp$,
which gives the proof.
\end{proof}

\begin{corollary}\label{cor:oldC}
The map $\Delta_\CG$ sends the unique closed $\CG_k$-orbit in $\RM_{\CG,\mu}\otimes_{O_E}k$ to the unique closed torus orbit in the special fiber of $\br Y_{\CG,\mu}$.
\end{corollary}
\begin{proof}
Let $\tau_{\rm min}$ denote the smallest element of ${\rm Adm}^\CG(\{\mu\})$, corresponding to the unique closed $\CG_k$-orbit in $\RM_{\CG,\mu}\otimes_{O_E}k$. Since $\tau_{\rm min}\leq t^\lambda$ for all $\lambda\in\Lambda_{\{\mu\}}$, Theorem \ref{conj:fac} gives 
$|\Delta_\CG|(\tau_{\rm min})={\rm Face}_{\sigma_{\CG,\mu}}\{v_\lambda\ |\ \lambda\in\Lambda_{\{\mu\}}\}=\sigma_{\CG,\mu}$.
\end{proof}
 
\begin{remark}\label{rem:oldBandC}
a) When $\CG=\CI$ is Iwahori and $\sigma_{\CI,\mu}\neq \{0\}$, the  elements $v_\lambda$, $\lambda\in \Lambda_{\{\mu\}}$  give, in fact, exactly the set of extremal rays of the cone $\sigma_{\CI,\mu}=\sigma_{\{\mu\}}$.  To see this we  argue as follows: The vectors $v_\lambda$ generate, by definition, the strictly convex cone $\sigma_{\CI, \mu}$ 
 and they form a single orbit under the relative Weyl group $W_0$. Hence, $\sigma_{\CI, \mu}$ is $W_0$-invariant. 
 By averaging and rescaling we can find a  $W_0$-invariant linear functional $h$ which is positive on every non-zero element of the cone and with $h(v_\lambda)=1$, for all $\lambda\in \Lambda_{\{\mu\}}$. The convex hull $C$ of $\{v_\lambda\ |\ \lambda\in \Lambda_{\{\mu\}}\}$  is then the intersection $\sigma_{\CI,\mu}\cap h^{-1}(1)$
 and the extremal rays of the cone $\sigma_{\CI,\mu}$ correspond to the vertices of $C$. Every vertex of  the convex hull of a finite set belongs to that finite set. Hence, the set of extremal rays of $\sigma_{\CG,\mu}$ is a subset of the rays given by $v_\lambda$, $\lambda\in \Lambda_{\{\mu\}}$. To show the converse,  choose a $W_0$-invariant positive-definite inner product $\langle\ ,\ \rangle$ and set 
 $f_\lambda(x)=\langle v_\lambda, x \rangle$. Since all vectors $v_{\lambda'}$, $\lambda'\in \Lambda_{\{\mu\}}$, have the same norm $R=\langle v_\lambda ,v_\lambda \rangle^{1/2}$, 
 \[
f_\lambda(v_{\lambda'})=R^2-\frac{1}{2}\langle v_\lambda- v_{\lambda'} , v_\lambda- v_{\lambda'} \rangle<R^2,
 \]
 when $v_\lambda\neq v_{\lambda'}$. Hence, the functional $f_\lambda$ restricted to $C$ has its unique maximum at $x=v_\lambda$. But if a linear functional has a unique maximum at some point of a convex polytope, then that point is a vertex. 
Hence, $v_\lambda$ is a vertex of $C$ and so the ray given by $v_\lambda$ is an extremal ray of $\sigma_{\CI,\mu}$.

 b) The conclusion of a) does not always hold in the  parahoric case: An element $v_\lambda$, $\lambda\in \Lambda_{\{\mu\}}$, does not always generate an extremal ray of the cone $\sigma_{\CG,\mu}$. A simple example is provided by $G=\GL_4$, $\mu$ the coweight given by $(1,1,0,0)$, and $\CG$ 
  the parahoric corresponding to $J=\{s_1, s_3\}$, i.e. the stabilizer of a lattice chain $p\Lambda\subset \Lambda'\subset \Lambda$ with $\dim_{\BF_p}\Lambda/\Lambda'=2$. Then the projection $X_*(T)\to X_*(T_\CG)$ is given by $(x_1,x_2,x_3,x_4)\mapsto (x_1+x_2, x_3+x_4)$
 and the projection of the Weyl group orbit of $(1,1,0,0)$ is the set $\{(2,0), (1,1), (0,2)\}\subset X_*(T_\CG)_\BR=\BR^2$. The cone $\sigma_{\CG,\mu}$ is $\BR^2_{\geq 0}$ but the image $(1,1)$ of $(1,0,1,0)$ 
  is in the interior and does not generate an extremal ray.
\end{remark}

There is the following variant of the face map. Let $(G, \{\mu\}, \CG)$ be a triple as usual, and let $(G_\ad, \{\mu_\ad\}, \CG_\ad)$ be the corresponding adjoint triple. We do not assume that $\{\mu\}$ is minuscule but we assume that $\mu_\ad$ is non-trivial in all simple factors of $G_\ad\otimes_{\BQ_p}\br\BQ_p$. Recall from \S \ref{ss:projtor} the polytope 
$P^\CG_{\mu_{\ad}}$ which is the convex hull of the image of $\Lambda_{\{\mu\}}$ in $X_*(T_{\CG_\ad})_\BR$. Then  $P^\CG_{\mu_{\ad}}$ contains the origin in its interior. An obvious variant of the definition of the face map defines the \emph{polytope face map}
\begin{equation}\label{polfac}
\nabla_\CG: {\rm Adm}^\CG(\{\mu\})\to {\mathscr F}(P^\CG_{\mu_\ad}).
\end{equation}
Now let $(G, \{\mu\}, \CG)$  as in \S \ref{ss:comfa}, in particular, $\sigma_{\CG, \mu}$ is a strictly convex cone. Then there are bijections between the following three sets: (a) the set of vertices of  $P^\CG_{\mu_\ad}$; (b) the set of extreme rays of the fan $\Sigma^\CG_{\mu_\ad}$ of \S\ref{ss:projtor}; (c) the set of extreme rays of $\sigma_{\CG, \mu}$. Under the resulting identification ${\mathscr F}(P^\CG_{\mu_\ad})={\mathscr F}(\sigma_{\CG, \mu})\setminus \{0\}$, the  face map $|\Delta_\CG |^{\rm f}$ coincides with the composition of $\nabla_\CG$ and the inclusion ${\mathscr F}(P^\CG_{\mu_\ad})\hookrightarrow{\mathscr F}(\sigma_{\CG, \mu})$. 

\begin{remark}\label{rem:yu}
 The map \eqref{polfac} can be useful in analyzing the set $ {\rm Adm}^\CG(\{\mu\})$ which has a notoriously difficult structure. In response to some of our conjectures, Q. Yu indicated to us that he can prove the following properties of  $\nabla_\CG$ (for details comp. \cite{Yu}). 
 \begin{altenumerate}
 \item $\nabla_\CG$ is surjective but not injective  (the source has potentially many more elements than the target).
 \item There is a criterion for when two elements $w, w'\in\Adm^\CG(\{\mu\})$ have the same image under $\nabla_\CG$. 
 \item The fibers of $\nabla_\CG$ give a disjoint decomposition of ${\rm Adm}^\CG(\{\mu\})$ into some pieces which are ``primitive'' and some pieces which can be identified with admissible subsets of smaller groups.
 \end{altenumerate}
 \end{remark}
 
 \bigskip

\subsection{Lang covers}

 \subsubsection{Lang torsor for unramified  tori}
Let $T$ be a torus over $\BZ_p$. 
We set $\overline T=T\otimes_{\BZ_p}\BF_p$ for its reduction modulo $p$ and denote by ${\rm Frob}: \overline T\to \overline T$
the Frobenius which is an isogeny of $\overline T$. We also consider the Lang isogeny 
 \[
 \overline T\to \overline T,\quad x\mapsto {\rm Frob}(x)x^{-1}.
 \]
 Recall that reduction modulo $p$ gives an isomorphism
\[
{\rm Hom}_{\Spec(\BZ_p)}(T, T)\xrightarrow{\sim }{\rm Hom}_{\Spec(\BF_p)}(\overline T, \overline T)
\]
which takes isogenies to isogenies and preserves degrees (see, for example, \cite{Conrad}, App. B.3, in particular, Thm. B.3.2.)
 Hence there is a unique isogeny ${\rm Fr}: T\to T$ lifting the Frobenius
${\rm Frob}$ and, similarly, there is
\begin{equation}
L:  T\to T, \quad L(x)={\rm Fr}(x)x^{-1},
\end{equation}
 the unique isogeny lifting the Lang isogeny over $\BF_p$. We also call this isogeny the \emph{Lang isogeny} of $T$. The degree of $L$ is prime to $p$ and, in fact,
the kernel of $L$ is $\underline{ T(\BF_p)}$, the constant group scheme given by the finite  abelian group 
$T(\BF_p)$. Hence, we have a  sequence of group scheme homomorphisms over $\BZ_p$
 \[
 1\to \underline{ T(\BF_p)}\to T\xrightarrow{\ L\ } T\to 1,
 \]
 which is exact for the \'etale topology. The Lang isogeny is functorial, in the sense that if $T\to T'$ is a homomorphism of tori over $\BZ_p$, then the  diagram 
\begin{equation*}\label{LangComp}
\begin{aligned}
 \xymatrix{
        T \ar[r] \ar[d]_{L}  &  T' \ar[d]^{L'} \\
        T\ar[r]  & T'.
        }
        \end{aligned}
\end{equation*}
 is commutative. 
 \begin{remark}
 Under the anti-equivalence $T\mapsto X^*(T)$  of the category of unramified tori over $\BQ_p$ and the category of unramified Galois representations of $\Gal(\bar\BQ_p/\BQ_p)$ on finitely generated free $\BZ$-modules, the Lang isogeny $L$ corresponds to $L^*: X^*(T)\to X^*(T)$ given by $\chi\mapsto p\sigma(\chi)-\chi$, where $\sigma$ lifts the Frobenius. 
   \end{remark}

 \subsubsection{Lang covers of toric embeddings}

 We continue with the same set-up as above. In particular, $T$ is a torus over $\BZ_p$. Let $E_0$ be a finite unramified extension of $\BQ_p$ and let $T_{O_{E_0}}\hookrightarrow Y=Y_\tau$ be an affine torus embedding 
 over $O_{E_0}$ defined by a rational polyhedral cone $\tau\subset X_*(T)_\BR$ which is $\Gal(\br\BQ_p/E_0)$-stable. Let $S=\tau^\vee\cap X^*(T)$ the corresponding semigroup so that $Y_S=Y_\tau$.
 
  Let $\wt Y_{S} $ be the integral closure of $Y_{S}$  in the Lang morphism cover  $L\otimes_{\BZ_p}O_{E_0}: T_{O_{E_0}}\to T_{O_{E_0}}$.  We have a Cartesian diagram
  \begin{equation}\label{LangNormalization}
\begin{aligned}
 \xymatrix{
        T _{O_{E_0} }\ar@{^{(}->}[r] \ar[d]_{L}  &  \wt Y_{S}  \ar[d]^L \\
        T_{ O_{E_0}} \ar@{^{(}->}[r]  & Y_{S}.
        }
        \end{aligned}
\end{equation}
The  morphism $L: \wt Y_{S} \to Y_S$ is a finite $T (\BF_p)$-cover: The group $T (\BF_p)$ acts on $\wt Y_{S} $ preserving $L$ and the   morphism $L$ identifies $Y_S$ with the scheme-theoretic quotient $T (\BF_p)\bs \wt Y_S$.  The top horizontal arrow is also an affine toric embedding, see Proposition \ref{prop:normalToric} below.

\begin{proposition}\label{prop:normalToric} The normalization $\ti Y_S$ can be identified with $Y_{L^*(S)^{\rm sat}}$, where $L^*(S)^{\rm sat}\subset X^*(T_\CG)$ is the saturation of image of $S$ under the group homomorphism $L^*: X^*(T )\to X^*(T )$, $\chi\mapsto p\sigma(\chi)-\chi$. Alternatively, if $\tau\subset X_*(T_\CG)_\BR$ is the rational polyhedral cone which gives $S=\tau^\vee\cap X_*(T )$ and $Y_S=Y_\tau$, then $\wt Y_S$ can be   identified with the toric scheme $Y_{L_*^{-1}(\tau)}$ given by the inverse image cone $L_*^{-1}(\tau)\subset X_*(T )_\BR$, where $L_*: X_*(T )_\BR\to X_*(T )_\BR$ is given by $\mu\mapsto p\sigma(\mu)-\mu$. The Lang cover $L: 
\wt Y_{S} \to  Y_S$ can be identified with the morphism of toric schemes 
\[
Y_{L_*^{-1}(\tau)}\to Y_\tau
\]
induced by $L_*: X_*(T)\to X_*(T)$.
\end{proposition}

 \begin{proof} The morphism $Y_{L_*^{-1}(\tau)}\to Y_\tau$ can be identified with the toric morphism between the two affine toric schemes for the (same) cone $\tau\subset X_*(T)_\BR$ and the two lattices $L_*(X_*(T))$ and $X_*(T)$, which is obtained by 
  the finite index inclusion   $L_*(X_*(T))\subset  X_*(T)$ in $X_*(T)_\BR$, see \cite[Prop. 1.3.18]{CLS}. This morphism 
  is finite and, since $Y_{L_*^{-1}(\tau)}$ is normal, it agrees with the normalization $\wt Y_S$  of $Y_\tau=Y_S$ 
  in the cover $L: T\to T$.
 \end{proof}

\begin{remark}\label{remark:splitLang}
Note that if the torus is split, i.e., $T \simeq \BG_m^r$ over $\BZ_p$, then $L: \BG_m^r\to \BG_m^r$ is given by $L(x_1,\ldots, x_r)=(x^{p-1}_1,\ldots, x^{p-1}_r)$ and both
$L_*: X_*(T )\to X^*(T )$ and  $L^*: X^*(T )\to X^*(T_\CG)$, are given as dilation by a factor of $p-1$. Then, for every choice of semigroup $S$, we have 
$L^*(S)=(p-1)S$   and so $L^*(S)^{\rm sat}=S$. Alternatively, $L_*^{-1}(\tau)=\tau$. In this case, $\wt Y_S=Y_S$ and the Lang cover 
$L: Y_S\to Y_S$ is given by the ring homomorphism on semigroup algebras which sends $x$ to $x^{p-1}$, for all $x\in S$. 
\end{remark}

\subsubsection{Ramification along the boundary}\label{ss:ramify}
Recall that the boundary 
\[
[ Y_{\tau}\otimes_{O_{E_0}}{\br\BZ_p} \setminus T_{\br\BZ_p} ]=\bigcup\nolimits_{\rho}D(\rho)
\]
 of the toric embedding is a union of divisors $D(\rho)$, which are parametrized by the extremal rays $\rho=\BR_{> 0}\cdot \lambda$ ($1$-dimensional faces) of the cone 
 $\tau\subset X_*(T)_\BR$, cf. \cite[Thm. 3.2.6]{CLS}. Let us denote by $\lambda_\rho\in \rho\cap X_*(T)$  the minimal generator of the ray $\rho$ which belongs to $X_*(T)$, in particular
$\rho=\BR_{>0}\cdot \lambda_\rho$. 

A similar description of the boundary holds for the toric embedding $Y_{L_*^{-1}(\tau)}$. The Lang map given by 
 $L_*: X_*(T)_\BR\to X_*(T)_\BR$ maps each extremal ray $\wt \rho $ of $L_*^{-1}(\tau)$ to some extremal ray $\rho=L_*(\wt \rho)$ of $\tau$. 
 If $\lambda_{\wt \rho}\in \wt\rho\cap X_*(T)$ is the minimal generator of the ray $\wt \rho$, we can write
 \begin{equation}\label{defe}
 L_*(\lambda_{\wt \rho})=e_{\wt\rho/\rho}\cdot \lambda_{\rho}, 
 \end{equation}
 where $e_{\wt\rho/\rho}$ is a positive integer. 
 
 \begin{lemma}\label{lemma:ram}
The integer $e_{\wt\rho/\rho}$in \eqref{defe} is equal to the ramification degree of the extension 
 of dvrs $\CO_{\wt Y_S, D(\wt \rho)}/\CO_{Y_S, D(\rho)}$ which is 
 obtained by localizing the Lang cover $L: (\wt Y_S)_{\br \BZ_p}\to (Y_S)_{\br\BZ_p}$ at the generic points of the
  boundary divisors $D(\wt \rho)$ and $D(\rho)$.
 \end{lemma}
 
 \begin{proof}
It follows from the construction of the Lang cover and  the proof of \cite[Prop. 4.1.1]{CLS}.  
 \end{proof}

\subsubsection{An example}\label{ex:ramify}
 Suppose $T=\Res_{\BZ_q/\BZ_p}\BG_m$, with $q=p^d$. Then 
 \[
 X_*( T)=\BZ[\BZ/d\BZ]=\oplus_{i=0}^{d-1} \BZ\sigma^i
 \]
 is a permutation $\Gal(\BZ_{p^d}/\BZ_p)=\{1, \sigma,\ldots, \sigma^{d-1}\}$-module.
 We have 
 \[
 T\otimes_{\BZ_p}\BZ_q\simeq \BG_m\times\cdots \times \BG_m=\BG_m^d
 \]
 and $L: T\otimes_{\BZ_p}\BZ_q\to T\otimes_{\BZ_p}\BZ_q$ is given by
 \[
 L(a_0,a_1,\ldots ,a_{d-1})=L((a_i)_{i\in \BZ/d\BZ})=((a_{i-1}^pa_{i}^{-1})_{i\in \BZ/d\BZ}).
 \]
 Then, $L(a_0,a_1,\ldots, a_{d-1})=1$ amounts to $a_i^p=a_{i+1}$, for all $i$, which gives $a_0^{p^d-1}=1$. Hence, $\ker(L)\simeq \BF_{p^d}^*= T(\BF_p)$.
 
 We now consider the standard toric embedding $T=\Res_{\BZ_q/\BZ_p}\BG_m \subset Y=\Res_{\BZ_q/\BZ_p}\BA^1$.  
 The corresponding cone in $X_*(T)_\BR=\BR^d$ is $\tau=\{(a_i)\in \BR^d\ |\ a_i\geq 0,\forall i\}$ with extremal rays generated by the standard basis. 
 Since $L_*: \BR^d\to \BR^d$ is given by 
 $(a_i)\mapsto (pa_{i-1}-a_i)$, the inverse image is
 \[
 L_*^{-1}(\tau)=\{(a_i)\in \BR^d\ |\ pa_{i-1}\geq a_i,\forall i\}.
 \]
 The minimal generators of the extremal rays of $L_*^{-1}(\tau)$ are the vectors with  coordinates $1, p, \ldots , p^{d-1}$, cyclically permuted. For example, for $d=3$, we have $(1,p,p^2)$, $(p^2,1,p)$, $(p,p^2,1)$. Applying $L_*$ to any such vector gives $p^d-1$ times a standard basis vector. Hence, by Lemma \ref{lemma:ram}, the ramification degree of the cover $L: \ti Y\to Y$ along every boundary component is $p^d-1$, i.e. the cover is totally ramified along the boundary.

In this case, we can also show directly that the base change of the cover $L: \wt Y\to Y$  to $\BZ_q$ is given by the spectrum  of the $R= \BZ_q[x_0,x_1,\ldots, x_{d-1}]$-algebra 
 \[
\wt R= R[(u_i)_{i=0}^{d-1} ]/((u_{i}^p-x_{i+1}u_{i+1})_{i\in \BZ/d\BZ}, (u_0u_1\cdots u_{d-1})^{p-1}-x_0x_1\cdots x_{d-1}).
 \]
Indeed, over the open subscheme   of $\BA^d_{\BZ_q}=\Spec( R)$ given by the torus $T\otimes_{\BZ_p}\BZ_q=\BG_m^d$, we have $x_i=u_{i-1}^pu_i^{-1}$, for all $i$, and so this agrees with the  Lang map.
The morphism $\Spec(\wt R)\to \BA^d_{\BZ_q}$  is finite. We can also see that the fiber of the morphism over each point $x$ of $\BA^d_{\BZ_q}$ is the spectrum of a $k(x)$-algebra which has rank $p^d-1$ as a $k(x)$-vector space. Hence, $\Spec(\wt R)\to \BA^d_{\BZ_q}$ is  finite flat and it follows that $\wt R$ is Cohen-Macaulay. Over the open subscheme  $U(i)=\cup_{j\neq i}D(x_j)\subset \BA^d_{\BZ_q}$, only $x_i$ is not a unit. We can see, by an explicit calculation,  that 
\[
\Spec(\wt R)\times_{\BA^d_{\BZ_q}}U(i)\simeq \underline{\Spec}(\CO_{U(i)}[u_{i-1}]/(u_{i-1}^{p^d-1}-x_i\cdot v)),
\]
 with $v$ a unit in $\CO_{U(i)}$. (This also verifies the statement about total ramification along the boundary deduced above and is an analogue of Abhyankar's lemma in our context. Note that the complement of $\cup_{i=0}^{d-1} U(i)$ in $\BA^d_{\BZ_q}$ has codimension $2$).
Using this, we see  that $\wt R$ 
 is regular in codimension $1$ and by Serre's criterion normal. Hence,  $\wt R$ is indeed the integral closure as the definition requires and the base change of $ L: {\wt Y} \to   Y$ to $\BZ_q$ is given by
 \[
 R=\BZ_q[(x_i)_{i=0}^{d-1} ]\to  \wt R=R[(u_i)_{i=0}^{d-1} ]/((u_{i}^p-x_{i+1}u_{i+1})_{i\in \BZ/d\BZ}, (u_0u_1\cdots u_{d-1})^{p-1}-x_0x_1\cdots x_{d-1})
 \]
 as we wanted to show. Note that, in this example, $L:\wt Y\to Y$ is flat.

   \begin{remark}\label{rem:ramify}
a) Suppose $T$ is split as in Remark \ref{remark:splitLang}. Then for all choices of $S$, $e_{\wt\rho/\rho}=p-1$, since 
 $L_*: X_*(T)\to X_*(T)$ is dilation by $p-1$. Hence, if $T$ is a split torus and $p>2$, then all Lang covers $\wt Y_S\to Y_S$
 of $T$-embeddings are ramified over each component of the boundary. 
 
 b) Suppose $T={\rm Res}_{O_F/\BZ_p}(\BG_m)_{O_F}$, $[F:\BQ_p]=d=2$, and the torus embedding $Y=Y_\tau$ is given by the cone in $X_*(T)_\BR=\BR^2$ with extremal rays $(-1,p)$ and $(p,-1)$. The Lang torsor $L: T\to T$ induces $L_*: X_*(T)\to X_*(T)$ with $L_*(1,0)=(-1,p)$, $L_*(0,1)=(p, -1)$. Proposition \ref{prop:normalToric} implies that the normalization $\wt Y=Y_{L^{-1}_*(\tau)}$ is given by the first quadrant cone with extremal rays $(1,0)$ and $(0,1)$.
 Hence, $\wt Y={\rm Res}_{O_F/\BZ_p}\BA^1_{O_F}$. The cover $L: \wt Y\to Y$ is \'etale over the complement of the origin and the ramification degrees along the two divisors at the boundary are both $1$. 
\end{remark}

 \subsubsection{Lang covers of toric embeddings for $(\CG,\{\mu\})$}
Recall we assume $(\CG, \{\mu\})$ is a local model pair which is strictly convex. 
Let  $S\subset X^*(T_{\CG})$ be a semigroup as in \S \ref{moregen}. This produces the affine toric embedding $T_{\CG, O_{E_0}}\hookrightarrow  Y_{S}$ and we can apply the above constructions for $T=T_\CG$ to obtain the Lang cover $L: \wt Y_S\to Y_S$ by normalization: 

 \begin{equation}\label{LangNormalization2}
\begin{aligned}
 \xymatrix@R=3em@C=3em{
        T_{\CG, O_{E_0}} \ar@{^{(}->}[r] \ar[d]_{L}  &  \wt Y_{S}  \ar[d]^L \\
        T_{\CG, O_{E_0}} \ar@{^{(}->}[r]  & Y_{S}.
        }
        \end{aligned}
\end{equation}

 \begin{remark}\label{rem:YvsZ}
For this remark, we suppose $T_{\CG,\mu}\hookrightarrow T_\CG\otimes_{\BZ_p}O_{E_0}$ is a proper embedding, so assumption \S \ref{blanket} (ii) does not hold. We suppose however  that $T_{\CG,\mu}\hookrightarrow T_\CG$ is  defined over $\BZ_p$; this is the case when $E_0=\BQ_p$. The same normalization construction can be applied to the Lang morphism of the torus $T_{\CG,\mu}$ and the toric embedding $T_{\CG,\mu}\otimes_{\BZ_p}O_{E_0}\to Z_S$, see \S \ref{ss:generalsemigroups}. This produces a $T_{\CG,\mu}(\BF_p)$-cover $L: \wt Z_S\to Z_S$. Using the functoriality of the Lang morphism for $T_{\CG,\mu}\to T_\CG$ we can induce and obtain
 \[
 T_\CG\times_{T_{\CG,\mu}}\wt Z_S\to  Y_S=T_\CG\times_{T_{\CG,\mu}}Z_S; \quad (t, z)\mapsto (L(t), L(z)).
 \]
This agrees with the $T_\CG(\BF_p)$-cover $L: \ti Y_S\to Y_S$ as given directly above.
 \end{remark}
 The Lang covering $\wt Y_S\to Y_S$ is an isomorphism if $T_{\CG}(\BF_p)$ is trivial (the case of \emph{pointless groups}).  This is the case if and only if $p=2$ and $T_{\CG}$ is split \cite{MO}\footnote{ Here is a sketch of the argument. Assume that $T_\CG(\BF_p)=\{1\}$. Then by looking at $X_*(T_\CG)$ we conclude that $\det(p\sigma-1)=\pm 1$.  Let $\sigma$ have order $n$. Then the LHS is a product of factors of the form $p\zeta-1$, where $\zeta$ is an $n$-th root of unity. But $|p\zeta-1|\geq p-1$, with equality iff $p=2$ and $\zeta=1$, which implies the claim. }.

 \begin{proposition}\label{flatL}
a) If $Y_S$ is smooth over $O_{E_0}$ then  the Lang cover $L: \wt Y_S\to Y_S$ is flat. 

b) Assume that the torus $T_\CG$ is split, as in Remark \ref{remark:splitLang}. If $p=2$, the Lang cover is an isomorphism.  
If $p>2$ and the Lang cover $L: \wt Y_S\to Y_S$ is flat then $Y_S$ is smooth over $O_{E_0}$.
\end{proposition}

\begin{proof}   a)  The scheme $\wt Y_{S}$ is Cohen-Macaulay, as is any toric variety. Hence, if $Y_S$ is smooth over $O_{E_0}$, $\wt Y_S\to Y_S$ is flat by the ``miracle flatness" theorem.

b) Here we assume $T_\CG$ is split. The assertion regarding the case $p=2$ is obvious. Now let $p>2$ and assume $L: \wt Y_S\to Y_S$ is flat. Using the construction in Remark \ref{rem:YvsZ}, we see that the Lang cover $L: \wt Z_S\to Z_S$, for the torus $T_{\CG,\mu}$, is flat. Recall $Z_S=Z_{\ov S}$ and $\wt Z_S=\wt Z_{\ov S}$; these toric schemes have no torus factor.
 Let $x_1,\ldots, x_s$ be the minimal set of generators of $\ov S$ (the ``Hilbert basis"). Recall that since $\ov S$ generates the group $X^*(T_{\CG,\mu})$, we have $s\geq r$ with equality if and only if $\ov S$ is free. The ideal $(x_1,\ldots, x_s)$ cuts out the unique  closed $T_{\CG,\mu}$-orbit in $Z_S$ and $x_1,\ldots, x_s$ minimally generate the maximal ideal  there, cf.  \cite[Lemma 1.3.10]{CLS}.
By the above description of the cover, the fiber of $L: (Z_S)_k\to (Z_S)_k$ over the closed orbit is the spectrum of the $k$-algebra $k[\ov S]/(x^{p-1}_1,\ldots, x^{p-1}_s)$. Assuming flatness of $L: (Z_S)_k\to (Z_S)_k$, the elements $x^{p-1}_1,\ldots , x^{p-1}_s$ are ``independent'' in $k[\ov S]$  by \cite[Lem. 51.17.5]{Stacks}, in the sense defined there. It then follows by 
\cite[Lem. 51.17.4]{Stacks} that $k[\ov S]/(x^{p-1}_1,\ldots, x^{p-1}_s)$ has 
rank $(p-1)^s$
over $k$. By flatness, this is equal to $(p-1)^r$, which is the  degree of the cover generically. Since $p>2$, we deduce $r=s$. Since $\ov S$ is free, $Z_S$ and then $Y_S=T_\CG\times_{T_\CG,\mu}Z_S$ are both smooth. %The converse is easy, as explained above.
\end{proof}

The general case, when $T_\CG$ is not necessarily split, is more involved. However, there is evidence for the following:

 \begin{conjecture}\label{conjLco} Let $(\CG,\{\mu\})$ be a strictly convex local model pair, with  
  $\CG$  parahoric. Assume that $p>2$ and that $\dim\sigma_{\CG,\mu}=\rank T_{\CG}$, or equivalently $T_\CG=T_{\CG,\mu}$.
Then the Lang cover
\[
\wt Y_{\CG,\mu}\to Y_{\CG,\mu}
\]
is flat if and only if $Y_{\mathcal G,\mu}$ is smooth over $O_{E_0}$.
 \end{conjecture}

\begin{remark}\label{rem:Flatcounterex}  When $T_\CG$ is split, this follows from Proposition \ref{flatL}. When $T_\CG$ is not split,
there are  examples produced by ChatGPT-6 which show that the hypotheses $p>2$ and $\dim\sigma_{\CG,\mu}=\rank T_{\CG}$, i.e. $Y_{\CG,\mu}$ has no torus factor,
are both necessary in the above conjecture. Also, when $S\neq S_{\CG,\mu}$, there are examples  with $p>2$ and $\dim\sigma_{\CG,\mu}=\rank T_{\CG}$,
of  flat Lang covers $\wt Y_S\to Y_S$ with $Y_S$ not smooth. In fact, these last examples are for  Iwahori $\CG$.
\end{remark}

\begin{remark}\label{rem:Altmann}
We mention here some work of K. Altmann in progress \cite{Alt} which gives further evidence for Conjecture \ref{conjLco}. It is based on a combinatorial flatness criterion for the map $\wt Y\to Y$ between affine toric varieties of the same torus $T$ which prolongs the map $T\to T$ given by $x\mapsto \tau(x)^nx^{-1}$. Here $\tau$ is a finite order automorphism of $T$ and $n$ is a positive integer.  Using this criterion, he has confirmed Conjecture \ref{conjLco} in the following cases in which it is always assumed that $G$ splits over $\breve\BQ_p$ and $\CG$ is Iwahori:
\begin{enumerate}
\item When $G$ is a non-split form of $\GSp_{2g}$ obtained from a quaternionic hermitian form. The conjecture holds in this case.
When $g>1$, the Lang cover is never flat.
\item In the Hilbert-Siegel case, see \S \ref{ss:HilbertSiegel}. Then flatness occurs only when $g=d=1$.
\item When $G$ is of the form $\GL_n(D^*)$, where $D$ is a division algebra over $\BQ_p$. 
\item When $G$ is of the form $\GU(V)$, for some $K/\BQ_p$-hermitian vector space $V$. 
\item When $G$ is of the form $\GSpin(V)$ for an  orthogonal $\BQ_p$-vector space  of odd dimension $2m+1\geq 5$
with Witt index $m-1$ and $\mu$ the standard Hodge type coweight. Then the Lang cover is never flat.
\end{enumerate}
\end{remark}

\section{The root stack local model}\label{s:RootStack}

\subsection{Definitions}  We assume the pair $(\CG,\{\mu\})$ is strictly convex (\S \ref{blanket} assumption (i)).

 Furthermore, we choose $S\subset S_{\CG,\mu}$ to satisfy the conditions listed in \S\ref{ss:generalsemigroups}. 
 Under these assumptions,  the toric schemes $Y_{\CG,\mu}$ and $Y_S$ have (relative) dimension over $O_{E_0}$ equal to the rank of 
 $T_\CG$.

\begin{definition}
We define the stack $\Mlocroot_{\CG,\mu}$ as the fiber product of the Lang cover and $\delta$,
\begin{equation}\label{LMrootstack}
\begin{aligned}
 \xymatrix{
      \Mlocroot_{\CG,\mu} \ar[r] \ar[d]_L  &   [T_\CG \bs \wt Y_S ] \ar[d]^L \\
       \Mloc_{\CG,\mu}\ar[r]^\delta &  [T_\CG \bs Y_S ]  .
        }
        \end{aligned}
\end{equation}
 If $S=S_{\CG,\mu}$, we will  denote $\Mlocroot_{\CG,\mu}$ by $\Mlocrooto_{\CG,\mu}$.
\end{definition}
 The  motivation for calling $\Mlocroot_{\CG,\mu}$ the \emph{root stack local model} is given in  Remark \ref{intro:rootstack} in the Introduction.

\begin{proposition}\label{prop:DM}
 The stack $\Mlocroot_{\CG,\mu}$ is a (separated, finite type) Deligne-Mumford stack over $O_E$. 
\end{proposition}

\begin{proof}
Let $f: U\to \Mloc_{\CG,\mu}$ be an open neighborhood such that the pull back $f^*{\rm P}_{\CG,\mu}$ of the $T_\CG$-torsor 
${\rm P}_{\CG,\mu}\to \Mloc_{\CG,\mu}$ is trivial.
A section of $f^*{\rm P}_{\CG,\mu}$ gives a morphism $U\to Y_S$ which we compose with the natural projection $Y_S\to [T_\CG\bs Y_S]$. By 
(\ref{LMrootstack})
\[
U\times_{\Mloc_{\CG,\mu}}   \Mlocroot_{\CG,\mu} \simeq U\times_{[T_\CG\bs Y_S]}[T_\CG\bs \wt Y_S]\simeq 
U\times_{Y_S}(Y_S\times_{[T_\CG\bs Y_S]}[T_\CG\bs \wt Y_S]).
\]
On the other hand, we have
\[
Y_S\times_{[T_\CG\bs Y_S]}[T_\CG\bs \wt Y_S]\simeq [\und{T_\CG(\BF_p)}\bs \wt Y_S].
\]
Hence the diagram 
\begin{equation}\label{DMrstack}
\begin{aligned}
 \xymatrix{
      U\times_{\Mloc_{\CG,\mu}}   \Mlocroot_{\CG,\mu}\ar[r] \ar[d]  &   [\und{T_\CG(\BF_p)} \bs \wt Y_S ] \ar[d] \\
       U\ar[r] &  Y_S 
        }
        \end{aligned}
\end{equation}
is cartesian. Since $\und{T_\CG(\BF_p)} =T_\CG(\BF_p)$ is a finite constant group acting on $\ti Y_S$, $[\und{T_\CG(\BF_p)}\bs \wt Y_S]$ is Deligne-Mumford and then the assertion follows, cf.  \cite[Thm. 2.3.3]{Cadman}.  
\end{proof}

Since by Proposition \ref{prop:delta}, the morphism $\delta:  \Mloc_{\CG,\mu}\to [T_\CG \bs Y_S ]$ is $\CG$-equivariant for the trivial action on the target, the smooth group scheme $\CG$ acts on $ \Mlocroot_{\CG,\mu}\to  \Mloc_{\CG,\mu}$. It makes sense to consider the quotient stack which we call the \emph{root stack local model} associated to $(\CG, S)$,
\[
{\mathfrak M}^{\sqrt{S}}_{\CG,\mu}:=[\CG \bs \Mlocroot_{\CG,\mu}].
\]

\subsection{Properties} 

Here we discuss some (local) properties of  the stacks $\Mlocroot_{\CG,\mu}$ and the stack covers $\Mlocroot_{\CG,\mu}\to \Mloc_{\CG,\mu}$. 
We continue to assume the divisor conjecture, Conjecture \ref{divconj}.

\subsubsection{The (R1) property for $\Mlocroot_{\CG,\mu}$}

Throughout this paragraph, we assume both  \S \ref{blanket} (i) and (ii), i.e. $(\CG,\{\mu\})$ is strictly convex and $T_{\CG,\mu}=T_\CG$.

 Note that it makes sense to ask if the  Deligne-Mumford stack $\Mlocroot_{\CG,\mu}$ is (R1), i.e. regular in codimension $1$. 
 Recall from \S \ref{remark:psi} the definition of the homomorphism
  \[
  \phi_\CG: X_*(T)_I\to X_*(T)^I_\BQ\to X_*(T_\CG)_\BQ .
  \]
   By the definition of  the cone $\sigma_{\CG,\mu}\subset X_*(T_\CG)_\BR$, the non-zero vectors $\phi_\CG(\bar\mu')$, for $\bar\mu'\in \Lambda_{\{\mu\}}$,
   generate $\sigma_{\CG,\mu}$. Unless $\CG$ is Iwahori, their rays 
    \[
 \rho^\CG_{\bar\mu'}:=\BR_{>0}\cdot \phi_\CG(\bar\mu')
 \]
 need not all be extremal, see Remark \ref{rem:oldBandC}. By Theorem \ref{conj:fac}, the image under the divisor map of the generic point of the irreducible component of $\RM_{\CG,\mu}\otimes_{O_E}k$ indexed by $\bar\mu'$ lies in the orbit corresponding to the smallest face containing $\phi_\CG(\bar\mu')$. A zero vector corresponds to the open torus orbit.
Let $e=e_{E/\BQ_p}$ be the absolute ramification degree of the reflex field $E$. Observe that,  for all $\bar\mu'\in \Lambda_{\{\mu\}}$, the multiple  $e\cdot \phi_\CG(\bar\mu')$ lies in the lattice $X_*(T_\CG)$.

\begin{proposition}\label{prop:R1a}
Suppose that, for all $\bar\mu'\in \Lambda_{\{\mu\}}$, the element $e\cdot \phi_\CG(\bar\mu')$  is either zero or a primitive generator of an extremal ray of 
$\sigma_{\CG,\mu}$  in the lattice $X_*(T_\CG)$. Then the stack $\Mlocrooto_{\CG,\mu}$ is (R1).
 \end{proposition} 
 
 We postpone the proof to discuss the Iwahori case. Suppose that $\CG=\CI$ is Iwahori. Then $X_*(T)^I=X_*(T_\CI)$ and 
 \[
 e\cdot \phi_\CI(\bar\mu')={\rm Norm}_{\br E/\br\BQ_p}(\mu')=\sum_{\gamma\in I/I_{\bar\mu'}} \gamma\cdot \mu'\in X_*(T)^I.
 \]
 In this case, the condition in Proposition \ref{prop:R1a} is that,  for all $\bar\mu'\in \Lambda_{\{\mu\}}$, ${\rm Norm}_{\br E/\br\BQ_p}(\mu')$ is either zero or not divisible in the lattice $X_*(T)^I$, or equivalently in the lattice $X_*(T)$. 

If $\mu'$ is defined over $\br\BQ_p$, i.e. $e=1$, and $\mu'_\ad\neq 1$, then this holds since ${\rm Norm}_{\br E/\br\BQ_p}(\mu')=\mu'$ and $\mu'\in X_*(T)$ is minuscule. Indeed, unless $\mu'_\ad=1$, a minuscule coweight $\mu'$ is not divisible. We obtain:

 \begin{corollary}\label{cor:R1}
Suppose that  $E$ is unramified over $\BQ_p$ and let $\CG=\CI$ be Iwahori  and $\mu_\ad\neq 1$. 
Then the stack 
$\Mlocrooto_{\CG,\mu}$ is (R1). \qed
 \end{corollary}

The condition in Proposition  \ref{prop:R1a} can easily fail if $\CG$ is not Iwahori, even if $G$ is split.
It can also fail in the Iwahori case when the reflex field is ramified, i.e. when $e>1$. This is displayed by the following two examples.
  
\begin{example}\label{ex:sstableGortz}
  Consider $G=\GL_n$, $n>2$, $\mu=(1^{(r)}, 0^{(n-r)})$ with $1<r<n$. Take $\CG$ to be the parahoric (not Iwahori) group scheme defined by the stabilizer of a lattice chain 
\[
p\Lambda_0 \subset \Lambda_1\subset\Lambda_0
\]
 in $\BQ_p^n$ with $\Lambda_0/\Lambda_1\simeq \BZ/p\BZ$. In this case $T=\BG_m^n\to T_\CG=\BG_m\times\BG_m$ is given by 
 \[
 (a_1,\ldots, a_n)\mapsto (a_1, a_2\cdots a_n).
 \]
 The map $\Lambda_{\{\mu\}}\to X_*(T_\CG)=\BZ^2$ obtained by $\phi_\CG: X_*(T)\to X_*(T_\CG)$ has image the two element set given by $(1, r-1)$ and $(0, r)$; the second element is divisible in $X_*(T_\CG)$.

 \end{example}

\begin{example}\label{ex:RamGLn}
Consider $G={\rm Res}_{F/\BQ_p}\GL_n$, where $F/\BQ_p$ is a cyclic totally ramified extension with Galois group $\Gamma=\Gal(F/\BQ_p)=\{1,\gamma,\gamma^2,\gamma^3\}$ of order $4$. Fix $\iota: F\to \ov\BQ_p$ and let $\phi_i=\iota\cdot \gamma^i $ be corresponding embeddings $\phi_i: F\to \bar\BQ_p$. Take $\mu=(\mu_{\phi_i})_{i=1,2,3,4}$ to be given by $(1,0^{(n-1)})$, $(1, 0^{(n-1)})$, $(0, 0^{(n-1)})$, $(0, 0^{(n-1)})$, and $\CG=\CI$ Iwahori. Then the reflex field of $\mu$ is $E=F$ and so $e=4$. As above, we have
\[
 e\cdot \phi_\CI(\bar\mu)={\rm Norm}_{\br E/\br\BQ_p}(\mu)=2\cdot \nu
\]
where $\nu=(\nu_{\phi_i})_{i=1,2,3,4}$ is the $\Gamma$-invariant coweight 
given by $(1,0^{(n-1)})$, $(1, 0^{(n-1)})$, $(1, 0^{(n-1)})$, $(1, 0^{(n-1)})$.
Hence, in this case, $e\cdot \phi_\CG(\bar\mu)$ is divisible.
\end{example}

We also have a partial converse to Proposition  \ref{prop:R1a}:

\begin{proposition}\label{prop:R1b}
Suppose that $p>2$, that the torus $T_\CG$ is split and that  there is  
$\bar\mu'\in \Lambda_{\{\mu\}}$ such that  $e_{E/\BQ_p}\cdot \phi_\CG(\bar\mu')$ is  non zero and divisible in the lattice $X_*(T_\CG)$. Then the stack $\Mlocrooto_{\CG,\mu}$ is not (R1).
\end{proposition}
 
Also, when the semi-group $S$ is strictly smaller than the maximal choice $S_{\CG,\mu}$, the (R1) property fails quite easily.
 
\begin{proposition}\label{prop:R1c}
Suppose that $p>2$, that the torus $T_\CG$ is split and that $S\neq S_{\CG,\mu}$.  Let $\tau$ be the cone corresponding to $S$, and assume that, for some $\bar\mu'\in \Lambda_{\{\mu\}}$, the vector $\phi_\CG(\bar\mu')$ is non zero and its ray is not extremal in $\tau$.
Then the stack 
$\Mlocroot_{\CG,\mu}$ is not (R1).
\end{proposition}

We will now give the proofs of Propositions \ref{prop:R1a}, \ref{prop:R1b} and \ref{prop:R1c}. We start with the set-up which is the same for all these proofs.

Let $R$ be the local ring of $\Mloc_{\CG,\mu}\otimes_{O_E}O_{\br E}$ at the generic point $\eta$ of the irreducible component of its special fiber which corresponds
to the coset of $\bar\mu'\in \Lambda_{\{\mu\}}\subset \wt W$; this is a dvr with uniformizer $\varpi_E$. Choose a trivialization of the pull-back of the  $T_\CG$-torsor ${\rm P}_{\CG,\mu}\to \Mloc_{\CG,\mu}$ over $R$,
which gives $\delta: \Spec(R)\to Y_S$. The stack $\Mlocroot_{\CG,\mu}$ is (R1) along the inverse image of $\eta$ along $\Mlocroot_{\CG,\mu}\to \Mloc_{\CG,\mu}$ if and only if  the semi-local ring $\wt R$ defined by the fiber product
\begin{equation}\label{fiberedrootstack}
\begin{aligned}
 \xymatrix{
      \Spec(\wt R) \ar[r] \ar[d]_L  &     \wt Y_S   \ar[d]^L \\
       \Spec(R)\ar[r]^\delta &   Y_S   
        }
        \end{aligned}
\end{equation}
is  (R1).

 Since finite group quotients preserve codimension, every vertical codimension $1$ point of
the cover lies above a generic point of the special fiber. The generic fiber of the cover is smooth.
If $\phi_\CG(\bar\mu')=0$, then $\delta(\eta)$ lies in the open torus orbit, so $\wt R$ is \'etale over $R$, hence regular.

Suppose first that the image $\delta(\eta)$ lands on a codimension $1$ orbit $O(\rho)$ of the toric scheme $Y_S$ which corresponds to a ray $\rho$ in the cocharacter cone:

Note that $Y_S$ is smooth along $O(\rho)$ and so the divisor  $\ov{O(\rho)}$ is principal when restricted to some open neighborhood of $\delta(\eta)$ in $Y_S$. Using $T_\CG$-equivariance and Lemma \ref{lemma:ram}, we see that there is an  \'etale neighborhood $\Spec(\CO)\to Y_S$ of $\delta(\eta)$ such that the restriction of the cover $\wt Y_S\to Y_S$ is isomorphic to 
\[
\bigsqcup_{\wt\rho\mapsto \rho}\Spec(\CO[T]/(T^{e(\wt \rho/\rho)}-t)\to \Spec(\CO),
\]
where $t\in \CO$ cuts out the inverse image of $O(\rho)$ in $\Spec(\CO)$. (Note that $(e(\wt\rho/\rho), p)=1$.) By base changing the diagram (\ref{fiberedrootstack}) along $\Spec(\CO)\to Y_S$, we now see that $\wt R$ is \'etale locally  isomorphic to 
\[
\bigsqcup_{j\in J}\Spec(R'[U]/(U^{e(\wt \rho/\rho)}-\delta^*(t))\to \Spec(R'),
\]
where $\Spec(R')\to \Spec(R)$ is an \'etale neighbourhood  of $\eta$. We can assume that $R'$ is the strict henselization of the dvr $R$. 

Recall that the multiplicity of the zero of the function on $Y_S$ given by $\chi\in S\subset X^*(T_\CG)$ along $O(\rho)$ is given as the value of the pairing 
$\langle \lambda_\rho, \chi\rangle_\CG$, where $\lambda_\rho\in X_*(T_\CG)$ is a minimal generator of the ray $\rho$. 

\textit{Proof of Proposition \ref{prop:R1a}:}
In this we suppose  that $S=S_{\CG,\mu}$ and $Y=Y_{\CG,\mu}$ and let the $\eta\in \Mloc_{\CG,\mu}$ be the generic point of the irreducible component of the geometric special fiber which corresponds to the double coset of $\bar\mu'$
in the $\{\mu\}$-admissible set.  The zero vector case was treated above.
Otherwise, by Theorem \ref{conj:fac} and the extremality hypothesis, the image $\delta(\eta)$ indeed  lands on the codimension $1$ orbit $O(\rho^\CG_{\bar\mu'})$ of the toric scheme $Y$ which corresponds to the ray $ \rho^\CG_{\bar\mu'}=\BR_{>0}\cdot \phi_\CG(\bar\mu')$. The divisor formula (\ref{divsum}) with the above
now implies that the valuation of $\delta^*(t)$ in the dvr $R'$ is the ratio 
\[
n(\bar\mu'):=\frac{  e\cdot \phi_\CG(\bar\mu')}{\lambda_{\rho_{\bar\mu'}}}
\]
Hence, $\delta^*(t)=\varpi_E^{n(\bar\mu')}\cdot u$, $u\in (R')^*$. Since we take $R'$ to be the strict henselization of $R$ we conclude that the ring $\ti R$, after base changing to $R'$, is isomorphic to a finite direct sum of copies of 
\begin{equation}\label{ringcover}
R'[U]/(U^{e(\bar\mu')}-\varpi^{n(\bar\mu')}_E).
\end{equation}
In this, $e(\bar\mu'):=e(\ti\rho/\rho_{\bar\mu'})$  is prime to $p$ and is independent of the choice of $\wt\rho$ above $\rho_{\bar\mu'}$. (Note that the unit $u$ above has an 
$e(\bar\mu')$-th root  in $R'$; this is used to absorb
its appearance by a simple substitution.) 

Observe now that, since $R'$ is a dvr with uniformizer $\varpi_E$, the ring (\ref{ringcover}) is (R1), equivalently a dvr, in exactly two cases: When 
$n(\bar\mu')=1$ which amounts to $e\cdot \phi_\CG(\bar\mu')$ being a minimal generator of the ray $\rho_{\bar\mu'}$, or when $e(\bar\mu')=1$ which amounts to the Lang cover 
$L: \wt Y\to Y$ being unramified over $O(\rho^\CG_{\bar\mu'})$. Since  $e\cdot \phi_\CG(\bar\mu')$ is a minimal generator of the ray $\rho^\CG_{\bar\mu'}\subset X_*(T_\CG)_\BR$ if and only if it is indivisible in the lattice $X_*(T_\CG)$, the proof is complete. \qed
\smallskip

\textit{Proof of Proposition \ref{prop:R1b}:} Note that, when $T_\CG$ is split and $p>2$,
the Lang cover $\wt Y_{\CG,\mu}\to Y_{\CG,\mu}$  ramifies everywhere along the boundary of the torus embedding, cf.  Remark \ref{rem:ramify} (a). 
 If the   smallest face containing a nonzero divisible vector $e\cdot \phi_\CG(\bar\mu')$   has dimension at least two, the cotangent calculation
in the proof of Proposition \ref{prop:R1c} below   shows that $\Mlocroot_{\CG,\mu}$ is not (R1). Otherwise it lies on an extremal ray.   In the latter case, we have by assumption $e(\bar\mu')>1$.
As  in the proof above, if $e\cdot \phi_\CG(\bar\mu')$ is  divisible in the lattice $X_*(T_\CG)$, then $n(\bar\mu')>1$. Hence, by the above, if there is $\bar\mu'\in \Lambda_{\{\mu\}}$, with $e\cdot \phi_\CG(\bar\mu')$  nonzero and divisible, then $\ti R$ and hence $\Mlocrooto_{\CG,\mu}$ is not (R1). This concludes the proof. \qed

\begin{remark}\label{rem:relax}
We see that the condition that $T_\CG$ is split in Proposition \ref{prop:R1b} can be relaxed to the following condition: \emph{The Lang cover $\wt Y_{\CG,\mu}\to Y_{\CG,\mu}$  ramifies everywhere along the boundary of the torus embedding}. Note that if $T_\CG$ is split, this  condition  is satisfied, cf.  Remark \ref{rem:ramify} (a).
\end{remark}
  
\textit{Proof of Proposition \ref{prop:R1c}:}
Since $S\neq S_{\CG,\mu}$, the cone $\tau$ which corresponds to $S$
is strictly larger than the cone $\sigma_{\CG,\mu}$.  By  hypothesis, there exists $\bar\mu'$ such that the nonzero ray
$\rho^\CG_{\bar\mu'}$
fails to be extremal in $\tau$ and, instead, lies in a higher dimensional face.

Let $\eta$ be the generic point of the irreducible component of the geometric special fiber of $\Mloc_{\CG,\mu}$ corresponding to the coset of  $\bar\mu'$ and let $R$ be the local ring of $\Mloc_{\CG,\mu}\otimes_{O_E}O_{\br E}$ at $\eta$, as before. The  ring $\wt R$ of (\ref{fiberedrootstack}) is   isomorphic to 
\begin{equation}\label{tildering}
 R[S]/(s^{p-1}-\delta^*(s))_{s\in S}=(\br\BZ_p[S(1), S(2)]/
 (s(1)^{p-1}-s(2))_{s(1), s(2)\in S}\otimes_{\br\BZ_p[S(2)], \delta^*}R
\end{equation}
with $S(1)=S(2)=S$ two copies of the semigroup $S$ and $s(1)$, $s(2)$ labelled to distinguish the ``same'' element $s$ in the two copies.
 (Here we are using that $T_\CG$ is split, so $\wt Y_S=Y_S$ and the Lang cover $L: Y_S\to Y_S$ is obtained by taking $p-1$-th roots, cf. Remark \ref{remark:splitLang}.) 
 
 Since $\delta: \Spec(R)\to Y_S$ is the composition $\Spec(R)\to Y_{\CG,\mu}\to Y_S$,  and by the additional hypothesis, there are distinct extremal rays $r_1,\ldots, r_a$, $a\geq 2$, of $\tau$ which form a face $\tau'$ and such that $\delta(\eta)\in O(\tau')\subset \cap_{i=1}^a\ov{O(r_i)}$.  Now there are  $2$ elements $s$, $s'\in S$, which vanish on $O(\tau')$ and hence on $\delta(\eta)$, and 
 whose images  in the  transverse cotangent space of $Y_S$ at $\delta(\eta)$ are linearly independent.
  Hence, $\delta^*(s)$, $\delta^*(s')$    both belong to the maximal ideal of $R$. 
  We can now see, using (\ref{tildering})  that $s(1)\otimes 1$ and $s'(1)\otimes 1$ give two linearly independent elements in the cotangent space of $\wt R$ at each point above $\delta(\eta)$, provided that $p>2$. 
   Indeed, invert the monomials nonvanishing on $O(\tau')$ and
quotient the resulting monoid by its group of invertible elements.
This removes the orbit torus directions and gives a finitely generated
sharp saturated monoid $P$ describing the transverse monomials, with
$\rank P^{\mathrm{gp}}=\dim\tau'\ge2$.
Choose $s$, $s'$ to represent two indecomposable generators of $P$.
After strict henselization the torus power cover splits, and the covering
special fiber is of the form $\kappa[P]/(u^{p-1})_{\{u|u\in P\setminus\{0\}\}}$, with $\kappa$ a separably closed field. For $p>2$ these relations
lie in the square of the monomial maximal ideal, so the classes of $s(1)\otimes 1$ and $s'(1)\otimes 1$  in the cotangent space
remain independent. Therefore this space has dimension at least $2$. This implies that $\wt R$ is not (R1) and the proof is complete.
 \qed

\begin{example}
We take $G=\GL_n$, $n>2$, $\mu=(1^{(r)}, 0^{(n-r)})$ with $r>1$ and the parahoric $\CG$ as in Example \ref{ex:sstableGortz} and continue 
with the same notations.
We have 
 \[
 \sigma^\vee_{\CG,\mu}=\{(x, y)\in X^*(T_\CG)_\BR=\BR^2\ |\  yr\geq 0, x+y(r-1)\geq 0\},
 \] 
$S_{\CG,\mu}$ is free, generated by $(1,0)$ and $(-(r-1), 1)$, and the toric scheme $Y_{\CG,\mu}$ is $\BA^2_{\BZ_p}$
with the toric embedding $\BG_m^2\subset \BA^2$ given by $\BZ_p[u,v]=\BZ_p[t_1, t_2/t_1^{r-1}]\subset \BZ_p[t_1^{\pm1}, t_2^{\pm 1}]$.
By \cite[Prop. 12.1]{HPR}, an affine chart for the local model $\Mloc_{\CG,\mu}$ which covers all strata can be chosen to be 
\[
\Spec(\BZ_p[x,y, t_1,\ldots , t_{d-1}]/(xy-p)),
\]
where $d$ is the relative dimension of $\Mloc_{\CG,\mu}$. 
Furthermore, we can arrange so that the components $x=0$ and $y=0$ of the special fiber correspond to the cosets of $(1, r-1)$ and $(0, r)$ respectively, and the divisor morphism is defined and is equal to
\[
\Spec(\BZ_p[x,y, z_1,\ldots , z_{d-1}]/(xy-p))\to  \BA^2_{\BZ_p}=\Spec(\BZ_p[t_1, t_2/t_1^{r-1}]), 
\]
given by
\[
t_1\mapsto x,\quad  t_2/t_1^{r-1}\mapsto y^r.
\]
Over the generic point of the divisor of $\Mloc_{\CG,\mu}$  given by $y=0$, the stacky cover $\Mlocrooto_{\CG,\mu}$
is \'etale locally isomorphic to 
\[
[\BF_p^*\bs \Spec(\BZ_p[w, w', z_1,\ldots , z_{d-1}]/(w^{p-1}-p^r))]
\]
with $\BF_p^*$ acting on $w$ by Teichmuller multiplication. For $p>2$ and $r>1$, this is not (R1).
\end{example}

\subsubsection{Flatness of the cover $\Mlocroot_{\CG,\mu}\to \Mloc_{\CG,\mu}$}

Since 
 the Lang cover $L: \ti Y_S\to Y_S$ is finite, the pullback map $\Mlocroot_{\CG,\mu}\to \Mloc_{\CG,\mu}$  in (\ref{LMrootstack})  is also proper and quasi-finite. 
 For the next lemma, we only  assume that $(\CG,\{\mu\})$ is strictly convex. 
 
 \begin{lemma}\label{rootback}
The stacky cover $\Mlocroot_{\CG,\mu}\to \Mloc_{\CG,\mu}$ is flat if and only if the Lang cover $L: \ti Y_S\to Y_S$ is flat. 
 \end{lemma}
 
 \begin{proof} The one direction is easy. To show the other implication, assume that $\Mlocroot_{\CG,\mu}\to \Mloc_{\CG,\mu}$ is flat. Since $L: \ti Y_S\to Y_S$ is finite, $T_\CG$-equivariant and $Y_S$ integral, to show that $L$ is flat is enough to check that the rank of the fiber of the $\CO_{Y_S}$-coherent sheaf $L_*\CO_{\wt Y_S}$
 over a point in the unique closed $T_\CG$-orbit $O(\tau)$ is equal to the generic rank, i.e. equal to $\# T_\CG(\BF_p)$. Here $\tau$ is the cone corresponding to $S$ and we have $\sigma_{\CG,\mu}\subset \tau\subset X_*(T_\CG)_\BR$. 
  The cones have the same span and so the relative interior of $\sigma_{\CG,\mu}$ is contained in the relative interior of $\tau$.
 By Corollary \ref{cor:oldC}, we see that the unique closed $\CG_k$-orbit of $\Mloc_{\CG,\mu}$ maps to $O(\tau)$. Let $s$ be a $k$-valued point of $\Mloc_{\CG,\mu}$ in this closed orbit. We have a fiber product
 \begin{equation}\label{fiberedClosed}
\begin{aligned}
 \xymatrix{
      \Spec(\wt A) \ar[r] \ar[d]_L  &     \wt Y_S   \ar[d]^L \\
       \Spec(k)\ar[r]^{\delta(s)} &  \, Y_S.   
        }
        \end{aligned}
\end{equation}
Since $\Mlocroot_{\CG,\mu}\to \Mloc_{\CG,\mu}$ is flat, ${\rm rank}_k(\wt A)=\# T_\CG(\BF_p)$.
Hence, the fiber of $L_*\CO_{\wt Y_S}$ at $\delta(s)$ has rank $\# T_\CG(\BF_p)$ and hence, by the above, $L:\wt Y_S\to Y_S$ is flat.
\end{proof}

 Flatness of the Lang cover $\wt Y_S\to Y_S$ holds if  $S$ is free. 
  Conversely, if $T_{\mathcal G}$ is split and $p>2$, by Proposition \ref{flatL} flatness implies that $Y_S$ is smooth. When 
 $T_{\CG,\mu}=T_{\CG}$, this is equivalent to $S$ being free.

\begin{remark}\label{rem:5.2.12} If $E_0=\BQ_p$,  $T_\CG$ is split and $T_\CG=T_{\CG,\mu}$, there always exists $S$ such that $S$ is free. Indeed, $\sigma_{\CG, \mu}^\vee$ contains  a basis of $X^*(T_\CG)$ and we may take for $S$ the semigroup generated by it.  The Galois stability assumption is automatic  as  $T_\CG$ splits.  Proposition \ref{prop:R1c} applies when, in addition to the hypothesis that  $T_\CG=T_{\CG,\mu}$ is split and that $S\neq S_{\CG, \mu}$, we assume that   $\sigma_{\CG, \mu}$ has full dimension and that $S$ is free.    Indeed,
the cone $\tau$ for a free $S$ is simplicial and we have $\sigma_{\CG,\mu}\subset \tau$. 
Not all extremal rays of $\sigma_{\CG,\mu}$ can be extremal rays of $\tau$:  otherwise, because the extremal rays of $\tau$ give a basis of $X^*(T_\CG)_\BR$, we would have $\sigma_{\CG,\mu}=\tau$. Hence there indeed exists an extremal ray of $\sigma_{\CG,\mu}$ which is not  extremal   in $\tau$.
\end{remark}

We now assume that $\CG$ is Iwahori and that, in addition, $T_{\CG,\mu}=T_\CG$, i.e \S \ref{blanket} (ii) is satisfied. Preliminary to the question whether $S_{\CG, \mu}$ is free, we first address the question when $\sigma_{\CG,\mu}$ is \emph{simplicial}, i.e., the number of its extremal rays is equal  to the dimension of the space
$X_*(T_\CG)_\BR$.  We identify $X_*(T_\CG)$ with $X_*(T)_I$. 
The number of extremal rays of $\sigma_{\CG, \mu}$ is equal to the cardinality  of the image of $\Lambda_{\{ \mu\}}$ in $X_*(T_\ad)_I$, i.e., the number of elements of $W/W_{\bar\mu_\ad}$, where $W=W_0$ is the relative Weyl group over $\br\BQ_p$, and where  $W_{\bar\mu_\ad}$ is the stabilizer of the image of some element $\bar\mu\in \Lambda_{\{\mu\}}$ in $X_*(T_\ad)_I$. 

We make the assumption that $G_\ad\otimes_{\BQ_p}\br\BQ_p$ is simple.  We also assume that $(G, \{\mu\})$ is ab-nondegenerate. Then $\dim (X_*(T)_I)_\BR={\rm rank}_{\br\BQ_p}(G_\ad\otimes_{\BQ_p}\br\BQ_p)+1$, cf. Proposition \ref{dimT1}.  We are thus led to the following estimate. It is the analogue for Weyl groups of the   relation between the dimension of a partial flag variety $G/P$ and the rank of a simple adjoint group $G$, cf. \cite[Prop. 1.4]{OR}. 
\begin{proposition}\label{ORforW}
Let $(V, R)$ be an irreducible root system, with Weyl group $W$. Let $W_J\subset W$ be a proper parabolic subgroup. Then there is the inequality
$$\# (W/W_J)\geq \dim V +1, $$
with equality if and only if $(V, R)$ is of type $A_n$ and $J=\{s_1, s_2,\ldots, s_{n-1}\}$ or $J=\{s_2, s_3,\ldots, s_{n}\}$.
\end{proposition}
\begin{proof}
It obviously suffices to prove the statement for a maximal parabolic subgroup $W_J$. These are the Weyl groups obtained by removing one vertex from the Coxeter graph, in the table of \cite[ch. VI, \S 4, 1]{Bou} (we do not need the information of the Dynkin diagram, the Coxeter graph is enough).
\begin{altitemize}
\item \emph{Type $A_n, n\geq 1$}: If the node $v_i$ is removed, the  quotient by the corresponding parabolic subgroup is $S_{n+1}/S_i\times S_{n+1-i}$ and has order $\binom{n+1}{i}$, which has order $\geq n+1$, with equality iff $i=1$ or $i=n$.
\item \emph{Type $B_n, n\geq 2$}: If the node $v_n$ is removed, the  quotient by the corresponding parabolic subgroup is $(\BZ/2\BZ)^n\rtimes S_{n}/S_n$ and has order $2^n>n+1$. If the node $v_{n-1}$ is removed, the  quotient by the corresponding parabolic subgroup is $(\BZ/2\BZ)^n\rtimes S_{n}/S_{n-1}\times S_2$ and has order $2^{n-1}n>n+1$. If the node $v_{n-i}$ is removed with $2\leq i\leq n-1$, the  quotient by the corresponding parabolic subgroup is $(\BZ/2\BZ)^n\rtimes S_{n}/S_{n-i}\times (\BZ/2\BZ)^{i}\rtimes S_{i}$ and has order $2^{n-i}\binom{n}{i}>n+1$. 
\item\emph{Type $C_n, n\geq 3$}: This is dual to type $B_n$.
\item\emph{Type $D_n, n\geq 4$}: If the node $v_n$ or $v_{n-1}$ is removed, the  quotient by the corresponding parabolic subgroup is $(\BZ/2\BZ)^{n-1}\rtimes S_{n}/S_{n}$ and has order $2^{n-1}>n+1$.  If the node $v_{n-2}$ is removed, the  quotient by the corresponding parabolic subgroup is $(\BZ/2\BZ)^{n-1}\rtimes S_{n}/S_{n-2}\times S_2\times S_2$ and has order $2^{n-3}n(n-1)>n+1$. If the node $v_{n-i}$ is removed with $2<i\leq n-1$, the  quotient by the corresponding parabolic subgroup is $(\BZ/2\BZ)^{n-1}\rtimes S_{n}/S_{n-i}\times (\BZ/2\BZ)^{i-1}\rtimes S_{i}$ and has order $2^{n-i}\binom{n}{i}>n+1$. 
\item\emph{Type $E_6$}: The orders of the quotients are
$
3^3, \, 2^3\cdot 3^2, \,  2^3\cdot 3^3,\, 2^4\cdot 3^2\cdot 5,
$
all larger than $7$. 
\item\emph{Type $E_7$}: The orders of the quotients are
$$
2^3\cdot7, \quad 2^2\cdot 3^3\cdot  7,\quad 2^6\cdot 3^2\cdot 7, \quad 2^5\cdot 3^2\cdot 5\cdot 7, \quad 2^6\cdot 3^2, \quad 2^5\cdot 3^2\cdot 7, \quad 2\cdot 3^2\cdot 7,
$$
all larger than $8$. 
\item\emph{Type $E_8$}: The orders of the quotients are
$$
2^4\cdot 3\cdot 5, \quad 2^6\cdot 3\cdot 5\cdot 7, \quad 2^6\cdot 3^3\cdot 5\cdot 7,\quad 2^8\cdot 3^3\cdot 5\cdot 7, \quad 2^9\cdot 3^3\cdot 5\cdot 7,\quad
2^7\cdot 3^3\cdot 5, \quad 2^9\cdot 3^3\cdot 5, \quad 2^4\cdot 3^3\cdot 5, 
$$
all larger than $9$. 
\item\emph{Type $F_4$}: The orders of the quotients are
$
2^3\cdot 3,\, 2^5\cdot 3, 
$
both larger than $5$. 
\item\emph{Type $G_2$}: The orders of the quotients are all
$
2\cdot 3,
$
larger than $3$. \end{altitemize}
\end{proof}

\begin{remark}  There is a direct relation between the previous result and \cite[Prop. 1.4]{OR}. Indeed, let  $G$ be the simple adjoint group $G$ with root system $R$ and fix a Borel subgroup $B$ corresponding to a Weyl chamber $C$ in $V$. Choose $\lambda\in \bar C$  such that the stabilizer of $\lambda$ in $W$ is $W_J$, and let $\CV_{\lambda}$ be the irreducible representation with highest weight $\lambda$.
 Let $P=P_J$ be the standard parabolic corresponding to $J$: this is also the stabilizer of the line of $\CV_\lambda$ given by the highest weight space. We obtain a closed embedding $G/P\subset \BP(\CV_\lambda)$. We therefore obtain the estimate
$$
\dim G/P\leq \dim \CV_\lambda-1. 
$$
By \cite[Prop. 1.4]{OR} we have the inequality $r(G)\leq \dim G/P$, where $r(G)=\dim V$ denotes the rank of $G$, with equality iff $G$ is of type $A_n$ and $J=\{s_1, s_2,\ldots, s_{n-1}\}$ or $J=\{s_2, s_3,\ldots, s_{n}\}$. 

Assume now that $\lambda$ is minuscule. 
Then the set of weights of $\CV_\lambda$ forms one orbit under $W$ and can be identified with $W/W_J$, and $\dim \CV_\lambda=\# (W/W_J)$.
  Putting these inequalities together, we obtain 
$$
\# (W/W_J)\geq \dim V+1,
$$
with equality iff $G$ is of type $A_n$ and $J=\{s_1, s_2,\ldots, s_{n-1}\}$ or $J=\{s_2, s_3,\ldots, s_{n}\}$. 
\end{remark}

We deduce the following statement.
\begin{corollary}\label{classfree}
Let $(G, \CG, \mu)$, where $\CG$ is Iwahori,  such that  $G_\ad$ is absolutely simple.  We also assume that $(G, \{\mu\})$ is ab-nondegenerate
and that $T_{\CG,\mu}=T_\CG$.

\begin{altenumerate}
\item If the cone $\sigma_{\CG, \mu}$ is simplicial, then the pair $(G_\ad\otimes_{\BQ_p}\br\BQ_p, \mu)$ is isomorphic to either $(\PGL_n, \varpi^\vee_1)$ or $({\rm PU}_3, \mu_{2,1})$, where ${\rm PU}_3$ is the adjoint unitary group for a ramified quadratic extension $K$ of $\breve\BQ_p$---and conversely.

\item If $S_{\CG, \mu}$ is free, then the pair $(G_\ad\otimes_{\BQ_p}\br\BQ_p, \mu)$ is isomorphic to $(\PGL_n, \varpi^\vee_1)$  (the \emph{Drinfeld case}). 
\end{altenumerate}
\end{corollary}
\begin{proof} The first statement follows  from determining all cases when  the echellonage  root system is of type $A$ and satisfies the criterion in Proposition \ref{ORforW}. 

 For the second statement, the case $({\rm PU}_3, \mu_{2,1})$ needs to be excluded.
Assume $(G,\mu)$ satisfies our assumptions and $(G_\ad\otimes\br\BQ_p,\mu_\ad)\simeq ({\rm PU}_3, \mu_{2,1})$. 
Note that $X_*(T_\ad)^I\simeq \BZ^-$, where we denote by $\BZ^{\pm }$ the $W_0$-module with underlying group $\BZ$ where the generator of the Weyl group $W_0\simeq\BZ/2\BZ$ acts as multiplication by $\pm 1$.

We   write $G$ as the almost direct product of $G_\der$ and the central torus $Z^0=Z(G)^0$ so that there is an exact sequence
\[
1\to H\to Z^0\times G_\der\to G\to 1
\]
where $H\hookrightarrow Z^0\times G_\der\to G_\der$ identifies $H$ with a subgroup scheme of the center of $G_\der$. Restricting to maximal tori gives
\[
1\to H\to Z^0\times T_\der\to T\to 1.
\]
Our conditions (ab non-degeneracy and $T_{\CG,\mu}=T_\CG$) together with Proposition \ref{dimT1} imply that $X_*(T_\CG)=X_*(T)^I$   is free of rank $2$ so it is isomorphic to $\BZ^2$ (as a group).

There are two cases: 

a) $G_\der={\rm PU}_3$, then $G=Z^0\times {\rm PU}_3$ is a direct product.  Consider the maximal torus $T=Z^0\times T_\ad$. Then 
$X_*(Z^0)^I \simeq \BZ^+$ since $W_0$ acts trivially on the center. Hence, 
\[
X_*(T)^I\simeq \BZ^+\times\BZ^-.
\]

b) $G_\der={\rm SU}_3$.
The center of ${\rm SU}_3$ is a twisted form $\ti\mu_3$ of the group scheme $\mu_3$: We have $\ti\mu_3\otimes_{\br\BQ_p}K\simeq \mu_3\otimes_{\br\BQ_p}K$
and the action of $\langle \tau\rangle=\Gal(K/\br\BQ_p)$ on $X^*(\ti\mu_3)=\BZ/3\BZ$ is given by $\tau(a)=-a$. In this case, $H$ is   subgroup scheme of $\ti \mu_3$ and hence either $H=(1)$ or $H=\ti \mu_3$. We have
\[
0\to X_*(Z^0)\times X_*(T_\der)\to X_*(T)\to X^*(H)^*\to 0.
\]
Note that $X^*(H)^*$ has order $1$ or $3$ and so it is cohomologically trivial for the action of $\Gal(K/\br\BQ_p)$ which acts by $-1$. 
By the above, $X^*(H)^*_I=(X^*(H)^*)^I=(0)$. Taking  invariants gives
\[
0\to X_*(Z^0)^I\times X_*(T_\der)^I\to X_*(T)^I\to  (X^*(H)^*)^I=0.
\]
Hence, $X_*(T)^I\simeq X_*(Z^0)^I\times X_*(T_\der)^I$ as Weyl modules. A similar argument as above starting with
\[
1\to \ti\mu_3\to T_\der\to T_\ad\to 1
\]
gives $X_*(T_\der)^I\simeq X_*(T_\ad)^I\simeq \BZ^-$. Hence,   we obtain again
\[
X_*(T)^I \simeq \BZ^+\times\BZ^-.
\]

We now consider the image $\mu^\diam$ of $\mu$ in $X_*(T)^I_\BQ$ under the averaging map
\[
 X_*(T)\to X_*(T)^I_\BQ\simeq (\BZ^+\oplus \BZ^-)\otimes_\BZ\BQ,\quad \lambda\mapsto \lambda^\diam.
\] 
A primitive element of $X_*(T)^I\simeq \BZ^+\times\BZ^-$ which is in the ray spanned by $\mu^\diam$ is of the form $(a,b)$ with $ab\neq 0$. Here, $a\neq 0$ because $(G,\mu)$ is ab non-degenerate, and $b\neq 0$ because $\bar\mu_\ad$ is not trivial. The generator of $W_0$ takes $(a,b)$ to $(a,-b)$. It follows that the cone $\sigma_{\CG,\mu}\subset X_*(T)^I_\BR\simeq \BR\times \BR$ has extremal rays with primitive elements $(a, b)$ and $(a, -b)$ in $X_*(T_\CG)=X_*(T)^I\simeq \BZ^+\times\BZ^-$. The corresponding semigroup $S_{\CG,\mu}$ is not free since, for example, the determinant of the corresponding matrix is always divisible by $2$.
\end{proof}
\begin{remark}
The converse to Corollary \ref{classfree}, (ii) does not hold, as is shown by Example 3) in \S \ref{ss:Exam}: Indeed, the pairs $(\GL_n, \{\mu_1\})$ and $(\GL_n, \{\mu_{n-1}\})$ become isomorphic after passing to adjoint groups, comp. also Example \ref{DrversAnti}.  For the first pair, the semi-group $S_{\CG, \mu}$ is free but for the second it is not. 
\end{remark}

From Lemma \ref{rootback} and Corollary \ref{classfree}, we deduce the following statement.

\begin{corollary}\label{cor:NonFlat}
Let $(G, \CG, \mu)$, where $\CG$ is Iwahori,  such that  $G_\ad$ is absolutely simple. We also assume that $(G, \{\mu\})$ is ab-nondegenerate and that $T_{\CG,\mu}=T_\CG$.  If $T_\CG(\BF_p)$ is trivial (i.e., $p=2$ and $T_\CG$ is split), then the map $\Mlocrooto_{\CG,\mu}\to \Mloc_{\CG,\mu}$ is  an isomorphism. If $p>2$, then the map $\Mlocrooto_{\CG,\mu}\to \Mloc_{\CG,\mu}$ is flat only if $(G_\ad\otimes_{\BQ_p}\br\BQ_p, \mu)$ is isomorphic to $(\PGL_n, \varpi^\vee_1)$ (the \emph{Drinfeld case}), provided that Conjecture \ref{conjLco} on the flatness of Lang coverings of tori is satisfied. In particular, this holds if $T_\CG$ is split. 
\end{corollary}
Indeed, the case when $T_\CG$ is split follows from Proposition \ref{flatL}.\qed

\section{Shimura varieties}\label{sec:SV}

\subsection{The conjecture}\label{ss:Conjecture}

In this section, we formulate the main conjecture of the paper and describe some consequences.

\subsubsection{Toric abelian covers of Shimura varieties}\label{sss:ToricCovers}

 We start by fixing  the notation. Let  $(\RG, \RX)=(\RG, \{h\})$ be
  a  Shimura datum in the sense of Deligne, with associated conjugacy class of cocharacters $\mu(z)=h_{\BC}(z,1)$ and with reflex field $\RE\subset \BC$. 
 
 Let $p$ be a rational prime. We  fix an  embedding $\nu$ of $\RE$ in an algebraic closure $\bar\BQ_p$. We set $G=\RG_{\BQ_p}$, and denote by $\{\mu\}$  the conjugacy class of cocharacters over $\bar\BQ_p$, and by $E=\RE_\nu$ its local reflex field. We  choose a parahoric subgroup $K_0$ of $G(\BQ_p)$ with corresponding group scheme $\CG$ over $\BZ_p$. 
 We also choose a (sufficiently small) compact open subgroup $K^p\subset G(\BA^p_f)$  and set $\RK_0=K_0\cdot K^p\subset G(\BA_f)$.
Next we consider  the kernel $K_1$ of the natural homomorphism
\[
 \CG(\BZ_p)\to \ov\CG(\BF_p)\to \ov\CG_{\red, \rm ab}(\BF_p)
\]
 and, as in \S \ref{ss:211}, set $T_\CG$ for the torus over $\BZ_p$ which lifts the torus $\ov\CG_{\red, \rm ab}$ over $\BF_p$.
 In particular, $T_\CG(\BF_p)=\ov\CG_{\red, \rm ab}(\BF_p)$.
 We also set  $\RK_1=K_1\cdot K^p\subset G(\BA_f)$.

  We now consider the Shimura varieties ${\rm Sh}_{\RK_1}(\RG, \RX)$ and ${\rm Sh}_{\RK_0}(\RG, \RX)$ which are defined over the reflex field $\RE$.
  We can also consider $ {\rm Sh}_{ K_1 }(\RG, \RX)$ and  ${\rm Sh}_{ K_0}(\RG, \RX)$ obtained by taking inverse limits over the prime to $p$ level subgroups $K^p$. The natural morphisms  give  unramified $T_\CG(\BF_p)$-covers
 \[
 \pi_{K^p}: {\rm Sh}_{K_1 }(\RG, \RX)\to {\rm Sh}_{K_0}(\RG, \RX),\quad  \pi: {\rm Sh}_{\RK_1}(\RG, \RX)\to {\rm Sh}_{\RK_0}(\RG, \RX)
 \]
  defined  over $\RE$.
  We denote by $ {\rm Sh}_{ K_1 }(\RG, \RX)_E$ and  ${\rm Sh}_{ K_0}(\RG, \RX)_E$ the base changes under $\RE\to E$.

  \subsubsection{Integral models}\label{sss:612}

  We have the corresponding local model pair $(\CG,\{\mu\})$ with local model $\Mloc_{\CG,\mu}$. 
 We  assume that there is an 
  integral model $\CS_{K_0}:=\CS_{K_0}(\RG, \RX)$ of ${\rm Sh}_{\RK_0}(\RG, \RX)_E$ over $O_E$ with $\RG({\BA}_f^p)$-action
  and a smooth morphism
 \begin{equation}\label{LMmap}
 \varphi: \CS_{K_0}\to \fkM_{\CG,\mu}:=[\CG\bs\Mloc_{\CG,\mu}].
 \end{equation}
 Such an integral model is known to exist when $p>2$ and $(\RG, \RX)$ is of abelian type, 
 under some mild additional technical restrictions, cf. \cite{KP}, \cite{KPZ}. (When these additional hypotheses are violated, a slightly weaker result is known.
 Since this is not pertinent to the discussion in this paper, we just refer the reader to  \cite[\S 7]{KPZ} for details.)
The morphism $\varphi$ is a stacky formulation of the local model diagram, and it induces an isomorphism for each $x\in \CS_{K_0}(k)$ of strictly henselian local rings
 \begin{equation}\label{eq:sh}
 \CO^{\rm sh}_{\Mloc_{\CG,\mu},\varphi(x)}\isoarrow \CO^{\rm sh}_{\CS_{K_0}, x}.
 \end{equation}
 
Let us assume, in addition, that $(\CG,\{\mu\})$ is strictly convex, see \S \ref{ss:nondegeneracy}.  
We also assume the validity of the divisor conjecture (Conjecture \ref{divconj}).
(For example, both these assumptions are  satisfied when  $(G,\{\mu\})$ is of Hodge type, and $G$ splits over a tamely ramified extension of $\BQ_p$, and $p$ does not divide 
the order of $|\pi_1(G_\der(\ov\BQ_p))|$, by the combination of Proposition \ref{propAo}, Corollary \ref{cor:Hodgetype}, and Theorem \ref{tameconj}.)

 Choose a semigroup $S\subset S_{\CG,\mu}$ as in \S \ref{ss:generalsemigroups}. This determines the toric scheme $Y_S$ and the Lang cover $L: \wt Y_S\to Y_S$.  Definition \ref{def:DivisorMap} gives the divisor  
 morphism $\Delta_S: \fkM_{\CG,\mu}\to [T_\CG\bs Y_S]$. We also have the stack $\fkM^{\sqrt{S}}_{\CG,\mu}$, which fits in a fiber product
 diagram
 \begin{equation}\label{CDrootstack}
\begin{aligned}
 \xymatrix{
      \fkM^{\sqrt{S}}_{\CG,\mu} \ar[r] \ar[d]  &   [T_\CG \bs \wt Y_S ] \ar[d]^L \\
       \fkM_{\CG,\mu}\ar[r]^{\Delta_S} &  [T_\CG \bs Y_S ]  .
        }
        \end{aligned}
\end{equation}

\subsubsection{The statement of the conjecture} 
\begin{conjecture}\label{globconj}
 There exists an $O_E$-integral model  of the Shimura variety 
 ${\rm Sh}_{K_1}(\RG, \RX)_E$, i.e. an $O_E$-scheme $\CS_{K_1, S}$ together with an isomorphism $\CS_{K_1, S}\otimes_{O_E}E\simeq {\rm Sh}_{K_1}(\RG, \RX)_E$,   which has the following properties:
 
 \begin{itemize}
\item[i)] There is a morphism \[\CS_{K_1, S}\to \mathcal \CS_{K_0}
\] which extends $\pi: {\rm Sh}_{K_1}(\RG, \RX)_E\to {\rm Sh}_{K_0}(\RG, \RX)_E$ on the generic fibers.
 
 \item[ii)] The action of $T_\CG(\BF_p)$ on ${\rm Sh}_{K_1 }(\RG, \RX)_E$ extends to $\CS_{K_1, S}$ and the morphism 
 $\CS_{K_1, S}\to \mathcal \CS_{K_0}$ of (i)
  identifies $\CS_{K_0}$ with the scheme quotient $T_\CG(\BF_p)\bs \CS_{K_1, S}$.
 
  \item[iii)] The scheme $\CS_{K_1, S}$ and the $T_\CG(\BF_p)$-cover $\CS_{K_1, S}\to \mathcal \CS_{K_0}$ are $\RG({\BA}_f^p)$-equivariant for a $\RG({\BA}_f^p)$-action that extends the natural action on the generic fiber.
  
  \item[iv)] There is a morphism $\varphi_{1, S}: \CS_{K_1, S}\to \fkM^{\sqrt{S}}_{\CG,\mu}$ which fits  in a $2$-commutative diagram
  \begin{equation}\label{CDconjecture}
\begin{aligned}
 \xymatrix{
     \CS_{K_1, S}\ \ar[r]^{\varphi_{1,S}} \ar[d]  &  \fkM^{\sqrt{S}}_{\CG,\mu} \ar[d] \\
       \CS_{K_0}\ar[r]^{\varphi}  & \fkM_{\CG,\mu},
        }
        \end{aligned}
\end{equation}
 and which induces
   an isomorphism of stacks 
 \[
 [T_\CG(\BF_p)\bs \CS_{K_1, S}]\xrightarrow{\sim}  \CS_{K_0}\times_{\fkM_{\CG,\mu}}\fkM^{\sqrt{S}}_{\CG,\mu }.
 \]
 \end{itemize}
 \end{conjecture} 
 
 \begin{remark}\label{rem:GenericIso}
 Note that, since the $T_\CG(\BF_p)$-action on the generic fiber $\CS_{K_1, S, E}={\rm Sh}_{K_1 }(\RG, \RX)_E$ is free with quotient $\CS_{K_0, E}={\rm Sh}_{K_0}(\RG, \RX)_E$, there is an isomorphism $[T_\CG(\BF_p)\bs \CS_{K_1, S, E}]\simeq \CS_{K_0, E}$. This is consistent with Conjecture \ref{globconj} since, by construction, the stacky cover $\fkM^{\sqrt{S}}_{\CG,\mu}\to \fkM_{\CG,\mu}$ is an isomorphism over the generic fibers. 
 \end{remark}

 \begin{remark}\label{rem:variant}
  A similar conjecture  can be  formulated for the covers 
  \[
  {\rm Sh}_{K}(\RG, \RX)\to {\rm Sh}_{K_0}(\RG, \RX)
  \]
  for intermediate $p$-level subgroups 
  $
  K_1\subset K\subset K_0 ,
 $
  where $K$ is the inverse image of $\ov Q(\BF_p)$ under the map $K_0=\CG(\BZ_p)\to \ov\CG_{\red, \rm ab}(\BF_p)$, for some subtorus $\ov Q \subset \ov\CG_{\red, \rm ab}$. In this case, we can use toric schemes and Lang covers for the quotient torus $T=T_\CG/Q$, where $Q\subset T_\CG$ is the subtorus over $\BZ_p$ which lifts $\ov Q\subset  \ov\CG_{\red, \rm ab}$. These  depend on the choice of a semigroup $S\subset X^*(T_\CG/Q)$
   satisfying conditions like in \S \ref{ss:generalsemigroups}. There are corresponding root stack local models
leading to a formulation of Conjecture \ref{globconj} for such $K$. In general, this does not seem to follow from Conjecture \ref{globconj} for $K_1$. In any case, we have decided to restrict
 to the case of $K_1$ to simplify the discussion and notation.
  \end{remark}

 \begin{remark}\label{rem:conj}
1)  Note that there is some freedom in selecting the semi-group $S$ and, hence, the toric scheme $Y_S$. When we use the maximal choice $S=S_{\CG,\mu}=X^*(T_\CG)\cap \sigma^\vee_{\CG,\mu}$, which gives the toric scheme $Y_{\CG,\mu}$, we omit $\sqrt{S}$ or $S$ and just write $\CS_{K_1}$, and  $\varphi_1$, and $\fkM^{\sqrt{\phantom{a}}}_{\CG,\mu}$. The choice $S=S_{\CG,\mu}$ can be thought of as the canonical one.

2) The crucial condition in the conjecture is (iv): This allows us to identify the ramification structure of the cover
$\CS_{K_1,S}\to \CS_{K_0}$ with that of the ``root stack" cover $\Mloc^{\sqrt{S}}_{\CG,\mu}\to \Mloc_{\CG,\mu}$. This identification follows from (iv) in view of (\ref{eq:sh}), and the definition of $\fkM^{\sqrt{S}}_{\CG,\mu }\to \fkM_{\CG,\mu }$ as the $\CG$-quotient of $\Mloc^{\sqrt{S}}_{\CG,\mu}\to \Mloc_{\CG,\mu}$.
In fact, condition (iv) and (\ref{eq:sh}) imply that,   \'etale locally on the base,   $[T_\CG(\BF_p)\bs \CS_{K_1,S }]\to \CS_{K_0}$ is isomorphic to $\Mloc^{\sqrt{S}}_{\CG,\mu}\to \Mloc_{\CG,\mu}$. In the next paragraph, we expand on this to give an explicit description of the \'etale local structure of $\CS_{K_1, S}$.

3) We warn the reader that, in many cases, $\Mloc^{\sqrt{S}}_{\CG,\mu }$ and hence by the above $\CS_{K_1,S}$, are not well-behaved schemes. For example,   sometimes $\Mloc^{\sqrt{S}}_{\CG,\mu }$ is not even flat over $O_E$, see Remark \ref{rem:badSiegel}. Then  $\CS_{K_1,S}$ as in the conjecture will also not be flat over $O_E$. In general, our treatment of the \'etale local properties of $\Mloc^{\sqrt{S}}_{\CG,\mu }$ in \S \ref{s:RootStack} becomes relevant
and predicts corresponding properties of $\CS_{K_1,S}$.

 4) The existence of the smooth morphism (\ref{LMmap}) implies that the integral model $\CS_{K_0}$ is flat over $O_E$. Since the morphism $\CS_{K_1, S}\to \mathcal \CS_{K_0}$ is finite surjective by ii), the integral model $\CS_{K_1, S}$ will be topologically flat over $O_E$. Therefore, even though, as remarked in 3), the model $\CS_{K_1, S}$ can have bad commutative algebra properties, it is well-suited for topological questions, such as the determination of sheaves of vanishing cycles.
\end{remark}

\subsubsection{A local model for $\CS_{K_1, S}$} 

Here, continuing with Remark \ref{rem:conj} (2), we give an explicit description of the strict henselizations  of $\CS_{K_1, S}$ which follows from Conjecture \ref{globconj}:  
 
Let $x\in \CS_{K_0}(k)$ and set $y=\varphi(x)\in \Mloc_{\CG,\mu}(k)$
 for a corresponding point on the local model (only the $\CG_k$-orbit of $y$ is well-defined). Pick an affine chart $U=\Spec(A)\subset \Mloc_{\CG,\mu}$ containing $y$ such that the restriction $({\rm P}_{\CG,\mu})_{|U}$ of the $T_\CG$-torsor ${\rm P}_{\CG,\mu}\to \Mloc_{\CG,\mu}$ to $U$ is trivial. We now argue as in the proof of Proposition \ref{prop:DM}: Choosing a section of $({\rm P}_{\CG,\mu})_{|U}$ gives $\delta: U=\Spec(A)\to Y_S$ and we obtain elements $\delta^*(s)\in A$, for each  $s$ in the semigroup $S$.  Assuming the conjecture, by (\ref{eq:sh})  and the argument in the proof of Proposition \ref{prop:DM}, we see that there is an isomorphism
 \begin{equation}\label{eq:sh2}
 \CS_{K_1, S}\times_{\CS_{K_0}}\Spec(\CO^{\rm sh}_{\CS_{K_0}, x})\simeq  \wt Y_S\times_{Y_S,\delta}\Spec(\CO^{\rm sh}_{\Mloc_{\CG,\mu}, y})
 \end{equation}
 which respects the $T_\CG(\BF_p)$-actions on both sides.  
 
 The right hand side of (\ref{eq:sh2}) can be written   more explicitly as follows: Suppose that the semigroup $S\subset S_{\CG,\mu}$ is associated to the cone $\tau\supset \sigma_{\CG,\mu}$ of $X_*(T_\CG)_\BR$. We consider the (saturated) semigroup $\wt S$ of $X^*(T_\CG)$ which is associated to the cone $L_*^{-1}(\tau)\subset X_*(T_\CG)_\BR$; then the injective homomorphism $L^*: X^*(T_\CG)\to X^*(T_\CG)$ given by the Lang isogeny restricts to a semigroup homomorphism $L^*: S\to \wt S$. (We have $\wt S=L^*(S)^{\rm sat}$ in the notations of Proposition \ref{prop:normalToric} and so $\wt Y_S=Y_{\wt S}$.)
 
 For simplicity of notation, set $R=\CO^{\rm sh}_{\Mloc_{\CG,\mu}, y}$. Then the RHS 
 of (\ref{eq:sh2}) is isomorphic to the spectrum of 
\[
R[\wt S]/(L^*(s)-\delta^*(s))_{s\in S}.
\]
 Here, the quotient is by the ideal of $R[\wt S]$ generated by $L^*(s)-\delta^*(s)$, where $s$ runs over all elements of the semigroup $S$.
 If $\{s_1,\ldots , s_n\}$ is a minimal generating set of the semigroup $S$, then the ring above is also
\begin{equation}\label{explicitLM}
R[\wt S]/(L^*(s_1)-\delta^*(s_1), \ldots , L^*(s_n)-\delta^*(s_n)).
\end{equation}
Hence,  (\ref{eq:sh2}) gives a $T_\CG(\BF_p)$-equivariant isomorphism
 \begin{equation}\label{eq:sh3}
 \CS_{K_1, S}\times_{\CS_{K_0}}\Spec(\CO^{\rm sh}_{\CS_{K_0}, x})\simeq \Spec(R[\wt S]/(L^*(s_1)-\delta^*(s_1), \ldots , L^*(s_n)-\delta^*(s_n))).
\end{equation}
On the RHS, the action of $T_\CG(\BF_p)$ is given by $t\cdot \ti s=\ti s(t)\ti s$, where $\ti s(t)\in \br\BZ_p^*$ is the image of $t\in T_\CG(\BF_p)$ under
\[
T_\CG(\BF_p)=\ker(L)(\BZ_p)\subset T_\CG(\BZ_p)\xrightarrow{\ti s} \br\BZ_p^*.
\]
(For $t\in \ker(L)$, $s\in S$, we have $L^*(s(t)s)=L^*(s(t))L^*(s)=L^*(s)$, where we also view $L^*$ as pull-back of regular functions by $L: T_\CG\to T_\CG$.
Hence, this action is well-defined.)

 Note that if the  torus  is split  this looks a bit simpler. Indeed,   suppose $T_\CG\simeq \BG_m^r$. Then $\wt S=S$, $\wt Y_S=Y_S$ and $L^*: X^*(T_\CG)\to X^*(T_\CG)$ is dilation by $p-1$.
 Therefore the ring above is
\[
R[S]/(s_1^{p-1}-\delta^*(s_1), \ldots , s_n^{p-1}-\delta^*(s_n)).
\]
On this,   $t\in T_\CG(\BF_p)$ acts by $t\cdot s=s(t) s$, where $s(t)$ is the image of $t$ under
\[
T_\CG(\BF_p) \subset T_\CG(\BZ_p)\xrightarrow{s} \BZ_p^*.
\]
(The inclusion is the Teichmuller lift.)

   \subsection{A general construction}\label{ss:generalconstruction}
   
Here we give a blueprint for constructing the schemes $\CS_{K_1,S}$ as in the conjecture.

Consider the exact sequence
 \begin{equation}\label{Exact6a}
 1\to \und{T_\CG(\BF_p)}\to T_\CG \xrightarrow{L} T_\CG \to 1
 \end{equation}
 of group schemes over $\BZ_p$. This induces the exact sequence of Galois modules
 \begin{equation}\label{Exact6b}
 0\to  X^*(T_\CG)\xrightarrow{L^*} X^*(T_\CG)\to X^*(T_\CG(\BF_p))\to 0,
 \end{equation}
and the Cartesian diagram of classifying stacks 
 \begin{equation}\label{PicBG}
\begin{aligned}
 \xymatrix{
      B\und{T_\CG(\BF_p)} \ar[r] \ar[d]  & BT_\CG \ar[d]_{L}  \\
      \{1\}\ar[r]  & BT_\CG.
        }
        \end{aligned}
\end{equation}
We now have:

 \begin{proposition}\label{prop:torsors}
 The Picard category of $T_\CG(\BF_p)$-torsors over a $\BZ_p$-scheme $U$ is equivalent to the Picard category of pairs 
 \[
 (Q, a: T_{\CG}\xrightarrow{\sim}  L_*(Q))
 \]
consisting of a $T_{\CG}$-torsor $Q$ over $U$, together with a section $a$ of its push-out by $L: T_{\CG}\to T_{\CG}$,
 \[
 L_*(Q):=T_{\CG}\times_{L, T_{\CG}}Q
 \]
The equivalence is given by pushing out a $T_\CG(\BF_p)$-torsor by $\und{T_\CG(\BF_p)}\to T_{\CG}$. 
\end{proposition}

\begin{proof}
Follows from the Cartesian diagram (\ref{PicBG}).
\end{proof}

We now apply this to the $T_\CG(\BF_p)$-torsor given by the unramified cover
  \[
 \pi: {\rm Sh}_{\RK_1 }(\RG, \RX)_E\to {\rm Sh}_{\RK_0}(\RG, \RX)_E
 \]
 of Shimura varieties. We will denote by $(Q_\pi, a_\pi)$ the pair over ${\rm Sh}_{\RK_0}(\RG, \RX)_E$ which is obtained  from the $T_\CG(\BF_p)$-torsor  $\pi$ by the above equivalence.  We assume Conjecture \ref{divconj} holds, in particular we have $(\RP_{\CG,\mu}, s_{\CG,\mu})$. (This is very often the case by Theorem \ref{tameconj}.)
  
 \begin{conjecture}\label{conj:torsors}
 There exists a $T_\CG$-torsor $P$ over $\CS_{K_0}$ together with an isomorphism of $T_\CG$-torsors over 
 $\CS_{K_0}[1/p]={\rm Sh}_{\RK_0}(\RG, \RX)_E$,
 \[
 \alpha: P[1/p]\xrightarrow{\sim} Q_\pi,
 \]
  and an isomorphism of pairs 
 \[
 \beta: \big(L_*(P), L_*(\alpha^{-1})\circ a_\pi: T_{\CG}[1/p]\xrightarrow{\sim} L_*(P)[1/p]\big)\xrightarrow{\sim} \varphi^*\big(\RP^{(-1)}_{\CG,\mu}, s_{\CG,\mu}\big),
 \]
 where the pair  $(\RP_{\CG,\mu}, s_{\CG,\mu})$ is as in Conjecture \ref{divconj} and the pull-back is by the local model diagram morphism $\varphi: \CS_{K_0}\to \fkM_{\CG,\mu}
 =[\CG\bs \Mloc_{\CG,\mu}]$.
 \end{conjecture}
 
We stress that the pairs involved in $\beta$ are pairs consisting of a $T_\CG$-torsor over $\CS_{K_0}$ and  a section of its restriction to the generic fiber $\CS_{K_0}[1/p]={\rm Sh}_{\RK_0}(\RG, \RX)_E$. 

\begin{remark}
Note that  the triple $(P,\alpha,\beta)$ as in the Conjecture is unique up to (unique) isomorphism, if it exists: Indeed, suppose that $(P',\alpha',\beta')$ is another such triple.
Then, we can consider the $T_\CG$-torsor $R:=P'\cdot P^{-1}$. Our data gives a trivialization of the push-out $L_*(R)$ and of the generic fiber $R[1/p]$ which are compatible. Hence, $R$ is obtained from a $T_\CG(\BF_p)$-torsor over $\CS_{K_0}$ and our data produce a trivialization of this $T_\CG(\BF_p)$-torsor over the generic fiber $\CS_{K_0}[1/p]$. By normality, the trivialization extends to a trivialization over $\CS_{K_0}$. In particular, $R:=P'\cdot P^{-1}$ is a trivial $T_\CG$-torsor and the rest follows.
 \end{remark}

\begin{remark}\label{rem:ConjInterp}
Suppose that $T_\CG$ is split over $\BZ_p$. A choice of an isomorphism $T_\CG\simeq \BG_m^r$ gives a basis $\eta_i$, $i=1,\ldots ,r$, of $X^*(T)$. Then giving the $T_\CG$-torsor $P$ over a $\BZ_p$-scheme $U$ amounts to giving the collection of $r$ line bundles $\CM_i=\CM_{\eta_i}$ which correspond to the $\BG_m$-torsors $(\eta_i)_*P$. If $\chi=\sum_i c_i\,\eta_i\in X^*(T)$, we   set
\[
 \CM_\chi:=\bigotimes_i \CM_i^{\otimes c_i}.
\]
 There are natural isomorphisms $\CM_{\chi\chi'}\simeq \CM_\chi\otimes\CM_{\chi'}$ and $\chi\mapsto \CM_\chi$ is the tensor functor which corresponds to the $T_\CG$-torsor $P$ by the Tannakian equivalence.  
 
 In general, we can first base change to $O=W(\BF_q)$, $q=p^d$, which splits $T_\CG$ and then choose $T_{\CG}\otimes_{\BZ_p}O\simeq \BG_{m,O}^r$ to obtain line bundles
 $\CM_i$ and $\CM_{\chi}$ over $U_O=U\otimes_{\BZ_p}O$. Let us denote by $F$ the fraction field of $O=W(\BF_q)$.
 Denote by $M_i$, $i=1,\ldots, r$, the line bundles over ${\rm Sh}_{\RK_0}(\RG, \RX)_{EF}$ with isomorphisms
 \[
 a_i: \CO_{{\rm Sh}_{K_0, F}}\xrightarrow{\sim} M_{L^*(\eta_i)} 
 \]
 which are obtained from the $T_{\CG}$-torsor $Q_\pi\otimes_EEF$.
 Then,  assuming Conjecture \ref{conj:torsors}, we have  line bundles $\CM_i$, $i=1,\ldots, r$, over $\CS_{K_0, O}$ together with isomorphisms
 \begin{align}\begin{split}\label{data1}
 \CM_i[1/p]\xrightarrow{\sim}&\ M_i,\\
 \big(\CM_{L^*(\eta_i)}, \CO_{{\rm Sh}_{K_0, F}}\xrightarrow{\sim} \CM_{L^*(\eta_i)}[1/p]\big)\xrightarrow{\sim} &\
\varphi^*\big(\CL^{-1}_{\eta_i}, s^{-1}_{\eta_i}: \CO_{\Mloc_{\CG,\mu, O}}[1/p]\xrightarrow{\sim} \CL^{-1}_{\eta_i}[1/p]\big).
 \end{split}
 \end{align}
 Here, the pair $(\CL_{\eta_i}, s_{\eta_i})$ over the base change $\Mloc_{\CG,\mu, O}:=\Mloc_{\CG,\mu}\otimes_{O_E}OO_E$
 is as in \S \ref{ss:BundlesLines}. In fact, the  data $\CM_i$ with isomorphisms (\ref{data1}) uniquely determine  the base change of $(P,\alpha, \beta)$ in Conjecture \ref{conj:torsors} to $O$; to determine $(P,\alpha, \beta)$, one needs to also consider suitable finite \'etale descent data for the extension $O/\BZ_p$.
 
 If $O=\BZ_p$, i.e. $T_\CG$ splits over $\BZ_p$, then $L_*(\eta_i)=\eta^{p-1}_i$. In this case, 
 the data in Conjecture \ref{conj:torsors} actually  amounts to 
 line bundles $\CM_i$, $i=1,\ldots, r$, over $\CS_{K_0}$ together with isomorphisms
 \begin{align}\begin{split}\label{data2}
 \CM_i[1/p]\xrightarrow{\sim}&\ M_i,\\
  (\CM_{i}^{\otimes(p-1)}, \CO_{{\rm Sh}_{K_0}}\xrightarrow{\sim} \CM_{i}^{\otimes(p-1)}[1/p])\xrightarrow{\sim}&\ 
 \varphi^* (\CL_{\eta_i}^{-1}, s_{\eta_i}^{-1}: \CO_{\Mloc_{\CG,\mu, O}}\xrightarrow{\sim} \CL^{-1}_{\eta_i}[1/p]).
 \end{split}\end{align}
We will use these, more explicit, interpretations of Conjecture \ref{conj:torsors} for the PEL examples.
 \end{remark}

\subsubsection{Conjecture  \ref{conj:torsors} implies Conjecture  \ref{globconj}}
\label{sss:ConjImplyConj} Assuming Conjecture \ref{conj:torsors} holds, we will now  construct $\CS_{K_1, S}$
which satisfies Conjecture \ref{globconj}. In particular, this shows that the one conjecture implies the other.

Proposition \ref{prop:torsors} gives an isomorphism  of stacks over $\Spec(\BZ_p)$,
 \[
 [T_\CG  \LBS T_\CG ]=[T_\CG  (\BF_p)\bs *]=\RB  \und{T_\CG(\BF_p)} .
\]
Here the first quotient is for the (left) action of $T_\CG $ on itself via the Lang isogeny, i.e. $t\cdot t'=L(t)t'$. 
It agrees with the quotient stack $ [T_\CG\,  \bs_{\vphantom{A^1}{\rm Fr}}\, T_\CG ]$ of   conjugation   by the Frobenius lift $t\cdot t'={\rm Fr}(t)t't^{-1}$.

The \'etale $T_\CG(\BF_p)$-cover 
\[
\pi: {\rm Sh}_{\RK_1}(\RG, X)_E\to {\rm Sh}_{\RK_0}(\RG, X)_E
\]
gives
\begin{equation}\label{gfiber1}
[\pi]: {\rm Sh}_{\RK_0}(\RG, X)_E\to  [T_\CG  \LBS T_\CG ]=\RB  \und{T_\CG(\BF_p)}.
\end{equation}
Recall now that we choose a semigroup $S\subset S_{\CG,\mu}$ as in \S \ref{ss:generalsemigroups} which determines the $T_\CG$-toric scheme $Y_S$. 
 The toric embedding $T_\CG\hookrightarrow Y_S$ defines a morphism
\[
\RB  \und{T_\CG (\BF_p)}= [T_\CG \LBS T_\CG ]\to  [T_\CG  \LBS Y_S ].
\]
(This can be viewed as a ``partial compactification" of the classifying stack of $T_\CG(\BF_p)$-torsors.)
On the other hand, the effectivity  implied by Conjecture \ref{divconj} gives that for all   $\chi \in S $ the trivialization
  \[
s_\chi: \CO_{\Mloc_{\CG,\mu, O}}[1/p]\xrightarrow{\sim} \CL_{\chi}[1/p]
\]
 extends to  
 \[
 s_\chi: \CO_{\Mloc_{\CG,\mu, O}}\to \CL_{\chi},
 \]
which is a $\CG$-equivariant global section of $\CL_\chi$ over $\Mloc_{\CG,\mu, O}$. In view 
of the isomorphisms in (\ref{data1}) above,  this implies that the trivialization given by $\pi$
\[
\CO_{{\rm Sh}_{K_0, F}}\xrightarrow{\sim} M^{-1}_{L^*(\chi)}\xrightarrow{\sim} \CM^{-1}_{L^*(\chi)}[1/p]
\]
 extends to 
\[
a_\chi: \CO_{\CS_{K_0,O}}\to \CM^{-1}_{L^*(\chi)} ,
\]
i.e., to a global (regular) section $a_\chi$ of $\CM^{-1}_{L^*(\chi)}$ over $\CS_{K_0,O}$.  Since $Y_{S,O}=\Spec(O[S])$, this together with descent along $OO_E/O_E$, implies that the $T_\CG$-torsor $P\to \CS_{K_0}$,
assumed to exist by Conjecture  \ref{conj:torsors}, supports a morphism $f$ occurring in the diagram 
\[
\begin{matrix}
P& \xrightarrow{\ f\ } & Y_S\\
\downarrow &&\\
\CS_{K_0}&&
\end{matrix}
\] 
which is $T_\CG$-equivariant via the Lang map, i.e. $f(t\cdot p)=L(t)f(p)$. The above diagram corresponds to a $1$-morphism
\begin{equation}\label{extend1}
\wt {[\pi]}: \CS_{K_0}\to  [T_\CG  \LBS Y_S ].
\end{equation}
We are going to use  the morphism $[T_\CG  \LBS Y_S ]\to [T_\CG\, \bs\, Y_S]$ induced by the identity on $Y$ and $L: T_\CG\to T_\CG$. Composing   $\wt{[\pi]}$ with $[T_\CG  \LBS Y_S ]\to [T_\CG\, \bs\, Y_S]$ gives the morphism $\CS_{K_0}\to  [T_\CG\,  \bs\, Y_S ]$ defined by the pair $(L_*(P), f')$ of the $T_\CG$-torsor $L_*(P)=T_\CG\times_{L, T_\CG} P$  and the $T_\CG$-equivariant morphism $f': L_*(P)\to Y_S$ given by $f'(t, p)=t\cdot f(p)$. The existence of the isomorphism $\beta$ in Conjecture \ref{conj:torsors} implies that this morphism is $2$-isomorphic to   the composition
  \[
 \CS_{K_0}\xrightarrow{\ \varphi\ } \fkM_{\CG,\mu}=[\CG\bs \Mloc_{\CG,\mu}]\xrightarrow{\ \Delta_S\ }  [T_\CG\, \bs\,  Y_{S}].
  \]
This, together with the existence of the isomorphism $\alpha$ 
 in Conjecture \ref{conj:torsors}, implies that we 
 have a $2$-commutative diagram
\begin{equation}\label{StackDia1}
\begin{aligned}
 \xymatrix{
    {\rm Sh}_{\RK_0}(\RG, X)_E \ar[r]^{\rm\ \ incl} \ar[d]_{[\pi]}  &     \CS_{K_0}\ar[r]^{\varphi} \ar[d]_{\wt{[\pi]}} & [\CG\bs \Mloc_{\CG,\mu}]\ar[d]_{\Delta_S}\\
      [T_\CG  \LBS T_\CG ] \ar[r]  & [T_\CG \LBS Y_S ]  \ar[r]  & [T_\CG\, \bs\, Y_S ].
        }
        \end{aligned}
\end{equation}
Here the two top horizontal morphisms are the inclusion of the generic fiber and the morphism obtained from the  local model diagram.
The two bottom horizontal morphisms are induced by the toric embedding $T_\CG\hookrightarrow Y_S$ as above and by  the 
Lang map $L: T_\CG \to T_\CG $ with the identity $Y_S\xrightarrow{=}Y_S$.
It is instructive to note here that the morphism $[T_\CG \LBS Y_S ]  \to [T_\CG\, \bs\, Y_S ]$ is an \'etale gerbe banded by $\und{T_\CG(\BF_p)}=\ker(L)$.

  We now use the normalization $L: \wt Y_S\to Y_S$ of the Lang isogeny $L: T_\CG\to T_\CG$. This gives  a morphism
$
  [T_\CG\, \bs\, \wt Y_S]\to [T_\CG  \LBS Y_S].
$
The action of $T(\BF_p)$ on $\wt Y_S$ gives an action on $[T_\CG\, \bs\, \wt Y_S]$ and this morphism is $T_\CG(\BF_p)$-equivariant with the trivial action on the target. 
The morphism $ L: [T_\CG\,\bs\, \wt Y_S]\to [T_\CG\, \bs\,  Y_S]$ of  \eqref{LMrootstack} factors as 
\begin{equation}\label{factorsT}
 [T_\CG\,\bs\, \wt Y_S]\to [T_\CG  \LBS Y_S]\to [T_\CG\, \bs\, Y_S].
\end{equation}
  \begin{lemma}\label{lem:representable}
  The morphism
  $
  [T_\CG\,\bs\, \wt Y_S]\to [T_\CG  \LBS Y_S] 
 $
  is relatively representable and finite.
  \end{lemma}
  
  \begin{proof} To ease notation, set $T=T_\CG$. Suppose $Z\to [T\LBS Y_S]$ is given by a $T$-torsor $W\to Z$ with a $L$-$T$-equivariant morphism $f: W\to Y_S$ and consider the fiber product
\[
Z\times_{[T  \LBS Y_S]}[T\,\bs\, \wt Y_S].
\]
 A $Z'$-valued point of $Z\times_{[T  \LBS Y_S]}[T\,\bs\, \wt Y_S]$ corresponds to a $T$-equivariant morphism
\[
f': W\times_Z Z'\to \wt Y_S
\]
such that  the following diagram is commutative,
\begin{equation} 
\begin{aligned}
 \xymatrix{
    W\times_Z Z' \ar[r]^{f'} \ar[d]  &   \wt  Y_S  \ar[d] \\
     W \ar[r]^{f}  & Y_S.
        }
        \end{aligned}
\end{equation}
 To prove relative representability we can reduce to the case that $Z$ is a local $O$-scheme. Then the $T$-torsor $W\to Z$ is trivial, i.e., $W=T\times Z$, and $f$ is determined by $g=f_{|1\times Z}: Z\to Y_S$ which, in turn, is  determined  by the monoid homomorphism $g^*: S\to \CO_Z(Z)$. The $T$-equivariant morphisms $f'$ with the above property, are determined by $g': Z'\to \wt Y_S$ such that
 \begin{equation} 
\begin{aligned}
 \xymatrix{
   Z' \ar[r]^{g'} \ar[d]  &   \wt  Y_S  \ar[d] \\
     Z \ar[r]^{g}  & Y_S.
        }
        \end{aligned}
\end{equation}
commutes. Hence, the fiber product is represented by $Z\times_{g,Y_S}\wt Y_S$ which is finite over $Z$.
\end{proof}

 \begin{definition}\label{def:generalconstruction2} We  define $\CS_{K_1, S}$ by the fiber product diagram
\begin{equation}\label{def:Fiberprod}
\begin{aligned}
 \xymatrix{
      \CS_{K_1,S } \ar[r] \ar[d]  &  [T_\CG \bs \wt Y_S ] \ar[d]  \\
       \CS_{K_0}\ar[r]  & [T_\CG \LBS Y_S ].
        }
        \end{aligned}
\end{equation}
The morphism $ \CS_{K_1,S }\to \CS_{K_0}$ supports an $T_\CG(\BF_p)$-action, trivial on the target.
 \end{definition}

\begin{proposition}\label{prop:Conj}
The  $\CS_{K_0}$-scheme  $\CS_{K_1, S}$  satisfies the conditions of Conjecture \ref{globconj}.
\end{proposition}
\begin{proof}
By Lemma \ref{lem:representable}, $\CS_{K_1, S}$ is a scheme. Recall that $L: [T_\CG\,\bs\, \wt Y_S]\to [T_\CG\,\bs\,  Y_S]$ factors 
 \[
 [T_\CG\,\bs\, \wt Y_S]\to [T_\CG  \LBS Y_S]\to [T_\CG\, \bs\, Y_S] ,
 \]
 as in \eqref{factorsT}. By composing the right vertical morphism in \eqref{def:Fiberprod}
 with $[T_\CG  \LBS Y_S]\to [T_\CG\, \bs\, Y_S]$ above, we obtain a commutative (not fiber product) diagram
 \begin{equation}\label{CDRootStack3}
\begin{aligned}
 \xymatrix{
      \CS_{K_1,S } \ar[r] \ar[d]  &  [T_\CG \bslash \wt Y_S ] \ar[d]  \\
       \CS_{K_0}\ar[r]  & [T_\CG \bslash Y_S ].
        }
        \end{aligned}
\end{equation}
Using \eqref{StackDia1} we see that, up to  $2$-morphisms,  this factors  
\begin{equation}\label{StackDia2}
\begin{aligned}
 \xymatrix{
     \CS_{K_1,S } \ar[r]^{\varphi_{1,S} } \ar[d]_{ }  &     \fkM^{\sqrt{S}}_{\CG,\mu}\ar[r]  \ar[d]  & [T_\CG\, \bs\, \wt Y_S]\ar[d]^{L}\\
     \CS_{K_0 } \ar[r]^{\varphi}  &  \fkM_{\CG,\mu}  \ar[r]^{\Delta_S}  & [T_\CG\, \bs\, Y_S ] ,
        }
        \end{aligned}
\end{equation}
where the right square is \eqref{CDrootstack}. The gives the diagram in (iv) of
Conjecture \ref{globconj} and the rest of the statements in Conjecture \ref{globconj} follow.
\end{proof}

 \begin{remark}\label{rem:NonNormalvariant}
 a)  The crucial point in the construction is the existence of the morphism 
  \[
  \wt{[\pi]}: \CS_{K_0}\to [T_\CG\LBS Y_S]
  \]
  fitting in the diagram \eqref{StackDia1}. 
  
 b)  Assume that   $L: Y'\to Y_S$ is a finite $T_\CG(\BF_p)$-cover which extends the Lang map $L: T_\CG\to T_\CG$. Assume further given $ \wt{[\pi]}: \CS_{K_0}\to [T_\CG\LBS Y']$ which extends $[\pi]$. 
   We can then define a finite $T_\CG(\BF_p)$-cover $\CS_{K_1, S, Y'}\to \CS_{K_0}$ as the fiber product $\CS_{K_0}\times_{[T_\CG\LBS Y_S]}[T_\CG\,\bs\, Y']$, and obtain in this way an integral model of ${\rm Sh}_{\RK_1}(\RG, X)_E$.

  Such  $Y'$ can be obtained from  a Galois stable non-saturated semigroup $L^*(S)\subset S'\subset X^*(T_\CG)$ which   generates $X^*(T_\CG)$; then $Y'$ is not normal.  
  We will encounter this variation of our construction in \S \ref{s:PELex}, see Remark \ref{rem:HilbertSiegel} (1).
 \end{remark}
 
We can give a direct functor description of the base change $\CS_{K_1, O}$ to $O$, in  line with the interpretation of Conjecture  \ref{conj:torsors}  via (\ref{data1}) and (\ref{data2}) as in the previous paragraph.  This will be very useful later and especially when we discuss examples in \ref{s:PELex}.  

Recall that  $S\subset X^*(T_\CG)$ gives $Y_{S,O}=\Spec(O[S])$ and we also have the semigroup $\wt S \subset X^*(T_\CG)$ defining the normalization $\wt Y_{S, O}=\Spec(O[\wt S])$ of the Lang cover. We have $L^*: S\hookrightarrow \wt S$ giving the cover $\wt Y_{S,O}\to Y_{S, O}$. 
 We assume Conjecture  \ref{conj:torsors} and use the notations of Remark \ref{rem:ConjInterp}.

 \begin{proposition}\label{def:generalconstruction}  The $\CS_{K_0}$-scheme $\CS_{K_1,S, O}$ of Definition \ref{def:generalconstruction2} represents the following functor on $({\rm Sch}/\CS_{K_0,O})$. Let $f: U\to \CS_{K_0, O}$ be an $U$-valued point of $\CS_{K_0,O}$. Then $\CS_{K_1,S, O}(U)$ is   the set   $(z_{\wt\chi}(U))_{\wt\chi\in \wt S}$
 of sections 
 \[
 z_{\wt \chi}(U)\in \Gamma(U,  \CM^{-1}_{\wt \chi})
 \] 
such that
  \begin{itemize}
  \item[i)] $z_{L^*(\chi)}=f^*(a_\chi)$, for all $\chi\in S$,
  \item[ii)] $ z_{\wt\chi\wt\chi'}=z_{\wt \chi}\otimes z_{\wt \chi'}$, for all $\wt \chi$, $\wt \chi'\in \wt S\subset X^*(T_\CG)$, via the identification given by the tensor functor $\psi\mapsto \CM^{-1}_{\psi}$, $\psi\in X^*(T_\CG)$.
\end{itemize}
The group $T_\CG(\BF_p)$ acts on $(z_{\wt\chi}(U))_{\wt\chi\in \wt S}$ by
\[
g\cdot z_{\wt\chi}(U)=\wt\chi(g)\cdot z_{\wt\chi}(U) 
\]
where   $\wt\chi(g)\in O^*$ is defined by evaluating $\wt\chi: T_{\CG}(\BZ_p)\to O^*$ at $g\in T_\CG(\BF_p)\subset T_\CG(\BZ_p)$.
 \end{proposition}
 
 \begin{proof}
 Follows as in the proof of Lemma \ref{lem:representable} by remarking that $O$-scheme morphisms $V\to \wt Y_S$, resp. $V\to Y_S$, are determined by semigroup homomorphisms
 $S\to \CO_V(V)$, resp.  $\wt S\to \CO_V(V)$.
  \end{proof}
 
 \begin{remark}\label{rem:alternative} We can view the above as giving an alternative construction of $\CS_{K_1, S}$
 (always assuming Conjecture  \ref{conj:torsors}):
 It is clear that the functor $U\mapsto (z_{\wt\chi}(U))_{\wt\chi\in \wt S}$ of the statement is represented by a finite scheme over $\CS_{K_0, O}$ which supports
a $T_\CG(\BF_p)$-action. The line bundles $\CM_\psi$, for $\psi\in X^*(T)$, and sections $a_\chi$ of $\CM^{-1}_{L^*(\chi)}$, for $\chi\in S$, are equipped with finite \'etale descent data for the extension $O/\BZ_p$ 
and so the same is true for the functor $U\mapsto (z_{\wt\chi}(U))_{\wt\chi\in \wt S}$. Hence, the scheme above is the base change $\CS_{K_1, S, O}$ of a uniquely determined finite scheme 
$
\CS_{K_1, S}\to \CS_{K_0}
$
 with $T_\CG(\BF_p)$-action.
 \end{remark}

\section{PEL type Shimura varieties}\label{s:PEL}

Here we discuss PEL type Shimura varieties, and give our main result for these, under some additional conditions that we introduce.
The proof uses Oort--Tate/Raynaud theory of finite group schemes.

\subsection{PEL data and moduli}\label{ss:PELdata}

\subsubsection{PEL data}\label{sss:PELdata} In this section we will follow the set-up of \cite[Ch. 6]{RZbook}. Let $\RB$ be a semi-simple algebra over $\BQ$ and let $\ast$ be a positive
involution on $\RB$. Let $\RV$ be a finite-dimensional $\BQ$-vector space with
a nondegenerate alternating bilinear form $(\ ,\ )$ with values in $\BQ$.
We assume that $\RV$ is equipped with a $\RB$-module structure such that
$$
(bv,w)= (v,b^\ast w),\ \ v,w\in V,\ \ b\in \RB .
$$
 Let $\RG \subset
\GL_\RB(\RV)$ be the closed algebraic subgroup over $\BQ$ such that
 \begin{equation}\label{PELgroup}
 \RG (\BQ)=\{ g\in \GL_\RB(\RV)\ \vert\
(gv,gw)=c(g)(v,w),\ c(g)\in\BQ^\times\} .
\end{equation}
Note that the group $\RG $ is not always connected. We will denote by $\RG^\circ$ its neutral component.
 Let ${\BS}={\rm Res}_{\BC/\BR} \BG_m$ and
let $h:{\BS}\to \RG^\circ_\BR\subset \RG_\BR$ be a homomorphism that defines on $\RV_\BR$ a Hodge structure of type $\{(1, 0), (0,1)\}$ such that $(v, h(\sqrt{-1})w)$ is a symmetric positive-definite bilinear form on $\RV_\BR$, cf. loc. cit.. We have a corresponding Hodge
decomposition 
$$
\RV\otimes \BC = \RV_0\oplus \RV_1,
$$
where $\BS$ acts on $\RV_0$ via the character $\bar z$ and on $V_1$ via the character $z$. Let $\mu$ be the   corresponding
cocharacter $\mu$ of $\RG$ defined over $\BC$, i.e., $\mu(z)=h(z, 1)$ under the identification $\BS_\BC=\BG_m\times\BG_m$, via $z\mapsto (z, \bar z)$. We let $\RE\subset\overline\BQ$
be the corresponding reflex field. This is a subfield of the reflex field $\RE^\circ$ of the Shimura datum\footnote{The group $\RG$ is connected outside of type $\RD$ and $\RE\neq \RE^\circ$ is possible only for non-split groups of type $\RD$.} $(\RG^\circ, X)$, i.e. $\RE\subset \RE^\circ$.
We now fix a prime number $p$ and
choose an embedding $\overline\BQ\to\overline\BQ_p$. The corresponding
$v$-adic completion of $\RE$ will be denoted $E$. Let $K^p\subset
\RG({\BA}_f^p)$ be an open compact subgroup.

We consider an order $O_\RB$ of $\RB$ such that $ O_\RB\otimes\BZ_p$ is a maximal
order of $ \RB\otimes \BQ_p$. We assume that $O_\RB\otimes \BZ_p$ is invariant
under the involution. Let us   write 
\[
\RB\otimes\BQ_p=B_1\times\cdots \times B_m,
\]
a product of simple $\BQ_p$-algebras such that
\[
O_\RB\otimes\BZ_p=O_{B_1}\times\cdots \times O_{B_m},
\] 
a product of maximal orders.
Also write
\[
\RV\otimes\BQ_p=V_1\oplus\cdots \oplus V_m
\]
for the corresponding decomposition of $\RV\otimes\BQ_p$ into (left) $B_j$-modules. 

We also fix a selfdual periodic multichain ${\mathscr L}$ of $O_\RB\otimes \BZ_p$-lattices in $\RV\otimes\BQ_p$ with respect to the
alternating form $(\ ,\ )$ and $*$, see \cite[Def. 3.13]{RZbook}. By the definition of ``multichain''   in \cite[Def. 3.4]{RZbook}, there are periodic $O_{B_j}$-lattice chains $\mathscr L_j$ in $V_j$ such that the lattices $\Lambda$  in $\mathscr L$ are exactly all the  direct sums 
\[
\Lambda=\Lambda_1\oplus \cdots \oplus \Lambda_m,
\]
with $\Lambda_j$ ranging over all the lattices in the chain $\mathscr L_j$. Here, we also assume in addition that $\mathscr L$ is self-dual for $*$, i.e. that if $\Lambda\in \mathscr L$, then $\Lambda^*\in \mathscr L$, see loc. cit.

\subsubsection{PEL moduli problems}\label{sss:PELmoduli} 
We recall from loc. cit. the definition of a moduli problem ${\CA}_{K^p}$ over $(Sch/\Spec({O}_{E}))$. It associates to a
${O}_{E}$-scheme $S$ the following set of data up to isomorphism:
\begin{itemize}
\item[1.] An ${\mathscr L}$-set of abelian varieties $A=\{ A_\Lambda\}$.
\item[2.] A $\BQ$-homogeneous principal polarization $\overline\lambda$ of
the ${\mathscr L}$-set $A$.
\item[3.] A $K^p$-level structure
$$
\overline\eta :\RH_1(A,{\BA}_f^p)\xrightarrow{\sim} \RV\otimes_\BQ {\BA}_f^p\ {\rm mod}\
K^p ,
$$
 which respects the bilinear forms on both sides up to a
constant in $({\BA}_f^p)^\times$.
\end{itemize}

 We require an identity of characteristic polynomials\footnote{Here we deviate from the normalization in \cite{RZbook} and adopt  the normalization of \cite[footnotes p. 200, p. 230]{SWberkeley}. }, 
$$
 {\rm{det}}
(T\cdot I-b\ \vert\  \Lie(A_\Lambda))={\rm{det}} (T\cdot I-b\ \vert\
\RV_1),\ \ b\in O_\RB,\ \Lambda\in {\mathscr L} . 
$$
 For the definitions of the
terms employed here we refer to \cite[\S 6.3--6.8]{RZbook}. We only mention that
$A$ is a functor from the category ${\mathscr L}$ to the category of abelian
schemes over $S$ up to isogeny of order prime to $p$, with $O_\RB$-action,
and that a polarization $\lambda$ is a ${O}_\RB$-linear homomorphism
from $A$ to the dual ${\mathscr L}$-set $\tilde A$ (for which $\tilde
A_\Lambda=(A_{\Lambda^*})^\wedge)$.

The functor ${\CA}_{K^p}$ is representable by a quasi-projective scheme
over ${O}_{E }$, provided that $K^p$ is sufficiently small.
It follows from \cite[\S 3.26]{RZbook}, that there is a ``local model diagram morphism''
 \[
 \varphi^{\rm naive}: \CA_{K^p}\to [\CG\bs \Mloc^{\rm naive}_{\CG,\mu}]
 \]
 which is smooth. Here, $\Mloc^{\rm naive}_{\CG,\mu}$ is the ``naive" local model of loc. cit.,  
 which is not always flat over $O_E$.

We would like to compare our set-up so far to the set-up and notation of \S \ref{sss:612}: 
 There we discuss Shimura varieties for data $(\RG,X)$, with parahoric level $K_0=\CG(\BZ_p)$. We also have
  the ``canonical" integral model $\CS_{K_0}$, local model
 $\Mloc_{\CG,\mu}$, and local model diagram morphism
 \[
 \varphi : \CS_{K_0}\to [\CG\bs \Mloc_{\CG,\mu}].
 \]
This comparison is,  in general, somewhat complicated.   More specifically, the complications come from the following sources, see the discussion in \cite[\S 8.2]{PZ}:
 
 \begin{itemize}
 \item[1)] The group $\RG$ of (\ref{PELgroup}) is not always connected and so $(\RG, X)$ as in (\ref{PELgroup})  does not fit directly within the usual formalism of Shimura pairs.
 
 \item[2)] Due to the possible failure of the Hasse principle, the generic fiber $\CA_{K^p}[1/p]$ is, in general, just a union of Shimura varieties for groups related to $\RG$.
 
 \item[3)] Even if $G=\RG\otimes\BQ_p$ is connected, the stabilizer group scheme  $\CG={\rm Aut}_{O_B, (\ ,\ )}(\mathscr L)$, see below, is sometimes not connected. Then, the level $\CG(\BZ_p)$ is not a parahoric subgroup but just ``quasi-parahoric."  
 \end{itemize}
 
  When both the generic and special fibers of the group scheme $\CG$ over $\BZ_p$ are connected this comparison is more straightforward.

\subsubsection{The main theorem}\label{sss:MainPEL} We now set $V=\RV\otimes\BQ_p$, $B=\RB\otimes\BQ_p$ and $O_B=O_\RB\otimes\BZ_p$. 
We also consider $\CG={\rm Aut}_{O_B, (\ ,\ )}(\mathscr L)$, the group scheme of automorphisms of the \emph{polarized} self-dual 
multichain $\mathscr L$ up to common similitude which lies in $\BZ_p^*$. By \cite[Thm. 3.16, App. to Ch. 3]{RZbook} this is a smooth group scheme over $\BZ_p$ with generic fiber $G=\RG\otimes\BQ_p$, at least when $p\neq 2$.
Most of our results are under the following assumption:
\smallskip

(IW) \emph{The special fiber $\CG\otimes \BF_p$ is connected and $\CG^\circ$ is 
an Iwahori group scheme for $G^\circ$.} 
\smallskip
\begin{remark}
When the generic fiber $G=\CG\otimes \BQ_p$ is connected, then the condition that the special fiber $\CG\otimes \BF_p$ is also connected implies that
$\CG^\circ=\CG$. So, if $G$ is connected, condition (IW) just amounts to requiring that $\CG$ is Iwahori.
However,  it is possible 
 that the special fiber of $\CG={\rm Aut}_{O_B, (\ ,\ )}(\mathscr L)$ is connected while the generic fiber $G$ is \emph{not}. For example, this happens when $\mathscr L$ is obtained from a complete periodic self-dual lattice chain in an  even dimensional split orthogonal space, i.e. with a Witt basis, over $\BQ_p$, $p\neq 2$, see \cite[Lem. 4.3.2]{Smithling}.
\end{remark}

We can also consider $\mathscr L$ simply as a periodic multichain of $O_B$-lattices in $V$ by forgetting the additional structures. Let ${\rm Aut}_{O_B}({\mathscr L})$ be the  group scheme of $O_B$-linear automorphisms of $\mathscr L$ over $\BZ_p$, as defined in \cite[Ch. 3]{RZbook}. 
By \cite[Thm. 3.11]{RZbook} this is a smooth group scheme over $\BZ_p$ and, hence, it is a 
Bruhat-Tits parahoric group scheme for $\GL_B(V)$. By definition, there is a closed immersion of group schemes
\begin{equation}\label{groupClosedImm}
\CG\hookrightarrow \GL_{O_B}(\mathscr L).
\end{equation}
 For each $j=1,\ldots,m$, we write $B_j={\rm M}_{n_j}(D_j)$ and  $O_{B_j}={\rm M}_{n_j}(O_{D_j})$, where $D_j$ is a central division algebra over the local field $F_j$ and where $O_{D_j}$ is a maximal order. We also fix a prime element $\Pi_j$ of $O_{D_j}$ and denote by $k_{D_j}=O_{D_j}/\Pi_jO_{D_j}$ the residue field. By reindexing the lattice chains $\mathscr L_j=\{\Lambda_{j, i}\}_{i\in \BZ}$, we can assume that  
$\Pi_j\Lambda_{j,i}=\Lambda_{j, i-r_j}$ for some $r_j\geq 1$. In fact, we can arrange so that for each $j\in \{1,\ldots, m\}$ there is $j^*\in\{1,\ldots, m\}$
and $a_j$ which is $0$ or $1$ so that
\begin{equation}
\Lambda_{j, i}^*=\Lambda_{j^*, -i+a_j},\qquad \forall i\in \BZ.
\end{equation}
Note that $j\mapsto j^*$ is the involution which is induced by the action of $*$ on the center of the algebra $B=\RB\otimes\BQ_p$. 
Set
\[
Q_{j, i}=\Lambda_{j,i}/\Lambda_{j, i-1}
\]
which is a (left) $\RM_{n_j}(k_{D_j})$-module. Using Morita equivalence, we can write 
\[
Q_{j, i}=\RM_{n_j}(k_{D_j})\otimes_{k_{D_j}}W_{j,i}
\]
where $W_{j,i}$ is a $k_{D_j}$-vector space of rank $l_{j,i}$.

As in \cite[App. to Ch. 3]{RZbook}, (\ref{groupClosedImm})
is given by a closed immersion
\begin{equation}\label{CLimm}
\CG\hookrightarrow \GL_{O_B}(\mathscr L)\hookrightarrow  \prod_{j=1}^m\prod_{i=1}^{r_j} {\rm Aut}_{\RM_{n_j }(O_{D_j})}(\Lambda_{j,-i+1}).
\end{equation}

In  \S \ref{sss:proofImpliesIW} we will show:

\begin{proposition}\label{impliesIW}
Suppose  (IW), i.e. $\CG\otimes \BF_p$ is connected and $\CG^\circ$ is 
an Iwahori group scheme for $G^\circ$. Then ${\rm Aut}_{O_B}(\mathscr L)$ is an Iwahori group scheme for
$\GL_B(V)$. 
\end{proposition}

Note that the conclusion that
${\rm Aut}_{O_B}(\mathscr L)$ is an Iwahori group scheme for
$\GL_B(V)$ is equivalent to requiring that, for each $1\leq j\leq m$, the $O_{B_j}$-lattice chain $\mathscr L_j$ in $V_j$ is \emph{complete},
i.e., that for each $j=1,\ldots, m$, $i\in \BZ$, $\Lambda_{j,i}/\Lambda_{j, i-1}$ is a simple $\RM_{n_j}(k_{D_j})$-module.

\begin{remark}\label{rem:IW}
 Conversely, if ${\rm Aut}_{O_B}(\mathscr L)$ is Iwahori, the neutral component $\CG^\circ$ of $\CG$ is an Iwahori group scheme of $G^\circ$. Indeed,  it is enough to show that the maximal reductive quotient $(\CG^\circ\otimes_{\BZ_p}\BF_p)_{\red}$ of its reduction $\CG^\circ\otimes_{\BZ_p}\BF_p$ is a torus. But, by  the proof of Proposition \ref{impliesIW},  $(\CG^\circ\otimes_{\BZ_p}\BF_p)_{\red}$ is a subgroup scheme of the maximal reductive quotient of the reduction  ${\rm Aut}_{O_B}(\mathscr L)\otimes_{\BZ_p}\BF_p$ and this
 is a torus.  On the other hand, the property that  ${\rm Aut}_{O_B}(\mathscr L)$ is Iwahori does not imply that $\CG\otimes\BF_p$ is connected.  An example is given by the group $\RG$ of unitary similitudes of a $\RF/\RF_0$-hermitian vector space $\RV$ of odd dimension (with similitude factor in $\BQ^*$), where $\RF_0\otimes\BQ_p=\BQ_p^d$ with $d\geq 2$ and $\RF_p=\RF_0\otimes K$, where $K/\BQ_p$ is a ramified quadratic extension. There exists a complete selfdual periodic lattice multichain $\CL$ such that the successive quotients are all one-dimensional. Let $\U=\Res_{F_0/\BQ_p}(\U(V))$. Then $\pi_1(\U)_\Gamma=(\BZ/2\BZ)^d$, and the Kottwitz invariant induces a surjective map $\CG(\BZ_p)\cap \U(\BQ_p)\to  \pi_1(\U)_\Gamma$, comp. \cite[p. 132]{PRtw}. But then the image in $\pi_1(G)_\Gamma$ of $\CG$ under the Kottwitz invariant is also non-trivial, i.e., $\CG\otimes\BF_p$ is disconnected.

  Also, asking that the neutral component $\CG^\circ$ of $\CG$ is Iwahori for $G^\circ$ is not enough to guarantee that ${\rm Aut}_{O_B}(\mathscr L)$ is Iwahori.  
For example, in some unitary and orthogonal groups there are Iwahori group schemes which are given using complete periodic lattice chains, where the successive quotients have dimension larger than $1$.  In the unitary case, ${\rm Aut}_{O_B}(\mathscr L)$ Iwahori excludes only the case when the quadratic extension is ramified, and the hermitian space has even dimension $n=2m$ and is non-split (in this case, there is no $\pi$-modular lattice, i.e., no vertex lattice of type $m$), cf. \cite{Jac}. In the orthogonal case, things are less neat. When the quadratic space $V$ is split (i.e., there exists a Witt basis), then ${\rm Aut}_{O_B}(\mathscr L)$ is Iwahori;  but when there is a non-trivial anisotropic part, the condition ${\rm Aut}_{O_B}(\mathscr L)$ Iwahori can fail. For instance, if  $V$ is anisotropic of dimension $2$, then $V$ is of the form $(E, q)$, where $E/F$  is a quadratic field extension and $q= d \Nm_{E/F}$, for some $d\in F^*$, cf. \cite[\S 2.2.2]{HZ}.  If  $E/F$ is  unramified, then $V$ contains only one vertex lattice which is either of type $0$ or of type $2$, hence ${\rm Aut}_{O_B}(\mathscr L)$ is not Iwahori.  If $E/F$ is ramified, then again $V$ contains only one vertex lattice which is now of type $1$, hence ${\rm Aut}_{O_B}(\mathscr L)$ is  Iwahori.  See also the proof of Proposition \ref{impliesIW}.
 \end{remark}

We now assume that  ${\rm Aut}_{O_B}(\mathscr L)$ is an Iwahori group scheme for
$\GL_B(V)$. 
We write $I=\sqcup_{j=1}^m \{1,\ldots, r_j\}$ for the indexing set of the direct product \eqref{CLimm}, with $\sum_{j=1}^mr_j$ elements. For $s\in I$ we denote by $j(s)\in \{1,\ldots, m\}$ its projection and write $D_s=D_{j(s)}$. Set $d_s=[k_{D_s}:\BF_p]$.
 The composition of maps in \eqref{CLimm} induces a map on maximal torus quotients and hence gives a surjection of $\Gal(\ov\BF_p/\BF_p)$-modules
 \begin{equation}\label{galoismod1}
 \bigoplus_{s\in I}X^*(\Res_{k_{D_{s}}/\BF_p}\BG_m)=  \bigoplus_{s\in I} \BZ[\Gal(k_{D_{s}}/\BF_p)]\to X^*(T_{\CG^\circ}),
\end{equation}
where we recall that $\BZ[\Gal(k_{D_{s}}/\BF_p)]=\BZ[\Gal(\BF_{p^{d_s}}/\BF_p)]$.

Assuming ${\rm Aut}_{O_B}(\mathscr L)$ is Iwahori,  consider the following additional splitness hypothesis:
\smallskip

(S) \emph{There is a subset $J\subset I$ such that  the 
composition 
\[
 \bigoplus_{s\in J}\BZ[\Gal(k_{D_{s}}/\BF_p)]\subset \bigoplus_{s\in I}\BZ[\Gal(k_{D_{s }}/\BF_p)]\xrightarrow{(\ref{galoismod1})} 
 X^*(T_{\CG^\circ})
 \]
either gives an isomorphism
\begin{equation}\label{star1}
X^*(T_{\CG^\circ})= \bigoplus_{s\in J}\BZ[\Gal(k_{D_{s}}/\BF_p)],
\end{equation}
or a Galois equivariant direct summand of $\BZ\cdot c$ in $ X^*(T_{\CG^\circ})$ such that 
\begin{equation}\label{star2}
X^*(T_{\CG^\circ})= \BZ\cdot c\oplus \bigoplus_{s\in J}\BZ[\Gal(k_{D_{s}}/\BF_p)].
\end{equation}}
In the above, we denote by $c\in X^*(T_{\CG^\circ})$ the character obtained 
from the restriction $\CG^\circ\to \BG_m$ of the similitude $c: G\otimes\BQ_p\to \BG_m$. 
Note that (S) implies that $T_{\CG^\circ}$ is an induced torus.

When (S) is satisfied, it is useful to also consider the torus direct summand 
\begin{equation}\label{dsTorus}
T'_{\CG^\circ}:= \prod_{s\in J} \Res_{W(k_{D_{s}})/\BZ_p}\BG_m\subset T_{\CG^\circ}
\end{equation}
 of $T_{\CG^\circ}$ over $\BZ_p$ which
corresponds to   $\bigoplus_{s\in J}\BZ[\Gal(k_{D_{s}}/\BF_p)]$. Of course, in the first case  (\ref{star1}), we have $T'_{\CG^\circ}=T_{\CG^\circ}$.

Our main result in this section is the following:

 \begin{theorem}\label{thm:PEL}
 Assume that $(\RB,\RV,( \,,\,), *, h)$ are PEL data as in \S \ref{sss:PELdata}.  Choose an odd prime $p$, an order $O_\RB\subset \RB$, and a periodic self-dual $O_\RB\otimes {\BZ_p}$-lattice multichain $\mathscr L$ in $\RV\otimes\BQ_p$,  as in \S \ref{sss:PELdata} above. Assume the following hold.
 \begin{itemize}
 \item[1)]  
 the special fiber $\CG\otimes\BF_p$ of
 the group scheme $\CG={\rm Aut}_{O_B,(\, )}(\mathscr L)$   is connected
and the neutral component $\CG^\circ$ is an Iwahori group scheme for $G^\circ$,
 
 \item[2)]  the generic fiber $G=\CG\otimes_{\BZ_p}\BQ_p$ splits over a tamely ramified extension of $\BQ_p$,  
 
 \item[3)]   condition (S) is satisfied. 
 \end{itemize}

 Then Conjectures \ref{conj:torsors} and \ref{globconj}  for the 
 $T_{\CG^\circ}(\BF_p)$-covers of the corresponding PEL type Shimura varieties for the Shimura data $(\RG^\circ, X)$ with $\CG^\circ(\BZ_p)=\CG(\BZ_p)$-level at $p$ are true.
 \end{theorem}

We will give the proof in \S \ref{ss:ConjPEL}. A main tool in the proof is Oort--Tate/Raynaud theory for group schemes which we review in \S \ref{ss:OTR}.

\subsubsection{The proof of Proposition \ref{impliesIW}}\label{sss:proofImpliesIW}
 We continue with the above notations and start with some preliminaries:
Taking $(\ ,\ )$-duals gives an involution 
\[
\GL_{O_B}(\mathscr L)\to \GL_{O_B}(\mathscr L),\qquad g\mapsto g^\vee,
\]
and by definition
\[
\CG=\{g\in \GL_{O_B}(\mathscr L)\ |\ g^\vee\cdot g=c(g), \ c(g)\in \BZ_p^*\}.
\]
This gives an exact sequence
\[
0\to \CG_1\to \CG\xrightarrow{\ c\ } \BG_m\to 1
\]
of group schemes over $\BZ_p$. Here,
\begin{equation}\label{FixedPoints}
\CG_1=( \GL_{O_B}(\mathscr L))^{\tau=1}
\end{equation}
where $\tau(g)=(g^\vee)^{-1}$. This realizes $\CG_1$ as the fixed point group scheme for the $\{1, \tau\}$-group action
on $\GL_{O_B}(\mathscr L)$. 

\begin{remark}\label{rem:Edixhoven} Since $p$ is assumed odd, the $\tau$-action above is tame. This observation provides an alternative approach to  \cite[App. 3]{RZbook}. More precisely,
by \cite[Prop. A.4]{RZbook}, the group scheme $\GL_{O_B}(\mathscr L)$ is smooth. Therefore, the presentation \eqref{FixedPoints}, combined with Edixhoven's lemma \cite{Edix}, yields the smoothness of $\CG_1$. Similarly, by \cite[Prop. A.4]{RZbook}, the scheme of trivializations $\und{\rm Isom}_{O_B}(\{M_\Lambda \}, \{\Lambda \})$ is a torsor under  $\GL_{O_B}(\mathscr L)$. It follows that the fixed point  scheme $\und {\rm Triv}:=\und{\rm Isom}_{O_B}(\{M_\Lambda \}, \{\Lambda \})^{\tau=1}$  of trivializations  of a polarized multichain  of $O_B$-modules $\{M_\Lambda\}$ is  smooth and with free $\CG_1$-action, i.e., $\und {\rm Triv}$ is a ``pseudo-torsor'' for $\CG_1$. The main theorem  \cite[Thm. 3.16]{RZbook},  which gives a normal form for $\{M_\Lambda\}$,  is equivalent to the claim that $\und {\rm Triv}$ is actually a $\CG_1$-torsor. This is now easier to prove.  Indeed, we only have to  check that all  fibers of $\und {\rm Triv}$ are non-empty, and since when $p$ is invertible this is easy, we even may assume that the base is an $\ov\BF_p$-scheme. This problem has still to be handled as in  \cite[A.6]{RZbook}, but  the cases that have to be  checked are now simpler: for example, the quaternion algebras that appear in loc.~cit.  are now split and so are various symmetric/hermitian forms.
\end{remark}

We now consider the special fibers. Taking $\tau$-fixed points in the exact sequence 
\[
1\to U\to \GL_{O_B}(\mathscr L)\otimes\BF_p\to (\GL_{O_B}(\mathscr L)\otimes\BF_p)_\red\to 1
\]
remains exact, since the kernel $U$ is unipotent and $p\neq 2$.
This gives 
\[
(\CG_1\otimes\BF_p)_\red =((\GL_{O_B}(\mathscr L)\otimes\BF_p)_\red)^{\tau=1}.
\]
We have 
\begin{equation}\label{eq:GLred}
(\GL_{O_B}(\mathscr L)\otimes\BF_p)_\red=\prod_{j=1}^m\prod_{i=-r_j+1}^0{\rm Aut}_{\RM_{n_j}(k_{D_j})}(Q_{j,i})=\prod_{j=1}^m\prod_{i=-r_j+1}^0{\rm Aut}_{k_{D_j}}(W_{j,i}),
\end{equation}
and the fixed point scheme $((\GL_{O_B}(\mathscr L)\otimes\BF_p)_\red)^{\tau=1}$ 
decomposes as a product over the set of orbits of the involution $j\mapsto j^*$ on $[1,m]:=\{1,\ldots, m\}$.

A) First consider an orbit $\{j,j^*\}$, $j^*\neq j$, consisting of two elements. This contributes to the fixed point scheme the factor
\begin{equation}\label{Factor1}
\prod_{i=-r_j+1}^{0}{\rm Aut}_{\RM_{n_j }(k_{D_j})}(Q_{j,i}))={\rm Aut}_{k_{D_j}}(W_{j,i}),
\end{equation}

B) Next consider an orbit $\{j\}$, $j^*=j$ consisting of one element: This contributes to the fixed point scheme  a factor according to the following recipe
in which to simplify notation, we omit the subscript $j$:

1) Suppose $a=a_j=0$. 

1a) Case $r=2r'+1$ is odd. By using periodicity and Morita equivalence, the form $(\ ,\ )$, gives perfect pairings 
\[
W_{-r'}\times W_{-r'}\to \BF_p,
\]
and
\[
W_0\times W_{-r+1}\to \BF_p,\ \ldots,\ W_{-r'+1}\times W_{-r'-1}\to \BF_p.
\]
These induce the involution $\tau$ on \eqref{eq:GLred}. In particular, the first pairing 
induces an involution $\ov\tau$ on ${\rm Aut}_{ k_{D}}(W_{-r'})$.
The factor contributing to  the fixed point scheme is 
\begin{equation}\label{Factor1a}
{\rm Aut}_{ k_{D}}(W_{-r'})^{\ov\tau=1}\times \prod_{i=-r'+1}^0 {\rm Aut}_{ k_{D}}(W_{i}).
\end{equation}

1b)  $r=2r'$ is even. We  have perfect pairings 
\[
W_0\times W_{-r+1}\to \BF_p,\ \ldots,\ W_{-r'}\times W_{-r'+1}\to \BF_p
\]
inducing the involution. The factor contrubuting to the fixed point scheme  is 
\begin{equation}\label{Factor1b}
  \prod_{i=-r'}^0 {\rm Aut}_{k_{D}}(W_{i}).
\end{equation}

2) $a=a_j=1$. 

2a) $r=2r'+1$ is odd. We  have perfect pairings 
\[
W_{1}\times W_{1}\to \BF_p,
\]
and
\[
W_0\times W_{-r+2}\to \BF_p,\ \ldots,\ W_{-r'+1}\times W_{-r'}\to \BF_p.
\]
The first pairing induces an involution $\ov\tau$ on ${\rm Aut}_{ k_{D}}(W_{1})$.
The factor contributing to the fixed point scheme is 
\begin{equation}\label{Factor2a}
{\rm Aut}_{ k_{D}}(W_{1})^{\ov\tau=1}\times \prod_{i=-r'+1}^0 {\rm Aut}_{ k_{D}}(W_{i}).
\end{equation}

2b) $r=2r'$ is even. We  have perfect pairings
\[
W_{1}\times W_{1}\to \BF_p,\quad W_{-r'+1}\times W_{-r'+1}\to \BF_p,
\]
and
\[
W_0\times W_{-r+2}\to \BF_p,\ \ldots,\ W_{-r'+2}\times W_{-r'}\to \BF_p.
\]
The first two pairings induce  involutions $\ov\tau$ on ${\rm Aut}_{ k_{D}}(W_{1})$, resp.  ${\rm Aut}_{ k_{D}}(W_{-r'+1})$. 
The factor contributing  to the fixed point scheme is 
\begin{equation}\label{Factor2b}
{\rm Aut}_{ k_{D}}(W_{1})^{\ov\tau=1}\times {\rm Aut}_{ k_{D}}(W_{-r'+1})^{\ov\tau=1}\times \prod_{i=-r'+2}^0 {\rm Aut}_{ k_{D}}(W_{i}).
\end{equation}

We now proceed with the actual proof.
We have the exact sequence
\[
1\to (\CG_1\otimes\BF_p)_{\red}\to (\CG\otimes\BF_p)_{\red}\xrightarrow{\ c\ }\BG_m\otimes\BF_p\to 1,
\]
where $(\CG\otimes\BF_p)_{\red}$ is the quotient of $\CG\otimes\BF_p$ by its maximal normal unipotent subgroup scheme.
Under our assumption, $(\CG\otimes\BF_p)_{\red}$ is a torus and so $(\CG_1\otimes\BF_p)_{\red}$ is a diagonalizable group scheme. We now apply the description of 
$(\CG_1\otimes\BF_p)_{\red}$ as a product of factors (\ref{Factor1}) -- (\ref{Factor2b}) as above.
We conclude that for all $j$ with $j^*\neq j$, and all $i$, $W_{j,i}$ is $1$-dimensional over $k_{D_j}$ and hence $Q_{j,i}$ is a simple $\RM_{n_j}(k_{D_j})$-module. Indeed, by periodicity, it is enough to check the indices $i$ in the range $[-r_j+1, 0]$.

The same is true for all $j$ with $j^*=j$, with $i$ in the appropriate range according to our cases (1a), (1b), (2a), (2b).
(For example, for all $i$ in case (1b), etc.) We now consider
the only remaining factors for ``exceptional indices" $i$. These are of the form
\[ 
{\rm Aut}_{k_{D_j}}(W_{j,i} )^{\ov\tau=1}=\GL_{l_{j,i}}(k_{D_j})^{\eta_{j,i}=1}
\]
for $j=j^*$, with $\ov\tau=\eta_{j,i}$ a non-trivial involution of ${\rm End}_{k_{D_j}}(W_{j,i})\simeq \RM_{l_{j,i}}(k_{D_j})$. By our assumption, these are also diagonalizable group schemes.  To simplify notation, we fix a pair $j$, $i$, and omit it from the notation.

\begin{itemize}
\item If $\eta$ is an involution of the second type, 
then $\GL_{l}(k_{D})^{\eta=1}$ is the unitary group $\RU(l)$ over the fixed field $(k_{D})^{\eta=1}$ and we consider its Weil restriction of scalars to $\BF_p$; this is commutative only if $l=1$.

\item If $\eta$ is an involution of the first type, then $\GL_{l }(k_{D})^{\eta=1}$ is commutative only if $l=1$, or if $l=2$ and 
$\eta_j$ is an orthogonal involution. When $l=2$, $\GL_{l}(k_{D})^{\eta=1}=\RO_2(W, q)$ for some quadratic form $q$. The case $l=2$
will  be ruled out below.
\end{itemize}

We first consider the case that, in addition, the invariants $F_0=F^{*=1}$ of the involution $*$ on the center $F$ of $B$ form a field.
Then $m=1$   and  $j=1$ is the unique index which will be omitted from the notation. Suppose we are in case (1a). Then, from the above, we have
\[
(\CG_1\otimes\BF_p)_{\red}\simeq  (\prod_{i=-r'+1}^0 \Res_{k_D/\BF_p}\BG_m)\times \RO_2(W_{-r'}, q_{-r'}),
\]
\[
(\CG\otimes\BF_p)_{\red}\simeq   (\prod_{i=-r'+1}^0 \Res_{k_D/\BF_p}\BG_m)\times \GO^*_2(W_{-r'}, q_{-r'}),
\]
where $\GO^*_2(W_{-r'},q_{-r'})$ signifies the subgroup of elements in $\GO^*_2(W_{-r'}, q_{-r'})$ with similitude in $\BF_p^*$. However, the groups
$\GO_2(W_{-r'}, q_{-r'})$ and $\GO^*_2(W_{-r'}, q_{-r'})$ are not connected. A similar argument works in the other cases (2a), (2b) for which we have such indexes to consider.

This rules out $l_{j,i}=2$ when $F_0$ is a field.  If $F_0$ is not a field, and there are more than one $j$ with $j^*=j$ with exceptional indices $i$ such that $l_{j,i}=2$, there are more orthogonal factors of the same form and 
$(\CG\otimes\BF_p)_{\red}$ and hence $\CG\otimes\BF_p$ is still not connected.  This contradicts our assumption. Hence $l_{j,i}=2$ is ruled out again. 
It follows that $l_{j,i}=1$ in all cases. This concludes the proof of Proposition \ref{impliesIW}. \qed

\subsubsection{More on condition (S)}\label{sss:Sprime} Suppose now that $\CG\otimes\BF_p$ is connected and $\CG^\circ$ is Iwahori. Then by Proposition \ref{impliesIW}, we know that 
${\rm Aut}_{O_B}(\mathscr L)$ is Iwahori and it makes sense to consider assumption (S). In the context of the notation and set-up in the proof of Proposition \ref{impliesIW} we can see that the following assumption implies (S).
\smallskip

\begin{itemize}
\item[(S')] \emph{There is at most one index $j$ with $j^*=j$. In addition, this $j$ is not in case (2b), so it has only at most one exceptional index $i$.
In addition, when this exceptional index $i$ occurs, for this $j$ and $i$ we have $k_{D_j}^{\eta_{j,i}=1}=\BF_p$.}
\end{itemize}

To explain this, let us assume (S'). In the proof of Proposition \ref{impliesIW} we have determined $(\CG_1\otimes\BF_p)_\red$. Using this, we can now determine
$(\CG\otimes\BF_p)_\red=T_{\CG^\circ}\otimes\BF_p$. Under our assumptions, this is as follows:

 A) If there is no $j=j^*$, we have 
 \[
T_{\CG^\circ}\otimes\BF_p=(\CG\otimes\BF_p)_{\red}=(\prod_{\{j,j^*\}, j^*\neq j} \Res_{k_{D_j}/\BF_p}\BG^{r_j}_m)\times \BG_m.
\]

B) If there exists a (unique) $j_0=j^*_0$, we have:
\begin{itemize}
\item In case (1a) or (2a) with $k_{D_0}=\BF_p$,
\[
T_{\CG^\circ}\otimes\BF_p=(\CG\otimes\BF_p)_{\red} = (\prod_{\{j,j^*\}, j^*\neq j} \Res_{k_{D_j}/\BF_p}\BG^{r_j}_m)\times  \BG_m^{r'_0}\times \BG_m.
\]

\item In case (1a) or (2a) with $[k_{D_0}:\BF_p]=2$,
\[
T_{\CG^\circ}\otimes\BF_p=(\CG\otimes\BF_p)_{\red} = (\prod_{\{j,j^*\}, j^*\neq j} \Res_{k_{D_j}/\BF_p}\BG^{r_j}_m)\times  \Res_{k_{D_0}/\BF_p}\BG^{r_0'}_m\times \Res_{k_{D_0}/\BF_p}\BG_m.
\]

\item In case (1b),
\[
T_{\CG^\circ}\otimes\BF_p=(\CG\otimes\BF_p)_{\red} = (\prod_{\{j,j^*\}, j^*\neq j} \Res_{k_{D_j}/\BF_p}\BG^{r_j}_m)\times  (\Res_{k_{D_0}/\BF_p}\BG^{r_0'}_m)\times\BG_m.
\]
\end{itemize}

(In all of this, $r'_0$ is the number of non-exceptional factors.) 

In these presentations, the similitude 
\[
c: (\CG\otimes\BF_p)_{\red}\to \BG_m
\]
is the projection to the last factor $\BG_m$ when there is no $j=j^*$ or in case (1b) or in cases (1a), (2a) when $k_{D_0}=\BF_p$ (in these cases, the alternative \eqref{star2} holds).  In cases (1a) or (2a) with $[k_{D_0}:\BF_p]=2$, the similitude is the projection to the last factor followed by the norm (in these cases, the alternative \eqref{star1} holds). 
In all cases, (S) is satisfied. 

\begin{remark}
Condition (S'),  and also (S), can fail in cases (1a) and (2a), when the fields $k_{D_0}$ are ``large". 
\end{remark}

\subsection{Oort--Tate/Raynaud theory}\label{ss:OTR}
Let $\BF_q$, $q=p^d$, be a finite field and set $O=W(\BF_q)$. In this paragraph, we review some constructions related to Raynaud's description of certain rank $1$ $\BF_q$-vector space schemes (``Raynaud group schemes") over $O$-schemes. 

\subsubsection{Recollections of Raynaud theory} Let $S$ be an $O$-scheme and let $\pi: G\to S$ be a commutative finite flat finite presentation group scheme  which is killed by $p$ and is, in addition, a $\BF_q$-vector space scheme in the sense of \cite[Def. 1.2.1]{Ray}. (For $d=1$ this is no extra structure.) Consider the Teichmuller character
\[
\chi_0: \BF_q^*\to O^*,\quad \chi_0(a):=[a]
\]
 and its powers $\chi_i=\chi_0^{p^i}$, for $i\in \BZ/d\BZ$. Then $\chi_i$ are fundamental characters of $\BF_q^*$, see \cite[Def. 1.1]{Ray}. For simplicity, we write $\CO_G$ instead of $\pi_*\CO_G$ for the   ``Hopf algebra" of $G\to S$; this is  a coherent locally free  sheaf of $\CO_S$-algebras over $S$.
 Since $G$ is an $\BF_q$-vector space scheme, for each $a\in \BF_q$ there is an endomorphism $[a]^*_G: \CO_G\to \CO_G$ which preserves the augmentation ideal $\CI_G:=\ker(\epsilon^*: \CO_G\to \CO_S)$, with $\epsilon: S\to G$ denoting the identity section.
 We now decompose  into isotypical components: 
 \[
 \CI_G=\bigoplus_{\chi \in (\BF^*_q)^*}\CI_\chi.
 \]
Here, for each open $U\subset S$, $\Gamma(U,\CI_\chi)$ is the set of sections $s$ of $\CI_\chi$
 with $[a]^*_G\cdot s=\chi(a)s$, for all $a\in \BF_q^*$. 
 
 We say that $G$ is a \emph{Raynaud $\BF_q$-vector space scheme} 
 if all $\CI_\chi$ are invertible $\CO_S$-modules, i.e. they have rank $1$. Then, $G$ has rank $p^d$ over $S$.
(Note that if $d=1$ and  $G$ has rank $p$ then all $\CI_\chi$ are invertible. Then $G$ is an \emph{Oort--Tate group scheme} over $S$.)

Assume $G$ is a Raynaud $\BF_q$-vector space scheme over $S$. We have $\CO_S$-module homomorphisms
 \[
 d_i: \CI_{\chi_i}^{\otimes p}\to \CI_{\chi_{i+1}},\quad c_i: \CI_{\chi_{i+1}}\to \CI_{\chi_{i}}^{\otimes p}
 \]
 given by multiplication, resp. comultiplication in $\CO_G$. By \cite[1.4]{Ray}, $d_i\cdot c_i=w\cdot \id_{\CI_{\chi_{i+1}}}$,
 $c_i\cdot d_i=w\cdot \id_{\CI_{\chi_i}^{\otimes p}}$. Here, $w=w_q$ is a uniformizer of $O$ which is independent of $G$ and $i$.
 
\begin{theorem}\label{thm:ray} (\cite[Thm. 1.4.1]{Ray}). The construction 
 \[
 G\mapsto ( \CI_{\chi_i}, c_i, d_i)_{i\in \BZ/d\BZ}
 \]
  gives an isomorphism between  the stack of Raynaud $\BF_q$-vector space schemes over $O$
 and the stack over $O$ whose $S$-points are given by systems
 \[
 (\CL_i, \ga_i, \de_i)_{i\in \BZ/d\BZ},
 \]
 where the $\CL_i$,  are invertible $\CO_S$-modules, and where $\ga_i: \CL_{i+1}\to\CL_i^{\otimes p}$ and  $\de_i: \CL_i^{\otimes p}\to \CL_{i+1}$,
 are $\CO_S$-module homomorphisms such that $\de_i\cdot \ga_i=w\cdot \id_{\CL_{i+1}}$ for all $i\in \BZ/d\BZ$. 
 \end{theorem}

 Note that we can also describe the stack of systems as above as the quotient stack
 \[
[ \BG_m^d\bs \Spec(O[x_i, y_i]_{i\in \BZ/d\BZ}/(x_iy_i-w)_{i\in \BZ/d\BZ})] 
 \]
 with $\und\lambda=(\lambda_i)_{i\in \BZ/d\BZ}$ in $\BG_m^d$ acting by 
 \[
\und\lambda\cdot x_i=\lambda_i^p\lambda_{i+1}^{-1} x_i,\quad \und\lambda\cdot y_i=\lambda_i^{-p}\lambda_{i+1} y_i.
 \]
 In the course of the proof of the above theorem, Raynaud gives the Hopf algebra of the group scheme which corresponds to $ (\CL_i, \ga_i, \de_i)_{i\in \BZ/d\BZ}$ as a quotient of the symmetric algebra ${\rm Symm}_{\CO_S}(\oplus_{i\in \BZ/d\BZ} \CL_i)$ by the ideal generated by
 \[
 (\de_i-\id)\CL^{\otimes p}_i,\quad  i\in \BZ/d\BZ.
 \]

\subsubsection{Generators of Raynaud group schemes}\label{sss:Generators} Suppose $S$ is an $O$-scheme and   $G$ is a Raynaud $\BF_q$-vector space scheme over $S$. 

By \cite[\S 5]{P95}, an \emph{$\BF_q$-generator} of $G$ over $S$   is an $S$-valued point $P\in G(S)$ with the property that the multiples $[a]_G(P)\in G(S)$, $a\in \BF_q$, form a full set of sections of $G\to S$ in the sense of Katz-Mazur, \cite{KatzMazur}.

By \cite[Prop. 5.1.5]{P95}, the functor of $\BF_q$-generators of $G$ is represented by a closed subscheme $G^*\hookrightarrow G$ which is finite locally free of rank $p^d-1$ over $S$. The subscheme $G^*$ has a simple description in terms of the system  $(\CL_i, \ga_i, \de_i)_{i\in \BZ/d\BZ}$: Note that by tensoring we have
\[
\otimes_{i\in \BZ/d\BZ} \de_i:  \bigotimes_{i\in \BZ/d\BZ}  \CL_i^{\otimes p}\to \bigotimes_{i\in \BZ/d\BZ}  \CL_{i+1}.
\]
This corresponds to a  $\CO_S$-homomorphism
\[
\delta_{\otimes} =\otimes_{i\in \BZ/d\BZ} \de_i: ( \bigotimes_{i\in \BZ/d\BZ}  \CL_i)^{\otimes (p-1)}\to   \CO_S.
\]
By loc. cit. $G^*\subset G$ is cut out by the locally principal ideal in $\CO_G$ which is generated by 
\[
(\delta_\otimes-\id )((\bigotimes_{i\in \BZ/d\BZ} \CL_i)^{\otimes(p-1)}).
\]

Suppose $S=\Spec(R)$ and that all $\CL_i$ are free. By picking a basis $u_i$ of $\CL_i$ we obtain $\ga_i$, $\de_i\in R$, with $\ga_i\cdot \de_i=w$. Then
\[
\CO_G=R[u_i]_{i\in \BZ/d\BZ}/((u_{i}^p-\de_{i}u_{i+1})_{i\in \BZ/d\BZ}), 
\]
\[
 \CO_{G^*}=R[u_i]_{i\in \BZ/d\BZ}/((u_{i}^p-\de_{i}u_{i+1})_{i\in \BZ/d\BZ}, (u_1\cdots u_d)^{p-1}-\de_1\cdots \de_d).
\]

\begin{remark}\label{rem:Generators}
1) The above notion of an $\BF_q$-generator of the Raynaud group scheme $G$ over $S$ agrees with the notion of a primitive element of $G$ over $S$ as defined by Kottwitz-Wake, see  \cite[\S 1.5]{KoWa}. Their notion of primitive element applies to all finite locally free commutative group schemes.

2) This notion of an $\BF_q$-generator differs from the corresponding notion of a ``Raynaud generator" in \cite[Def. 3.4]{SLiu}. Liu's generators are not given by points of the group scheme $G$ and their definition depends on a choice.
Consider $R=\BZ_{p^2}[\gamma_0,\gamma_1,\delta_0,\delta_1]/(\gamma_0\delta_0-w_p, \gamma_1\delta_1-w_p)$ and the ``universal" Raynaud $\BF_{p^2}$-vector space scheme 
\[
G=\Spec(R[u_0, u_1]/(u_0^p-\delta_0u_1, u_1^p-\delta_1u_0))
\] 
over $S=\Spec(R)$. Then the $S$-scheme of Raynaud generators of $G/S$, as defined in \cite[Def. 3.4]{SLiu}, is 
\[
\Spec(R[x ]/(x ^{p^2-1}-\delta_0\delta_1^p)).
\]
This is not R1 over the generic point of the divisor $\delta_1=\gamma_0=0$ and is not a subscheme of $G$.
\end{remark}

\subsubsection{The \'etale case}\label{sss:Etale}
 Suppose that  the Raynaud $\BF_q$-vector space scheme $G\to S$ is \'etale. This is equivalent to $\de_i:\CL_i^{\otimes p}\to \CL_{i+1}$ being isomorphisms
  for all $i\in \BZ/d\BZ$. In this case, $G^*\to S$ is also \'etale. In fact, then the closed subscheme $G^*$ is also open in $G$ and is given 
  as the complement of the zero section. It follows that $G^*\to S$ is a $\BF_q^*$-torsor with $a\in \BF_q^*$ acting by $[a]_G$. 
  
  Set $T=\Res_{O/\BZ_p}\BG_m$ so that $T(\BF_p)=\BF_q^*$ and we have an exact sequence of group schemes over $\Spec (\BZ_p)$,
  \[
  1\to \und \BF_q^* \to T\xrightarrow{L} T\to 1.
  \]
 In this case, we can apply the construction of Proposition \ref{prop:torsors} to the $\BF_q^*$-torsor $G^*\to S$. Recall that $S$ is an $O$-scheme. We choose the isomorphism $T_O=\BG_m^d$ so that the character $\eta_i: T_O\to \BG_m$ given by the $i$-th projection restricts on the subgroup 
 $\BF_q^*=T(\BF_p)\subset T(O)$ to $\chi_i: \BF_q^* \to O^*$ above. Then the Frobenius generator of the Galois group $\Gal(\BF_q/\BF_p)$ acts on $X^*(T)=\BZ^d$ by sending $\eta_i$ to $\eta_{i-1}$ and $L^*(\eta_i)=p\eta_{i-1}-\eta_{i}$.   The $T$-torsor $Q$ over $S$ is given by $(\CL_i)_i$.
 The push out  $T$-torsor $L_*(Q)$ over $S$ is given by  $L_*(Q)=(\CL_{i-1}^{\otimes p}\otimes\CL_{i}^{-1})_i$ and the trivialization of $L_*(Q)$ is given by $(\de_{i-1}: \CO_S\xrightarrow{\sim} \CL_{i-1}^{\otimes -p}\otimes\CL_{i})_i$. 
 
 \subsubsection{The corresponding $\Res_{O/\BZ_p}\BG_m$-torsor}\label{sss:RaynaudTorsor}
 
 Suppose now that $G\to S$ is a (not necessarily \'etale) Raynaud $\BF_q$-vector space scheme over the $O$-scheme $S$. We still have a $T$-torsor 
 over $S$ given by $(\CL_i)_i$, with the notation of the previous paragraph.  The push out  $T$-torsor $L_*(Q)$ over $S$ is given by  $(\CL_{i-1}^{\otimes p}\otimes\CL_{i}^{-1})_i$; the duals of these are equipped with sections $\de_{i-1}: \CO_S\to \CL_{i-1}^{\otimes -p}\otimes\CL_{i}$.  The sections give a trivialization of the $T$-torsor $L_*(Q)$ over the generic fiber $S[1/p]$,   where the group scheme is \'etale.

 \subsection{The conjectures   in Iwahori  PEL cases}\label{ss:ConjPEL}
 
 We return to the assumptions and notations of \S \ref{ss:PELdata}. We assume that (IW) is satisfied so we are in the Iwahori case. Later we will also assume (S) and eventually prove Theorem \ref{thm:PEL}. Our goal is the construction of line bundles and trivializations as in (\ref{data1}), under these assumptions. 
 
 \subsubsection{Some preliminaries}\label{sss:prelim} Let us denote by $F=F_1\times\cdots\times F_m$ the center of the simple algebra $B$ and by $F_0$ the part of $F$ fixed by the involution $*$. We will assume that $F_0$ is a field; the general case can be reduced to this case. We can then break up into considering various cases as in \cite[p. 135]{RZbook}. 
 \medskip
 
 (I) $F=F_0\times F_0$ and $*$ induces on $F$ the transposition. Then there is a central division algebra $D$ over $F_0$, with
 \[
 B=B_1\times B_2=\RM_{n\times n}(D)\times \RM_{n\times n}(D^{\rm opp}).
 \]
As in loc. cit., we can write $V=W\oplus \widetilde W\simeq W\oplus W^\vee$, with $W^\vee={\rm Hom}_{\BQ_p}(W,\BQ_p)$, and we have
 \[
G\simeq {\rm Aut}_{\RM_{n\times n}(D)}(W)\times \BG_m.
\]
 We can further write $W=D^n\otimes_D U$ for some $D$-module $U$; then by Morita equivalence
 \[
 G\simeq \GL_{D}(U)\times \BG_m.
 \]
 In this case, we see as in loc. cit. that the self-dual lattice multichain $\mathscr L$ is determined by a periodic $O_D$-lattice chain 
 $\{\Gamma_i\}_{i\in \BZ}$ in $U$, i.e., $\Lambda_i=O_D^n\otimes_{O_D}\Gamma_i$. Then the multichain $\mathscr L$ is given by $\Lambda_i\oplus \Lambda^\vee_{i'}$, with $\Lambda^\vee={\rm Hom}_{\BZ_p}(\Lambda, \BZ_p)$. 
 
 We have $\CG=\CG^\circ$. Condition (IW) is equivalent to $\{\Gamma_i\}_{i\in \BZ}$ being a complete chain and therefore to $\CG$ being an Iwahori group scheme of $G$. Assuming (IW) we can see that (\ref{star2}) and so (S) holds:
 Indeed, $c$ is given by projection to the $\BG_m$ factor above and we have
 \[
 X^*(T_\CG)=(\bigoplus_{i=1}^{\dim_D(U)} \BZ[\Gal(k_D/\BF_p)])\oplus \BZ\cdot c, 
 \]
where the summands $\BZ[\Gal(k_D/\BF_p)]$ are given by the characters of the automorphisms of the quotients $\Gamma_i/\Gamma_{i-1}$, $i=1,\ldots, \dim_D(U)$. Note that by (IW), $\Gamma_i/\Gamma_{i-1}\simeq k_D$, for all $i$.
\medskip

(II-IV)  In the rest of the cases, (II), (III), (IV) of \cite[p. 135]{RZbook}, the center $F$ is a field and  $B=B_1$ is central simple over $F$.
 We also write $D=D_1$, $d=[k_D:\BF_p]$, and, in general, we omit the subscript $j=1$. Then $\mathscr L=\{\Lambda_i\}_{i\in \BZ}$ is a periodic self-dual chain of $\RM_{n\times n}(O_D)$-lattices. By Morita, we can write $V=D^n\otimes_D W$ and $\Lambda_i=O_D^n\otimes_{O_D}\Gamma_i$, with $\mathscr G:=\{\Gamma_i\}_{i\in \BZ}$ a periodic chain of $O_D$-lattices in $W$.

If (\ref{star1}) holds, i.e.
  \[
X^*(T_{\CG^\circ})\simeq  \bigoplus_{i\in J}\BZ[\Gal(k_{D}/\BF_p)], 
\]
 then a basis of $X^*(T_{\CG^\circ})$ is given by $\eta_{i, a}$ with $i\in J\subset \{1,\ldots, r\}$ and $a\in \BZ/d\BZ$.
 In the other case (\ref{star2}), we have a basis given by $\eta_{i, a}$ with $i\in J$ and $a\in \BZ/d\BZ$, and by $c$, the similitude character.
 \medskip

 \subsubsection{The proof of Theorem \ref{thm:PEL}}\label{sss:pf712}
 \begin{proof}  We assume   that  the special fiber of  $\CG$ over $\BZ_p$ is connected and that $\CG^\circ$ is Iwahori for $G^\circ$ so that $\CG(\BZ_p)=\CG^\circ(\BZ_p)$ is an Iwahori subgroup of $G^\circ(\BQ_p)$. We also assume (S) and that the generic fiber $G^\circ=\CG^\circ\otimes_{\BZ_p}\BQ_p$ splits over a tamely ramified extension of $\BQ_p$ so we can apply Theorem \ref{tameconj}. This tameness   is also a blanket assumption in \cite{PZ}. 
 
 We will continue with the above notations. We first assume that  the invariants $F_0=F^{*=1}$ of the involution of the center of $B$ is a field, see \S\ref{sss:prelim}.

 In all cases, both (I) and (II-IV), we have $\RM_{n\times n}(O_D)$-lattice chains $\{\Lambda_i\}_{i\in \BZ}$ and corresponding 
 $O_D$-lattice chains $\{\Gamma_i\}_{i\in \BZ}$ such that $\Lambda_i=O_D^n\otimes_{O_D}\Gamma_i$.

 Recall $O=W(k_D)$. For an $S$-valued point of $\CA_{K^p,O}$, set $A_{i}=A_{\Lambda_{i}}$ for the corresponding abelian scheme.  We then have isogenies of abelian schemes
 \[
 0\to H(i)\to A_{i-1}\xrightarrow{\phi_i} A_i\to 0.
 \]
 Here $H(i)$ is a finite flat group scheme with $\RM_{n\times n}(O_D)$-action which is killed by $\Pi\in O_D$. Therefore, 
 it has $\RM_{n\times n}(k_D)$-action and it decomposes
 \[
 H(i)   =e_{11}H(i)\times\cdots \times e_{nn}H(i)=G(i)^n.
 \]
Here, $e_{jj}$ is the standard idempotent of the matrices and  $G(i)$ has $k_D$-action. Similarly we can set $B_i=e_{11}A_i[p^\infty]$ which is a $p$-divisible group with $O_D$-action. So, for each $S$-valued point of $\CA_{K^p,O}$, 
 there is a $\mathscr G$-system of $p$-divisible groups $B_i=B_{\Gamma_i}$,
  with $O_D$-action and $O_D$-linear isogenies
  \begin{equation}\label{isogp}
  0\to G(i)\to  B_{i-1}\xrightarrow{\psi_i} B_i\to 0 .
  \end{equation}
Since ${\rm Aut}_{O_B}(\sL)$ is Iwahori, the $k_D$-vector spaces $W_i$ are $1$-dimensional, and so the kernel $G(i)$  of $\psi_i$ is a $k_D$-vector space scheme.
 In fact,  when $S$ factors through the flat closure $\CA^{\rm flat}_{K^p,O}$ of $\CA_{K^p,O}$, then $G(i)$ is a Raynaud $k_D$-vector space scheme. Indeed, it is enough to verify the condition about the rank of the isotypic components  for the universal $G(i)^{\rm univ}$ over the flat $O$-scheme $\CA^{\rm flat}_{K^p,O}$. This can be checked on the generic fiber $\CA^{\rm flat}_{K^p,O}[1/p]=\CA_{K^p,O}[1/p]$ where it is automatic, see \cite[Prop. 1.2.2]{Ray}.
We let 
\[
(\CL(i)_a, \gamma_a(i), \delta_a(i))_{a\in \BZ/d\BZ}
\]
 be the system which corresponds by Theorem \ref{thm:ray} to the universal  
Raynaud $k_D$-vector space scheme  $G^{\rm univ}(i)$ over  $\CA^{\rm flat}_{K^p, O}$.
We now take  
\[
 \CM_{i,a}:=\CL(i)_a. 
\]
We obtain the $T'_{\CG}$-torsor given by $\CM=\oplus_{i, a}\CM_{i, a}$. Here we recall the notation $T'_{\CG}$ from
 (\ref{dsTorus}). Then $\CM_{i, a}=\CM_{\eta_{i, a}}$ (see the last line of \S \ref{sss:prelim} for the definition of $\eta_{i, a}$).
Then
 \[
  \CM^{-1}_{L^*(\eta_{i,a})} =\CL(i)_{a-1}^{\otimes -p}\otimes \CL(i)_{a}.
 \]
 There is a $T'_{\CG}(\BF_p)$-torsor  
 \begin{equation}\label{pi1A}
\pi':  \CA'_{K^p, 1, O}[1/p]\to \CA_{K^p,O}[1/p]
 \end{equation}
 which classifies choices of Raynaud generators for all $G(i)^{\rm univ}$, $i\in J$, over the generic fiber. 
 
By \S \ref{sss:Etale}, we see that the construction of Proposition \ref{prop:torsors} applied to this torsor gives the line bundles 
$\CM_{i,a}[1/p]$ with the trivializations
 \[
 \CO_{\CA_{K^p, O}}[1/p]\xrightarrow{\sim}  \CM^{-1}_{L^*(\eta_{i, a})}[1/p]
  \]
 which are given by  
 \[
 \delta(i)_{a-1}:  \CL(i)_{a-1}^{\otimes p}\to  \CL(i)_{a} ,
 \]
cf.  \S\ref{sss:RaynaudTorsor}. We will now produce the second isomorphism as in (\ref{data1}).

For an $S$-valued point of $\CA^{\rm flat}_{K^p, O}$, given by the abelian schemes $A_i/S$ with additional structures, we consider 
\[
\omega_{A_i/S}=e^*\Omega^1_{A_i/S}
\]
the corresponding Hodge bundle and 
\[
\Lie(A_i/S)=\omega_{A_i/S}^\vee
\]
its dual. These are modules (right and left, respectively) over 
\[
O_\RB\otimes O=O_B\otimes_{\BZ_p}O=\RM_{n\times n}(O_D)\otimes_{\BZ_p}O
\]
Hence, $\omega_{A_i/S}e_{11}$ is a right $O_D\otimes_{\BZ_p}O$-module.
Using $O\subset O_D$ and $O\otimes_{\BZ_p}O=\oplus_{a\in \BZ/d\BZ} O$ we decompose
\[
\omega_{A_i/S}e_{11}=\bigoplus_{a\in \BZ/d\BZ} \Omega_{i,a/S} 
\]
into isotypic components $\Omega_{i,a/S}:= (\omega_{A_i/S}e_{11})_a$. These are locally free 
over $S$ of the same rank. Denote by $\Omega_{i,a} $ the corresponding locally free coherent sheaves 
over $\CA^{\rm flat}_{K^p, O}$ and let $\Omega^\vee_{i,a} $ be their duals. We have
\[
\Omega^\vee_{i,a/S} =(e_{11}\Lie(A_i/S))_a.
\]

Pulling back by the isogeny $\phi_i: A_{i-1}\to A_i$ gives  
$
\phi^*_i: \omega_{A_{i}/S}\to \omega_{A_{i-1}/S}.
$
This respects the $O_\RB\otimes O$-module structure. Hence, the universal isogeny gives
\[
\phi^*_i: \Omega_{i,a} \to \Omega_{i-1,a}.
\]
 
 \begin{proposition}\label{prop:Fitting}
For each  $i$, $a$, there is an  isomorphism  
 \begin{equation*}\label{Fitt}
 \CM^{-1}_{L^*(\eta_{i,a})} =\CL(i)_{a-1}^{\otimes -p}\otimes \CL(i)_{a}\xrightarrow{\sim} \det(\Omega_{i-1,a})\otimes \det(\Omega_{i,a} )^{\otimes -1} 
  \end{equation*}
 of invertible sheaves over $\CA^{\rm flat}_{K^p, O}$, which takes the section $\delta(i)_{a-1}$ to the section given by $\det(\phi^*_i)$. 
\end{proposition}

\begin{proof} By \cite[Lem. 4.1]{SLiu}, there is the following expression for the  isotypical components of the co-Lie complex of $G(i)^{\rm univ}$ (this holds for any Raynaud group scheme)
\[
(\ell_{G(i)^{\rm univ}})_a=[\CL(i)_{a-1}^{\otimes p}\xrightarrow{\delta_{a-1}} \CL(i)_a].
\]
Recall that $\ell_{G(i)^{\rm univ}}$ is computed in the proof of \cite[Lem. 4.1]{SLiu} by constructing a specific regular (i.e. complete intersection) embedding $\iota: G=G(i)^{\rm univ}\hookrightarrow X$ over $S=\CA^{\rm flat}_{K^p, O}$. This gives a complex
\[
[e^*\CI/\CI^2\xrightarrow{\delta} e^*\iota^*\Omega^1_{X/S}]
\]
and an exact sequence
\[
e^*\CI/\CI^2\xrightarrow{\delta} e^*\iota^*\Omega^1_{X/S}\to e^*\Omega^1_{G/S}\to 0.
\]
In our situation, the exactness extends on the left, i.e. the first arrow $\delta$ is injective. Indeed, both $e^*\CI/\CI^2$ and
$e^*\iota^*\Omega^1_{X/S}$  are locally free over $S$ and $\delta$ is an isomorphism over the 
generic fiber $S[1/p]$ since $G[1/p]\to S[1/p]$ is finite \'etale. Since $S=\CA^{\rm flat}_{K^p, O}$ is flat over $\BZ_p$
the injectivity of $\delta$ follows. Hence, $\ell_{G(i)^{\rm univ}}$ is given by the sheaf $\omega_G:=e^*\Omega^1_{G/S}$ 
placed in degree $0$. By taking isotypic components, we obtain
the following isomorphism of objects in ${\rm D}^{[-1,0]}(\CO_S)$:
\[
[\CL(i)_{a-1}^{\otimes p}\xrightarrow{\delta_{a-1}} \CL(i)_a] \xrightarrow{\sim} (\omega_{G(i)^{\rm univ}} )_a.
\]
On the other hand,  we have the exact sequence 
\[
0\to \omega_{A^{\rm univ}_i/S}\xrightarrow{\phi^*}\omega_{A^{\rm univ}_{i-1}/S}\to \omega_{H(i)^{\rm univ}}\to 0
\]
obtained from the isogeny $\phi_i: A^{\rm univ}_{i-1}\to A_i^{\rm univ}$. The exactness of this sequence is clear everywhere except on the left where it follows by an argument as above
using the flatness of $S=\CA^{\rm flat}_{K^p, O}$. By applying idempotents and then taking isotypic components this gives the exact sequence 
\[
0\to \Omega_{i,a} \xrightarrow{\phi^*_i} \Omega_{i-1,a}\to (\omega_{G(i)^{\rm univ}})_a\to 0.
\]
 This yields another isomorphism of objects in ${\rm D}^{[-1,0]}(\CO_S)$:
\[
[\Omega_{i,a} \xrightarrow{\phi^*_i} \Omega_{i-1,a}] \xrightarrow{\sim} (\omega_{G(i)^{\rm univ}})_a.
\]
Composing gives
\[
[\CL(i)_{a-1}^{\otimes p}\xrightarrow{\delta_{a-1}} \CL(i)_a] \xrightarrow{\sim} [\Omega_{i,a} \xrightarrow{\phi^*_i} \Omega_{i-1,a}]. 
\]
Taking the determinant (\cite{KnM}, see also \cite[5A]{SLiu})  yields the isomorphism  of pairs consisting of a line bundle and a global section
\[
\big(\CL(i)_{a-1}^{\otimes{-p}}\otimes\CL(i)_a, \delta(i)_{a-1}\big)\xrightarrow{\sim} \big(\det(\Omega_{i-1,a})\otimes\det(\Omega_{i,a})^{\otimes{-1}}, \det(\phi^*_i)\big).
\]
This completes the proof.
\end{proof}

\begin{remark}\label{rem:SLiuDerived}
The above generalizes \cite[Cor. 5.5]{SLiu}. Note that the isomorphism   in the statement of Proposition \ref{prop:Fitting} is uniquely determined. 
\end{remark}

We can now consider the Shimura variety ${\rm Sh}_{\RK_0}(\RG^\circ, \RX)_{E^\circ}$. This is an  open and closed subscheme  of $\CA_{K^p}\otimes_{O_E}E^\circ$ (for the compatible choice of $K^p$ and the prime $v$ of $\RE$ below our implicit choice of prime of $E^\circ$).  
Note that, by our assumptions, $\CG^\circ$ is an Iwahori group scheme 
for $G^\circ$. There is a $\CG^\circ$-equivariant morphism 
\[
\Mloc_{\CG^\circ,\mu}\to \Mloc^{\rm naive}_{\CG,\mu}\otimes_{O_E}O_{E^\circ},
\]
see \cite[p. 215-217]{PZ}. By construction,  the \emph{canonical integral model} $\CS_{K_0K^p }$ in the sense of  \cite{PRg} is the normalization of the closure of ${\rm Sh}_{\RK_0}(\RG^\circ, \RX) $ in $\CA^{\rm flat}_{K^p }\otimes_{O_E}O_{E^\circ}$ and there is a commutative diagram
   \begin{equation}\label{2LM}
\begin{aligned}
 \xymatrix{
      \CS_{K_0K^p} \ar[r]^{\varphi} \ar[d]  &    [\CG^\circ\bs \Mloc_{\CG,\mu}]\ar[d]  \\
       \CA_{K^p}\ar[r]^{\varphi^{\rm naive}} &  [\CG\bs \Mloc^{\rm naive}_{\CG,\mu}],
        }
        \end{aligned}
\end{equation}
where the left vertical morphism factors through $\CA^{\rm flat}_{K^p }$. Under our assumptions,  $\varphi$ is smooth. The 
  $T'_{\CG}(\BF_p)$-torsor   
  \[
 \pi': {\rm Sh}_{\RK'_1 }(\RG^\circ , \RX)_{E^\circ}\to {\rm Sh}_{\RK_0}(\RG^\circ , \RX)_{E^\circ}
 \]
  is obtained by base changing the 
  $T'_{\CG }(\BF_p)$-torsor  $\CA'_{K^p, 1,  O}[1/p]\to \CA_{K^p,O}[1/p]$ of (\ref{pi1A}) above to that open and closed subscheme.
  Here we recall $T'_{\CG }(\BF_p)$ from \eqref{dsTorus} and write  $\RK'_1=K_1'K^p$ where  $K'_1\supset K_1$  is the kernel 
   \[
  K'_1:=\ker(K_0=\CG (\BZ_p)\to T'_{\CG }(\BF_p)).
  \]
 We obtain universal isogenies, group schemes $G(i)^{\rm univ}$   and data 
  $(\CL(i)_a, \gamma_a(i), \delta_a(i))_{a\in \BZ/d\BZ}$ over $\CS_{K_0, O}$ by pulling back along $\CS_{K_0,O}\to \CA^{\rm flat}_{K^p }$. (We do not change notation.) 
 After pulling back the isomorphism of  Proposition \ref{prop:Fitting} we obtain an isomorphism of line bundles over $\CS_{K_0,O}$
 \begin{equation}\label{Fitt2}
 \CM^{-1}_{L^*(\eta_{i,a})} =\CL(i)_{a-1}^{\otimes -p}\otimes \CL(i)_{a} \xrightarrow{\sim} \det(\Omega_{i-1,a})\otimes \det(\Omega_{i,a} )^{\otimes -1} ,
  \end{equation}
 which takes the section $\delta(i)_{a-1}$ to the section given by $\det(\phi^*_i)$.

Recall the smooth morphism $\varphi: \CS_{K_0,O}\to [\CG\bs \Mloc_{\CG,\mu}]$ and the invertible sheaf $\CL_{\eta_{i,a}}$ over the base change of the local model $\Mloc_{\CG,\mu}\otimes_{O_E}O_EO$. Recall also that our goal is to produce isomorphisms over $\CS_{K_0,O}$ 
  \[
  \big(\CM_{L^*(\eta_{i,a})}, \CO_{{\rm Sh}_{K_0, F}}\xrightarrow{\sim} \CM_{L^*(\eta_{i,a})}[1/p]\big)\xrightarrow{\sim} 
 \varphi^*\big(\CL^{-1}_{\eta_{i,a}}, s^{-1}_{\eta_{i,a}}: \CO_{\Mloc_{\CG,\mu, O}}[1/p]\xrightarrow{\sim} \CL^{-1}_{\eta_{i,a}}[1/p]\big), 
  \]
 as required in (\ref{data1}). In view of Proposition \ref{prop:Fitting}, it will be enough to show
 
\begin{proposition}\label{prop:LBLM}
There is an isomorphism 
  \begin{equation}\label{HodgeIso}
 \det(\Omega_{i-1,a})\otimes \det(\Omega_{i,a} )^{\otimes -1}\xrightarrow{\sim} \varphi^*\CL_{\eta_{i,a}}
 \end{equation}
of invertible sheaves over $\CS_{K_0,O}$, which maps the  section $\det(\phi^*_i)$   to the section $s_{\eta_{i,a}}$ of $\CL_{\eta_{i,a}}[1/p]$ given in  
 \S \ref{ss:BundlesLines}. 
 \end{proposition}
  
   \begin{proof}
   We start by discussing the ``naive" local models $\Mloc^{\rm naive}_{\GL_{O_B}(\mathscr L),\mu}$ for the group ${\rm Aut}_B(V)$ and an Iwahori (i.e. complete) lattice chain  $\mathscr  L=\{\Lambda_i\}_{i\in \BZ}$ of $O_B=\RM_{n\times n}(O_D)$-modules.  By \cite[Def. 3.27]{RZbook}, this is the scheme whose $S$-valued points parametrizes compatible periodic sequences $(\CF_i)_{i\in \BZ}$ of  $\RM_{n\times n}(O_D)\otimes \CO_S$-submodules
   \[
   \begin{matrix}
   \CF_{i-1}&\subset &\Lambda_{i-1}\otimes_{\BZ_p}\CO_S\\
   \downarrow && \downarrow\\
    \CF_{i}&\subset &\Lambda_{i}\otimes_{\BZ_p}\CO_S,
   \end{matrix}
   \]
   such that the $\CF_i$ are locally direct summands of $\Lambda_i\otimes_{\BZ_p}\CO_S$ as $\CO_S$-modules, have
    fixed rank and fixed 
    $\RM_{n\times n}(O_D)\otimes_{\BZ_p} \CO_S$-representation type,  depending on the choice of $\mu$. 
    Set 
    \[
    \CQ_i:=(\Lambda_i\otimes_{\BZ_p}\CO_S)/\CF_i
    \]
    for the quotient locally free $\CO_S$-modules.
    In this case, the naive local model is the canonical local model, i.e.,
     \[
    \Mloc^{\rm naive}_{\GL_{O_B}(\mathscr L),\mu}=\Mloc_{\GL_{O_B}(\mathscr L),\mu},
    \]
     by \cite[\S 7.2, \S 8.2]{PZ}; this uses a straightforward extension of the result of G\"ortz \cite{Goertz}.

    The definition of the local model diagram morphism $\varphi^{\rm naive}$ gives (essentially tautologically)  canonical $\RM_{n\times n}(O_D)\otimes \CO_S$-isomorphisms
    \[
   \Lie(A_i/S)= \omega^\vee_{A_i/S}\simeq (\varphi^{\rm naive})^*(\CQ_i).
    \]
    These take $\phi_{i,*}:  \Lie(A_{i-1}/S)\to  \Lie(A_{i}/S)$ to the pull-back by $\varphi^{\rm naive}$ of 
    $\CQ_{i-1}\to \CQ_i$ in the above diagram.  
    
    We can also consider the $O_D\otimes_{\BZ_p}\CO_S$-modules $ e_{11}\CQ_{i}$, and, if $S$ is an $O$-scheme, their isotypic components $\CQ_{i, a}:=(e_{11}\CQ_i)_a$, for each $a\in \BZ/d\BZ$. By taking idempotents and decomposing into isotypic components the above isomorphisms give
    \[
       \Omega_{i, a}^\vee = (e_{11}\Lie(A_i/S))_a\simeq (\varphi^{\rm naive})^* \CQ_{i,a}.
    \]
   By the above, we obtain canonical isomorphisms
   \[
  \det(\Omega_{i,a}^\vee)\otimes \det(\Omega^\vee_{i-1,a} )^{\otimes -1}  \simeq (\varphi^{\rm naive})^*\big(\det(\CQ_{i,a})\otimes \det(\CQ_{i-1,a})^{-1}\big).
   \]
   Note that there is a canonical isomorphism
   \[
 \det(\Omega_{i-1,a})\otimes \det(\Omega_{i,a} )^{\otimes -1}    \xrightarrow{\sim}   \det(\Omega_{i,a}^\vee)\otimes \det(\Omega^\vee_{i-1,a} )^{\otimes -1},
  \]
  taking $\det(\phi_i^*)$ to $\det(\phi_{i,*})$. Here, $\det(\phi_{i,*})$ is induced by $\phi_{i,*}:  \Lie(A_{i-1}/S)\to  \Lie(A_{i}/S)$
  after taking isotypic components and applying the determinant. By composing, we obtain 
     \[
 \det(\Omega_{i-1,a})\otimes \det(\Omega_{i,a} )^{\otimes -1}   \xrightarrow{\sim}  (\varphi^{\rm naive})^*\big(\det(\CQ_{i,a})\otimes \det(\CQ_{i-1,a})^{-1}\big).
   \]
This takes  $\det(\phi_i^*)$ to the pull-back by $(\varphi^{\rm naive})^*$ of the section of $\det(\CQ_{i,a})\otimes \det(\CQ_{i-1,a})^{-1}$ induced by  $\CQ_{i-1,a }\to \CQ_{i,a}$.
  
 Using the commutativity of (\ref{2LM}) we see that it is enough to show that there 
 is an isomorphism
 \[
 \CL_{\eta_{i,a}}\simeq \det(\CQ_{i,a})\otimes \det(\CQ_{i-1,a})^{-1}
 \]
 which takes the section $s_{\eta_{i,a}}$ of  $\CL_{\eta_{i,a}}[1/p]$  in \S \ref{ss:BundlesLines} to the   section of  $\det(\CQ_{i,a})\otimes \det(\CQ_{i-1,a})^{-1}$
  induced by  $\CQ_{i-1,a }\to \CQ_{i,a}$. Given the comparison between the naive local model $\Mloc^{\rm naive}_{\GL_{O_B}(\mathscr L),\mu}$ and the local model of \cite{PZ} discussed above, this follows by a straightforward extension of 
  the argument in Example \ref{sss:ExampleLoop}.
   \end{proof}

 The composition of (\ref{Fitt2}) and (\ref{HodgeIso})   gives the second isomorphism in (\ref{data1}). 
 
 This completes the construction in case (\ref{star1}). 
 
 In the case (\ref{star2}), we also have to make a construction for the similitude character $c$: For this we set $\CM_c=\CO_{\CS_{K_0}}$ to be the trivial line bundle with section $\CO_{\CS_{K_0}}\xrightarrow{} \CM_c^{\otimes (1-p)}=\CO_{\CS_{K_0}}$ 
 given by multiplication by $w_p=p\cdot u$, $u\in \BZ_p^*$. In this case, $\CL_c=\CO_{\Mloc_{\CG,\mu}}$ and since 
 $\langle \mu, c\rangle=1$, we have $s_c=p\cdot v$, $v\in \BZ_p^*$. The isomorphism $\CO_{\CS_{K_0}}=\CM_c^{\otimes (1-p)}\simeq \varphi^*\CL_c=\CO_{\CS_{K_0}}$ is given by multiplication by $v/u$.
 
We have now provided data as in (\ref{data1}) over the unramified Galois extension $O/\BZ_p$ which splits the torus $T_\CG$. The data (\ref{data1}) that we have constructed  support descent data for the finite unramified Galois extension $OO_E/O_E$ since 
the isogenies and group schemes are all defined over $O_E$.  This  shows Conjecture \ref{conj:torsors}. Hence, by Proposition \ref{prop:Conj}, Conjecture \ref{globconj} follows with the integral model $\CS_{K_1, S}$ as defined in Definition \ref{def:generalconstruction2}. This concludes the proof of Theorem \ref{thm:PEL} in the case that $F_0$ is a field.  The  case of a general $F_0$ can now be reduced to this case by taking products.
 \end{proof}

\section{Some examples}\label{s:PELex}

Here we describe more explicitly some cases of Theorem  \ref{thm:PEL} and a few variations. We relate the integral models $\CS_{K_1,S}$ to moduli schemes which parametrize Raynaud generators of the kernels of the universal isogenies.

   \subsection{The Siegel case} \label{ss:Siegel}
    Here we discuss the basic example of the Siegel Shimura variety with level the pro-unipotent radical of an Iwahori subgroup. This was considered in unpublished notes of Haines-Stroh, see also \cite{HLS}  and \cite{Shadrach}. 
    
  We take $\RB=\BQ$, $\RV=\BQ^{2g}=\oplus_{i=1}^{2g}\BQ\cdot e_i$ with alternating form $(\ ,\ )$ determined by $(e_i, e_{2g+1-j})=\delta_{ij}$, $*={\rm id}$. Then $\RG={\rm GSp}_{2g}(\BQ)$. Let $h$ be the standard Deligne cocharacter corresponding to the (upper and lower) Siegel domain $S_g^{\pm}$.  We take $\mathscr L$ to be the standard self-dual {\rm complete} periodic $\BZ_p$-lattice chain in $V=\RV\otimes\BQ_p$ 
  \[
\cdots \subset \Lambda_{-g}\subset\cdots \subset \Lambda_{-1}\subset  \Lambda_0\subset \Lambda_1\subset\cdots\subset \Lambda_g\subset\cdots
  \]
  determined by
  \[
  \Lambda_0=(e_1,\ldots, e_{2g}),\quad 
  \Lambda_{-i}=(pe_1,\ldots, pe_i, e_{i+1},\ldots, e_{2g}),\quad 1\leq i\leq g,
  \]
  \[
 \Lambda_{i} =\Lambda_{-i}^\vee =(e_1,\ldots, e_{2g-i}, p^{-1}e_{2g+1-i},\ldots,  p^{-1}e_{2g}), \quad 1\leq i\leq g.
  \]
   Then (IW) and (S) are satisfied: Condition (IW) is obvious and $K_0=\CG(\BZ_p)$ is the ``standard" Iwahori subgroup. 
   For   condition (S), recall that we have
   \[
   T_\CG=\{\diag(r_1,\ldots, r_g, cr_g^{-1},\ldots, cr^{-1}_1)\}\simeq \BG_m^{g+1}=\{(r_1,\ldots, r_g, c)\},
   \]
   with projection to the last coordinate $c$ giving the similitude character.
   Hence, (\ref{star2}) is satisfied with $J=\{-g+1,\ldots, 0\}\subset \{-g+1,\ldots, g\}$, i.e. from the direct summands given by the $2g$ graded pieces ${\rm Aut}_{\BF_p}(\Lambda_i/\Lambda_{i-1})=\BG_m$, $-g+1\leq i\leq g$, we pick the first half, see also \S \ref{sss:Sprime}. Note that here $T_\CG=\BG_m^{g+1}$ is split, so $O=\BZ_p$ and the Lang isogeny is given by raising all coordinates to the $(p-1)$-st power.

  Now we continue with notations of \S \ref{ss:Exam} (5). The semigroup $S=S_{\CG,\mu}\subset X^*(T_\CG)=\BZ^{g+1}$ is generated by $e_i=\eta_i$ and $f_i=c-\eta_i$, $1\leq i\leq g$, with relations $e_i+f_i=e_j+f_j$, for all $i$, $j$, where $c=\eta_{g+1}\in X^*(T_\CG)$ is the similitude character.
 This presentation of the semi-group allows us to give a more compact explanation and description of the functor describing $\CS_{K_1}$ by Proposition \ref{def:generalconstruction}, as follows.
 
As above, we let $G(i)^{\rm univ}$ be the kernel of $A^{\rm univ}_{i-1}\to A^{\rm univ}_{i}$, for $-g+1\leq i\leq 0$. For simplicity, we will sometimes omit the superscript. This is an Oort--Tate group scheme over the Iwahori level integral model $\CS_{K_0}=\CA_{K_0}$. We obtain data $(\CL(i), \gamma(i), \delta(i))$, where $\CL(i)$ is a line bundle and $\gamma(i): \CL(i)\to \CL(i)^{\otimes p}$ and  $\delta(i): \CL(i)^{\otimes p}\to \CL(i)$; we consider $\delta(i)$ as section of $\CL(i)^{\otimes(1-p)}$. We set  
\[
 \CM_{e_i}=\CM_{\eta_i}=\CL(i-g), \  a_{e_i}=\delta({i-g})^{-1},\ \hbox{\rm for $1\leq i\leq g$}, 
\]
\[
\CM_{c}=\CO_{\CS_{K_0}}, \ a_{c}=w_p.
\]
Then we also have
\[
\CM_{f_i}=\CM_{c-\eta_i}=\CM_{\eta_i}^{-1},\quad a_{f_i}=a_{c-\eta_i}=w_p\delta({i-g}), \ \hbox{\rm for $1\leq i\leq g$}.
\]
\begin{proposition} The $\CS_{K_0}$-scheme $\CS_{K_1}$ of Definition \ref{def:generalconstruction2}, parametrizes $2g+1$-tuples 
\[
(z_1,\ldots, z_g, z, z'_g, \ldots, z'_1),
\]
 with  $z_i\in \Gamma(\CS_{K_0}, \CM^{-1}_{e_i})$, $z'_i\in \Gamma(\CS_{K_0}, \CM_{e_i})$, 
for $1\leq i\leq g$, and $z=z_c\in  
\Gamma(\CS_{K_0}, \CO_{\CS_{K_0}})$, which satisfy
\[
z_{i}^{\otimes(p-1)}=a_{e_i},  \  \ z^{p-1}=w_p,  \  \ {z'_i}^{\otimes(p-1)}=w_pa_{e_i}^{-1}, \ \ z_iz'_i=z.
\]
\end{proposition}

\begin{proof} We start by applying Proposition \ref{def:generalconstruction} which gives an explicit description of the functor of $\CS_{K_1}$.
Set $z_{e_i}=z_i$, $z_{f_i}=z'_i$, $z_c=z$. Given the equations $z_iz'_i=z=z_jz_j$, the presentation of the semigroup $S=S_{\CG,\mu}$ given above shows that we can define  $z_\chi\in \Gamma(\CS_{K_0}, \CM^{-1}_{\chi})$,  for all $\chi\in S$, extending these values on the generators. This implies that a $2g+1$-tuple as above determines and is determined by a unique point of $\CS_{K_1}$.
\end{proof}

The above proposition shows that   $\CS_{K_1}$ agrees with the integral model considered in Haines-Li-Stroh \cite{HLS}. 
It also implies that, in this case, $\CS_{K_1}$ has a rather simple moduli interpretation: 
 
\begin{corollary}\label{sieOT}
The $\CS_{K_0}$-scheme $\CS_{K_1}$ parametrizes Oort--Tate generators 
of the $2g$  group schemes $G(i)=\ker(A_{i-1}\to A_i)$, for $-g+1\leq i\leq g$, which are ``compatible with duality" (see below).
\end{corollary}
 \begin{proof}

 Note that by definition, $z_i \in \Gamma(\CS_{K_0},  \CM^{-1}_{e_i} )$ with $z_i^{\otimes(p-1)}=a_{e_i}$ is an Oort--Tate generator of $G(i-g)$, for all $1\leq i\leq g$, see \S\ref{sss:Generators}.
Similarly $z'_i\in \Gamma(\CS_{K_0}, \CM_{e_i})$ with ${z'_i}^{\otimes(p-1)}=w_pa_{e_i}^{-1}$ is an Oort--Tate generator of the Cartier dual of $G(i-g)$, for all $1\leq i\leq g$. The polarization gives  isomorphisms between $G(1-i)$ and the Cartier dual $G(i)^\vee$ of $G(i)$, for each $-g+1\leq i\leq g$. Hence, we can 
consider  $z'_i$ as an Oort--Tate generator of   $G(1-(i-g))=G(g+1-i)$, for $1\leq i\leq g$. ``Compatibility with duality" means that the equations $z_iz'_i=z_jz'_j$ as sections of $\CO_S$  are satisfied  for all $1\leq i, j\leq g$. In terms of Oort--Tate generators, this means that the image of $(z_i, z_i')$ under the canonical map
\[
G(i-g)\times G(i-g)^\vee\to \CO_S
\]
is independent of $i$. 
\end{proof}

\begin{remark}\label{rem:badSiegel}
The morphism $\CS_{K_1}\to \CS_{K_0}$ is finite but not flat if $g>1$. This follows from Corollary \ref{cor:NonFlat}. In addition, the scheme $\CS_{K_1}$ is not flat over $\BZ_p$ and hence not normal. This  failure of flatness of $\CS_{K_1}$ over $\BZ_p$ was first observed   by computer calculations of the first named author and then shown by Marazza \cite{Marazza} in general. In addition, the following conjectural statements are supported by 
computer calculations (assume $g>1$): The special fiber of $\CS_{K_1}\to \Spec(\BZ_p)$ has embedded primes of codimension $1$. The special fiber of the flat closure $\CS_{K_1}^{\rm flat}$ of $\CS_{K_1}$ also has embedded primes of codimension $1$. In particular, $\CS_{K_1}^{\rm flat}$ is also not normal. Of course, all these statements are obtained by examining the corresponding root stacks $\Mloc_{\CG,\mu}^{\sqrt{S}}$ for $S=S_{\CG,\mu}$. Indeed, $\Mloc_{\CG,\mu}^{\sqrt{S}}$ give local models for $\CS_{K_1,S}$ by   (\ref{eq:sh2}). 
\end{remark}

    \subsection{The Hilbert-Siegel case}\label{ss:HilbertSiegel}

 Next we turn to a ``Hilbert-Siegel" case. 
 Let $\RF$ be a totally real field of degree $d$ with integers $O_\RF$ and let $\RV=\oplus_{i=1}^{2g}\RF e_i$ be a $2g$-dimensional $\RF$-vector space
 with the alternating $\RF$-bilinear form determined by $(e_i, e_{2g+1-i})=\delta_{ij}$. Let $\RG'={\rm Res}_{\RF/\BQ}{\rm GSp}(V, h)$ and let $\RG$ the subgroup of  $\RG'$ given by elements with similitude in $\BQ^*$. We take $X$ to be the product of $d$ copies of 
 the (upper and lower) Siegel domain $S^{\pm}_g$ on which $\RG(\BR)$ acts. Then $(\RG, X)$ is the Hilbert-Siegel Shimura datum with reflex field $\RE=\BQ$.  
 
 Let $p$ be a rational prime which is inert  in $\RF$ (remains prime), and write  $F=\RF_p$ for the completion. 
 We let $K_0$ be the standard Iwahori subgroup of $G(\BQ_p)$ stabilizing the $O_F$-lattice chain $\{\Lambda_i\}_{i\in \BZ}$ 
 given as in \S \ref{ss:Siegel} and $K_1$ its pro-unipotent radical. In this case, we have
 \[
 T_\CG=({\rm Res}_{O/\BZ_p}\BG_m)^g\times \BG_m,
 \]
 where $O=O_F=W(\BF_{p^d})$. 
 Choose $F\hookrightarrow \ov\BQ_p$ so we can identify the set $\{\phi\}$ of $\BQ_p$-embeddings $\phi: F\to \ov\BQ_p$ with $\Gal(F/\BQ_p)\simeq \Gal(\BF_{p^d}/\BF_p)\simeq \BZ/d\BZ$. 
 Then 
 \[
 X^*(T_\CG)=(\bigoplus_{a\in \BZ/d\BZ} \BZ)^g\oplus \BZ=\bigoplus_{i=1}^g \BZ[\Gal(\BF_{p^d}/\BF_p)]\oplus \BZ\cdot c,
 \]
 which has natural basis $\eta_{i, a}$, for $1\leq i\leq g$, $a\in \BZ/d\BZ$, and $c$ corresponding to the generator of the last factor coming from the similitude. 
 
 We can again see that (IW) and (S) are satisfied. In fact, here (\ref{star2}) holds.
 
 We now discuss the corresponding toric schemes: A calculation based on \S \ref{ss:Exam}, Case (5) shows that for $S=S_{\CG,\mu}$, the base change $Y_S\otimes_{\BZ_p}O$ is the spectrum of the quotient of the polynomial algebra in $2gd$ variables
\[ 
R=O[ e_{i, a},      f_{i,a}]_{ 1\leq i\leq g, a\in \BZ/d\BZ}
\]
by the ideal generated by  
\[
e_{i, a}f_{i, a}-e_{j, a}f_{j, a},\quad e_{i, a}f_{i, a}-e_{i, a'}f_{i, a'}, \quad \forall\ i, j  ,\ \forall\ a, a' .
\]
Here the generators $e_{i,a}$ in $S\subset X^*(T_\CG)$ correspond to the characters $\eta_{i,a}$ and 
the generators $f_{i,a}$ to $\eta'_{i, a}:= c- \eta_{i, a}$.

Hence, $Y_S\otimes_{\BZ_p}O$ is isomorphic to the base change to $O$ of the toric scheme for $\GSp_{2dg}$ and the Iwahori subgroup, described in \S \ref{ss:Exam}, Case 5).  

The definition of the scheme $\CS_{K_1,S, O}$ involves the semigroup $\wt S\subset X^*(T_\CG)$ giving the normalization $L_{O}: \wt Y_S\otimes_{\BZ_p}O\to Y_S\otimes_{\BZ_p}O$ via $\wt Y_S\otimes_{\BZ_p}O=\Spec(O[\wt S])$. In general, it seems difficult to obtain a simple   presentation for $\wt S$. We can check that the semigroup $\wt S$ giving  $\wt Y_S\otimes_{\BZ_p}O=\Spec(O[\wt S])$  contains  $e_{i, a}$ and $f_{i, a}= c\cdot e_{i, a}^{-1}$, for all $i$ and $a$. Indeed, the characters  $ \eta_{i, a} $ and $\eta'_{i, a} $ belong to the cone $L^*( \sigma^\vee_{\CG,\mu})\subset X^*(T_\CG)_\BR$.
(However, computer calculations indicate that $ \eta_{i, a} $, $\eta'_{i, a} $ and $L^*(S)$, are not enough to generate $\wt S$, unless $d=1$.) 
Hence, an $U$-valued point of $\CS_{K_1,S, O}$ over $S_{K_0,O}$ gives, in particular, sections $z_{i,a}:=z_{e_{i, a}}$,
 $z'_{i,a}:=z_{f_{i, a}}$, $z:=z_c$, of the pull-backs of  $\CM^{-1}_{e_{i, a}}$, $\CM^{-1}_{f_{i, a}}$, $\CM^{-1}_c=\CO_{\CS_{K_0}}$, by $U\to \CS_{K_0,O}$. Note that we have 
 global sections $a_{i,a}:=a_{e_{i,a}}$ of $\CM^{-1}_{L^*(e_{i,a})}$ given as $\delta(i)_{a-1}$ of the construction in the proof. 
 We also have $a'_{i,a}:=a_{f_{i,a}}=w a_{i,a}^{-1}$.
 Condition (i) in Proposition \ref{def:generalconstruction} for $\chi=\eta_{i, a}$ and $\chi=\sum_{a}\eta_{i,a}$  gives, respectively,
 \[
 z_{ i, a-1}^p=a_{i,a}z_{ i, a},\qquad 
 (\prod_{a\in \BZ/d\BZ}z_{i,a})^{p-1}=\prod_{a\in \BZ/d\BZ}a_{i,a}.
 \]
 Similarly, the condition (i) for $\chi=\eta'_{i, a}$ and $\chi=\sum_{a}\eta'_{i, a}$ gives
 \[
{ z'}_{ i, a-1}^p=a'_{i,a}z'_{ i, a}, \qquad (\prod_{a\in \BZ/d\BZ}z'_{i,a})^{p-1}=\prod_{a\in \BZ/d\BZ}a'_{i,a}.
 \]
 We also have 
 \[
 z_{i,a}z'_{i,a}=z,\quad z^{p-1}=w_p,
 \]
 for all $i$, $a$.
 
 These conditions imply that 
$
z_i:=(z_{i, a})_{a\in \BZ/d\BZ}
$
gives a   generator of the Raynaud group scheme $G(i-g)$ and that
$
z'_i:=(z'_{i, a})_{a\in \BZ/d\BZ}
$
gives a  generator of the Raynaud group scheme $G(g+1-i)$, for all $1\leq i\leq g$. The equations  $z_{i,a}z'_{i,a}=z_{j,a}z'_{j,a}$
and $z_{i,a}z'_{i,a}=z_{i,a'}z'_{i,a'}$ again reflect the duality relations. 

\begin{remark}\label{rem:HilbertSiegel}
1) In this case, we can also consider the $\CS_{K_0, O}$-scheme $\CS'_{K_1, O}$ parametrizing generators $z_i$ of all the Raynaud group schemes $G(i)$, for $-g+1\leq i\leq g$, which are compatible with duality in the above sense. This descends to $\CS'_{K_1}$ over $\BZ_p$ and our discussion shows that there is a natural morphism
\[
\CS_{K_1}\to \CS'_{K_1}
\]
over $\CS_{K_0}$. This is an isomorphism when $d=1$. When $d>1$, our calculations suggest that it is not an isomorphism, although the morphism induces an isomorphism over the generic fibers for all $d\geq 1$. The scheme $\CS'_{K_1}$ can also be viewed as obtained 
by the variant of our construction mentioned in Remark \ref{rem:NonNormalvariant}, which uses, instead of $\wt S$, the non-saturated sub-semigroup $\wt S'$ of $\wt S$ which is generated by $L^*(S)$ together with $\eta_{i,a}$ and $\eta'_{i,a}$. Such a variant involves using non-normal torus embeddings and falls outside our basic framework.

2) The morphisms $\wt Y_S\to Y_S$ and $\CS_{K_1, S}\to \CS_{K_0}$  are finite. They are never flat in the Hilbert-Siegel case, unless $g=d=1$.
This follows from work of Altmann \cite{Alt}, see Remark \ref{rem:Altmann}. Calculations show that the same failure of flatness appears to hold for $\CS'_{K_1}\to \CS_{K_0}$ as given in (1) above. In the Siegel case $d=1$,  $\CS_{K_1}= \CS'_{K_1} $ and this non-flatness follows from Corollary \ref{cor:NonFlat}. 

3) Contrary to what is stated in \cite[Rem. 3.10]{SLiu}, the morphism $Y_1\to Y_0$ is not flat (notations of loc. cit.) unless $g=1$.
This already happens for $\RF=\BQ$, i.e. for the Siegel moduli stack, see Remark \ref{rem:badSiegel}.  
 In fact, computer calculations show that the integral model $Y_1$ of \cite{SLiu} is also not flat over $O$. (For $d=1$ this non-flatness was also shown by Marazza \cite{Marazza}.) So Liu's notion of ``Raynaud generator" seems to give no improvement on flatness.
 \end{remark}

  \subsubsection{A larger level subgroup}\label{sss:HSsmaller}
 We continue to assume  that $p$ is inert in the totally real field $\RF$ of degree $d$. 
Instead of the pro-unipotent $K_1$ of the standard Iwahori $K_0=\CG(\BZ_p)$, let us consider the slightly larger subgroup $K$ of $K_0$ given by the inverse image of the subgroup
\[
\{(1,\ldots, 1,\lambda,\ldots, \lambda)\ |\ \lambda\in \BF_p^*\}
\]
of $T_\CG(\BF_p)$ under $K_0=\CG(\BZ_p)\to (\CG\otimes\BF_p)_{\red,\rm ab}=T_\CG(\BF_p)$.  The quotient $K_0/K$ is isomorphic to $\prod_{i=1}^g\BF_{p^d}^*=T(\BF_p)$, where $T=\Res_{O/\BZ_p}\BG_m^g$, while $K/K_1\simeq \BF_p^*$. Note that $T$ can be identified with the direct factor torus $T'_\CG$ of (\ref{dsTorus}).  

Now we can consider a mild variant of our constructions in which $T_\CG$ is replaced by the quotient torus $T$, as also explained 
in Remark \ref{rem:variant}: The image of the Weyl orbit of the Shimura cocharacter $\mu$ under $X_*(T_\CG)\to X_*(T)=\oplus_{i=1}^{dg}\BZ$,
is the set of the $2^{dg}$ cocharacters $x\mapsto (a_1,\ldots, a_{dg})$, where each $a_i$ is $0$ or $1$. The toric scheme $Y_S$ is $\Res_{O/\BZ_p}\BA^g_O$ and we have 
\[
Y_S\otimes O=\Spec(O[e_{i,a}]_{1\leq i\leq g, a\in \BZ/d\BZ}).
\]
An argument as in Example \ref{ex:ramify} shows that the normalization of the Lang cover $\wt Y_S\otimes O$ is given by the spectrum of 
\[
 O[e_{i,a}, u_{i,a}] /((u_{i,a-1}^p-e_{i, a}u_{i, a})_{i, a}, ((\prod_{a} u_{i,a})^{p-1}-\prod_{a} e_{i,a})_i).
\]
In this case,  $L: \wt Y_S\to Y_S$ is flat by miracle flatness.  The construction in the proof of Theorem \ref{thm:PEL} again applies to give a corresponding integral model $\CS_K$ over $\CS_{K_0}$ which satisfies Conjecture \ref{globconj} for the quotient $T$. Here,   $\CS_{K, O}$  
 is the moduli scheme classifying generators of the kernels of the universal $O_\RF$-linear isogenies $A_{i-1}\to A_i$, for $-g+1\leq i\leq 0$,
 over $\CS_{K_0, O}$; these kernels are Raynaud group schemes. In this case, $\CS_{K}\to \CS_{K_0}$ is finite and flat. Hence, $\CS_K$ is flat over 
 $\BZ_p$. 
 
 \begin{remark}\label{remarkHB}
 a) When $g=1$ (the Hilbert-Blumenthal case), the above moduli scheme and corresponding integral model $\CS_{K}$ was constructed in \cite{P95}.
 
 b) When $d=1$ (the Siegel case) and $g>1$, the above moduli scheme and corresponding integral model $\CS_{K}$ was constructed and discussed in \cite[\S 4.1]{Shadrach}. Shadrach proceeds to construct a resolution of $\CS_K$ when $g=2$.
\end{remark}

\subsection{Unitary groups: The unramified (quasi-split)  case}

Take $B=\RK\subset \BC$, an imaginary quadratic number field, $\RV=\oplus_{i=1}^n \RK\cdot e_i$, $*=$ complex conjugation. 
Fix an element $\alpha\in \RK$ such that $\bar\alpha=-\alpha$. Let $\psi: \RV\times \RV\to \RK$ be a non-degenerate hermitian form and take $(v, w)={\rm Tr}_{\RK/\BQ}(\alpha \psi(v,w))$ for the corresponding perfect alternating $\BQ$-bilinear form. Assume that the $\BC$-hermitian form $\psi\otimes_\RK\BC$ has signature $(r,s)$ with $1\leq r,s\leq n$ and fix a basis of $\RV\otimes_\RK\BC$ for which the matrix of $\psi\otimes_\RK\BC$ is $\diag( 1^{(r)}, (-1)^{(s)})$. Then $\RG$ is a unitary similitude group:
\[
\RG(R)=\{g\in \GL(\RV\otimes_\BQ R)\ |\ \psi(gv,gw)=c(g)\psi(v,w), c(g)\in R^*\},
\]
with $\RG_\BR\simeq {\rm GU}(r,s)$, and we  take the corresponding cocharacter $h: \BS\to {\rm GU}(r,s)$ as in \S \ref{sss:PELdata}  given by $h(\sqrt{-1})=\diag(\sqrt{-1}^{(r)}, -\sqrt{-1}^{(s)})$.  

Now suppose $p$ is an odd prime which is unramified in $\RK$ and assume $\alpha$ is a unit at $p$.

 \subsubsection{Split primes}\label{sss:split} Assume that $p$ splits in $\RK$ so that $\RK\otimes_\BQ\BQ_p=\BQ_p\times \BQ_p$, $V=\RV\otimes_\BQ\BQ_p=W\oplus W^\vee$. Then we are in case (I) of \S\ref{sss:prelim} and $G=\RG\otimes_\BQ\BQ_p=\GL(W) \times \BG_m$. Choosing a complete periodic $\BZ_p$-lattice chain in $W$ determines an Iwahori  subgroup of $G(\BQ_p)$. Since (IW) and (S) are satisfied Theorem \ref{thm:PEL} applies. The relevant Shimura varieties have level the 
pro-unipotent radical of an Iwahori of $G(\BQ_p)$, i.e. the level subgroup $K_1\times \{a\in \BZ_p^*\ |\ a\equiv 1\mod p\}\subset G(\BQ_p)$
where $K_1$ is the pro-unipotent radical of an Iwahori of $\GL_n(\BQ_p)$.  We obtain integral models related to toric stacks for 
$\GL_n\times\BG_m$. 

Note that in this situation $X^*(T_\CG)^{\mu, 0}\neq 0$ and so the toric schemes have a toric factor.
Although the corresponding local models are the same, the toric schemes for $\GL_n\times\BG_m$ are  different from the corresponding toric schemes for $\GL_n$;
they are obtained by inducing the ones for $\GL_n$. (For these, see Examples \S\ref{ss:Exam} (1)-(4)). The toric schemes for $\GL_n$ directly relate   to Shimura varieties with level subgroup $K=K_1\times \BZ_p^*$, where $K_1$ is the pro-unipotent radical of an Iwahori of $\GL_n(\BQ_p)$. This follows by a straightforward extension of Theorem \ref{thm:PEL} to this situation. (Instead of taking $T_\CG\simeq \BG^n_m\times\BG_m$ we consider the factor torus $T=T'_\CG=\BG^n_m$.) 
It is instructive to explain the case $r=1$. Then there is only one choice for the semigroup $S$ and the corresponding $T$-toric scheme is $Y=\BA^n_{\BZ_p}=\Spec(\BZ_p[t_1,\ldots, t_n])$. Here we have the well-known semi-stable local model for $\GL_n$, for Iwahori level and   coweight $\mu=(1^{(1)}, 0^{(n-1)})$. So $\Mloc_{\CG,\mu}$ is  the ``Deligne scheme" of linked projective spaces $\BP(\Lambda)$ for a complete $\BZ_p$-lattice chain $\{\Lambda\}$  in $\BQ_p^n$. We can pick an open 
\[
U=\BZ_p[x_1, \ldots , x_n]/(x_1 \cdots x_n-p)\subset \Mloc_{\CG,\mu} 
\]
and choose the coordinates such that the restriction to $U$ of the divisor morphism  is   defined by the ring homomorphism
\[
\BZ_p[t_1,\ldots, t_n ] \to \BZ_p[x_1, \ldots , x_d]/(x_1\cdots x_n-p); \quad t_i\mapsto x_i.
\]
The corresponding chart $U\times_Y\wt Y$ of the stacky local model $\Mloc_{\CG,\mu}^{\sqrt{}}$ for $\CS_{K,S}$ is given, as a $\Mloc_{\CG,\mu}$-scheme, as 
\begin{equation}\label{dr-eq}
 \Spec(\BZ_p[u_1,\ldots, u_n]/(u^{p-1}_1\cdots u^{p-1}_n -p)), \quad x_i\mapsto u^{p-1}_i, \quad i=1,\ldots, n.
\end{equation} 
 
When $r>1$ there is a choice for $S$. The integral models $\CS_{K, S}$ of these Shimura varieties, when we choose the free semigroup $S$ (as in Example \S\ref{ss:Exam}, Case 4), have been studied  in \cite{HR}, \cite{HLS} and by  Shadrach \cite{Shadrach} and Marazza \cite{Marazza}. For this level $K$ and semigroup $S$, the integral model $\CS_{K, S}$ can be given 
as a moduli scheme parametrizing generators of the kernels of the universal isogenies; these are Oort--Tate group schemes. 
Apparently, the models for the choice $S=S_{\CG,\mu}$ that we regard as canonical here have not been considered before.

\subsubsection{Inert primes} Assume that $p$ remains prime in $\RK$ and set $K=\RK\otimes_\BQ\BQ_p$. Suppose that there is a $K$-basis of $V$ such that $\psi(e_i, e_{n+1-j})=\delta_{ij}$, i.e. the hermitian form $\psi$ is split at $p$. Consider the
complete standard periodic $O=O_K$-lattice chain $\mathscr L=\{\Lambda_i\}_{i\in \BZ}$ in $V=\RV\otimes_\BQ\BQ_p$ determined by
\[
\Lambda_{-i}=(pe_1,\ldots, pe_i, e_{i+1},\ldots ,e_n),
\]
for $0\leq i\leq n$, and $p^k\Lambda_i=\Lambda_{i-kn}$. This chain is also self-dual for $\psi$ and the form $(\ ,\ )$.
Condition (IW) is obviously satisfied and the group scheme $\CG={\rm Aut}_{O_K,\psi}(\mathscr L)$ has connected fibers.
Set $k=O/(p)\simeq\BF_{p^2}$.
We split into two cases:
 \medskip

\noindent\emph{(A)\ $n=2m+1$ is odd.} Then we have
\[
T_\CG=\{(a_1,\ldots, a_m, b, b\bar b\bar a_m^{-1},\ldots, b\bar b\bar a_1^{-1})\ |\ a_i, b\in O^*\}\simeq \prod_{i=1}^m \Res_{O/\BZ_p}\BG_m\times \Res_{O/\BZ_p}\BG_m.
\]
In this case,
\[
T_\CG\otimes\BF_p\hookrightarrow {\rm Aut}_{k}(\Lambda_0/\Lambda_{-1})\times\cdots \times {\rm Aut}_{k}(\Lambda_{-n+1}/\Lambda_{-n})
\]
and projecting to the first $m+1$ factors  
gives an isomorphism
\[
T_\CG\otimes\BF_p\xrightarrow{\sim} {\rm Aut}_{k}(\Lambda_0/\Lambda_{-1})\times\cdots \times {\rm Aut}_{k}(\Lambda_{-m}/\Lambda_{-m-1}).
\]
So (\ref{star1}) and hence (S) is satisfied.

\medskip

\noindent\emph{(B)\ $n=2m$ is even.} 
Then we have
\[
T_\CG=\{(a_1,\ldots, a_m, c\bar a_m^{-1},\ldots, c\bar a_1^{-1})\ |\ a_i \in O^*, c\in \BZ_p^*\}\simeq \prod_{i=1}^m \Res_{O/\BZ_p}\BG_m\times  \BG_m.
\]
Again
\[
T_\CG\otimes\BF_p\hookrightarrow {\rm Aut}_{k}(\Lambda_0/\Lambda_{-1})\times\cdots \times {\rm Aut}_{k}(\Lambda_{-n+1}/\Lambda_{-n}).
\]
Now projecting to the first $m$ factors  and taking the similitude $c$
gives an isomorphism
\[
T_\CG\otimes\BF_p\xrightarrow{\sim} {\rm Aut}_{k}(\Lambda_0/\Lambda_{-1})\times\cdots \times {\rm Aut}_{k}(\Lambda_{-m+1}/\Lambda_{-m})\times \BG_m.
\]
So (\ref{star2}) and hence (S) is satisfied.
\medskip

\subsubsection{The case $n=3$} For a  concrete example, choose $n=3$ and $(r,s)=(1,2)$  in the above. Write $X_*(T_\CG)=\BZ[\tau]\oplus \BZ[\tau]$, with $\tau$ the Galois involution.
It is convenient to use the isomorphism 
\[
\iota: X_*(T_\CG)=\BZ[\tau]\times \BZ[\tau]\xrightarrow{\sim} \BZ^3\times \BZ
\]
 given by
$
\iota (a_1+a_2\tau, b_1+b_2\tau)= (a_1, b_1, b_1+b_2-a_2, b_1+b_2).
$

The Weyl orbit of $\mu$ in $X_*(T_\CG)=\BZ^3\times\BZ$ is $(1,0,0,1)$, $(0,1,0,1)$, $(0,0,1,1)$, spanning the cone $\sigma=\sigma_{\CG,\mu}$.
We now discuss the toric schemes $Y_S$ and $\wt Y_S$ for $S=S_{\CG,\mu}$. We have
\[
Y_S=\Spec(O[t_1, t_2, t_3, \frac{t}{t_1t_2t_3}, \frac{t_1t_2t_3}{t}])\supset \BG_m^3\times\BG_m,
\]
with $(t_1,t_2,t_3,t)$ the   coordinates on $\BG_m^3\times\BG_m$. Since $Y_S$ is regular, the Lang cover $L: \wt Y_S\to Y_S$
is flat but the normalization $ \wt Y_S$ is complicated:
 
The Weyl orbit of $\mu$ in the coordinates $X_*(T_\CG)=\BZ^2\times\BZ^2=\BZ^4$ is $(1,1,0,1)$, $(0,1,0,1)$, $ (0,0,0,1)$.
Now $L_*: X_*(T_\CG)=\BZ^4\to X_*(T_\CG)=\BZ^4$ is given by multiplication by the matrix
\[
A=\begin{pmatrix}
-1&p&0&0\\
p&-1&0&0\\
0&0&-1&p\\
0&0&p&-1
\end{pmatrix}.
\]
The cone $(L_*)^{-1}(\sigma)$ is spanned by $A^{-1}\cdot (1,0,0,1)^t$, $A^{-1}\cdot (0,1,0,1)^t$, $A^{-1}\cdot (0,0,1,1)^t$. So,
the spanning rays of $(L_*)^{-1}(\sigma)$ are the rows of 
\[
\begin{pmatrix}
p+1&p&0\\
p+1&1&0\\
p&1&p\\
1&p&1
\end{pmatrix}.
\]
We have $\wt S=((L_*)^{-1}(\sigma))^\vee\cap \BZ^4$ but an explicit presentation of $\wt S$ depends on $p$ and is hard to pin down. Nevertheless, we can easily see that the unit vectors $e_1=(1,0,0,0)$, $e_2=(0,1,0,0)$, $e_3=(0,0,1,0)$, $e_4=(0,0,0,1)$ in $X^*(T_\CG)=\BZ^4$ belong
 to the dual cone $((L_*)^{-1}(\sigma))^\vee$ and so they belong to $\wt S$. These $4$ elements provide
 the coordinates of the Raynaud generators of the kernels of the two universal isogenies corresponding to $\Lambda_{-2}\to \Lambda_{-1}$ and $\Lambda_{-1}\to \Lambda_0$. However, the semigroup $\wt S$  requires more generators than $L^*(S)\cup\{e_1,e_2,e_3,e_4\}$. So, again, we do not have a straightforward moduli scheme description in terms of generators of kernels of the universal isogenies.

 \subsection{``Fake" unitary groups and the Drinfeld case}\label{ss:fak}

 Let $\RB$ be a central division algebra of degree $d^2$ over  an imaginary quadratic field $\RK$ and $*$ a positive involution (of the second kind) on $\RB$. We fix embeddings $\RK\subset \ov\BQ\subset \BC$. We also fix a left $\RB$-module $\RV$ of rank $1$ and an alternating non-degenerate $\BQ$-bilinear form $(\ ,\ )$ on $\RV$, as in \S \ref{sss:PELdata}.  We choose an isomorphism of $\BC$-algebras $\RB\otimes_K\BC\simeq \RM_{d\times d}(\BC)$ such that the involution is written as $X\mapsto ^t\!\!\!\bar X$. We may write $\RV\otimes_K\BC=\BC^d\otimes_\BC W$, where the action of $\RB\otimes_K\BC$ is via the first factor. We assume that the form $(\,,\, )$ is given by
 \[
 (Z_1\otimes X_1, Z_2\otimes X_2)=\tr_{\BC/\BR}(^t\!\bar Z_1 Z_2\, h(W_1, W_2)), \quad Z_1, Z_2\in\BC^d, X_1, X_2\in W .
 \]
Here, $h$ is an anti-hermitian form  on $W$ which is given, after choosing an isomorphism $W\simeq \BC^d$, by
 \[
 h(W_1, W_2)=^t\bar W_1 H W_2,
 \]
 with
 \[
 H=\diag(\sqrt{-1}^{(1)},-\sqrt{-1}^{(d-1)}) .
 \] 
This defines the reductive group $\RG$ over $\BQ$  with Deligne cocharacter $h: \Res_{\BC/\BR}\BG_m\to \RG_\BR$, see \cite[\S 6.37]{RZbook}.
 In loc.~cit., $\RG$ is called a  ``fake unitary similitude group". The real group $\RG_\BR$ is a   group of unitary similitudes, 
\[
\RG_\BR={\rm GU}(1,d-1)=\{A\in \RM_{d\times d}(\BC)\ |\ {}^t\bar AHA=c(A)H,\quad c(A)\in \BR^*\}. 
\]

We are now in the situation of  \S \ref{sss:PELdata}. The reflex field of the  Shimura datum $(\RG, X)$ is $\RE=\RK$ if $d>2$ and $\RE=\BQ$ if $d=2$. 

As in \cite[\S 6.38-6.40]{RZbook}, we also choose an odd prime $p$ which splits 
 $pO_\RK=\fkp_1\fkp_2$ in $\RK$ and we  fix an embedding $\ov\BQ\subset \ov\BQ_p$.
 This last choice also picks a prime of $\RK$ over $p$ and we can assume this is $\fkp_1$. We assume that 
 \[
 \RB\otimes_{\BQ}\BQ_p=
 (\RB\otimes_\RK \RK_{\fkp_1})\times  (\RB\otimes_\RK \RK_{\fkp_2})=D\times D^{\rm opp},
 \]
 where $D$ is a division algebra over $\RK_{\fkp_1}=\BQ_p$ with Hasse invariant $1/d$ (then $D^{\rm opp}$ has Hasse invariant $(d-1)/d$).
Then
 \[
 G=\RG\otimes_{\BQ}\BQ_p\simeq D^*\times\BQ_p^*.
 \]
 Next, we choose an order $O_\RB$ which is stable under the involution $*$ and 
 which is maximal at $p$. This situation at $p$ falls under case (I) above with $\mathscr L$
  the (unique) periodic self-dual $O_\RB\otimes\BZ_p=O_{D}\times O_D^{\rm opp}$-lattice multichain determined by the $O_D$-lattice chain 
 given by $\{\Pi^iO_D\}_{i\in \BZ}$. The corresponding stabilizer group $\CG(\BZ_p)={\rm Aut}_{O_B,(\ ,\ )}(\mathscr L)(\BZ_p)$ corresponds to the unique Iwahori subgroup $K_0\simeq O_D^*\times\BZ_p^*$ of $D^*$. 
 
  Let us assume $d\geq 3$ so that $\RE=\RK$.  We are interested in integral models of the Shimura variety for $(\RG, X)$ over the prime $\fkp_1$ of the reflex field $\RE=\RK$. This is a situation in which, for the unique choice of maximal compact level subgroup at $p$ given by $K_0=\CG(\BZ_p)\simeq O_D^*\times\BZ_p^*$, we have $p$-adic uniformization of the corresponding rigid analytic variety by the integral model Drinfeld's $p$-adic symmetric domain for the field $\BQ_p$, $\wh \Omega^d=\wh \Omega^d_{\BQ_p}$. In this case, $K_1\simeq O^*_{D,1}\times \{a\in \BZ_p\ |\ a\equiv 1\mod p\}$, where $O_{D,1}^*$ is the subgroup of elements in $O_D$ with norm $1$.
 
The residue field $k_D=O_D/\Pi O_D$ has degree $d$ over $\BF_p$ and we set  $O=W(k_D)$ as before.
We can easily see   
 \[
 T_\CG=(\Res_{O/\BZ_p}\BG_m)\times\BG_m,\qquad T'_\CG=\Res_{O/\BZ_p}\BG_m,
 \]
 and that conditions (IW) and (S) are satisfied.
The Weyl orbit $\Lambda_{\{\mu\}}$   in $X_*(T_\CG)\simeq \BZ^d\times\BZ$ is given by the $d$ vectors $(1,0,\ldots, 0,1)$, $(0,1,\ldots, 0,1),\ldots, (0,0,\ldots, 1,1)$ and we can see that 
\[
Y\otimes_{\BZ_p} O=Y_{\CG,\mu}\otimes_{\BZ_p} O=\Spec(O[t_1,\ldots, t_d, \frac{t}{t_1\cdots t_d}, \frac{t_1\cdots t_d}{t}]).
\]
A small variation of the argument in \S \ref{ex:ramify} shows that the normalization $\wt Y\otimes_{\BZ_p} O$ of the Lang cover is the spectrum of the quotient of
\[
 O[(t_a)_{a\in \BZ/d\BZ},  (u_a)_{a\in \BZ/d\BZ}, \frac{t}{t_1\cdots t_d}, \frac{t_1\cdots t_d}{t}, 
 \frac{u}{u_1\cdots u_d}, \frac{u_1\cdots u_d}{u}]
\]
by the ideal generated by
\[
u_a^p-t_{a+1}u_{a+1}, \quad a\in \BZ/d\BZ,
\]
\[
(u_1\cdots u_d)^{p-1}-t_1\cdots t_d, 
\]
\[
   \left(\frac{u_1\cdots u_d}{u}\right)^{p-1}-\frac{t_1\cdots t_d}{t}, \qquad  \left(\frac{u}{u_1\cdots u_d}\right)^{p-1}-\frac{t}{t_1\cdots t_d}.
\]
Similarly, we see that $L: \wt Y\otimes_{\BZ_p} O\to Y\otimes_{\BZ_p}O$ is flat.

In this case, the local model $\Mloc_{\CG,\mu}$ has semi-stable reduction: 
Indeed, the base change $\Mloc_{\CG,\mu}\otimes_{\BZ_p}O$ agrees with the base change to $O$ of the semi-stable local model for $\GL_d$, Iwahori level and   coweight $\mu=(1^{(1)}, 0^{(d-1)})$ as in \S \ref{sss:split} above.
We can pick an open 
\[
U=O[x_1, \ldots , x_d]/(x_1\cdots x_d-p)\subset \Mloc_{\CG,\mu}\otimes_{\BZ_p}O
\]
over which the $T_\CG$-torsor $\RP_{\CG,\mu}$ is trivial, and choose the coordinates so that the restriction to $U$ of the divisor morphism  is given by
$U\to Y\otimes_{\BZ_p} O $ defined by the homomorphism
\[
O[t_1,\ldots, t_d, \frac{t}{t_1\cdots t_d}, \frac{t_1\cdots t_d}{t}] \to O[x_1, \ldots , x_d]/(x_1\cdots x_d-p); \quad t_i\mapsto x_i, \ t\mapsto p.
\]

Hence, the chart $U\times_Y \wt Y$ of the root stack $\Mloc_{\CG,\mu}^{\sqrt{}}\otimes_{\BZ_p} O$ is
given by the spectrum of
\[
 O[x_1, \ldots , x_d, u_1,\ldots, u_d, u ]/((u_a^p-x_{a+1}u_{a+1})_{a\in \BZ/d\BZ}, (u_1\cdots u_d)^{p-1}-p, \left(\frac{u}{u_1\cdots u_d}\right)^{p-1}-1).
\]
This decomposes as the direct product
\[
\prod_{\lambda\in \BF_p^*} O[x_1, \ldots , x_d, u_1,\ldots, u_d ]/((u_a^p-x_{a+1}u_{a+1})_{a\in \BZ/d\BZ}, (u_1\cdots u_d)^{p-1}-p).
\]
where the factors correspond to $\lambda\in \BF_p^*$ for which 
\[
\frac{u}{u_1\cdots u_d}=[\lambda],
\]
with $[\lambda]$ the Teichmuller lift.

Using specialization of the cover of \S \ref{ex:ramify} along the divisor $x_1\cdots x_d=p$, we see that the $k^*_D\times \BF_p^*$-cover
\[
U\times_Y \wt Y\to U
\]
obtained by pulling back the Lang cover $L: \wt Y\to Y$ is finite and flat.  Hence, $U\times_Y \wt Y$ is Cohen-Macaulay and then, by Serre's criterion, normal. It follows that $\Mloc_{\CG,\mu}^{\sqrt{}}\to \Mloc_{\CG,\mu}$ is finite and flat and $\Mloc_{\CG,\mu}^{\sqrt{}}$ is normal.

Theorem \ref{thm:PEL} applies to this case and provides an integral model $\CS_{K_1}$ with $\CS_{K_1}\to \CS_{K_0}$ which extends the tame 
Galois cover 
\[
{\rm Sh}_{\RK_1}(\RG, X)\to {\rm Sh}_{\RK_0}(\RG, X)
\]
 with covering group $k^*_D\times \BF_p^*\simeq \BF_{p^d}^*\times\BF_p^*$. 
 
We can see from the proof and the above description of $Y$ and $\wt Y$ that 
$\CS_{K_1,O}$ can be also given directly as the moduli scheme over $\CS_{K_0, O}$ which parametrizes:
\begin{itemize}
\item[1)] A  generator for the Raynaud group scheme
$A^{\rm univ}[\Pi]$, where $A^{\rm univ}$ is the universal $p$-divisible group with $O_D$-action over $\CS_{K_0,O}$.

\item[2)]  A $p-1$-st root of $w_p$.
\end{itemize}

By the above,  $\CS_{K_1}$ is normal and $\CS_{K_1}\to \CS_{K_0}$ is finite and flat.
 \begin{remark}\label{rem:vanH}
The fact that $A^{\rm univ}[\Pi]$ is a Raynaud group scheme is established here as a consequence of the condition (IW) using the flatness of $\CS_{K_0}$ over $O_K$, comp. \S \ref{sss:pf712}. A direct proof is in \cite[Prop. 1.2]{Van}. 
\end{remark}
\begin{remark}\label{rem:nomult}
A small variation of the above, cf. Remark \ref{rem:variant}, gives information for the cover $
{\rm Sh}_{\RK}(\RG, X)\to {\rm Sh}_{\RK_0}(\RG, X)$ where $\RK=K\cdot K^p$ with level subgroup $K$ at $p$ given by $O^*_{D,1}\times \BZ_p^*$.
Then $\CS_{K, O}$ is the moduli scheme which just parametrizes the datum (1) above, i.e. a generator for the Raynaud group scheme
$A^{\rm univ}[\Pi]$. A corresponding local model is given by
\[
 O[x_1, \ldots , x_d, u_1,\ldots, u_d ]/((u_a^p-x_{a+1}u_{a+1})_{a\in \BZ/d\BZ}, (u_1\cdots u_d)^{p-1}-p).
\]

\end{remark}

 \VE

\section{Local Shimura varieties}
In this section, we discuss an analogue of the global set-up for integral local Shimura varieties \cite{SWberkeley}, comp. also \cite{PRint}.

\subsection{ The conjecture}

Consider and fix a local Shimura datum $(G,\{\mu\}, b)$ over $\BQ_p$ as in \cite{SWberkeley}, \cite{PRint}. Let $\CG$ be a quasi-parahoric group scheme for $G$ in the sense of Scholze-Weinstein, cf. \cite[\S 2.2]{PRint}. 
Let $E=E(G,\{\mu\})$ be the reflex field, a finite extension of $\BQ_p$ contained in a fixed algebraic closure $\bar\BQ_p$ of $\BQ_p$, and denote by $O_E$ its ring of integers.

For any  compact open subgroup $K\subset G(\BQ_p)$, Scholze-Weinstein \cite[\S 24.1.3]{SWberkeley} define a  $v$-sheaf   ${\rm Sht}_{K}(G,\mu, b)$ over $\Spd(\br E)$ and prove that this is a \emph{diamond} which is representable by a uniquely determined smooth rigid analytic space  over $\br E$. We denote this rigid-analytic space by the same symbol. It comes with a Weil descent datum down to $E$, cf. \cite[\S 3.1, \S 3.2]{PRg}. 
By definition, ${\rm Sht}_{K}:={\rm Sht}_{K}(G,\mu, b)$  is the \emph{local Shimura variety} for $(G,\{\mu\}, b)$ and level $K$.

In  \cite[Def. 25.1.1]{SWberkeley} a  $v$-sheaf $\CM^{\rm int}_{\CG,\mu, b}$ over $\Spd(O_{\br E})$ is defined. When $\CG$ is parahoric (and not only quasi-parahoric), its   generic fiber 
  is ${\rm Sht}_{\CG(\BZ_p)}(G,\mu, b)$. Scholze conjectures that this $v$-sheaf is representable by a formal scheme over $\Spf(O_{\br E})$ which is normal and flat locally of finite type. We will assume this conjecture and denote by the same symbol $\CM^{\rm int}_{\CG,\mu, b}$ this formal scheme (it is uniquely determined).  The conjecture is known, e.g.,  if $p\neq 2$ and $(G, b, \mu)$ is of abelian type, cf. \cite[Thm. 2.5.4]{PRint}. In particular, if $(G, b, \mu, \CG)$ comes from (local) EL data $(B, V, \{\mu\},b, O_B, \CL)$ or PEL data $(B, V, (\, , \, ), *, \{\mu\}, b,  O_B, \CL)$ in the sense of \cite{RZbook}, then the formal scheme  $\CM^{\rm int}_{\CG,\mu, b}$ is given by the corresponding Rapoport-Zink formal scheme (in its non-naive version, obtained by flat closure from its naive version). When $\CG$ is parahoric, 
  the rigid generic fiber of $\CM^{\rm int}_{\CG,\mu, b}$ is ${\rm Sht}_{\CG(\BZ_p)}(G,\mu, b)$.
  
  It is also conjectured in general that there is a local model diagram. In order to simplify the notation, we denote by the same symbol $\Mloc_{\CG,\mu}$ the $p$-adic completion $\wh\Mloc_{\CG,\mu}$ of the local model, and by $[\CG\bs  \Mloc_{\CG,\mu}]$ the $p$-adic completion  $[\wh \CG\bs \wh \Mloc_{\CG,\mu}]$ of the stack quotient. Then  the local model map is a formally smooth map of formal algebraic stacks,
\begin{equation}
\CM^{\rm int}_{\CG,\mu, b}\xrightarrow{\ \varphi\ } [\CG\bs \Mloc_{\CG,\mu}].
\end{equation}
The existence of the local model map is known at least in the (P)EL-case, cf. \cite{RZbook}.  Let us assume its existence and, in addition, that $(\CG,\{\mu\})$ is strictly convex, see \S \ref{ss:nondegeneracy}.  
We also assume the validity of the divisor conjecture (Conjecture \ref{divconj}).

Now all the ingredients used in the global case are in place in the local case, and we can formulate the local analogue of Conjecture \ref{globconj}. We assume that $\CG$ is a  parahoric,  and set $K_0=\CG(\BZ_p)$. 
Let $K_1$ be  the kernel  of the composition of  homomorphisms
\[
 \CG(\BZ_p)\to \ov\CG(\BF_p)\to \ov\CG_{\red, \rm ab}(\BF_p)
\]
 and, as in \S \ref{ss:211}, set $T_\CG$ for the torus over $\BZ_p$ which lifts the torus $\ov\CG_{\red, \rm ab}$ over $\BF_p$.

\begin{conjecture}\label{localconj} Choose a semigroup $S\subset S_{\CG,\mu}\subset X^*(T_\CG)$ as in \S \ref{ss:generalsemigroups}.  
 There exists a formal scheme $\CM^{\rm int}_{K_1, S}$ over $\Spf(O_{\br E})$ with rigid analytic generic fiber  the local Shimura variety 
 ${\rm Sht}_{K_1}:={\rm Sht}_{K_1}(G,\mu, b)$ which has the following properties:
 
 \begin{itemize}
\item[i)] There is a morphism 
\[
\CM^{\rm int}_{K_1, S}\to \mathcal \CM^{\rm int}_{K_0}
\] 
which extends $\pi: {\rm Sht}_{K_1} \to {\rm Sht}_{K_0} $ in the generic fiber.
 
 \item[ii)] The action of $T_\CG(\BF_p)$ on ${\rm Sht}_{K_1}$ extends to $\CM^{\rm int}_{K_1, S}$ and the morphism 
 $\CM^{\rm int}_{K_1, S}\to \mathcal \CM^{\rm int}_{K_0}$ of (i)
  identifies $\CM^{\rm int}_{K_0}$ with the formal scheme quotient $T_\CG(\BF_p)\bs \CM^{\rm int}_{K_1, S}$.

  \item[iii)] There is a morphism $\varphi_{1, S}: \CM^{\rm int}_{K_1, S}\to [\CG\bs \Mloc^{\sqrt{S}}_{\CG,\mu}]$ which fits  into a $2$-commutative diagram
  \begin{equation}\label{CDlocalconj}
\begin{aligned}
 \xymatrix{
     \CM^{\rm int}_{K_1, S}\ \ar[r]^{\varphi_{1,S}} \ar[d]  &  [\CG\bs \Mloc^{\sqrt{S}}_{\CG,\mu}] \ar[d] \\
       \CM^{\rm int}_{K_0}\ar[r]^{\varphi}  & [\CG\bs \Mloc_{\CG,\mu}],
        }
        \end{aligned}
\end{equation}
 and which induces
   an isomorphism of formal stacks 
 \[
 [T_\CG(\BF_p)\bs \CM^{\rm int}_{K_1, S}]\xrightarrow{\sim}  \CM^{\rm int}_{K_0}\times_{[\CG\bs \Mloc_{\CG,\mu}]}[\CG\bs \Mloc^{\sqrt{S}}_{\CG,\mu}].
 \]
 \end{itemize}
 \end{conjecture}

 \begin{remark}
 \begin{altenumerate}
 \item As in the global case, there is a variant of this conjecture 
   for  covers 
  \[
 \pi: {\rm Sht}_{K} \to {\rm Sht}_{K_0} 
  \]
  for intermediate  subgroups 
  $
  K_1\subset K\subset K_0 ,
 $
  where $K$ is the inverse image of $\ov Q(\BF_p)$ under the map $K_0=\CG(\BZ_p)\to \ov\CG_{\red, \rm ab}(\BF_p)$, for some subtorus $\ov Q \subset \ov\CG_{\red, \rm ab}$, cf. Remark \ref{rem:variant}. 
  \item The Weil descent datum on $\CM^{\rm int}_{K_0}$ should lift to a Weil descent datum on $\CM^{\rm int}_{K_1, S}$.
  
  \item The theory of non-archimedean uniformization (Rapoport-Zink uniformization) should give a compatibility between the global Conjecture \ref{globconj} and the local Conjecture \ref{localconj}. 
 \end{altenumerate}
  \end{remark} 

\subsection{Examples}\label{ss:localex}
There is a local analogue of Theorem \ref{thm:PEL} in the (P)EL case. Rather than formulating this precisely, we end by discussing  three examples. In these examples,  the parahoric is an Iwahori and the reflex field is $\BQ_p$. It turns out that in all these examples, except the last one, $\CM^{\rm int}_{K_1, S_\mu}$ can be obtained from $\CM^{\rm int}_{K_0}$  by adding a moduli datum explicitly in terms of Oort--Tate--Raynaud theory. These three examples are essentially the only ones of (P)EL type when the formal scheme $\CM^{\rm int}_{\CG,\mu,b}$ (with its Weil descent datum) are explicitly known. 

\subsubsection{The Lubin-Tate case}
This is the EL case for $G=\GL_n$, and $ \{\mu\}=(1, 0, \ldots, 0)$, and where $b$ is a representative of the unique basic element of $B(G, \mu^{-1})$ and  $\CG$ is the standard Iwahori.  It is the RZ space of complete periodic chains of isogenies of degree $p$ of formal $p$-divisible groups of height $n$ and dimension $1$ over schemes $S$ over $\Spf(\br\BZ_p)$,
\[
X_0\to X_1\to\cdots\to X_{n-1}\to X_0 ,
\]
together with a framing (quasi-isogeny) $\rho\colon  \bar X_0\to \BX\otimes_{\Spec(\bar\BF_p)}\bar S$. Here $\BX$ denotes a one-dimensional isoclinic formal group of height $n$ over $\bar\BF_p$. Such a framing object $\BX$ is unique up to quasi-isogeny.

In this case, $\CM^{\rm int}_{\CG,\mu,b}$  is isomorphic  to 
\begin{equation}
\Spf(\breve\BZ_p[[t_1,\ldots,t_n]]/(t_1t_2\cdots t_n-p))\times\BZ,
\end{equation}
with its descent datum down to $\Spf(\BZ_p)$ given by the product of the natural descent datum on the first factor and the translation by $1$ on the second factor, cf. \cite[3.78]{RZbook}, \cite[I.1.4]{Farguesi}.

In this case $Y_{\CG, \mu}$ is smooth. The covering $\CM^{\rm int}_{K_1, S_\mu}$ represents the functor over $\CM^{\rm int}_{K_0}$, where a Oort--Tate generator of the group scheme $G_i=\ker(X_i\to X_{i+1})$ 
is added to the moduli data, for $i=0,\ldots, n-1$, comp. \S \ref{sss:split}. The moduli scheme $\CM^{\rm int}_{K_1, S_\mu}$ is isomorphic to 
\begin{equation}
\Spf(\breve\BZ_p[[u_1,\ldots,u_n]]/(u^{p-1}_1u^{p-1}_2\cdots u^{p-1}_n-p))\times\BZ,
\end{equation}
and the map $\CM^{\rm int}_{K_1, S_\mu}\to \CM^{\rm int}_{K_0}$ is given by $t_i\mapsto u^{p-1}_i$, comp. \eqref{dr-eq}. Hence $\CM^{\rm int}_{K_1, S_\mu}$ is regular and the map $\CM^{\rm int}_{K_1, S_\mu}\to \CM^{\rm int}_{K_0}$ is finite flat. 

\subsubsection{The Drinfeld case}
This is the EL case for $G=D^*$ (the multiplicative group of the central division algebra $D$ over $\BQ_p$ with invariant $1/n$), where $\{\mu\}=(1, 0, \ldots, 0)$,  and where $b$ is a representative of the unique  element of $B(G, \mu^{-1})$ (which is basic), and  $\CG$ is the unique  Iwahori  (then $\CG(\BZ_p)=O_D^*$). It is the RZ space of \emph{special formal} $O_D$-modules $(X, \iota:O_D\to \End(X))$  over schemes $S$ over $\Spf(\br\BZ_p)$, 
together with a framing  $\rho\colon  \bar X\to \BX\otimes_{\Spec(\bar\BF_p)}\bar S$. Here $\BX$ denotes a special formal $O_D$-module  over
 $\bar\BF_p$. Such a framing object $\BX$ is unique up to quasi-isogeny.

In this case, $\CM^{\rm int}_{\CG,\mu,b}$  is isomorphic  to 
\begin{equation}
\big(\wh\Omega^n_{\BQ_p}\times_{\Spf(\BZ_p)}\Spf(\breve\BZ_p)\big)\times\BZ, 
\end{equation}
with its descent datum down to $\Spf(\BZ_p)$ given by the product of the natural descent datum on the first factor and the translation by $1$ on the second factor. Here $\wh\Omega^n_{\BQ_p}$ denotes formal Drinfeld space of dimension $n$ relative to the local field $\BQ_p$, cf. \cite[Thm. 3.72]{RZbook}.

Again $Y_{\CG, \mu}$ is smooth. 
The covering $\CM^{\rm int}_{K_1, S_\mu}$ represents the functor over $\CM^{\rm int}_{K_0}$, where a generator of the Raynaud group scheme $G=\ker(\iota(\Pi): X\to X)$ 
is added to the moduli data, comp. the passage before Remark \ref{rem:vanH}. The moduli scheme $\CM^{\rm int}_{K_1, S_\mu}$ is normal and the map $\CM^{\rm int}_{K_1, S_\mu}\to \CM^{\rm int}_{K_0}$ is finite flat. A local description is given in Remark \ref{rem:nomult}. 
\begin{remark}
By the compatibility of the local model with unramified base change, the local models in the Lubin-Tate case and in the Drinfeld case are isomorphic after base change to $\Spec(\br\BZ_p)$. Hence the formal schemes $\CM^{\rm int}_{\CG,\mu,b}$ are \'etale-locally isomorphic. This  ceases to be true for the coverings (in the Lubin-Tate case, $\CM^{\rm int}_{K_1, S_\mu}$ is regular, which is not true in the Drinfeld case), comp. Remark \ref{rem:nomult}.
\end{remark}
\subsubsection{The case of a non-split binary  group of unitary similitudes}
Let $p\neq 2$ and let $F/\BQ_p$ be a quadratic extension. Let $V$ be a $F/\BQ_p$-hermitian space of dimension $2$ which is anisotropic. Let $G=\GU(V)$. Then $G\otimes_{\BQ_p} F\simeq \GL_{2, F}\times \BG_{m, F}$ and we take  $\{\mu\}=(1,0;1)$. Also, we let $b$ be a representative of the unique  element of $B(G, \mu^{-1})$ (which is basic), and  let $\CG$ be the unique  Iwahori. 

 If $F/\BQ_p$ is unramified, we have $\CG(\BZ_p)=c^{-1}(\BZ_p^*)$, where $c\colon G(\BQ_p)\to \BQ_p^*$ is the multiplier map; it is the unique maximal compact subgroup of $G(\BQ_p)$. If $F/\BQ_p$ is ramified, the inclusion  $\CG(\BZ_p)\subset c^{-1}(\BZ_p^*)$ has index $2$. More precisely, in both cases there exists a unique quasi-parahoric $\wt\CG$ such that $\wt\CG(\BZ_p)=c^{-1}(\BZ_p^*)$; then  $\CG=\wt\CG$ when $F/\BQ_p$ is unramified and $\CG\subset\wt\CG$ has index $2$ when $F/\BQ_p$ is ramified. In fact $\wt\CG=\Aut_{O_F, (\, , \, )}(\CL)$, where $\CL$ is the unique periodic selfdual lattice chain in $V$. Let us distinguish the case when $F/\BQ_p$ is unramified from the case when $F/\BQ_p$ is ramified.

 \smallskip

{\it Case: $F/\BQ_p$ unramified.} Then $\CM^{\rm int}_{\CG,\mu,b}$  is the RZ space of triples $(X, \iota, \lambda)$, where $X$ is a $p$-divisible group of dimension $2$ and height $4$, and $\iota:O_F\to\End(X)$ satisfies the Kottwitz condition
\begin{equation}\label{Kou}
\det(T\cdot I-\iota(a)\mid\Lie X)=T^2-{\rm Tr}_{F/\BQ_p}(a)\cdot T+\Nm_{F/\BQ_p}(a) , \quad a\in O_F,
\end{equation}
 and $\lambda\colon X\to X^\vee$ is a polarization  compatible with $\iota$ such that $\ker(\lambda)$ is a $\BF_{p^2}$-Raynaud group scheme contained in $X[p]$. Here we have identified $O_F\otimes\BF_p$ with $\BF_{p^2}$. In addition, a framing $\rho\colon  \bar X\to \BX\otimes_{\Spec(\bar\BF_p)}\bar S$ is given. Such a framing object $\BX$ is unique up to quasi-isogeny. 

 In this case, 
$\CM^{\rm int}_{\CG,\mu,b}$ is isomorphic to 
\begin{equation}
\big(\wh\Omega^2_{\BQ_p}\times_{\Spf(\BZ_p)}\Spf(\br\BZ_p)\big)\times\BZ 
\end{equation}
with its descent datum down to $\Spf (\BZ_p)$ given by the product of the natural descent datum on the first factor and the translation by $1$ on the second factor, cf. \cite[Thm. 1.2]{KRf} (for the descent, use that the framing object is a Drinfeld s.f. $O_D$-module).

In this case $T_\CG=\BG_m^2$ and $Y_{\CG, \mu}$ is smooth. The covering $\CM^{\rm int}_{K_1, S_\mu}$ represents the functor over $\CM^{\rm int}_{K_0}$, where a Raynaud generator of  $\ker(\lambda)$ is  
added to the moduli data. The moduli scheme $\CM^{\rm int}_{K_1, S_\mu}$ is normal and the map $\CM^{\rm int}_{K_1, S_\mu}\to \CM^{\rm int}_{K_0}$ is finite flat. 

\smallskip

{\it Case: $F/\BQ_p$ ramified.} In this case $\CM^{\rm int}_{\wt\CG,\mu,b}$  is the RZ space of triples $(X, \iota, \lambda)$, where $X$ is a $p$-divisible group of dimension $2$ and height $4$, and $\iota:O_F\to\End(X)$ satisfies the Kottwitz condition \eqref{Kou},
 and $\lambda\colon X\to X^\vee$ is a principal polarization   compatible with $\iota$. In addition, a framing $\rho\colon  \bar X\to \BX\otimes_{\Spec(\bar\BF_p)}\bar S$ is given. Here the potential framing object is not unique; we choose $\BX$ such that $\inv(\BX)=-1$, cf. \cite[\S 5, case c)]{KRf}.
 
    In this case, 
$\CM^{\rm int}_{\wt\CG,\mu,b}$ is isomorphic to 
\begin{equation}
\big(\wh\Omega^2_{\BQ_p}\times_{\Spf(\BZ_p)}\Spf(\br\BZ_p)\big)\times\BZ 
\end{equation}
with its descent datum down to $\Spf (\BZ_p)$ given by the product of the natural descent datum on the first factor and the translation by $1$ on the second factor, cf. \cite[Thm. 1.2]{KRf}.

Passing from $\wt\CG$ to $\CG$ is problematic and we have no explicit description of $\CM^{\rm int}_{\CG,\mu,b}$. In this case, the condition (IW) is violated. We cannot prove Conjecture \ref{localconj} in this case.


\bigskip

 

 \begin{thebibliography}{aaaaaaaa}
 
 \bibitem[Al]{Alt} K.~Altmann, \emph{Flatness for semigroups.} In preparation.  
 
 \bibitem[An]{Ana} S. Anantharaman,  
\emph{Sch\'emas en groupes, espaces homog\`enes et espaces alg\'ebriques sur une base de dimension 1.} pp. 5–79,
Supplément au Bull. Soc. Math. France, Tome 101,
SMF, Paris, 1973
  
 \bibitem[AGLR]{AGLR} J.~Ansch\"utz, I.~Gleason,  J.~N.~P.~Louren\c co, T.~Richarz, \textit{On the $p$-adic theory of local models}.  arXiv:2201.01234  
 
 \bibitem[AB]{AB} M.~Atiyah, R.~Bott, \textit{The Yang-Mills equations over Riemann surfaces}. Phil. Trans. R. Soc. Lond. A 308 (1982), 523--615.
 


\bibitem[BLR]{BLR} 
S. Bosch,   W.  Lütkebohmert, M. Raynaud, Michel, \emph{N\'eron models}.
Ergeb. Math. Grenzgeb. (3), 21 
Springer-Verlag, Berlin, 1990, x+ 325 pp.

\bibitem[Bou]{Bou} N. Bourbaki, \emph{Lie groups and Lie algebras. Chapters 4--6}. Elements of Mathematics. Springer-Verlag,
Berlin, 2002. xii+300 pp.

    \bibitem[BL]{BL} T.~Braden, V.~Lunts, \textit{Equivariant-constructible Koszul duality for dual toric varieties.}
Adv. Math. 201(2006), no.2, 408--453.
 
  \bibitem[BP]{BP} M. Brion, P. Polo, \textit{Generic singularities of certain Schubert varieties}. Math. Z. 231 (1999), 301--324.
 
\bibitem[BTII]{BTII} F.~Bruhat, J.~Tits, \emph{Groupes r\'eductifs sur un corps local. II. Schémas en groupes. Existence d'une donn\'ee radicielle valuée.} Publ. Math. IHES 60 (1984), 197--376.

\bibitem[Ca]{Cadman}  C.~Cadman, \textit{Using stacks to impose tangency conditions on curves}. American Journal of Math. 
 129, 2, (2007),  405--427.
 
 
 
 \bibitem[Co]{Conrad} B.~Conrad, \emph{Reductive group schemes.}
  in Autour des schémas en groupes. Vol. I, Panoramas \& Synthèses, vol. 42/43, Soci\'et\'e Math\'ematique de France, Paris, 2014, p. 93--444. 
 
 \bibitem[CLS]{CLS} D. Cox, J. Little, H. Schenck,  
\emph{Toric varieties.}
Grad. Stud. Math., 124
American Mathematical Society, Providence, RI, 2011. xxiv+841 pp.

\bibitem[Da]{Dabrowski} R. Dabrowski, \emph{On normality of the closure of a generic torus orbit in $G/P$.}
Pacific J. Math. 172 (1996), no. 2, 321--330.


\bibitem[Del79]{DeligneCorvallis}
P.~Deligne, \emph{Vari\'et\'es de {S}himura: interpr\'etation modulaire, et
  techniques de construction de mod\`eles canoniques}. Automorphic forms,
  representations and {$L$}-functions, {P}art 2, Proc. Sympos.
  Pure Math., XXXIII, Amer. Math. Soc., Providence, R.I., 1979, 247--289.
 
  
 
 

\bibitem[DR73]{DeligneRa}
P.~Deligne, M.~Rapoport, \emph{Les sch\'emas de modules de courbes
  elliptiques}. Modular functions of one variable, {II} ({P}roc. {I}nternat.
  {S}ummer {S}chool, {U}niv. {A}ntwerp, {A}ntwerp, 1972), Springer, Berlin,
  1973, pp.~143--316. Lecture Notes in Math., Vol. 349.  
  
 

 \bibitem[Ed]{Edix} B. Edixhoven, \textit{ N\'eron models and tame ramification.} Compositio Math. 81 (1992), no. 3, 291--306.

\bibitem[FHLR]{FHLR} N. Fakhruddin,  T. Haines,  J. Louren\c co,   T. Richarz,  
\emph{Singularities of local models.}
Math. Ann. 391 (2025), no. 4, 6205--6250.

 \bibitem[Fa]{Farguesi} L.~Fargues, \emph{ L'isomorphisme entre les tours de Lubin-Tate et de Drinfeld et applications cohomologiques}.
Progr. Math., 262
Birkh\"auser Verlag, Basel, 2008, 1--325.
  
 \bibitem[GT]{GenTi} A. Genestier, J. Tilouine, \emph{Syst\`emes de Taylor--Wiles pour  ${\rm GSp}_4$.} 
Ast\'erisque No. 302 (2005), 177--290.


\bibitem[GL]{GL}  I. Gleason, J. Louren\c co, \emph{Tubular neighborhoods of local models.}
Duke Math. J. 173 (2024), no. 4, 723--743.
 

\bibitem[Go]{Goertz} U.~G\"ortz, \emph{On the flatness of models of certain Shimura varieties of PEL-type.} Math. Ann. 321 (2001), no. 3, p. 689--727.  



\bibitem[Hai]{Hai} T.~Haines, \emph{Dualities for root systems with automorphisms and applications to non-split groups.} Representation
Theory of the American Mathematical Society, 22 (2018), 1--26. 

 \bibitem[HLS]{HLS} T.~Haines, Q.~Li, B.~Stroh, In preparation.
 
 \bibitem[HR]{HR} T.~Haines, M.~Rapoport, \textit{Shimura varieties with $\Gamma_1(p)$-level via Hecke algebra isomorphisms: the Drinfeld case.}
Ann. Sci. Ec. Norm. Sup. (4) 45 (2012), no. 5, 719--785.

\bibitem[HRi]{HRi} T.~Haines, T.~Richarz, \textit{Normality and Cohen-Macaulayness of parahoric local models.}
J. Eur. Math. Soc. (JEMS) 25 (2023), no. 2, 703--729.
 
 \bibitem[HS]{HS} T.~Haines, B.~Stroh, Unpublished notes.
 
 \bibitem[HT]{HT} M.~Harris,  R.~Taylor, \emph{Regular models of certain Shimura varieties.} Asian J. Math. 6 (2002), 61--94.
 
 \bibitem[HZ]{HZ} Q.~He, R.~Zhou,  \textit{On the basic locus of GSpin Shimura varieties with vertex stabilizer level}. arXiv:2505.08911

\bibitem[HPR]{HPR} X.~He, G.~Pappas, M.~Rapoport, \textit{Good and semi-stable reductions of Shimura varieties.}
J. \'Ec. polytech. Math. 7 (2020), 497--571.

\bibitem[HP]{HP} B.~Howard, G.~Pappas, {\sl Rapoport-Zink spaces for spinor groups.}
  Compositio Math. 153 (2017), no. 5, 1050--1118.
        
\bibitem[Jac]{Jac} R.~Jacobowitz, \textit{Hermitian forms over local fields.}
 American J.  Math. 84 (1962), 441--465.
 


      \bibitem[KaP]{KalethaPrasad} T.~Kaletha, G.~Prasad,  \emph{Bruhat-Tits theory--a new approach.}
New Math. Monogr., 44
Cambridge University Press, Cambridge, 2023, xxx+718 pp.

        
 \bibitem[KM]{KatzMazur}    N.~M.~ Katz, B.~Mazur,   \textit{Arithmetic moduli of elliptic curves.}
Ann. of Math. Stud., 108
Princeton University Press, Princeton, NJ, 1985, xiv+514 pp.

 \bibitem[TEI]{TEI}  G.~Kempf, F.~Knudsen, D.~Mumford, B.~Saint-Donat, \emph{Toroidal embeddings I.}
Lecture Notes in Math., Vol. 339
Springer-Verlag, Berlin-New York, 1973, viii+209 pp.


 \bibitem[KP]{KP} M.~Kisin, G.~Pappas, \emph{Integral models of Shimura varieties with parahoric level structure.}  Publ. Math. IHES 128 (2018), p. 121--218. 
 
 \bibitem[KPZ]{KPZ} M.~Kisin, G.~Pappas, R.~Zhou, {\sl Integral models of Shimura varieties with parahoric level structure, II.} 
Forum of Math., Pi, to appear. arXiv:2409.03689 
 
\bibitem[KnM]{KnM} F.~Knudsen, D.~Mumford, \textit{The projectivity of the moduli space of stable curves. I. Preliminaries on ``det'' and ``Div''.}
Math. Scand. 39 (1976), no. 1, 19--55.


 \bibitem[Ko]{KoIsoI} R.~Kottwitz, \emph{Isocrystals with additional structure.}
Compositio Math. 56 (1985), no. 2, 201--220.

 \bibitem[KoW]{KoWa} R.~Kottwitz, P.~Wake, \emph{Primitive elements for $p$-divisible groups.} 
 Res. Number Theory 3 (2017), Paper No. 20, 11 pp.

 \bibitem[KRf]{KRf} S.~Kudla, M.~Rapoport, \emph{ An alternative description of the Drinfeld $p$-adic half-plane.} Ann. de l'Inst. Fourier, 64  (2014), 1203--1228.

 \bibitem[Le]{Levin} B.~Levin, \textit{Local models for Weil-restricted groups.}
Compos. Math. 152 (2016), no. 12, 2563--2601.

\bibitem[Lim]{Lim}D.~G.~Lim, \textit{Nonemptiness of single affine Deligne-Lusztig varieties}. arXiv: 2302.04976 

\bibitem[Liu]{SLiu} S.~Liu,  
\textit{Modèle local des schémas de Hilbert-Siegel de niveau $\Gamma_1(p)$. }
Algebra Number Theory 15 (2021), no.7, 1655–1698.

\bibitem[Lo]{Lourenco} J. Louren\c co, \emph{Grassmanniennes affines tordues 
sur les entiers,}
Forum Math. Sigma 11 (2023), Paper No. e12, 65 pp.

\bibitem[Mar]{Marazza} G.~Marazza, 
\emph{On the geometry of integral models of Shimura varieties with $\Gamma_1(p)$-level structure.} arXiv:2509.21198
 
\bibitem[Max]{Maxwell} G.~Maxwell, \emph{Wythoff's construction for Coxeter groups.}
J. Algebra 123 (1989), no. 2, 351--377.

\bibitem[MO]{MO}  \textit{https://mathoverflow.net/questions/424489/pointless-groups-iii}

\bibitem[MoR]{MoR} P.-L.~Montagard,  A.~Rittatore, \emph{Fano generic torus orbit closures in  $G/P$.}
J. Algebraic Combin. 57 (2023), no. 2, 439--460.

\bibitem[O]{Ol} M. Olsson, \emph{Logarithmic geometry and algebraic stacks.}
Ann. Sci. \'Ecole Norm. Sup. (4) 36 (2003), no. 5, 747--791.

\bibitem[OR]{OR} S.~Orlik, M.~Rapoport,
\textit{Deligne-Lusztig varieties and period domains over finite fields.}
J. Algebra 320 (2008), no. 3, 1220--1234.

\bibitem[P]{P} L.~Pan, \emph{First covering of the Drinfeld upper half-plane and Banach representations of $\GL_2(\BQ_p)$.}
Algebra Number Theory 11 (2017), no. 2, 405--503.

 \bibitem[Pa]{P95} G.~Pappas, \emph{Arithmetic models for Hilbert modular varieties.}
Compositio Math. 98 (1995), no. 1, 43--76.

\bibitem[PRtw]{PRtw} G.~Pappas, M.~Rapoport, \textit{ Twisted loop groups and their affine flag varieties.}
Adv. Math. 219 (2008), 118--188.
 \bibitem[PR]{PRLM3} G.~Pappas, M.~Rapoport, \textit{Local models, III. Unitary groups.} Journal of the Institute of Mathematics of Jussieu, 8 (3) (2009), 507--564.
 
  \bibitem[PRg]{PRg} G.~Pappas, M.~Rapoport, \textit{$p$-adic shtukas and the theory of global and local Shimura varieties.} Cambridge Journal of Mathematics 12 (2024), 1--164.
  
 \bibitem[PRl]{PRint} G.~Pappas, M.~Rapoport, \textit{On integral local Shimura varieties.} Journal of the Institute of Mathematics of Jussieu, 25 (1) (2026), 375--443.
 
  \bibitem[PRS]{PRS} G.~Pappas, M.~Rapoport, B. Smithling, \emph{Local models of Shimura varieties, I. Geometry and Combinatorics.} Handbook of moduli. Vol. III, 135--217.
Adv. Lect. Math. (ALM), 26, International Press, Somerville, MA, 2013
 
   \bibitem[PZ]{PZ} G.~Pappas, X.~Zhu, \emph{Local models of Shimura varieties and a conjecture of Kottwitz.}
Invent. Math. 194 (2013), no. 1, 147--254.

\bibitem[PY]{PrasadYu} G.~Prasad, J.-K.~Yu, \emph{On quasi-reductive group schemes}.
J. Algebraic Geom. 15 (2006), no. 3, 507--549.

 \bibitem[RZ]{RZbook} M.~Rapoport, T.~Zink,  {Period spaces for $p$--divisible groups}. Ann.\ of Math. Studies {141}, Princeton University Press, Princeton, NJ, 1996. 
 
 \bibitem[Ra]{Ray} M.~Raynaud, \emph{Sch\'emas en groupes de type $(p,\ldots, p)$.} Bull. Soc. Math. France 102 (1974), 241--280.

\bibitem[Re]{Renner} L.~Renner, \emph{Descent systems for Bruhat posets}. Bulletin de S.M.F, 102 (1974), p. 241-280.
J. Algebraic Combin. 29 (2009), no. 4, 413--435.



\bibitem[SW]{SWberkeley} P.~Scholze, J.~Weinstein, \emph{Berkeley lectures on $p$-adic geometry}. Ann. of Math. Studies, 207, Princeton University Press, Princeton, 2020. 

\bibitem[Sha]{Shadrach} R.~Shadrach, \emph{Integral models of certain PEL Shimura varieties with $\Gamma_1(p)$-type level structure.}  Manuscripta Math. 151:1--2 (2016), 49--86. 

\bibitem[Slo]{Slodowy}  P. Slodowy,
\emph{On the geometry of Schubert varieties attached to Kac-Moody Lie algebras.} Proceedings of the 1984 Vancouver conference in algebraic geometry, 405--442.
CMS Conf. Proc., 6
Published by the American Mathematical Society, Providence, RI;  1986

\bibitem[Sm]{Smithling} B.~Smithling,  \emph{Topological flatness of orthogonal local models in the split, even case. I}
Math. Ann. 350 (2011), no. 2, 381--416.

\bibitem[Sta]{Stacks} The Stacks project authors,
 \emph{The Stacks project}.
 {\url{https://stacks.math.columbia.edu}},
  year         = {2025}
  
 \bibitem[TO]{OT} J.~Tate, F.~Oort, \emph{Group schemes of prime order.} Ann. Sci. \'Ecole Norm. Sup. 3 (1970), 1--21.
 
 \bibitem[Ta]{Tak} Y. Takaya,  \emph{On depth-zero integral models of local Shimura varieties.} arXiv:2505.10000

 
 \bibitem[Te]{T} J.~Teitelbaum, \emph{Geometry of an \'etale covering of the  $p$-adic upper half plane.} Ann. Inst. Fourier (Grenoble) 40 (1990), no. 1, 68--78.
   
  \bibitem[Van]{Van} A.~Vanhaecke, \emph{Le crystal de Dieudonn\'e des sch\'emas en $\BF$-vectoriels.} Bull. Soc. Math. France
148 (3), 2020, p. 439--465.

\bibitem[W]{W} H.~Wang, \emph{L'espace sym\'etrique de Drinfeld et correspondance de Langlands locale I.}
Math. Z. 278 (2014),  829--857.


   \bibitem[Yu]{Yu} Q.~Yu, \emph{On the face map of the admissible set with Iwahori level.} arXiv:2605.15657
      
   \bibitem[Z1]{ZhuCoh} X.~Zhu, \emph{On the coherence conjecture of Pappas and Rapoport.}
Ann. of Math. (2) 180 (2014), no. 1, 1--85.

   \bibitem[Z2]{ZhuIntro} X.~Zhu, \emph{An introduction to affine Grassmannians and the geometric Satake equivalence.}
IAS/Park City Math. Ser., 24
American Mathematical Society, Providence, RI, 2017, 59–154.
   
   
   \end{thebibliography}
  \end{document}